\documentclass[11pt]{amsart}
\usepackage{a4wide}
\usepackage[utf8]{inputenc}
\usepackage{amsmath}
\usepackage{amsfonts}
 \usepackage{xcolor} 
\usepackage{amssymb}
\usepackage{amsthm}
\usepackage{subfigure}
\usepackage[mathscr]{eucal}
\usepackage{dsfont}
\usepackage{pstricks}
\usepackage{graphicx}

\usepackage[colorlinks=true, linkcolor=purple, citecolor=purple]{hyperref}

\newtheorem{theo}{Theorem}[section]
\newtheorem{prop}[theo]{Proposition}
\newtheorem{coro}[theo]{Corollary}
\newtheorem{lemm}[theo]{Lemma}

\theoremstyle{definition}
\newtheorem{def1}[theo]{Definition}
\theoremstyle{remark}
\newtheorem{rema}[theo]{Remark}

\newcommand{\nwc}{\newcommand}
\nwc{\eps}{\epsilon}
\nwc{\ep}{\epsilon}
\nwc{\vareps}{\varepsilon}
\nwc{\Oph}{\operatorname{Op}_\hbar}
\nwc{\la}{\langle}
\nwc{\ra}{\rangle}

\nwc{\mf}{\mathbf} %Latex (as in \bf not tilted math letters)
\nwc{\blds}{\boldsymbol} %Latex
\nwc{\ml}{\mathcal} %Latex

\nwc{\defeq}{\stackrel{\rm{def}}{=}}

\nwc{\cE}{\ml{E}}
\nwc{\cN}{\ml{N}}
\nwc{\cO}{\ml{O}}
\nwc{\cP}{\ml{P}}
\nwc{\cU}{\ml{U}}
\nwc{\cV}{\ml{V}}
\nwc{\cW}{\ml{W}}
\nwc{\tU}{\widetilde{U}}
\nwc{\IN}{\mathbb{N}}
\nwc{\IR}{\mathbb{R}}
\nwc{\R}{\mathbb{R}}
\nwc{\IZ}{\mathbb{Z}}
\nwc{\Z}{\mathbb{Z}}
\nwc{\N}{\mathbb{N}}
\nwc{\IC}{\mathbb{C}}
\nwc{\C}{\mathbb{C}}
\nwc{\IT}{\mathbb{T}}
\nwc{\T}{\mathbb{T}}
\nwc{\IS}{\mathbb{S}}
\nwc{\tP}{\widetilde{P}}
\nwc{\tPi}{\widetilde{\Pi}}
\nwc{\tV}{\widetilde{V}}
\nwc{\supp}{\operatorname{supp}}
\nwc{\rest}{\restriction}
\let \d \relax
\nwc{\d}{\partial}
\nwc{\Cor}{\mathscr{C}}
\nwc{\CLam}{\overline{\C}_+^\Lambda}

\definecolor{ao(english)}{rgb}{0.0, 0.5, 0.0}
\newcommand{\vio}{\color{violet}}

\nwc{\todo}[1]{$\clubsuit$ {\tt #1}}

\DeclareMathOperator{\Vol}{Vol}

\DeclareMathOperator{\Res}{Res}
\DeclareMathOperator{\Sp}{Sp}
\DeclareMathOperator{\Id}{Id}
\DeclareMathOperator{\WF}{WF}
\DeclareMathOperator{\dR}{dR}
\DeclareMathOperator{\Int}{Int}

\renewcommand{\Im}{\operatorname{Im}}
\renewcommand{\Re}{\operatorname{Re}}

\numberwithin{equation}{section}

\begin{document}

\title[Length orthospectrum of convex bodies on flat tori]{Length orthospectrum of convex bodies on flat tori}

\author{Nguyen Viet Dang}

\address{Sorbonne Université and Université Paris Cité, CNRS, IMJ-PRG, F-75005 Paris, France.}

\address{Institut Universitaire de France, Paris, France}

\email{dang@imj-prg.fr}

\author{Matthieu L\'eautaud}

\address{Laboratoire de Math\'ematiques d'Orsay, UMR 8628, Universit\'e Paris-Saclay, CNRS, B\^atiment 307, 91405 Orsay Cedex France} 

\email{matthieu.leautaud@universite-paris-saclay.fr}

\author{Gabriel Rivi\`ere}

\address{Laboratoire de Math\'ematiques Jean Leray, Nantes Universit\'e, UMR CNRS 6629, 2 rue de la Houssini\`ere, 44322 Nantes Cedex 03, France}

\address{Institut Universitaire de France, Paris, France}

\email{gabriel.riviere@univ-nantes.fr}

%\date

\begin{abstract}
In analogy with the study of Pollicott-Ruelle resonances on negatively curved manifolds, we define anisotropic Sobolev spaces that are well-adapted to the analysis of the geodesic vector field associated with any translation invariant Finsler metric on the torus $\mathbb{T}^d$. Among several applications of this functional point of view, we study properties of geodesics that are orthogonal to two convex subsets of $\T^d$ (i.e. projection of the boundaries of strictly convex bodies of $\mathbb{R}^d$). Associated with the set of lengths of such orthogeodesics, we define a geometric Epstein function and prove its meromorphic continuation. We compute its residues in terms of intrinsic volumes of the convex sets. 
We also prove Poisson-type summation formulae relating the set of lengths of orthogeodesics and the spectrum of magnetic Laplacians.
\end{abstract}

\maketitle

%
%\setcounter{tocdepth}{2}
%\tableofcontents

\section{Introduction}

Motivated by recent developments on analytical and spectral properties of geodesic flows on negatively curved manifolds, we study in this article related questions in the opposite setting of a completely integrable system.
Such an analysis is now known to have several types of applications ranging from the study of correlation functions to counting problems and equidistribution properties. 
Before discussing analytical and spectral properties of the geodesic flow on the torus (see Section~\ref{s:mainresult}), we thus start by presenting one of their applications in the context of convex geometry. All along the article, we use the following slightly abusive terminology.
\begin{def1}
\label{d:def-convex} 
We say that $K \subset \R^d$ is a strictly convex compact set if 
\begin{itemize}
\item either $K$ is a strictly convex compact set with nonempty interior and smooth boundary $\partial K= K \setminus \Int(K)$ having all its sectional curvatures \emph{positive};
\item or $K$ is a point.
\end{itemize}
\end{def1}
By a classical Theorem of Hadamard~\cite{Hadamard1897} and Sacksteder~\cite{Sacksteder60}, if $S$ is a smooth, compact, connected, orientable hypersurface embedded in $\IR^d$ and if $S$ has all its sectional curvatures \emph{nonnegative}, then it is the boundary of a convex body -- see also~\cite{DoCarmoLima69}.

We let $K_1$ and $K_2$ be two strictly convex and compact subsets of $\IR^d$ ($d\geq 2$) with smooth boundaries $\partial K_1$ and $\partial K_2$.  
 Through the canonical projection $\mathfrak{p}:\IR^d\rightarrow \IT^d:=\IR^d/2\pi\IZ^d$, the boundaries of $K_1$ and $K_2$ can be identified with immersed submanifolds of the flat torus that may have selfintersection points. We fix an orientation on each submanifold $\partial K_i$ either by the outgoing normal vector to $K_i$ or by the ingoing one. This orientation induces an orientation on $\Sigma_i:=\mathfrak{p}(\partial K_i)$. The choice of orientation is not necessarily the same on each $\partial K_i$ (hence on each $\Sigma_i$) and, once each orientation is fixed, we denote by $\mathcal{P}_{K_1,K_2}$ the set of geodesic arcs (parametrized by arc length) on $\mathbb{T}^d$ that are \emph{directly} orthogonal\footnote{In the case where $K_i$ is reduced to a point, every geodesic passing through $K_i$ is said to be orthogonal to it and we fix the natural orientation using the Euclidean volume on $\IR^d$.} to $\Sigma_1$ and $\Sigma_2$. The orientation of the sets $K_1,K_2$ is implicit in our notation (in the introduction); the results however depend on this choice.
Using the strict convexity of $K_1$ and $K_2$, one can first verify the following statement.
\begin{lemm}
\label{l:starting-point}
There exists $T_0>0$ large enough, such that, for any $T>0$ the subset 
$$
\left\{\gamma\in\mathcal{P}_{K_1,K_2}: T_0<\ell(\gamma) \leq T \right\}
$$ 
of $\mathcal{P}_{K_1,K_2}$ is finite.
\end{lemm}
Note that the complementary set in $\mathcal{P}_{K_1,K_2}$ might be uncountable, depending on the choice of orientations of $\Sigma_i$.
We will be interested on the asymptotic properties of the lengths of these orthogeodesics. We will for instance verify in Theorem~\ref{t:asymptotic-counting} below that, for $T_0>0$ as in Lemma~\ref{l:starting-point},
\begin{equation}\label{e:asymptotic-counting}
 \sharp \left\{\gamma\in\mathcal{P}_{K_1,K_2}: T_0<\ell(\gamma)\leq T\right\}=\frac{\pi^{\frac{d}{2}}T^d}{(2\pi)^{d}\Gamma\left(\frac{d}{2}+1\right)}+\mathcal{O}(T^{d-1}).
\end{equation}
In the case where $K_1=K_2=\{0\}$, this exactly amounts to count the number of lattice points in $2\pi\IZ^d$ of norm less than $T$ and understanding the optimal size of the remainder is a deep problem in number theory. Here, we consider the setting of orthogeodesics for much more general convex sets. Thus, in some sense less arithmetical tools are available and we do not necessarily expect as strong properties on the size of the remainder. In fact, rather than refining these asymptotic formulas, our main purpose is to study various zeta functions associated with these length orthospectra and some of their analytical properties. We also aim at determining the geometric quantities encoded by these functions. Recall that counting orthogeodesics to convex subsets in the setup of hyperbolic geometry was much studied (see e.g.~\cite{ParkkonenPaulin2016, BroiseParkkonenPaulin2019} and the references therein) and similar questions arise even if the asymptotic formulae are of different nature.

\subsection{Epstein zeta functions in convex geometry} The most natural way to form a zeta function starting from $\mathcal{P}_{K_1,K_2}$ is to define, for $T_0>0$ as in Lemma~\ref{l:starting-point}, a generalized Epstein zeta function:
\begin{equation}
\label{e:epstein-zeta-intro}
\boxed{\zeta_{\beta}(K_1,K_2,s):=\sum_{\gamma\in\mathcal{P}_{K_1,K_2}:\ell(\gamma)>T_0}\frac{e^{i\int_\gamma\beta}}{\ell(\gamma)^s},}
\end{equation}
where $\beta$ is a \emph{closed and real-valued} one form on $\mathbb{T}^d$. Recall that any such form writes $\beta = \sum_{i=1}^d \beta_i dx_i + d f$ where $\beta_i \in \R$ and $f \in \mathcal{C}^\infty(\T^d;\R)$. The one-form $[\beta]=\sum_{i=1}^d\beta_i dx_i$ is identified with the de Rham cohomology class of $\beta$. The first de Rham cohomology group will be denoted by $H^1(\T^d,\R) =H^1_{\dR}(\IT^d)$: it corresponds to the kernel of the Laplacian acting on smooth real-valued one-forms and it can be identified with the first singular cohomology group. See~\cite[Chapter~17 and 18]{Lee13} or~\cite[Chapter~10]{Lee09}. We say that the cohomology class $[\beta]$ of $\beta$ is in $H^1(\IT^d,\IZ)$ if, in the above decomposition, $\beta_i \in \Z$ for all $i \in \{1,\dots,d\}$.

Thanks to~\eqref{e:asymptotic-counting}, $s\mapsto \zeta_{\beta}(K_1,K_2,s)$ defines a holomorphic function on $\{\text{Re}(s)>d\}$ and our first main result describes its meromorphic continuation to $\mathbb{C}$:
\begin{theo}\label{t:maintheo-Mellin} Let $K_1$ and $K_2$ be two strictly convex compact sets of $\IR^d$ ($d\geq 2$) and let $\beta$ be a closed and real-valued one form on $\IT^d$.
The following holds:
\begin{enumerate}
 \item if the cohomology class $[\beta]$ of $\beta$ is in $H^1(\IT^d,\IZ)$, then
 $$s\in \{\operatorname{Re}(s)>d\}\mapsto \zeta_{\beta}(K_1,K_2,s)$$
 extends meromorphically to $\mathbb{C}$, its poles are located at $s=1,\ldots,d$ and are simple;
 \item otherwise, $\zeta_{\beta}(K_1,K_2,s)$ extends holomorphically to $\mathbb{C}$.
\end{enumerate}
\end{theo}
In the case where both $K_1$ and $K_2$ are reduced to points, this theorem recovers a classical property of Epstein zeta functions~\cite{Epstein03}. See~\S\ref{ss:number} below for a brief reminder on such arithmetic functions. Yet, to the best of our knowledge, this result seems to be new in the case of general smooth strictly convex subsets. Under lower regularity assumptions on the boundary of our convex sets and eventhough we did not write all the details, our proof should in principle allow to perform the analytic continuation up to $\{\text{Re}(s)>d-N\}$ with $\mathcal{C}^N$ being the regularity of the boundary and $N>d$.

As for classical zeta functions in number theory, it is natural to compute the explicit values of the residues and, due to the geometric nature of the problem, one would like to express them in terms of natural geometric quantities attached to the convex sets $K_1$ and $K_2$. In order to be more explicit on the residues when $\beta=0$, recall Steiner's formula for a compact and convex subset $K$ of $\IR^d$~\cite[\S4]{Schneider14}:
\begin{equation}\label{e:steiner}\text{ for all } t>0,\quad \text{Vol}_{\IR^d}\left(K+tB_d\right) =\sum_{\ell=0}^dV_{d-\ell}\left(K\right)\frac{\pi^{\frac{\ell}{2}}}{\Gamma\left(\frac{\ell}{2}+1\right)} t^{\ell},\end{equation}
where $V_{\ell}\left(K\right)\geq0$ is the $\ell$-\emph{intrinsic volume} of the convex set $K$, and~\eqref{e:steiner} may be taken as a definition of the numbers $V_{\ell}\left(K\right)$. Note that $V_0(K)=1$, $V_d(K)=\text{Vol}_{\IR^d}(K)$. Moreover, if $\partial K$ has smooth boundary, one finds by the Minkowski-Steiner formula~\cite[4.2]{Schneider14}~\cite[p.~86]{teissier2016bonnesen}:
$$V_{d-1}(K)=\frac{1}{2}\text{Vol}(\partial K),$$
where $\text{Vol}$ is the $(d-1)$-volume measure on $\partial K$ induced by the Euclidean structure on $\IR^d$. Observe that $V_{d-\ell}\left(K\right)=0$ for any $0\leq \ell\leq d-1$ when $K$ is reduced to a point. Other properties of these intrinsic volumes are their invariance under Euclidean isometries (i.e. any composition of a rotation with a translation), their continuity with respect to the Hausdorff metric and their additivity\footnote{A functional satisfying such an additive property is referred as a \emph{valuation}~\cite[\S6.1]{Schneider14}.} on convex subsets of $\IR^d$, i.e.
$$\forall\ 0\leq \ell \leq d,\quad V_{\ell}\left(K\right)+V_{\ell}\left(L\right)=V_{\ell}\left(K\cup L\right)+V_{\ell}\left(K\cap L\right),$$
whenever $K$, $L$, $K\cup L$, $K\cap L$ are convex subsets of $\IR^d$. In fact, a classical Theorem of Hadwiger states that any functional on the convex subsets of $\IR^d$ enjoying these three properties is a linear combination of these intrinsic volumes~\cite[Th.~6.4.14]{Schneider14}.

Our second main theorem expresses the residues of $\zeta_{0}(K_1,K_2,s)$ (stated for $\beta=0$ to improve readability) in terms of these intrinsic volumes:
\begin{theo}\label{t:maintheo-residues}
 Let $K_1$ and $K_2$ be two strictly convex compact sets of $\IR^d$ ($d\geq 2$). Suppose in addition that $\Sigma_1=\mathfrak{p}(\partial K_1)$ (resp. $\Sigma_2=\mathfrak{p}(\d K_2)$) is oriented by the outgoing (resp. ingoing) normal vector to $K_1$ (resp. $K_2$).
 
 Then, the function
 $$s\mapsto \zeta_{0}(K_1,K_2,s)-\frac{1}{(2\pi)^d}\sum_{\ell=1}^d\frac{\pi^{\frac{\ell}{2}}\ell}{\Gamma\left(\frac{\ell}{2}+1\right)}\frac{V_{d-\ell}\left(K_1-K_2\right)}{s-\ell}$$
 extends holomorphically from $\{\operatorname{Re}(s)>d\}$ to $\mathbb{C}.$
\end{theo}
Note that $-K_2$ is a convex set and thus so is $K_1-K_2$.
Here we only describe the case where $\beta=0$ and  geodesics are pointing outward $K_1$ and inward $K_2$. Yet our proof yields an explicit expression for any $\beta$ and for all possible orientation conditions on the $\Sigma_i$. See formula~\eqref{e:residues} and \S\ref{s:convex} for more details. 
%For the sake of simplicity, we restricted ourselves to this case here as it has a transparent expression in terms of intrinsic volumes while formulas are more involved in the general case. 
When $K_2$ is reduced to a point, this theorem in particular solves the following geometric inverse problem: recover all $\ell$-intrinsic volumes of $K_1$ for $0\leq\ell\leq d-1$ from the lengths of the geodesics of $\IR^d$ orthogonal to $K_1$ and joining $K_1$ to an element of $2\pi\IZ^d$ (see Figure~\ref{f:lattice}). 

\begin{figure}[ht]
\includegraphics[scale=0.3]{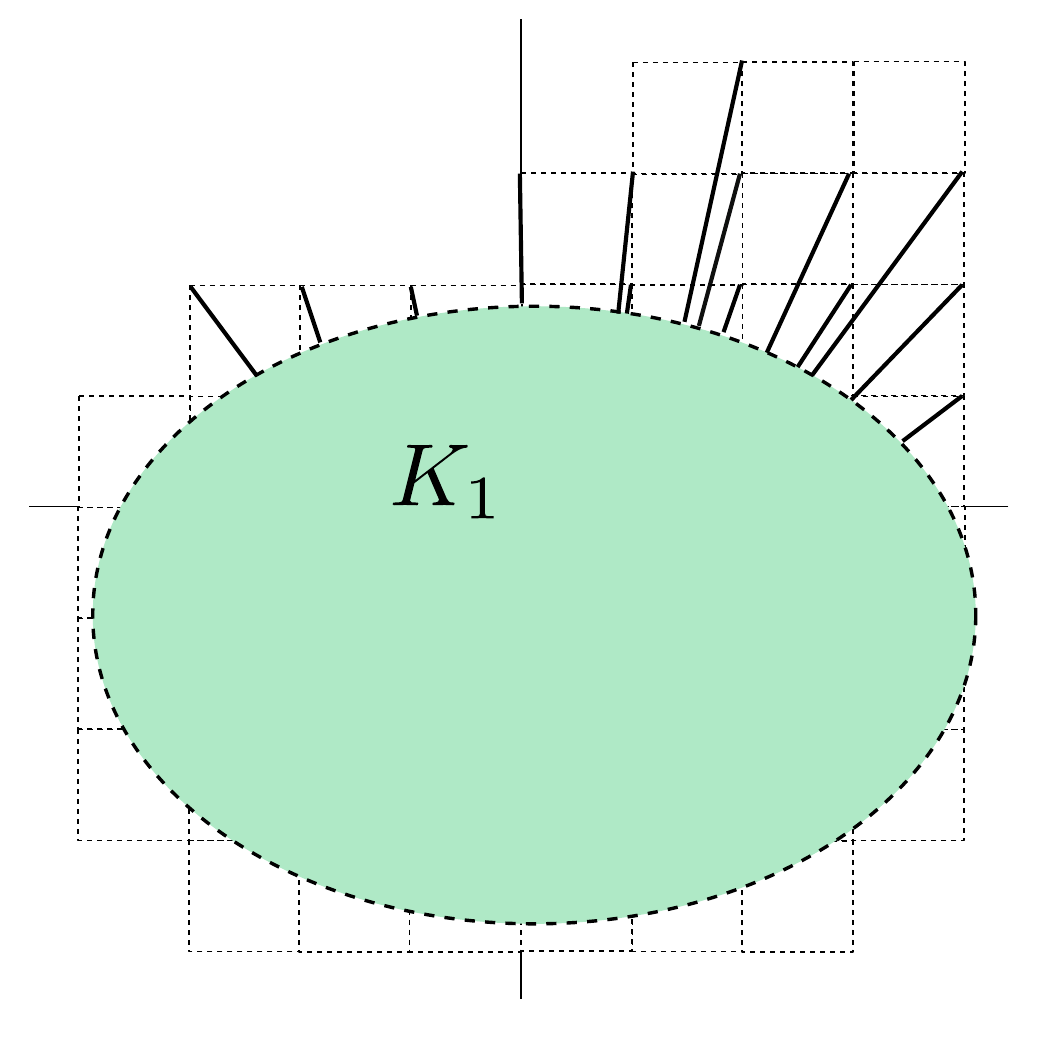}
\centering
\caption{\label{f:lattice}Lift of the orthogeodesic arcs when $K_2=\{0\}$.}
\end{figure}

\subsection{Poincar\'e series in convex geometry}
In analogy with the case of negatively curved manifolds~\cite{ParkkonenPaulin2016, BroiseParkkonenPaulin2019, DangRiviere20d}, one can also define generalized Poincar\'e series for the length orthospectrum:
\begin{equation}
\label{e:poincare-series-intro}
\boxed{\mathcal{Z}_{\beta}(K_1,K_2,s):=\sum_{\gamma\in\mathcal{P}_{K_1,K_2}:\ell(\gamma)>T_0}e^{i\int_\gamma\beta-s\ell(\gamma)},}
\end{equation}
which, as a consequence of~\eqref{e:asymptotic-counting}, is holomorphic on $\{\text{Re}(s)>0\}.$ As above, $\beta$ is a closed and real-valued one form on $\IT^d$. For such functions, we are able to describe the continuation up to $\text{Re}(s)=0$ in the following sense:
\begin{theo}[Continuous continuation of Poincar\'e series]
\label{t:maintheo-laplace}
 Let $K_1$ and $K_2$ be two strictly convex compact sets of $\IR^d$ ($d\geq 2$) and let $\beta$ be a closed and real-valued one form on $\IT^d$. 

 Then, the function
 $$s\in \{\operatorname{Re}(s)>0\}\mapsto \mathcal{Z}_{\beta}(K_1,K_2,s)$$
 extends continuously to
 $$\{\operatorname{Re}(s)\geq 0\}\setminus \{\pm i|\xi-[\beta]|:\xi\in\IZ^d\}.$$
Moreover, given $\xi_0\in\IZ^d$, one has
\begin{enumerate}
 \item if $\xi_0-[\beta]=0$, then there exist $a_{\beta}^{(1)}(K_1, K_2),\ldots,a_{\beta}^{(d)}(K_1, K_2)$ in $\mathbb{C}$ such that 
 $$\mathcal{Z}_{\beta}(K_1,K_2,s)-\sum_{\ell=1}^{d}\frac{a_{\beta}^{(\ell)}(K_1, K_2)}{s^\ell}$$
 converges as $s\rightarrow 0$ (with $\operatorname{Re}(s)>0$);
 \item if $\xi_0-[\beta]\neq0$ and $d$ is odd, then there exist $a_{\xi_0,\beta}^{(0)}(K_1, K_2),\ldots,a_{\xi_0,\beta}^{(\frac{d-1}{2})}(K_1, K_2)$ and $b_{\xi_0,\beta}(K_1, K_2)$ in $\IC$ such that 
 $$\mathcal{Z}_{\beta}(K_1,K_2,s)-\sum_{\ell=0}^{\frac{d-1}{2}}\frac{a_{\xi_0,\beta}^{(\ell)}(K_1, K_2)}{(s\mp i|\xi_0-[\beta]|)^{\frac{d+1}{2}-\ell}}-b_{\xi_0,\beta}(K_1, K_2)\ln(s\mp i|\xi_0-[\beta]|)$$
 converges as $s\rightarrow \pm i|\xi_0-[\beta]|$ (with $\operatorname{Re}(s)>0$);
 \item if $\xi_0-[\beta]\neq0$ and $d$ is even, then there exist $a_{\xi_0,\beta}^{(0)}(K_1, K_2),\ldots,a_{\xi_0,\beta}^{(d/2)}(K_1, K_2)$ such that 
 $$\mathcal{Z}_{\beta}(K_1,K_2,s)-\sum_{\ell=0}^{\frac{d}{2}}\frac{a_{\xi_0\beta}^{(\ell)}(K_1, K_2)}{(s\mp i|\xi_0-[\beta]|)^{\frac{d+1}{2}-\ell}}.$$
 converges as $s\rightarrow \pm i|\xi_0-[\beta]|$ (with $\operatorname{Re}(s)>0$).
\end{enumerate}
\end{theo}
This Theorem is a weakened version of Theorem~\ref{l:regularity-poincare} where the $\mathcal{C}^{k}$ regularity of the continuation of $\mathcal{Z}_{\beta}$ will also be discussed. Note that the set of singular points $\{\pm i|\xi-[\beta]|:\xi\in\IZ^d\}$ is linked to the spectrum of a natural magnetic Laplacian on $\T^d$, namely\footnote{The eigenvalues of $\Delta_{-[\beta]}$ and $\Delta_{[\beta]}$ coincide.} $\Delta_{-[\beta]} = (\partial_x-i[\beta])^2$ --  see the discussion in Sections~\ref{s:emergence-quantum} and~\ref{s:magnetic-laplacian}. For $\beta=0$ and for the orientation conventions of Theorem~\ref{t:maintheo-residues}, the ``residues'' at $s=0$ can be explicitly expressed as
$$\forall 1\leq \ell\leq d,\quad a_0^{(\ell)}:=\frac{(-1)^{d-1}\ell!\pi^{\frac{\ell}{2}}}{(2\pi)^d\Gamma\left(\frac{\ell}{2}+1\right)}V_{d-\ell}(K_1-K_2).$$

\subsection{Application to Poisson type formulas}

According to~\eqref{e:asymptotic-counting}, and for $T_0>0$ as in Lemma~\ref{l:starting-point} above, we emphasize that, on the imaginary axis, the Poincar\'e series
\begin{align}
\label{e:distrib-T-fourier}
 \mathcal{Z}_{\beta}(K_1,K_2,it)= \sum_{\gamma\in\mathcal{P}_{K_1,K_2}:\ell(\gamma)>T_0}e^{i\int_\gamma\beta} e^{-it \ell(\gamma)}
\end{align}
is the Fourier transform of the 
counting measure 
\begin{align}
\label{e:distrib-T}
\mathsf{T}_{\beta,K_1,K_2}(t) := \sum_{\gamma\in\mathcal{P}_{K_1,K_2}:\ell(\gamma)>T_0}e^{i\int_\gamma\beta} \delta(t-\ell(\gamma)) \quad \in \mathcal{S}'(\IR) 
\end{align}
which is a tempered distribution supported in the cone $(0,+\infty)$. Therefore, thanks to~\cite[Thm IX.16 p.~23]{ReedSimonII} 
(see also Proposition~\ref{prop:boundaryvaluesholo} below), $\widehat{\mathsf{T}}_{\beta,K_1,K_2}(t)$ can be obtained as the boundary value of a holomorphic function as follows:
\begin{align}
\label{e:bound-value-holo}
\mathcal{Z}_{\beta}(K_1,K_2,it+\alpha)\rightharpoonup\widehat{\mathsf{T}}_{\beta,K_1,K_2}(t)\quad \text{in}\ \mathcal{S}'(\IR),\quad \text{as}\ \alpha\rightarrow 0^+.
\end{align}
The holomorphic function is nothing but the analytic extension of the Fourier transform to the upper half--plane. Hence, the reader can think of Theorem~\ref{t:maintheo-laplace} as a loose version of the Paley--Wiener--Schwartz Theorem, stating that the Fourier transform of a distribution supported on the half--line $(0,+\infty)$ is the boundary value on $\R$ of a holomorphic function on the lower half--plane (sometimes called its Fourier-Laplace transform).
Note that, as a consequence of~\eqref{e:bound-value-holo}, the Poincar\'e series completely determines the distribution of the twisted length orthospectrum.
\begin{rema}
In case $\beta=0$, $\mathsf{T}_{0,K_1,K_2} $ gives precisely the distribution of length of orthogeodesics: namely if $\mathscr{L}(K_1,K_2) = \{\ell(\gamma), \gamma\in\mathcal{P}_{K_1,K_2}\}$ denotes the set of length of orthogeodesics, then we simply have 
\begin{align}
\mathsf{T}_{0,K_1,K_2}(t) := \sum_{\ell >T_0, \ell \in \mathscr{L}(K_1,K_2)} \mathsf{m}_\ell \delta(t-\ell), \quad \text{ where }\mathsf{m}_\ell = \sharp \, \{\gamma\in\mathcal{P}_{K_1,K_2}, \ell(\gamma) = \ell \} 
\end{align}
 denotes the multiplicity of the length $\ell \in \mathscr{L}(K_1,K_2)$.
\end{rema}
As a direct application of a refined version of Theorem~\ref{t:maintheo-laplace} (namely Theorem~\ref{l:regularity-poincare} below) together with~\eqref{e:bound-value-holo}, we also obtain a new Poisson-type summation formula, describing the distributional singularities of $\widehat{\mathsf{T}}$.
\begin{theo}[Poisson type formula]\label{e:maintheo-realaxis} Let $K_1$ and $K_2$ be two strictly convex compact subsets of $\IR^d$ ($d\geq 2$) and let $\beta$ be a closed and real-valued one form on $\IT^d$. Then, with $\mathsf{T}_{\beta,K_1,K_2}(\tau)$ defined in~\eqref{e:distrib-T}, the singular support of
 $$\widehat{\mathsf{T}}_{\beta,K_1,K_2}(\tau) = \sum_{\gamma\in\mathcal{P}_{K_1,K_2}:\ell(\gamma)>T_0}e^{i\int_\gamma\beta} e^{-i\tau \ell(\gamma)}$$
 is included in $\operatorname{Sp}(\pm \sqrt{-\Delta_{[\beta]}})$ and the singularities are explicitly described by Theorem~\ref{l:regularity-poincare}. 
\end{theo}
 %Here, we recall that the complementary of the singular support of a distribution $\mathsf{T}$ is the open set where the distribution is smooth. 
 As the singular support of the geometric distribution $\widehat{\mathsf{T}}_{\beta,K_1,K_2}$ is given by the eigenvalues of the magnetic Laplacian, it does not depend on the convex sets $K_1,K_2$. We would like to remark that Theorem \ref{e:maintheo-realaxis} looks like a trace formula and we refer to paragraph~\ref{sss:trace} for a more detailed comparison. The precise form of the singularities depends on the geometry of the convex sets $K_1$ and $K_2$ -- see Theorem~\ref{l:regularity-poincare} for precise expressions of the leading term at each singularity. We emphasize that the singularities are obtained as the boundary values of simple holomorphic functions as in~\cite[Th.3.1.11]{Hormander90}. 

In view of having simpler singularities and motivated by the recent developments on crystalline measures~\cite{Meyer2022}, one can symmetrize (and renormalize) the distribution $\mathsf{T}_{\beta,K_1,K_2}(t)$. This is the content of our last main result which extends in our geometric setup the Guinand--Meyer summation formula~\cite[Th.~5]{Meyer2016}.
 \begin{theo}[Guinand--Meyer type formula]\label{e:GuinandWeil} Let $K_1$ and $K_2$ be two strictly convex compact sets of $\IR^d$ ($d\geq 2$) and let $\beta$ be a closed and real-valued one form on $\IT^d$ such that $[\beta]\notin H^1(\IT^d,\IZ)$. Let $\mu$ be the signed measure defined as 
\begin{eqnarray*}
\boxed{\mu(t)= \sum_{\gamma\in\mathcal{P}_{K_1,K_2}:\ell(\gamma)>T_0}\frac{e^{i\int_\gamma\beta}}{\ell(\gamma)^{\frac{d-1}{2}}} \delta(t-\ell(\gamma))+(-i)^{d-1}\sum_{\gamma\in\mathcal{P}_{K_2,K_1}:\ell(\gamma)>T_0}\frac{e^{-i\int_\gamma\beta}}{\ell(\gamma)^{\frac{d-1}{2}}} \delta(t+\ell(\gamma)),}
\end{eqnarray*} 
where we take the same orientation conventions for\footnote{In particular, both sets are a priori distinct.} $\mathcal{P}_{K_2,K_1}$ and $\mathcal{P}_{K_1,K_2}$. 

Then, there exist complex numbers $(c_\lambda)_{\lambda\in \operatorname{Sp}(\pm \sqrt{-\Delta_{[\beta]}})}$ and $r$ belonging to $L^p_{\operatorname{loc}}(\IR)$ for every $1\leq p<\infty$ such that
\begin{eqnarray*}
\boxed{\widehat{\mu}(\tau)=\sum_{\lambda\in \operatorname{Sp}(\pm \sqrt{-\Delta_{[\beta]}})} c_\lambda \delta(\tau-\lambda)+r.}
\end{eqnarray*} 
\end{theo}
In the case where $\beta\in H^1(\IT^d,\IZ)$, the result would be similar except for an extra singularity of the Fourier transform at $\tau=0$ that may be more singular than the Dirac distribution. Following our proof, one could in fact describe explicitly this singularity at $\tau=0$ even if we do not carry out the calculation explicitly. In the case where $K_1$ and $K_2$ are distinct points and where $d=3$, it was in fact proved that $r\equiv 0$ in~\cite[Th.~5]{Meyer2016}. The proof of this last fact is briefly recalled in \S\ref{sss:guinand} using our formalism. We also explain how it can be extended to higher dimensions (when $d$ is odd) to give rise to \emph{crystalline distributions}, as first shown in~\cite[\S2]{LevReti2021}. We finally deduce from this discussion that $r$ is not identically $0$ as soon as $d\geq 5$, even in the case where $K_1,K_2$ are points. See also~\cite{Guinand59,LevOlevskii2016,RadchenkoViazovska2019} for earlier related results and~\cite{Meyer2022} for a review on recent developments in that direction.

\subsection{Related results} Before discussing the relation of these results to the analytical properties of geodesic flows, let us comment how these applications to zeta functions and Poisson formulas in convex geometry compare with similar properties and objects appearing in different contexts, most notably in arithmetic, spectral geometry and hyperbolic geometry.

\subsubsection{Comparison with zeta functions from analytic number theory}\label{ss:number}

The zeta functions appearing in Theorem~\ref{t:maintheo-Mellin} are natural generalizations in the setup of convex geometry of the Hurwitz zeta function~\cite[Ch.~12]{Apostol1998}:
$$\zeta_{\text{Hur}}(q,s):=\sum_{\xi\in \IZ: \xi\neq- q}\frac{1}{|\xi+q|^s},$$
where $q$ is some fixed element in $[0,1)$. In the case $q=0$, this is nothing else than twice the Riemann zeta function $\zeta_R(s)$. It is well known that these functions extend meromorphically from $\{\text{Re}(s)>1\}$ to $\IC$ with a simple pole at $s=1$ whose residue is equal to $2$. The relation with our zeta functions is as follows. Assume that both $K_1$ and $K_2$ are points in $\IT^1=\IR/2\pi\IZ$ that are at a distance $\ell=2\pi \min\{ q,1-q\}$ of each other. Then, one can verify that $\zeta_0(K_1,K_2,s)=(2\pi)^{-s}\zeta_{\text{Hur}}(q,s)$. The fact that we are in higher dimensions is responsible for the presence of extra poles at $s=1,\ldots d$ and Theorem~\ref{t:maintheo-residues} gives us an explicit expression of their residues in terms of geometric quantities. Due to our use of stationary phase arguments, we note that our proof does not work (strictly speaking) for $d=1$ even if the functions are of the same nature from the perspective of convex geometry.

Here we choose to call our functions generalized Epstein zeta functions in analogy with the zeta functions defined by Epstein~\cite[Eq.~(2)]{Epstein03} as higher-dimensional analogues of the Riemann zeta function:
$$\zeta_{\text{Eps}}(q,\beta,s):=\sum_{\xi\in\IZ^d\setminus\{-q\}}\frac{e^{2i\pi\xi\cdot\beta}}{|\xi+q|^s}.$$
where $q$ and $\beta$ are two fixed elements in $\IR^d$ and $\vert . \vert$ is the Euclidean norm. When $K_1$ and $K_2$ are reduced to two points $ x_1$ and $x_2$, one has 
$$\zeta_{\text{Eps}}\left(\frac{x_2-x_1}{2\pi},\beta,s\right)=(2\pi)^{s}e^{i(x_1-x_2)\cdot\beta}\zeta_{\beta}(K_1,K_2,s),$$ 
where $\beta\in\IR^d$ is identified with a closed one form, and $\zeta_\beta$ is defined in~\eqref{e:epstein-zeta-intro}. Hence, up to a multiplicative factors, our zeta functions $\zeta_{\beta}(K_1,K_2,s)$ are the natural extension of Epstein zeta functions when one considers general convex subsets of $\IR^d$ instead of points. It is well-known that the ``classical'' Epstein zeta functions extend meromorphically to the whole complex plane with at most a simple pole at $s=d$. Theorems~\ref{t:maintheo-Mellin} and~\ref{t:maintheo-residues} show that, for more general convex sets, one may also have poles at $s=1,\ldots, d-1$.
Note that we recover the continuation of the ``classical'' Epstein case since, if $K_1,K_2$ are both points, $V_{\ell}(K_1-K_2)=0$ for all $1\leq \ell \leq d-1$.

In Theorem~\ref{t:maintheo-Mellin}, we saw that if we weight our series with some unitary twist involving a closed and real-valued one form $\beta$, then our zeta functions have in fact a holomorphic extension as soon as $[\beta]\notin H^1(\IT^d,\IZ)$. These unitary twists can be thought of as geometric analogues of the (arithmetic) twisting factors used when one extends the Riemann zeta function to more general Dirichlet series~\cite[Ch.~12]{Apostol1998}. Recall that these are defined in the following manner. Fix a positive integer $D$ and a morphism $\chi: (\IZ/D\IZ)^*\rightarrow \IS^1:=\{z\in\IZ:|z|=1\}$ (the Dirichlet character). Such a morphism can be extended into a $D$-periodic function $\chi:\IZ\rightarrow \IS^1$ by letting $\chi(\xi)=0$ for every $\xi$ such that $\xi$ and $D$ are not coprime. Dirichlet series (or $L$-functions of weight $\chi$) are then defined as
$$L(\chi,s):=\sum_{\xi\in\IZ^*}\frac{\chi(\xi)}{|\xi|^s}=\sum_{r=1}^D\sum_{q\in\IZ\setminus\{-r/D\}}\frac{\chi(qD+r)}{|qD+r|^s}=\frac{1}{D^s}\sum_{r=1}^D\chi(r)\zeta_{\text{Hur}}\left(s,\frac{r}{D}\right),$$
and they have a holomorphic extension to $\IC$ except for the trivial character $\chi=1$ where one has a simple pole at $s=1$. Understanding the \emph{holomorphic} continuation of more general $L$-functions~\cite{Artin24} on algebraic number fields (for arbitrary irreducible representations) is in fact a classical topic in analytic number theory: this is for instance at the heart of Artin's conjecture. Here, we emphasize that our unitary twists do not have any particular arithmetic meaning and our (strictly) convex sets are a priori arbitrary. Despite that and thus for seemingly different reasons, these twisting factors have the same effect as Dirichlet characters for the Riemann zeta function in the sense that, under some natural ``non-rationality'' assumption on $\beta$, our zeta functions extend holomorphically to $\IC$.

\subsubsection{Relation with trace formulas}
\label{sss:trace}
Our main result on the singular support of the oscillatory series $\widehat{\mathsf{T}}_{0,K_1,K_2}(t)=\sum_\gamma e^{-it\ell(\gamma)}$ is very reminiscent to the celebrated wave trace formula proved by Chazarain~\cite{Chazarain74} and Duistermaat--Guillemin~\cite{DuistermaatGuillemin75} extending previous results by Selberg~\cite{Selberg56} and Colin de Verdi\`ere~\cite{ColindeVerdiere73}. These formulas may be seen as generalizations in spectral geometry of the Poisson summation formula, letting $\Sp(\sqrt{-\Delta_g})$ denote the spectrum of the square root of the Laplace-Beltrami operator $\Delta_g$, one considers the series
$$\widehat{\mathsf{T}}(t)=\sum_{\lambda\in \Sp(\sqrt{-\Delta_g})} e^{-it\lambda }\in \mathcal{S}^\prime(\mathbb{R}) ,$$ 
which converges in tempered distributions thanks to the Weyl law. The wave trace formula states that the singular support of the distribution $\widehat{\mathsf{T}}$ is exactly the set of lengths of periodic geodesic curves for the metric $g$. Furthermore, when the geodesic flow is nondegenerate, they described the singularity of $\widehat{\mathsf{T}}$ at each period in terms of geometric data attached to the periodic orbits and of distributions of the form $(t\pm \ell+i0)^{-1}$. In other words, the quantum spectrum determines the classical length spectrum and these wave trace formulas are often referred as generalized Poisson formulas. Recall from~\cite[p.72]{Hormander90} that the singularities in this formula can also be rewritten as follows
\begin{equation}\label{e:finite-part}\lim_{y\rightarrow 0^+}\frac{1}{t\pm \ell +iy}=(t\pm \ell+i0)^{-1}=\text{FP}\left(\frac{1}{t\pm\ell}\right)-i\pi\delta(t\pm \ell),\end{equation}
where $\text{FP}\left(.\right)$ is the finite part of the (non-integrable) function $(t\pm \ell)^{-1}$.

Theorem~\ref{e:maintheo-realaxis} has a similar flavour except that the correspondence is in the other sense and that it involves orthogeodesics of two given convex sets. More precisely, we start  from the length orthospectrum between two convex sets, we then form the series $\widehat{\mathsf{T}}_{\beta,K_1,K_2}(t)=\sum_\gamma e^{i\int_\gamma\beta}e^{-it\ell(\gamma)}$, and its singular support coincides with the quantum spectrum $\Sp(\pm i\sqrt{-\Delta_{[\beta]}})$ where $\Delta_{[\beta]}$ is the magnetic Laplacian. Another notable difference is that the singularities are more complicated in the sense that they involve distributions of the form $(t\pm \lambda-i 0)^{-k}$ with $k\geq 1$ that may not even be an integer if $d$ is even. We emphasize from~\eqref{e:finite-part} that, as in the Chazarain--Duistermaat--Guillemin formula, the singularities of $\widehat{T}_{\beta,K_1,K_2}$ are {\em not} purely Dirac type distributions (and their derivatives). This is due to the fact that the counting measure $T_{\beta,K_1,K_2}$ is supported on the half--line, hence its Fourier transform $\mathcal{Z}_{\beta,K_1,K_2}(it)$ must have its ($C^\infty$ and analytic) wave front set contained in the half cotangent cone $\{(t;\tau); \tau<0\}\subset T^*\mathbb{R}$. This prevents the presence of purely $\delta^{(k)}(t)$--like singularities whose contribution to the wave front set would contain both positive and negative frequencies $\tau$. 

As alluded above, the formulation in Theorem~\ref{e:GuinandWeil} is itself motivated by recent developments on crystalline measures~\cite{Meyer2022}, i.e. measures on $\IR$ carried by a discrete locally finite set, belonging to $\mathcal{S}'(\R)$, whose Fourier transform is still a measure carried by a discrete locally finite set. Here, we started from a complex valued measure carried by a discrete locally finite set on $\IR$ (defined from our convex orthospectrum) and we ended up with a \emph{Radon measure} carried by the spectrum of the magnetic Laplacian (which is a discrete locally finite set) modulo some (absolutely continuous) remainder lying in $L^p_{\text{loc}}$. Hence, in general, it does not fall into the category of crystalline measures due to this a priori nonvanishing remainder.

\subsubsection{Poincar\'e series on negatively curved manifolds}

Poincar\'e series appear naturally when one studies counting problems on a negatively curved manifold $(M,g)$~\cite{ParkkonenPaulin2016,BroiseParkkonenPaulin2019}. In that context, one aims for instance at couting the number of common perpendicular geodesic curves of two convex subsets of the universal cover $(\tilde{M},\tilde{g})$. Due to the exponential growth of the number of such orthogeodesics, it is natural to consider $e^{-s\ell(\gamma)}$ rather than $\ell(\gamma)^{-s}$ in order to ensure the convergence of the sums. The study of the meromorphic continuation of Poincar\'e series on compact manifolds of \emph{constant} negative curvature goes back to the works of Huber in the late fifties~\cite[Satz A]{Huber56},~\cite[Satz 2]{Huber59}. In that setting, one can obtain the meromorphic continuation through the relation between Poincar\'e series and the spectral decomposition of the Laplacian. In the case of variable negative curvature, the relation with the Laplacian is less explicit and one rather needs to exploit the ergodic properties of the geodesic flow directly. This approach was initiated by Margulis in~\cite{Margulis69, Margulis04}. Using this dynamical approach and the theory of Pollicott-Ruelle resonances, two of the authors recently proved the meromorphic continuation of Poincar\'e series on manifolds of \emph{variable} negative curvature~\cite{DangRiviere20d}. Here, as in the works of Huber, we will use the tools from harmonic analysis that are available on the torus to study the continuation of Poincar\'e series. Yet, rather than making the connection with the Laplacian\footnote{The fact that we aim at dealing with general convex sets (and not only points) seems to prevent us from working with the Laplacian on $\IT^d$.}, we will directly study the analytical properties of the geodesic flow on the torus when acting on spaces of distributions with anisotropic regularity as it was the case for negatively curved manifolds. See \S\ref{ss:strategy} for more details. In the negatively curved setting, it is shown in~\cite{DangRiviere20d} that one has meromorphic continuation beyond the threshold $\text{Re}(s)=h_{\text{top}}.$ In the case of the flat torus, Theorem~\ref{t:maintheo-laplace} shows that there is barrier at $\text{Re}(s)=h_{\text{top}}=0$ where logarithmic or square root singularities may occur at certain points that correspond to the eigenvalues of the (magnetic) Laplacian. Outside these singularities, we are however able to continuously/smoothly extend the function up to $\text{Re}(s)=0$. As already alluded, our study is intimately related to the analytic properties of the geodesic vector field 
\begin{align}
\label{e:def-V}
V:=\theta\cdot\partial_x
\end{align} on the unit tangent bundle
$$S\IT^d:=\{(x,\theta)\in\IT^d\times\IS^{d-1}\}.$$
When studying the resolvent of this operator, we will verify that there is a barrier at $\text{Re}(s)=0$ when trying to make some analytic continuation. This phenomenon is retated  to observations made by Dyatlov and Zworski at the end of the introduction of~\cite{DyatlovZworski15}, where stochastic perturbations of geodesic vector fields on Anosov manifolds are studied. In that reference, the authors studied stochastic perturbations of geodesic vector fields on Anosov manifolds. In the opposite setup of the flat $2$-torus, they described the spectrum of $P_\epsilon:=V+\epsilon\Delta_{S\IT^2}$ (with $V$ given by~\eqref{e:def-V}) and they observed that, in the limit $\epsilon\rightarrow 0^+$, the spectrum of $P_\epsilon$ fills $Y$-shaped lines in the halfplane $\{\text{Re}(s)\leq 0\}$ that are based at the same singularities as our Poincar\'e series. See e.g. Figure $3$ in that reference and the companion article~\cite{BonthonneauDangLeautaudRiviere22} for more details on this issue.

\subsubsection{Orthospectrum identities in hyperbolic geometry}
Finally, let us mention the following related problem in hyperbolic geometry. Consider some hyperbolic manifold $X$ with nonempty totally geodesic smooth boundary. In that framework, an orthogeodesic $\gamma$ is a geodesic arc which is properly immersed in $X$ and which is perpendicular to $\partial X$ at its endpoints. The lengths of these orthogeodesics verify certain identities connecting them to the volume of the boundary of $X$ (Basmajian’s identity)~\cite{Basmajian93}:
$$\text{Vol}(\partial X)=\sum_{\gamma} V_{d-1}\left(\ln\left(\text{coth}\frac{\ell(\gamma)}{2}\right)\right),$$
where $V_{d-1}(r)$ is the volume of a ball of radius $r$ on the hyperbolic space $\mathbb{H}^{d-1}$ (with the convention that $V_1(r)=2r$). Similar equalities also relate this length orthospectrum with the volume of the unit tangent bundle $SX$ (Bridgeman-Kahn's identity~\cite{BridgmanKahn2010}) and the analogues of these results on manifolds with cusps are due to McShane~\cite{McShane91, McShane98}. We refer to~\cite{BridgemanTan2016} for a recent review on this topic. In some sense, this formula has the same flavour as Theorem~\ref{t:maintheo-residues} as it relates some length orthospectrum with some volumes associated with our convex. However, while the right-hand side of Basmajian's identity converges in a standard sense, our zeta functions are defined by analytic continuation and the volumes appear as the residues of these functions.

\section{Analytical results: a functional setup for the geodesic vector field}
\label{s:mainresult}

Let us now discuss the relation of these problems from convex geometry with the analytical properties of geodesic flows on flat tori and come to the statement of our main analytical results. For simplicity, we now restrict ourselves to the case where $\beta=0$ and where we look at geodesic arcs pointing outside $K_1$ and inside $K_2$ (as in Theorem~\ref{t:maintheo-residues}). 

\subsection{Lifting the problem to the unit tangent bundle}
In order to prove Theorems \ref{t:maintheo-Mellin}, \ref{t:maintheo-residues} and \ref{t:maintheo-laplace}, one way is to rewrite the series we are interested in under an integral form as follows:
\begin{equation}\label{e:formule-magnifique}
 \sum_{\gamma\in\mathcal{P}_{K_1,K_2}}\chi(\ell(\gamma))= \int_{\IR^d}\delta_{[0]}(x)\left(\int_0^{+\infty}\chi(t)\delta_{\partial(K_1-K_2+tB_d)}(x,|dx|)|dt|\right),
\end{equation}
where $\chi$ is a nice enough function on $\IR_+^*$ (in the applications we have in mind, $\chi \in \mathcal{C}^{\infty}_c(\IR_+^*)$ or $\chi(t)=t^{-s}$ or $\chi(t)=e^{-st}$), where $\delta_{\partial(K_1-K_2+tB_d)}(x,|dx|)$ is the volume measure on $\partial(K_1-K_2+tB_d)$ induced by the Euclidean structure on $\IR^d$ and where
\begin{equation}\label{e:delta-periodique}\delta_{[0]}(x)=\frac{1}{(2\pi)^d}\sum_{\xi\in\IZ^d}e^{i\xi\cdot x}.
 \end{equation}
 A precise signification of the right hand-side of~\eqref{e:formule-magnifique} together with a proof of this formula are given in Appendix~\ref{a:proof}. With that expression at hand, proving our main results on convex geometry amounts to discuss the allowed functions $\chi$ in~\eqref{e:formule-magnifique}, to decompose $\delta_{[0]}$ according to~\eqref{e:delta-periodique} and to analyze the oscillatory integrals that come out. Yet, as explained in the beginning of the article, rather than doing that directly, we will obtain these results as a by-product of a more general analysis\footnote{Similar oscillatory integrals will of course appear in our analysis.} of the geodesic vector field on $S\IT^d$. In fact, since the seminal work of Margulis~\cite{Margulis69, Margulis04}, it is well understood that on negatively curved manifolds, it is convenient to lift this kind of geometric problems to the unit cotangent bundle of the manifold. For instance, properties of Poincar\'e series are related to the asymptotic properties of the geodesic flow, and more specifically to its mixing properties. In a recent work~\cite{DangRiviere20d}, two of the authors formulated this relation using the theory of De Rham currents and we will see that this still makes sense in the case of flat tori where the curvature vanishes everywhere. See Section~\ref{s:poincare} for details. Let us explain this connection without being very precise on the sense of the various integrals. We denote by $N(K_i)$ the outward unit normal to $K_i$ inside $S\IT^d$:
$$N(K_i):=\left\{(\mathfrak{p}(x),d\mathfrak{p}(x)\theta),  x\in \partial K_i,  \theta\text{ directly orthogonal to } \partial K_i \ \text{at}\ x\right\}.$$
Then, given any nice enough function $\chi(t)$ (say again in $\mathcal{C}^{\infty}_c(\IR_+^*)$, $t^{-s}$ or $e^{-st}$), we will prove that
\begin{equation}\label{e:current-zeta}\sum_{\gamma\in\mathcal{P}_{K_1,K_2}}\chi(\ell(\gamma))=(-1)^{d-1}\int_{S\IT^d}[N(K_1)] \wedge\int_{\IR}\chi(t)\iota_Ve^{-tV*}[N(K_2)]|dt|.
 \end{equation}
where $[N(K_i)]$ is the current of integration on $N(K_i)$ and where
\begin{align}
\label{e:geodesic-flow}
e^{tV}:(x,\theta)\in S\IT^d \rightarrow (x+t\theta, \theta)\in S\IT^d
\end{align}
is the geodesic flow. Compared with~\eqref{e:formule-magnifique}, this new formula has the advantage to explicitly involve the geodesic vector field. This current theoretic approach also allows to deal directly with the exponential weights appearing in our zeta functions together with the more general orientation conventions considered in the introduction. On the contrary, the approach using~\eqref{e:formule-magnifique} (performed in Appendix~\ref{a:proof}) seems to only apply (at least directly) to the outgoing/ingoing convention of Theorem~\ref{t:maintheo-residues}.
\begin{rema}
Formula~\eqref{e:current-zeta} derives from the observation that elements in $\mathcal{P}_{K_1,K_2}$ are in one-to-one correspondance with the geodesic orbits in $S\IT^d$ joining the two Legendrian submanifolds $N(K_1)$ and $N(K_2)$. In the framework of symplectic topology, such orbits are referred as the \emph{Reeb chords} of these two Legendrian submanifolds.
\end{rema}

\subsection{Defining a proper functional framework for the geodesic flow}\label{ss:strategy}
Hence, rather than proceeding directly to the calculation of zeta functions from~\eqref{e:formule-magnifique}, we choose to view this as a consequence of analytical properties of geodesic vector fields. More precisely, we will define appropriate functional frameworks to study the operators appearing in~\eqref{e:current-zeta}:
$$\widehat{\chi}(-iV):=\int_{\IR}\chi(t)e^{-tV*}|dt|,$$
where $\chi$ is a nice enough function (say $e^{-st}$ or $t^{-s}$). In the end, our main geometrical theorems on length orthospectra for convex bodies will be simple corollaries of this analysis -- see Section~\ref{s:poincare}. Even if slightly longer, we believe that this sharp analysis, which is the content of Sections~\ref{s:analysis} to~\ref{s:laplace-mellin}, is interesting on its own and that it allows to capture the dynamical mechanism at work when proving this kind of results. Along the way, it also has the advantage of applying directly to other questions such as equidistribution properties of the geodesic flow. See Theorems~\ref{t:maintheo-correlations} or~\ref{t:twisted-correlations} for instance. 

On top of these applications, this analysis is motivated by the study of similar questions arising on negatively curved manifolds where one defines appropriate spaces of anisotropic Sobolev distributions in order to make sense of the spectrum of the geodesic vector field: the so-called Pollicott-Ruelle spectrum~\cite{Ruelle76, Pollicott85}. More precisely, given any $N>0$, one aims at defining a Banach space $\mathcal{B}_N$ such that the geodesic vector field (viewed as an unbounded order $1$ differential operator) has discrete spectrum on $\{\text{Re}(s)>-N\}$. One of the difficulties when analyzing such an operator on the unit tangent bundle $SX$ of some manifold $(X,g)$ is that its symbol $H(x,\theta;\xi)=\xi(V(x,\theta))$ is not elliptic and that it vanishes on the noncompact set:
$$\mathcal{C}:=\{(x,\theta;\xi)\in T^*SX: \xi(V(x,\theta))=0\}.$$
In the case of a negatively curved manifold, this characteristic region is generated by two subbundles: the unstable direction and the stable one. Using this duality, one is able to construct Banach (or Hilbert) spaces adapted to the operator $V$ by requiring some negative (resp. positive) Sobolev regularity along the unstable (resp. stable) direction and by exploiting the contraction properties along these directions. The construction of such functional spaces was made explicit through different methods in various geometric contexts: Anosov flows~\cite{Liverani04, ButterleyLiverani07, Tsujii10, FaureSjostrand11, Tsujii12, GiuliettiLiveraniPollicott13, DyatlovZworski16, FaureTsujii17}, Axiom A flows~\cite{DyatlovGuillarmou16, Meddane21}, billiard dynamics~\cite{BaladiDemersLiverani18, BaladiDemers20}, Morse-Smale flows~\cite{DangRiviere20a}, manifolds with cusps~\cite{GuedesBonthonneauWeich17}, analytic Anosov flows~\cite{Jezequel21, GuedesBonthonneauJezequel20}, etc. We also refer the reader to~\cite{Baladi18} for a detailed account of the (related) case of hyperbolic diffeomorphisms. 

In our framework, the geodesic flow does not belong to any of these classes of flows as it is an integrable dynamical system without any hyperbolic property. Despite that and using the fact that the curvature is $0$ (and thus nonpositive), there is a notion of stable and unstable bundles~\cite[Ch.~3]{Ruggiero07}. Yet, as opposed to the negatively curved setting, both bundles are equal and they do not generate the whole characteristic region. They correspond in fact to the tangent space to $\IT^d$ intersected with $\mathcal{C}$. See~\S\ref{ss:riemannian} for details. As we will recall in~\S\ref{ss:riemannian}, this bundle is in some sense attractive/repulsive for the lifted dynamics on the cotangent bundle to $SX$. This observation is somehow enough to implement similar ideas (with of course also some major differences) as for geodesic flows on negatively curved manifolds and in order to define spaces with anisotropic regularity adapted to the geodesic flow on $S\IT^d$. On $S\T^d$, the ``mixing'' properties of the geodesic flow are much weaker than for geodesic flows on negatively curved manifolds, but they turn out to be sufficient in view of proving our main results using formula~\eqref{e:current-zeta}. To that aim, we will use tools from harmonic analysis that are available on the torus in order to construct the spaces adapted to $V$. In this respect, our approach is in some sense reminiscent of the one used by Ratner~\cite{Ratner87} to study the decay of correlations on hyperbolic surfaces. Even though anisotropic Sobolev spaces are not explicitly mentioned in her analysis, the Fourier type reduction made in~\cite[\S2]{Ratner87} and the way it is handled there is close to the strategy we will follow in Sections~\ref{s:correlation-diff-forms} and~\ref{s:analysis} of the present work. 

\subsection{Anisotropic Sobolev spaces}

Let us now describe more precisely the analytical properties we are aiming at in the simplified setting where we consider functions rather than currents of integrations as in~\eqref{e:current-zeta}. We define anisotropic Sobolev spaces of distributions on $S\IT^d$ as follows
$$\mathcal{H}^{M,N}(S\IT^d):=\left\{u\in\ml{D}^\prime(S\IT^d):\ \sum_{\xi\in\IZ^d} \langle \xi \rangle^{2N}\|\widehat{u}_\xi\|_{H^M(\IS^{d-1})}^2<+\infty\right\}, \quad \langle \xi \rangle = (1+|\xi|^2)^{\frac{1}{2}} , $$
where $(M,N)\in\IR^2$ and  
$$u(x,\theta)=\sum_{\xi\in\IZ^d}\widehat{u}_\xi(\theta)\frac{e^{i\xi\cdot x}}{(2\pi)^{\frac{d}{2}}},$$
with $\widehat{u}_\xi\in\ml{D}^{\prime}(\IS^{d-1})$, and where $\|.\|_{H^{M}}$ denotes the standard Sobolev norm on $\IS^{d-1}$. Roughly speaking, $u = u (x,\theta)\in \mathcal{H}^{M,N}(S\IT^d)$ if $u$ has $H^N$ regularity in the variable $x \in \T^d$ and $H^{M}$ regularity in the variable $\theta \in \IS^{d-1}$.
With this convention at hand, we will prove the following type of results:
\begin{theo}[Mellin transform, function case]\label{t:maintheo-mellin-function} Let $\chi\in\ml{C}^\infty_c([1,+\infty))$ such that $\chi=1$ in a neighborhood of $1$ and let $N\in\IZ_+$. Then, the operator
$$\mathcal{M}(s):=\int_1^{\infty}t^{-s}e^{-tV*}|dt|:\mathcal{C}^{\infty}(S\IT^d)\rightarrow \mathcal{D}^{\prime}(S\IT^d)$$
splits as
$$\mathcal{M}(s)=\mathcal{M}_0(s)+\mathcal{M}_\infty(s),$$
where
$$\mathcal{M}_0(s):=\int_1^{\infty}\chi(t)t^{-s}e^{-tV*}|dt|:\ml{H}^{N,-N/2}(S\IT^d)\rightarrow\ml{H}^{N,-N/2}(S\IT^d)$$
is a holomorphic family of bounded operators on $\mathbb{C}$ and where
$$\mathcal{M}_\infty(s):=\int_1^{\infty}(1-\chi(t))t^{-s}e^{-tV*}|dt|:\ml{H}^{N,-N/2}(S\IT^d)\rightarrow\ml{H}^{-N,N/2}(S\IT^d)$$
extends as a meromorphic family of bounded operators from $\{\operatorname{Re}(s)>1\}$ to $\{\operatorname{Re}(s)>1-N\}$
 with only a simple pole at $s=1$ whose residue is given by
 $$\forall\psi\in\mathcal{C}^{\infty}(S\IT^d),\quad\Res_{s=1}\left(\mathcal{M}_\infty(s)\right)(\psi)(x,\theta)=\frac{1}{(2\pi)^d}\int_{\IT^d}\psi(y,\theta)dy.$$
\end{theo}
In particular, this Theorem tells us that the operator
$$\mathcal{M}(s):=\int_1^{\infty}t^{-s}e^{-tV*}|dt|:\mathcal{C}^{\infty}(S\IT^d)\rightarrow \mathcal{D}^{\prime}(S\IT^d)$$
extends meromorphically from $\{\operatorname{Re}(s)>1\}$ to the whole complex plane with only a simple pole at $s=1$. Yet, the statement is more precise as it allows us to describe the allowed regularity for this meromorphic continuation. We emphasize that the mapping properties of $\mathcal{M}_0(s)$ are rather immediate from the definition of our anisotropic norms and the main difficulty in this statement is about the ``regularizing'' properties of $\mathcal{M}_\infty(s)$. This Theorem is a direct consequence of the much more general Theorem~\ref{t:general-mellin} (together with Proposition~\ref{p:terms-near-zero}), and it is one of the main results of this article. In other words, our meromorphic continuations are valid on spaces of distributions that are regular along the vertical bundle to $S\IT^d$ (i.e. the tangent space to $\IS^{d-1}$) and that may have negative Sobolev regularity along the horizontal bundle (i.e. the tangent space to $\IT^{d}$). In particular, the anisotropic Sobolev spaces $\mathcal{H}^{N,-N/2}$ contain the Dirac distribution $\delta_{[0]}(x)$ for $N>d$, and this is typically the kind of distributions that we will pick as test functions in order to derive our main applications on convex geometry using~\eqref{e:current-zeta}. In order to prove Theorems~\ref{t:maintheo-Mellin} and~\ref{t:maintheo-residues}, we will in fact need to prove more general statements for the action of $\mathcal{M}(s)$ on differential forms or more precisely on certain anisotropic Sobolev spaces of currents. Among other things, the action on differential forms will be responsible for the presence of the extra poles at $s=2,\ldots,d$ but this simplified statement already illustrates the kind of properties we are aiming at. 

The same spaces will also allow us to prove the following statement:
\begin{theo}[Laplace transform, function case, continuous continuation]
\label{t:maintheo-resolvent-function} Let $\chi\in\ml{C}^\infty_c([0,+\infty))$ such that $\chi=1$ in a neighborhood of $0$ and let $N\in 2\IZ_+^*+d$. Then, the operator
$$\mathcal{L}(s):=(V+s)^{-1}=\int_0^{\infty}e^{-st}e^{-tV*}|dt|:\mathcal{C}^{\infty}(S\IT^d)\rightarrow \mathcal{D}^{\prime}(S\IT^d)$$
splits as
$$\mathcal{L}(s)=\mathcal{L}_0(s)+\mathcal{L}_\infty(s),$$
where
$$\mathcal{L}_0(s):=\int_0^{\infty}\chi(t)e^{-st}e^{-tV*}|dt|:\ml{H}^{N,-N}(S\IT^d)\rightarrow\ml{H}^{N,-N}(S\IT^d)$$
is a holomorphic family of bounded operators on $\mathbb{C}$ and where
$$\mathcal{L}_\infty(s):=\int_0^{\infty}(1-\chi(t))e^{-st}e^{-tV*}|dt|:\ml{H}^{N,-N/2}(S\IT^d)\rightarrow\ml{H}^{-N,N/2}(S\IT^d)$$
extends continuously from $\{\operatorname{Re}(s)>0\}$ to 
 \begin{enumerate}
  \item $\{\operatorname{Re}(s)\geq 0\}\setminus\{0\}$ if $d\geq 4$,
  \item $\{\operatorname{Re}(s)\geq 0\}\setminus\{\pm i|\xi|:\xi\in\IZ^d\}$ if $d=2,3$.
 \end{enumerate}
Moreover, in any dimension, one has
 $$(V+s)^{-1}\psi=\frac{1}{(2\pi)^ds}\int_{\IT^d}\psi(y,\theta)dy+\mathcal{O}_{\ml{D}^\prime}(1), \quad  \text{ as } s \to 0, \Re(s)>0 ,$$
and, when $d=2,3$, one has, for $|\xi_0|\neq 0$,
\begin{align*}(V+s)^{-1}\psi
& =\frac{e^{\mp i\pi\frac{d-1}{4}}g_d(s\mp i|\xi_0|)}{(2\pi)^{\frac{d+1}{2}}|\xi_0|^{\frac{d-1}{2}}}\sum_{\xi:|\xi|=|\xi_0|}e^{i\xi\cdot x}\delta_0\left(\theta\mp \frac{\xi}{|\xi|}\right)\int_{\IT^d}\psi\left(y,\pm\frac{\xi}{|\xi|}\right)e^{-i\xi\cdot y}dy \\
& \quad +\mathcal{O}_{\ml{D}^\prime}(1), \quad   \text{ as } s\rightarrow \pm i|\xi_0| , \Re(s)>0 ,
\end{align*}
where
$$g_2(z):=\frac{\sqrt{2\pi}}{\sqrt{z}},\quad\text{and}\quad g_3(z):=-\ln(z).$$
\end{theo}
Again, this result is the consequence of the much more precise Theorem~\ref{t:general-laplace} (together with Proposition~\ref{p:terms-near-zero}) which is valid on certain anisotropic Sobolev spaces of currents and which will also lead us to the proof of Theorem~\ref{t:maintheo-laplace}. In Theorem~\ref{t:general-laplace}, the $\mathcal{C}^{k}$ continuation of $\mathcal{L}(s)$ is also discussed, and shows that the Laplace transform actually exhibits $\mathcal{C}^{k}$-singularities at the points $\{\pm i|\xi|:\xi\in\IZ^d\}$ in {\em any} dimension (but for larger values of $k$ in higher dimension).
 In the companion article~\cite{BonthonneauDangLeautaudRiviere22}, we show that this result can be ``improved'' if we replace the Sobolev norm on $\IS^{d-1}$ by some appropriate analytic norm built from the norms used in~\cite{GalkowskiZworski19} for the study of analytic pseudodifferential operators of order $0$. In fact, after Fourier decomposition, studying the resolvent of $V$ acting on functions amounts to study a family of resolvents of multiplication operators (i.e. of pseudodifferential operators of order $0$) on $\IS^{d-1}$. In Sobolev regularity (as we are dealing here), one could apply the results from~\cite[\S7.6]{AmreinBoutetGeorgescu96} (e.g. Th.~7.6.2) based on Mourre's commutator method. See also~\cite{ColindeVerdiereSaintRaymond20, DyatlovZworski19b} for recent developments for more general pseudodifferential operators of order $0$ in dimension $2$. Modulo some extra work to sum over all Fourier modes, this would yield in principle that
 $$\forall \epsilon>0,\ \forall N\in\IR,\quad (V+s)^{-1}:\mathcal{H}^{\frac{1}{2}+\epsilon, N}(S\IT^d)\rightarrow\mathcal{H}^{-\frac{1}{2}-\epsilon, -N}(S\IT^d)$$
 extends continuously from $\{\text{Re}(s)>0\}$ to $\{\text{Re}(s)\geq 0\}\setminus\{\pm i|\xi|:\xi\in\IZ^d\}.$ In view of our geometric applications to convex geometry, this analysis does not suffice since we aim at considering distributions having the same regularity as $\delta_{[0]}$, which does not belong to such spaces.

%\begin{rema} The operator $\ml{L}(s)\psi$ is the Laplace transform of the solution $u(t,\cdot) = e^{-tV*}\psi$ to the transport equation
%$$\partial_\tau u+ \theta\cdot\partial_xu=0,\quad u(\tau=0,x,\theta)=\psi(x,\theta),$$
%while
% $$\mathcal{M}(s+1) \psi=\int_0^{+\infty}e^{-s\tau}e^{-e^\tau V*}\psi|d\tau|$$
%is the Laplace transform of the map $\tau \mapsto v(\tau, \cdot) =e^{-e^\tau V*}(\psi)$ which solves the (reparametrized) transport equation
%$$\partial_\tau v + e^{\tau}\theta\cdot\partial_xv = 0,\quad v(\tau=0,x,\theta)=\psi(x,\theta).$$
%\end{rema}

\subsection{Emergence of quantum dynamics}
\label{s:emergence-quantum}
Theorems~\ref{t:maintheo-mellin-function} and~\ref{t:maintheo-resolvent-function} (as well as their analogues in the case of differential forms) are consequences of the fact that, through standard stationary phase asymptotics, we can give a full expansion of the Schwartz kernel of the geodesic flow. For instance, the first term in the asymptotic expansion reads
\begin{theo}[Time asymptotics of the geodesic flow, function case, leading term]
\label{t:maintheo-correlations} For every smooth function $\psi\in\ml{C}^{\infty}(S\IT^d)$, one has
\begin{eqnarray*}
t^{\frac{d-1}{2}}\left(\psi\circ e^{-tV}(x,\theta)-\frac{1}{(2\pi)^d}\int_{\IT^d}\psi(y,\theta)dy\right)&=& (2\pi)^{\frac{d-1}{2}}\sum_{\vareps\in\{\pm\}}\mathbf{P}_{\vareps}^\dagger\frac{e^{i\vareps \left(t\sqrt{-\Delta}-\frac{\pi}{4}(d-1)\right)}}{(-\Delta)^{\frac{d-1}{4}}}\mathbf{P}_{\vareps}\\
&+&\ml{O}_{\ml{D}^{\prime}(S\IT^d)}(t^{-1})
\end{eqnarray*}
where $\Delta=\sum_{j=1}^d\partial_{x_j}^2$ is the Euclidean Laplacian on $\IT^d$, 
$$\mathbf{P}_{\pm}:\psi\in\ml{C}^{\infty}(S\IT^d)\mapsto\sum_{\xi\neq 0}\frac{1}{(2\pi)^d}\int_{\IT^d}\psi\left(y,\pm\frac{\xi}{|\xi|}\right)e^{i(y-x)\cdot\xi}dy\in\ml{C}^{\infty}(\IT^d)$$
and
 $$\mathbf{P}_{\pm}^\dagger:f\in\ml{C}^{\infty}(\IT^d)\mapsto \sum_{\xi\neq 0}\frac{1}{(2\pi)^d}\left(\int_{\IT^d}f\left(y\right)e^{i(y-x)\cdot\xi}dy\right)\delta_0\left(\theta\mp\frac{\xi}{|\xi|}\right)\in\ml{D}^{\prime}(S\IT^d)$$
\end{theo}
This Theorem is a corollary of the much more precise statement given in Theorem~\ref{t:twisted-correlations} which provides a full asymptotic expansion with a precise description of the remainder terms at each step. Once again, this result could (and will) be expressed in terms of anisotropic Sobolev norms. Yet, due to the absence of integration over time, this requires a refined version of the spaces $\ml{H}^{N,-N/2}(S\IT^d)$ with an additional regularity imposed in the direction of the vector field $\theta\cdot\partial_x$.

In order to keep track of the comparison with negatively curved manifolds, such a result can be viewed as a simple occurence of the emergence of quantum dynamics (through the half-wave group $(e^{\pm it \sqrt{-\Delta}})_{t \in \R}$ on the torus) in the long time dynamics of geodesic flows (i.e. $(e^{tV})_{t\in\R}$ on $S\T^d$). This phenomenon was recently exhibited by Faure and Tsujii in the general context of contact Anosov flows~\cite{FaureTsujii15, FaureTsujii17, FaureTsujii17b, FaureTsujii21}. See also~\cite{DyatlovFaureGuillarmou2015} for related results of Dyatlov, Faure and Guillarmou in the particular case of geodesic flows on \emph{hyperbolic} manifolds. Compared with the results of Faure and Tsujii, we emphasize that our analysis heavily relies on the algebraic structure of our flows as in the hyperbolic settings treated in~\cite{Ratner87, DyatlovFaureGuillarmou2015}. Moreover, we are dealing with completely integrable systems which have in some sense opposite behaviours compared with the dynamical situations considered in all these references. In particular, due to the integrable nature of our system, the asymptotic expansion in terms of the quantum propagator is polynomial rather than exponential as in~\cite[Th.~1.2]{FaureTsujii21}. This is reminiscent to the much weaker mixing properties of the geodesic flow in this situation.
Finally, we refer to~\cite{DangLeautaudRiviereJEDP} for a short note presenting some of the results of the present paper as well a sketches of proofs.

\subsection{Organization of the article}

In Section~\ref{s:background-convex}, we review the necessary background on convex and differential geometry in order to define the Poincaré series and Epstein zeta functions.
 In particular, we introduce in Section~\ref{s:Finsler-Ham} the general Finsler structures on $\T^d$, to which our results apply and discuss some of their properties. In Section~\ref{ss:relation-resolvent-series}, we introduce the general versions of the Zeta function and Poincar\'e series studied in the paper, and describe some of their basic properties.

 In Section~\ref{s:correlation-diff-forms}, after some brief recollection of basic facts on de Rham currents, we give a current theoretic interpretation of mixed volumes in convex geometry and we define some dynamical correlation functions for currents. 
 We explain how the Poincaré series and Epstein zeta functions may be understood as the Mellin and Laplace transforms of the correlation function associated with certain currents.
 Then we close the section by relating the dynamical correlators with the Poincaré series and Epstein zeta function.

 The long time asymptotics of the dynamical correlations involves some estimates for oscillatory integrals.
The analysis of these oscillatory integrals is a standard topic in harmonic analysis~\cite{Herz62, Hormander90, Steinbook, DyatlovZworski19} and in Section~\ref{s:analysis}, we rediscuss some of their properties and pay some attention on the control of certain estimates in terms of the frequency parameter.
%
%In Section~\ref{s:correlation-diff-forms}, we review the necessary background on de Rham currents and we fix some conventions that will be used all along this work. Among other things, we reduce our analysis of $e^{-tV}$ to the study of certain oscillatory integrals. The reader only interested in dynamical properties (correlations of functions), and not on the geometrical questions (e.g. in relation with convex sets) may skip this section and consult directly Sections~\ref{s:analysis} and~\ref{s:correlation}.
%
%The analysis of these oscillatory integrals is a standard topic in harmonic analysis~\cite{Herz62, Hormander90, Steinbook, DyatlovZworski19} and in Section~\ref{s:analysis}, we rediscuss some of their properties and pay some attention on the control of certain estimates in terms of the frequency parameter. 

Then, in Section~\ref{s:correlation}, we apply this analysis to define spaces with anisotropic Sobolev regularity, in which we describe the asymptotic expansion of the operator $e^{-tV*}$ as $t\rightarrow+\infty$ acting on functions, as in Theorem~\ref{t:maintheo-correlations}. 

In Section~\ref{s:anis-spaces}, we come back to the general forms/currents setting of the article, we define anisotropic Sobolev spaces adapted to the dynamics i.e. on which the operator $\widehat{\chi}(-iV)$ is well defined. When $\chi$ depends on some complex parameter $s\in\IC$, we show in Section~\ref{s:laplace-mellin} that the operator $\widehat{\chi}(-iV)$ can be continued, as in Theorems~\ref{t:maintheo-mellin-function} and~\ref{t:maintheo-resolvent-function}.

In Section~\ref{s:poincare}, we apply these results to particular currents to make the connection between these kinds of operators and geometric zeta functions/Poincar\'e series, using the strategy initiated in~\cite{DangRiviere20d}. Along the way, we prove slightly more general versions of Theorems~\ref{t:maintheo-Mellin},~\ref{t:maintheo-laplace},~\ref{e:maintheo-realaxis} and~\ref{e:GuinandWeil}. 

We conclude the proof of Theorem~\ref{t:maintheo-residues} in Section~\ref{s:convex} by identifying the values of certain residues using tools from convex geometry~\cite{Schneider14}. 

%Finally, in Section~\ref{s:equidistribution}, we apply our approach to prove some equidistribution properties of convex subsets under the action of the geodesic flow. The   associated statements, that are variants of earlier results due to Randol~\cite{Randol84}, are not presented in the introduction and we refer the reader directly to this section. {\vio Phrase a virer si on vire l'equidistribution}

\subsection{Comments on generalizations}\label{ss:comments-general}

The choice of the lattice $2\pi\IZ^d$ is somewhat arbitrary and it makes the presentation slightly simpler. However, our analysis could be adapted to handle more general flat tori of the form $\mathbb{R}^d/\Gamma$ where $\Gamma$ is a lattice in $\mathbb{R}^d$ of maximal rank. 

 Up to this point, for the sake of the exposition, we only considered the natural Euclidean structure on $\T^d$, the associated distance function, the associated geodesic vector field $\theta\cdot \d_x$ and flow~\eqref{e:geodesic-flow} on $S\IT^d$. 
 Yet, our analysis works if we replace the Euclidean structure by any translation invariant Finsler structure on $\T^d$. As explained in Section~\ref{s:Finsler-Ham} below, the latter corresponds to studying on $S\IT^d$ the vector field $\mathbf{v}(\theta) \cdot \d_x$ and the associated flow:
$$(x,\theta)\in S\IT^d\mapsto (x+t\mathbf{v}(\theta),\theta) \in S\IT^d,$$
where $\theta\in\IS^{d-1}\mapsto \mathbf{v}(\theta)\in \IR^d$ is the parametrization by its outward normal of the boundary of a strictly convex compact subset $K$ having $0 \in \Int(K)$. This general set-up is described in Section~\ref{s:background-convex} and we state the results at this level of generality all along the article.

 %All allong this work, we consider the case of the geodesic flow on $S\IT^d$. Yet our analysis could handle more general flows  of the form:
%$$(x,\theta)\in S\IT^d\mapsto (x+t\tilde{x}(\theta),\theta) \in S\IT^d,$$
%where $\theta\in\IS^{d-1}\mapsto \tilde{x}(\theta)\in \IR^d$ is the parametrization of the boundary of a strictly convex compact subset $K$ by its outward normal, i.e. the inverse of the Gauss map. The key observation to extend our analysis to these flows is that, for every $\xi\neq 0$ in $\IR^d$, the map $\theta\in\IS^{d-1}\mapsto \xi\cdot \tilde{x}(\theta)$ has only two critical points at $\theta=\pm \xi/|\xi|$. Moreover, these two points are nondegenerate thanks to the strict convexity assumption, which allows to perform the stationary phase asymptotics of~\S\ref{s:analysis}. From the geometric perspective of the introduction, this would correspond to the more general setup where one considers dilations $K \mapsto TK$, $T>0$, of the arbitrary strictly convex subset $K$~\cite{vanderCorput20, Hlawka50, Herz62b, Randol66, ColindeVerdiere77}. In fact, as it will be clear in our proof, our geometric problem is related to the description of the lattice points in $K_1-K_2+T B_d$ (at least for a proper choice of orientation), where $B_d$ is the Euclidean unit ball of $\IR^d$ centered at $0$. The choice of $B_d$ (and thus of the geodesic flow) makes the presentation slightly easier.

\section{Background on convex and differential geometry}\label{s:background-convex}
In this preliminary section, we review a few facts from convex geometry. We also define precisely the general Finsler structures and the associated vector fields, as well as their zeta functions.

\subsection{Normal bundles to convex sets}\label{ss:convex}

Let $K$ be a compact and convex subset of $\mathbb{R}^d$ with smooth boundary $\partial K$.  We define the \emph{unit normal bundle to $\partial K$} as
$$N(\partial K):=\left\{(x,\theta)\in\partial K\times\IS^{d-1}:\forall v\in T_x\partial K,\ \theta\cdot v=0\right\}.$$
Except when $K$ is reduced to a point, this submanifold of $\mathbb{R}^d\times\mathbb{S}^{d-1}$ has two connected components and, in that case, we introduce the direct normal bundle to $K$ as 
$$N_+(K):=\left\{(x,\theta)\in N(\partial K):\ \theta\ \text{is pointing outward }K\right\},$$
and the indirect normal bundle to $K$ as $N_-(K):=N(\partial K)\setminus N_+(K)$.
The boundary $\partial K$ of $K$ is thus naturally oriented by the outward normal.
 In the case where $K$ is reduced to a point, we set $N_+(K)=N(\partial K)=N_-(K)$.
\begin{rema}\label{r:shape-operator}
Recall that, when $K$ is not reduced to a point, the shape operator of the smooth hypersurface $\partial K$ is the map
$$S(x):v\in T_x\partial K\mapsto \nabla_v\theta\in T_x\partial K,$$ 
where $x\in \partial K \mapsto\theta(x)\in N_{+,x}(\partial K)\subset\IS^{d-1}$ and where $\nabla$ is the (standard) covariant derivative in $\IR^d$~\cite[\S 4.2]{Lee09}. In particular, $S(x)$ is the selfadjoint map associated with the second fundamental form of $\Sigma$ and it is \emph{invertible if and only if $\partial K$ has nonvanishing Gauss curvature}~\cite[Def.~4.24]{Lee09}. 
If $K$ is a strictly convex body (not reduced to a point), then $\partial K$ has all its sectional curvatures~\cite[p.557]{Lee09} positive by definition (and thus nonvanishing Gaussian curvature). This is equivalent to saying that all the eigenvalues of the shape operator are non zero and have the same sign thanks to the Gauss curvature equation~\cite[Eq.~4.10, p.~172]{Lee09}.
\end{rema}
If we suppose that $K$ is strictly convex, then the Gauss map 
\begin{equation}\label{e:gaussmap}
G:(x,\theta)\in N_+(K)\mapsto \theta\in\IS^{d-1}
\end{equation} is a diffeomorphism and there exists a smooth map $x_K:\IS^{d-1}\rightarrow\IR^d$ such that $G^{-1}(\theta)=(x_K(\theta),\theta)$. The map $x_K$ is the \emph{inverse Gauss} map. Note that this remains true when $K=\{x_0\}$ by letting $x_K(\theta)=x_0$.
In both cases, it is natural to say that we can parametrize the convex set by the normal and, when $K$ is not reduced to a point, the map $x_K:\IS^{d-1}\rightarrow\partial K$ is in fact a diffeomorphism which is orientation preserving. For later purposes, we also define the following vector field on $\IR^d\times\IS^{d-1}$:
$$V_K^\pm:=x_K(\pm\theta)\cdot\partial_x.$$
If $\theta\mapsto x_K(\theta)$ parametrizes the convex $K$ by the outward normal, then $\theta\mapsto -x_K(\theta)$ parametrizes the 
reflected convex $-K$ by the inward normal, therefore $\theta\mapsto -x_K(-\theta)$ parametrizes $-K$ with the outward normal. In terms of the vector fields $V^{\pm}_{\pm K}$, 
we note that this correspondence reads $V_{-K}^+=-V_K^-$. \textbf{From this point on of the article,} we fix a strictly compact and convex body $K$ with smooth boundary in the sense of Definition~\ref{d:def-convex} that \emph{contains $0$ in its interior, $0 \in \Int(K)$}. We will be interested in the induced vector field on $S\IT^d=\IT^d\times\IS^{d-1}$:
\begin{equation}\label{e:vector-field}
 V:=\mathbf{v}(\theta)\cdot\partial_x,\quad\text{with}\quad \mathbf{v}:=x_K.
\end{equation}
The case of the classical geodesic flow described in the introduction corresponds to $K=B_d$, where $B_d$ is the Euclidean ball of radius $1$ centered at $0$ and $\mathbf{v}(\theta)=\theta$.
\begin{rema}\label{r:interior-0} The fact that $K$ has $0$ in its interior has the following consequence that will be used later on:
\begin{equation}\label{e:interior-0}
 \text{for all }  \theta\in\IS^{d-1},\quad \mathbf{v}(\theta)\cdot\theta>0.
\end{equation}
Indeed, if there is $\theta_0\in\IS^{d-1}$ such that $\mathbf{v}(\theta_0)\cdot\theta_0\leq 0$, we may suppose without loss of generality that $\theta_0=(1,0,\ldots,0).$ This would imply that $\mathbf{v}_1(\theta_0)\leq 0$ and contradict the fact that $\IR\theta_0\cap K$ is a closed interval containing $0$ in its interior.
\end{rema}

\subsection{Finsler geometry and Hamiltonian structure} 
\label{s:Finsler-Ham}
In this section, we introduce the general translation invariant Finsler geometry in which we consider our Epstein Zeta function and Poincar\'e series. We recall how it naturally enjoys a Hamiltonian structure and  how both are linked to vector fields of the form~\eqref{e:vector-field}.

\subsubsection{Generalities on convex and Finsler geometries}
% We would like to connect our flows, convex geometry and certain basic notions from Finsler geometry. We start from a real finite dimensional vector space $\mathbb{R}^d$, and a
For $\mathsf{K} \subset \R^d$ a convex set, we start by recalling the notation for its polar set 
$$
\mathsf{K}^\circ = \{p \in \R^{d} : p\cdot x\leq 1 \text{ for all } x \in \mathsf{K} \} , 
$$
and the relation $\mathsf{K}^{\circ\circ}=\mathsf{K}$ if $0 \in \Int(\mathsf{K})$, see~\cite[Theorem~1.6.1]{Schneider14}. We also recall the definition of the support function of the convex set $\mathsf{K}$:
\begin{align}
\label{e:defh_K}
h_\mathsf{K}(p) :=  \sup\{ p\cdot x : x \in \mathsf{K} \} , \quad p\in \R^{d}  .
\end{align}
Given a convex function $g :\R^d \to \R$, its Legendre transform is defined as 
$$
g^*(p) = \sup \{p\cdot x - g(x): x \in \R^d\} , 
$$
which is a convex function on $\R^{d}$.

Next, a function $F:\mathbb{R}^{d}\mapsto \mathbb{R}$ is called a Minkowski norm if it satisfies the following properties~\cite[p.~2]{ChernShen}:
\begin{itemize}
\item $F(\lambda y)=\lambda F(y), \forall \lambda>0$,
\item $F:\mathbb{R}^{d}\setminus \{0\}\mapsto \mathbb{R}$ is smooth and for any $y\in \mathbb{R}^{d}\setminus \{0\}$, the bilinear form $B_y$
\begin{eqnarray*}
B_y(v,w)=\frac{1}{2} \frac{d^2}{dsdt}F^2(y+sv+tw)|_{s=t=0}
\end{eqnarray*} 
is positive definite.
\end{itemize}
The pair $(\mathbb{R}^{d},F)$ is called a Minkowski space. Let us list some properties of the pair $(\mathbb{R}^{d},F)$ which are consequences of the above definition~\cite[p.~3]{ChernShen}~\cite[Thm 1.2.2 p.~6]{baochern}:
\begin{enumerate}
\item $F(y)\geqslant 0$ and $F(y)=0\implies y=0$,
\item the unit sphere $\{F=1\}$ is a smooth, strictly convex hypersurface diffeomorphic to the unit Euclidean sphere,
\item $F(x+y)\leqslant F(x)+F(y)$ with equality iff $x,y$ are colinear which means the Minkowski norm satisfies the strict triangle inequality.
\end{enumerate}

We collect in the following lemma links between Minkowski norms and convex sets.
\begin{lemm}
\label{l:finsler-hamilton}
Assume that $\mathsf{K} \subset \R^{d}$ is a strictly convex compact set in the sense of Definition~\ref{d:def-convex} such that $0 \in \Int(\mathsf{K})$.
One can associate with $\mathsf{K}$ its support function $h_\mathsf{K}$ and the Minkowski norm $F_\mathsf{K}$ s.t. $\mathsf{K}=\{F_K\leqslant 1\}$.
Then one has the following relations 
\begin{itemize}
\item the norm $F_\mathsf{K} $ and support functions $h_\mathsf{K}$ are related via Legendre transformation: $F_\mathsf{K}=h_\mathsf{K}^* $,
\item the map $\mathsf{K}\mapsto h_\mathsf{K}$ exchanges polarity $\mathsf{K}\mapsto \mathsf{K}^\circ$ and the Legendre transform which reads 
$$ h_{\mathsf{K}^\circ}=h_\mathsf{K}^* .$$
\end{itemize}

%
%
%
% Then $h_\mathsf{K}$ is a Minkowski norm on $\R^{d}$ and $F := h_\mathsf{K}^* = h_{\mathsf{K}^\circ}$ is a Minkowski norm on $\R^d$.
%
%Conversely, if $F$ is a Minkowski norm on $\R^d$, then $\mathsf{K} := \{F\leq 1\} \subset \R^d$ is a strictly convex compact set in the sense of Definition~\ref{d:def-convex} with $0 \in \Int(\mathsf{K})$, $h_\mathsf{K}$ is a Minkowski norm on $\R^{d}$ and $h_\mathsf{K}^* = h_{\mathsf{K}^\circ} = F$.% is a Minkowski norm on $\R^{d}$.

If we define the Lagrangian $L(y)=\frac{1}{2}F^2(y)$ on the velocity space $\mathbb{R}^d$ and the Hamiltonian 
  \begin{align}
  \label{e:hamiltonien}
H(p)=\frac{1}{2}(F^*)^2(p) = \frac{1}{2}h_K^2(p) 
\end{align} on the dual space $\R^{d*}\simeq\R^d$, we have $H^*=L$ and $L^*=H$. Moreover, the maps $x \mapsto \nabla L (x)$ and $u \mapsto \nabla H(u)$ are bijective and inverse to each other. 
%. The corresponding Hamiltonian $H$ is a function on momentum space and its value at $p$ is exactly~\cite[eq (14.8.2) p.~407]{baochern}: 
\end{lemm}
This lemma is essentially proved in~\cite[p55-56]{Schneider14}. 
Let us quickly remind how to recover the Minkowski norm $F$ from the convex $\mathsf{K}$ we started with. We would like to realize the boundary hypersurface $\partial \mathsf{K}$ as the unit sphere $F=1$. Note that the positive definiteness of the bilinear forms $(g_u)_{u\in \{F=1\}}$ in the definition of Minkowski norms exactly means that the Gauss map $x\in \{F=1\}\mapsto \frac{\nabla F(x)}{\vert \nabla F(x)\vert}\in \mathbb{S}^{d-1}$ is invertible. For every $x\in \mathbb{R}^d$, define $F(x)=\frac{1}{t}$ for $t\geqslant 0$ s.t. $tx\in \partial \mathsf{K}$, which according to~\cite[Lemma 1.7.13]{Schneider14} is $F=h_{\mathsf{K}^\circ}$.
%. The relation between $F$ and support functions  reads~\cite[Lemma 1.7.13]{Schneider14}:
%$
%F=h_{K^\circ}
%$
%where $h_{K^\circ}$ is the support function of the polar convex set $K^\circ$~\footnote{For more details on this relation between convex geometry and Finsler geometry, we invite the reader to look at~\cite[p.~56]{Schneider14}}.
The fact that $F^* = h_\mathsf{K}$ is proved in~\cite[Eq. (14.7.4), p.~404)]{baochern}.
%To emphasize the relationship with duality, for every $p\in \mathbb{R}^{d*}$ define the dual Minkowski norm~\cite[eq (14.7.4 p.~404)]{baochern}:
%\begin{eqnarray*}
%F^*(p)=\sup_{x\in \partial K}p\cdot x=h_K(p)
%\end{eqnarray*} 
%which we can identify with the support function of $K$. 
By~\cite[Lemma 3.1.2, p.~38]{shen2001}, $F^*$ is also a Minkowski norm on $\mathbb{R}^{d}$ and its unit ball is nothing but the polar convex set $\mathsf{K}^\circ$.
Finally that the Lagrangian $L(y)=\frac{1}{2}F^2(y)$ on velocity space corresponds to the Hamiltonian $H$ on momentum space given by~\eqref{e:hamiltonien}, follows for instance from~\cite[Eq. (14.8.2), p.~407]{baochern}. For more informations on the Legendre transformation and convex geometry, see~\cite[\S 14.8]{baochern}).

\medskip
On $T\mathbb{T}^d$, we endow each tangent space $T_x\mathbb{T}^d$ with the Minkowski metric $F$, then this turns $\mathbb{T}^d$ into a Riemann--Finsler manifold $(\T^d,F)$ whose Finsler metric is translation invariant. 
The associated distance function is 
$$
d_{F} (x,y) = \inf \left\{ \int_0^1 F(\dot{\gamma}(t))dt, \gamma \in W^{1,1}([0,1];\T^d), \gamma(0)=x, \gamma(1)=y \right\} .
$$
As in the Riemannian framework and in classical mechanics, one deduces from Lemma~\ref{l:finsler-hamilton} that curves on $\T^d$ minimizing the distance $d_F$ (which we call geodesic curves of $(\T^d,F)$ as in the Riemannian case) are projections of Hamiltonian flow on $\T^d$, namely $(x(t),p(t))$ where $\dot{x} =\nabla H(p), \dot{p}=0$.
That is to say $(x(t),p(t)) = (x_0 + t \nabla H(p_0) , p_0)$. In particular, translation invariance implies that geodesics are straight lines in $\T^d$.

%We would like to show that the flow $e^{-tV}$ is exactly the generalization of the geodesic flow in the Finsler setting. As in classical mechanics, we may now form the Lagrangian $L(y)=\frac{1}{2}F^2(y)$ on the velocity space $\mathbb{R}^d$. The corresponding Hamiltonian $H$ on momentum space is given by~\eqref{e:hamiltonien}, see~\cite[eq (14.8.2) p.~407]{baochern}. 
%%\begin{eqnarray*}
%%H(p)=\frac{1}{2}(F^*)^2(p)
%%\end{eqnarray*}
%which generalizes the geodesic flow for the Euclidean metric to the Finsler setting and $H$ is the Hamiltonian we found for our vector field $V$~\footnote{For more on the Legendre transformation and convex geometry, see~\cite[section 14.8]{baochern}}. 

\subsubsection{Parametrization by the unit normal}
Let us conclude this section with a Hamiltonian interpretation of the vector field $V=\mathbf{v}(\theta)\cdot\partial_x$ defined in~\eqref{e:vector-field}. Here, given $K \subset \R^{d}$ a strictly convex compact set in the sense of Definition~\ref{d:def-convex} such that $0 \in \Int(K)$, we recall that we have defined $\mathbf{v}=x_K$ as the inverse of the Gauss map on $\partial K$.
 Recalling~\eqref{e:hamiltonien} and~\eqref{e:defh_K}, one has:
$$H_K(x,\xi):=\frac{1}{2}h_K(\xi)^2=\frac{1}{2}\left(\xi\cdot \mathbf{v}\left(\frac{\xi}{|\xi|}\right)\right)^2,\quad (x,\xi)\in T^*\IT^d\setminus\underline{0}_{\IT^d},$$
 where the second equality comes from the study of the critical points of the function $\theta\in\IS^{d-1}\mapsto \mathbf{v}(\theta)\cdot\xi$ -- see \S\ref{s:analysis} for instance or~\cite[Remark~1.7.14 p53]{Schneider14}. Recalling that the normal to $K$ at $\mathbf{v}(\theta)$ is $\theta$ by construction, we deduce that the image of $d\left(\mathbf{v}\left(\frac{\xi}{|\xi|}\right)\right)$ is orthorgonal to $\xi$ and thus the corresponding Hamiltonian vector field is given by
$$X_{H_K}=\left(\mathbf{v}\left(\frac{\xi}{|\xi|}\right)+ d\left(\mathbf{v}\left(\frac{\xi}{|\xi|}\right)\right)^T\xi\right)\cdot\partial_x=\mathbf{v}\left(\frac{\xi}{|\xi|}\right)\cdot\partial_x.$$
This looks very much like our vector field except that the level line $\mathcal{E}:=\{H=\frac{1}{2}\}$ is not equal to $\IS^{d-1}$. In view of fixing this issue, we define the following diffeomorphism:
\begin{align}
\label{e:def-Phi-K}
\Phi:(x,\theta)\in S\IT^d\mapsto \left(x,\frac{\theta}{\mathbf{v}(\theta)\cdot\theta }\right)\in\mathcal{E},
\end{align}
and one finds $\Phi^*X_{H_K}=V$. In other words, our vector fields are pullbacks on $S\IT^d$ of Hamiltonian vector fields whose expressions are given in terms of $K$ or equivalently $(\Phi^{-1})^* e^{-tX_H}\Phi^*= e^{-tV}$ where $e^{-tV}$ acts on $S\T^d$ and $e^{-tX_H}$ acts on $\mathcal{E}\subset T^*\T^d$.

\begin{rema} Note that $e^{-tV}$ preserves the contact form. Indeed, by the Lie--Cartan formula:
$$\mathcal{L}_V \theta\cdot dx=d(\mathbf{v}(\theta)\cdot\theta)+\iota_{V}\left(d\theta\wedge dx\right)=\mathbf{v}(\theta)\cdot d\theta-\mathbf{v}(\theta)\cdot d\theta=0 $$
where we used the orthogonality relation $\theta\cdot d\mathbf{v}(\theta)=0$ and the fact that $V$ is horizontal hence $\iota_Vd\theta=0$. This property allows to give another proof of identity~(\ref{e:legendrian}). In fact, since $N_+(\Sigma_1)=e^{V_{K_1}^+*}(S_0\mathbb{T}^d)$ (with $V_{K_1}^+=x_{K_1}(\theta)\cdot \d_x$), it is the transport by some contact flow of the fiber $S_0\mathbb{T}^d$ which is Legendrian, $N_+(\Sigma_1)$. Hence it is also Legendrian. 
\end{rema}

\medskip
Finally, let us notice that conversely, if we are given a Hamiltonian 
\begin{lemm}
Let $H :\R^{d} \to \R$ be such that 
\begin{enumerate}
\item $\mathcal{E}=\{H=\frac{1}{2}\}$ is compact and connected,
\item $H$ is smooth near $\mathcal{E}$ and $\nabla H \neq 0$ on $\mathcal{E}$,
\item the Hessian $D^2H(p)( \cdot , \cdot )$ is positive definite for all $p \in \mathcal{E}$ when restricted to $T_p\mathcal{E} \times T_p\mathcal{E}$,
\item \label{e:fins-ham-ch} $p\mapsto p\cdot \nabla H(p)$ is constant on $\mathcal{E}$. 
\end{enumerate}
Then there is a strictly convex set $K$ in the sense of Definition~\ref{d:def-convex}, with $0 \in \Int(K)$ and a constant $c_0>0$ such that on $\mathcal{E}$, $e^{-tX_H}= e^{-c_0 t X_{H_{K}}} = \Phi^* e^{-c_0 tV} (\Phi^{-1})^*$ with $V=\mathbf{v}(\theta)\cdot \d_x$ where $\mathbf{v}(\theta)$ associated with $K$ and $\Phi$ defined by~\eqref{e:def-Phi-K}.
\end{lemm}
As a consequence of this lemma, the analysis of the flow $e^{-tX_H}$ on the set $\mathcal{E}\subset T^*\T^d$ is equivalent to that of $e^{-tV}$ on $S\T^d$.   
Therefore, the present article includes the study of a family of completely integrable Hamiltonian flows associated with strictly convex Hamiltonians. % (in the sense $D^2H$ positive definite).
Note that the condition~\eqref{e:fins-ham-ch} is reminiscent to the homogeneity of $H$ near $\mathcal{E}$. In particular, it shows that, up to conjugation by $\Phi$, our analysis includes the case of flows associated with Hamiltonians of the form $H(p):=\frac{1}{2}p\cdot Ap$, where $A$ is any positive definite symmetric matrix.

\begin{proof}
By assumption, $\mathcal{E}$ is a smooth compact, connected and oriented hypersurface embedded in $\R^{d}$. Moreover, the convexity assumption on $H$ implies that its sectional curvatures are positive. Hence, by Hadamard-Sacksteder Theorem~\cite{Hadamard1897, Sacksteder60, DoCarmoLima69}, it is the boundary of a strictly convex body $\mathsf{K} \subset \R^{d}$ (in the sense of Definition~\ref{d:def-convex}), i.e. $\partial\mathsf{K}= \mathcal{E}$. %The convexity assumption on $H$ implies that $\mathsf{K}$ is strictly convex in the sense of Definition~\ref{d:def-convex}, and together with the compactness assumption and
Moreover, assumption~\eqref{e:fins-ham-ch} implies that $0 \in \Int (\mathsf{K})$ and $p\mapsto p\cdot \nabla H(p) = c_0>0$ for all $p \in \mathcal{E}$.
Then, according to Lemma~\ref{l:finsler-hamilton}, $F := h_{\mathsf{K}}$ is a Minkowski norm on $\R^d$ and $h_{\mathsf{K}^\circ}$ is a Minkowski norm on $\R^{d}$ such that $\mathsf{K}= \{p \in \R^{d}, h_{\mathsf{K}^\circ} (p) \leq 1\}$. Setting $H_{\mathsf{K}^\circ} := \frac12 h_{\mathsf{K}^\circ}^2$, we have, locally near $\mathcal{E}$, $\mathsf{K}=\{H\leq\frac12\}=\{H_{\mathsf{K}^\circ}\leq\frac12\}$, and, more precisely, $H - \frac12$ and $H_{\mathsf{K}^\circ} - \frac12$ are two defining functions of $\partial \mathsf{K}$. Hence, there exists a smooth non-vanishing function $f$ such that $H_{\mathsf{K}^\circ}(p) - \frac12= f(p) \left( H(p) - \frac12 \right)$. Differentiating this identity and restricting it to $\mathcal{E}$, we obtain $\nabla H_{\mathsf{K}^\circ}(p) = f (p)\nabla H(p)$ for all $p \in \mathcal{E}$. 
This implies $p \cdot \nabla H_{\mathsf{K}^\circ}(p) = f (p) p\cdot \nabla H(p)$ for all $p \in \mathcal{E}$. The Hamiltonian $H_{\mathsf{K}^\circ}$ is homogeneous of degree two, whence $p \cdot \nabla H_{\mathsf{K}^\circ}(p) = 2 H_{\mathsf{K}^\circ} =1$ for all $p \in \mathcal{E}$. We deduce that $f(p) = c_0^{-1}$ for all $p \in \mathcal{E}$. This implies $\nabla H(p)= c_0 \nabla H_{\mathsf{K}^\circ}(p)$ and hence $e^{-tX_H}= e^{-c_0 t X_{H_{\mathsf{K}^\circ}}} = \Phi^* e^{-c_0 tV} (\Phi^{-1})^*$ where $V=\mathbf{v}(\theta)\cdot \d_x$ and $\mathbf{v}(\theta)$ is associated with the convex set $\mathsf{K}^\circ=K$.
\end{proof}

\subsection{Decomposition of the tangent space of $S\IT^d$}\label{ss:riemannian}

The tangent space to a point $(x,\theta)\in S\IT^d$ decomposes in a way which is adapted to the dynamical features of our problem. First, we write
$$T_{x,\theta}S\IT^d\simeq T_x\IT^d \times T_\theta\IS^{d-1}.$$ Given $\theta\in \IS^{d-1}$, we consider a family $(e_1(\theta),\ldots, e_{d-1}(\theta))$, depending smoothly on $\theta$, and such that the family $(\theta,e_1(\theta),\ldots,e_{d-1}(\theta))$ is orthonormal and $$\text{det}\left(\theta, e_1(\theta), \ldots, e_{d-1}(\theta)\right)>0.$$
At a given point $(x,\theta)\in S\IT^d$, we define the horizontal space as
$$\mathcal{H}_{x,\theta}:=\text{Span}_{T_x\IT^d}(e_1(\theta),\ldots, e_{d-1}(\theta))\times\{0\}\subset T_x\IT^d\times\{0\}\subset T_{x,\theta}(S\IT^d).$$
Similarly, we introduce the vertical space
$$\mathcal{V}_{x,\theta}:=\{0\}\times\text{Span}_{T_\theta\IS^{d-1}}(e_1(\theta),\ldots, e_{d-1}(\theta))=\{0\}\times T_\theta\IS^{d-1}\subset T_{x,\theta}(S\IT^d).$$
Note that $\ml{V}_{x,\theta}$ is the tangent space to the submanifold $S_{x}\IT^d$, or equivalently the kernel of the tangent map of $\Pi:(x,\theta)\in S\IT^d\mapsto x\in\IT^d.$ One has then some canonical identification of all tangent fibers as:
$$T_{x,\theta}S\IT^d=\IR \theta\cdot\partial_x \oplus \mathcal{H}_{x,\theta}\oplus \mathcal{V}_{x,\theta}.$$
In the terminology of symplectic geometry, $ \mathcal{H}_{x,\theta}\oplus \mathcal{V}_{x,\theta}$ is the kernel of the Liouville (contact) form $\alpha(x,\theta,dx,d\theta):=\theta\cdot dx.$ 
This decomposition of the tangent space allows to write the following nice expression of the tangent map $D(e^{tV}) :T_{x,\theta}S\IT^d \mapsto T_{x+t\mathbf{v}(\theta),\theta}S\IT^d\simeq  T_{x,\theta}S\IT^d $ at a point $(x,\theta)$:
\begin{equation}\label{e:tangentmap}[D(e^{tV})(x,\theta)]_{\IR \theta\cdot\partial_x\oplus \mathcal{H}_{x,\theta}\oplus \mathcal{V}_{x,\theta}}=\left(\begin{array}{ccc} 1 & 0 & 0\\
                                0 & \Id & t D\mathbf{v}(\theta) \\
                                0 & 0 & \Id
                               \end{array}\right).\end{equation}
Indeed, $e^{tV}$ commutes with translations on $\mathbb{T}^d$ therefore $D(e^{tV})=\Id$ on the horizontal part 
$\theta\cdot\partial_x\oplus \mathcal{H}_{x,\theta}$. Then for the other part, we just compute the derivative along a $C^1$ curve: $s\mapsto \theta(s)\in S_x\mathbb{T}^d$, we have $\frac{d}{ds}(x+t\mathbf{v}(\theta(s)),\theta(s))=(tD\mathbf{v}(\theta)(\theta^\prime(s)),\theta^\prime(s))$ which yields the full matrix of the differential.

\begin{rema}\label{r:stable-bundle}
Recall that $\mathbf{v}:\theta\in\IS^{d-1}\rightarrow \partial K$ is a diffeomorphism. Moreover, by construction, $D\mathbf{v}(\theta)$ can be identified with an isomorphism of $T_\theta\IS^{d-1}$, as the normal to $\partial K$ at $\mathbf{v}(\theta)$ is given by $\theta$. From this expression of the tangent map, we deduce that every vector in $\mathcal{H}_{x,\theta}\oplus \mathcal{V}_{x,\theta}$ is stretched in the direction of the horizontal bundle under the action of the tangent map as $t\rightarrow\pm\infty$. See Figure~\ref{f:kernel}.
\end{rema}

\begin{rema}\label{r:transverse-vectorfield} Recalling~\eqref{e:interior-0}, one knows that $V$ is transversal to $\mathcal{H}_{x,\theta}\oplus \mathcal{V}_{x,\theta}$.
\end{rema}

\begin{figure}[ht]
\includegraphics[scale=0.55]{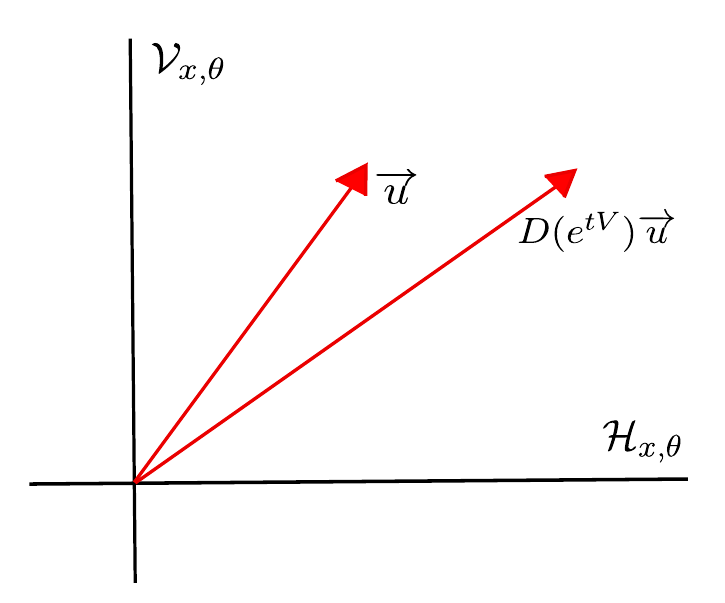}
\centering
\caption{\label{f:kernel}Action of the tangent map on the kernel of the contact form.}
\end{figure}

\begin{rema} We will denote by $E_0^*=\IR\alpha \subset T^*S\IT^d$ the annihilator of $\ml{H}\oplus \ml{V}$, by $\ml{H}^*\subset T^*S\IT^d$ the annihilator of $\IR V\oplus\ml{V}$ and by $\ml{V}^*\subset T^*S\IT^d$ the annihilator of $\IR V\oplus\ml{H}$. The action of the tangent map on $T^*S\IT^d$ reads
\begin{equation}\label{e:cotangentmap}[D(e^{tV})(x,\theta)^T]^{-1}_{\IR \alpha(x,\theta)\oplus \mathcal{H}_{x,\theta}^*\oplus \mathcal{V}_{x,\theta}^*}=\left(\begin{array}{ccc} 1 & 0 & 0\\
                                0 & \Id & 0 \\
                                0 & -tD\mathbf{v}(\theta)^T & \Id
                               \end{array}\right).\end{equation}
\end{rema}

\subsection{Admissible submanifolds}\label{ss:normal}
In view of defining our zeta functions, we will need the following admissibility property:
\begin{def1} 
\label{d:admissible}
We say that $\Sigma_{ 1}\subset\IT^d$ is \emph{admissible} if there exists a strictly convex and compact subset $K_{ 1}\subset \IR^d$ with smooth boundary such that 
$$\mathfrak{p}(\partial K_{ 1})=\Sigma_{1}.$$
%where $\partial K_{{\red 1}}=K_{{\red 1}}\setminus \mathring{K}_{{\red 1}}.$
\end{def1}
This definition includes the case where $\Sigma_{ 1}$ is reduced to a point. We also observe that the map $\mathfrak{p}:\partial K_{1}\rightarrow\IT^d$ is a smooth immersion but it is not necessarily injective, i.e. $\Sigma_{1}$ may have selfintersection points. Using the inverse of the Gauss map~\eqref{e:gaussmap}, one can then define the direct/indirect normal bundles to $\Sigma$:
\begin{equation}\label{e:gauss-coordinates}
N_\pm(\Sigma_{ 1}):=\left\{\left(\tilde{x}(\theta):=\mathfrak{p}\circ x_{K_{ 1}}(\theta),\pm\theta\right):\ \theta\in\IS^{d-1}\right\},\quad N(\Sigma_{ 1})=N_{+}(\Sigma_{ 1})\cup N_{-}(\Sigma_{ 1}),
\end{equation}
where $\theta\in\IS^{d-1}\mapsto\tilde{x}(\theta)\in\Sigma_1\subset \IT^d$ is the parametrization of $\Sigma_1$ by its normal. %and the indirect normal bundle 
%$$N_-(\Sigma_{ 1}):=\left\{\left(\tilde{x}(\theta):=\mathfrak{p}\circ x_{K_{1}}(\theta),-\theta\right):\ \theta\in\IS^{d-1}\right\}.$$ 
Even if $\Sigma_{1}$ is not a proper submanifold (as it may have selfintersection points), $N(\Sigma_{1})$ and $N_{\pm}(\Sigma_{1})$ are smooth, compact and embedded submanifolds of $\IT^d\times\IS^{d-1}$. Using the conventions of \S\ref{ss:riemannian}, one has
\begin{lemm}
\label{l:legendrian-transversal} Let $\Sigma_{1}\subset\IT^d$ be admissible. Then, for every $(x,\theta) = (\tilde{x}(\theta),\theta)\in N_{\pm}(\Sigma_{1})$,
\begin{align}%\label{e:tangent-space}
%& T_{\tilde{x}(\theta),\theta}N_+(\Sigma_{{\red 1}})=\left\{(D\tilde{x}(\theta)v,v): v\in \ml{V}_{\tilde{x}(\theta),\theta}\right\}\subset \IR V(\tilde{x}(\theta),\theta)\oplus\ml{H}_{\tilde{x}(\theta),\theta}\oplus \ml{V}_{\tilde{x}(\theta),\theta} ,\\ 
\label{e:legendrian}
&T_{x,\theta}N_\pm (\Sigma_{ 1})\subset \mathcal{H}_{x,\theta}\oplus \mathcal{V}_{x,\theta}, \\
\label{e:transversal}
&T_{x,\theta}S\IT^d=\IR  \theta\cdot\partial_x\oplus \ml{H}_{x,\theta}\oplus T_{x,\theta}N_\pm(\Sigma_{1}).
 \end{align}
 \end{lemm}
In the terminology of symplectic geometry,~\eqref{e:legendrian} says that $N_\pm(\Sigma_{1})$ is a \emph{Legendrian} submanifold as its tangent space lies in the kernel of the Liouville contact form. 
Property~\eqref{e:transversal} is a transversality property, see Figure~\ref{f:tangent}. It says that our unit normal bundle is never tangent to the horizontal bundle inside $S\mathbb{T}^d $. %Recalling Remark~\ref{r:stable-bundle}, the horizontal bundle coincides with the stable and unstable bundles in the case of flat tori. Hence, Property~\eqref{e:transversal}  agrees with the transversality assumption appearing in Margulis' works on Anosov flows~\cite[Ch.~7]{Margulis04}. 
\begin{figure}[ht]
\includegraphics[scale=0.6]{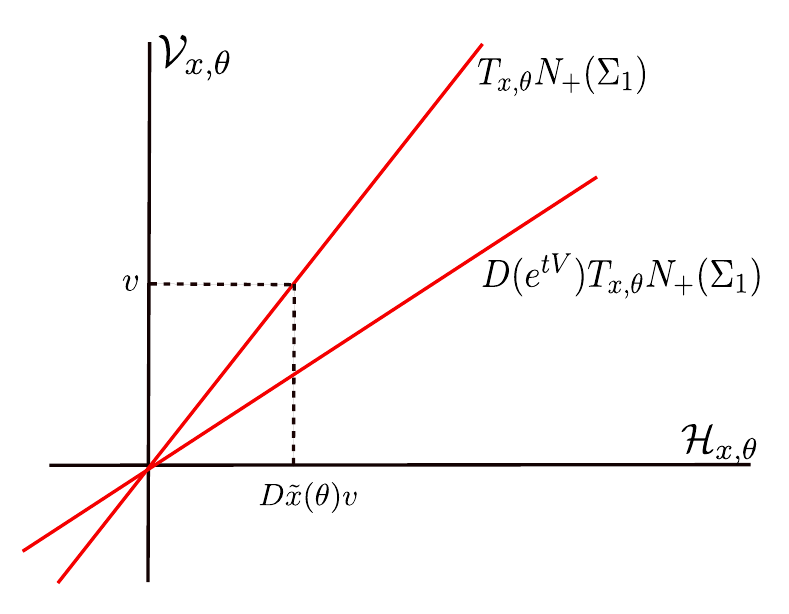}
\centering
\caption{\label{f:tangent}Tangent space to $N_+(\Sigma_{ 1})$.}
\end{figure}
\begin{proof}[Proof of Lemma~\ref{l:legendrian-transversal}] We start proving~\eqref{e:legendrian} and only discuss the case of $N_+(\Sigma_{1})$ (the other case can be handled similarly). We recall that $N_+(\Sigma_{1})$ is defined in~\eqref{e:gauss-coordinates}.  
In particular, the tangent set to a point $(\tilde{x}(\theta),\theta)$ is given by
\begin{equation}\label{e:tangent-space}
T_{\tilde{x}(\theta),\theta}N_+(\Sigma_{1})=\left\{(D\tilde{x}(\theta)v,v): v\in \ml{V}_{\tilde{x}(\theta),\theta}\right\}\subset \left(  \IR \theta\cdot\partial_x\oplus\ml{H}_{\tilde{x}(\theta),\theta} \right) \oplus \ml{V}_{\tilde{x}(\theta),\theta} 
\end{equation}
Thus, in order to prove~\eqref{e:legendrian}, we need to check that $D\tilde{x}(\theta)v \perp V(\tilde{x}(\theta),\theta)$ for every $v\in \ml{V}_{\tilde{x}(\theta),\theta}$. To see this, we only need to discuss the case where $D\tilde{x}(\theta)v\neq 0$. By construction, this means that $Dx_{K_{ 1}}(\theta)v\neq 0$. By definition, such a vector is an element in $T_{x_{K_1}(\theta)}\partial K_{1}$ and thus orthogonal to $\theta\cdot\partial_x$ which is the outward normal to $\partial K_{ 1}$ at $x_{K_{1}}(\theta)$. This implies that $D\tilde{x}(\theta)v$ is orthogonal to $V(\tilde{x}(\theta),\theta)$, hence the conclusion~\eqref{e:legendrian}.

Finally, recalling the decomposition of the tangent space~\eqref{e:tangent-space}, these Legendrian submanifolds verify the transversality property~\eqref{e:transversal}.
\end{proof}

\subsection{Epstein zeta functions and Poincar\'e series for admissible submanifolds}\label{ss:relation-resolvent-series}
Let now $\Sigma_1$ and $\Sigma_2$ be two admissible subsets of $\IT^d$ and let $\sigma_1$ and $\sigma_2$ be two elements in $\{\pm\}$. The ultimate goal of the present section is to prove Lemma~\ref{l:starting-point} stating finiteness of the set of certain geodesic curves of length $\leq T$, as well as an a priori upper bound on its cardinal, namely $\mathcal{O}(T^d)$. In turn, this will imply that generalized Epstein functions like~\eqref{e:epstein-zeta-intro} are well defined for $\Re(s)>d$ and that Poincar\'e series like~\eqref{e:poincare-series-intro} are well-defined for $\Re(s)> 0$.
 This relies on a reformulation of length of geodesics in terms of the geodesic flow on $S\T^d$.
 We follow this program in the general context of Finsler invariant structures on the torus described in Section~\ref{s:Finsler-Ham}, that is to say for general flows $(e^{-tV})$ associated with a strictly convex compact set in the sense of Definition~\ref{d:def-convex} such that $0 \in \Int(K)$. 
 
To this aim, we define for $t>0$,
\begin{align}
\label{e:def-Et}
\mathcal{E}_t(\Sigma_1,\Sigma_2) := N_{\sigma_1}(\Sigma_1)\cap e^{tV}(N_{\sigma_2}(\Sigma_2)) \subset S\T^d .
\end{align}
and describe the set of times $t\in\IR_+$ such that $\mathcal{E}_t(\Sigma_1,\Sigma_2) \neq\emptyset$.
Note that the orientations $\sigma_i$ are implicit in this notation.

\subsubsection{A priori bounds on the number of intersection points}
The first basic statement concerns finiteness of $\mathcal{E}_t(\Sigma_1,\Sigma_2)$ together with  the set of times $t$ for which $\mathcal{E}_t(\Sigma_1,\Sigma_2)$ is nonempty.
\begin{lemm}
\label{l:basic-simple}
There exists some $T_0>0$ such that, for every $t\geq T_0$, $\mathcal{E}_t(\Sigma_1,\Sigma_2)$ is a (possibly empty) finite set. Moreover, setting
$$m_{\Sigma_1,\Sigma_2}(t):=\sharp \, \mathcal{E}_t(\Sigma_1,\Sigma_2)<\infty, \quad t \geq T_0 ,$$
 for any $[a,b]\subset [T_0,+\infty)$,
$$\{t\in[a,b]:\  \mathcal{E}_t(\Sigma_1,\Sigma_2)\neq \emptyset \} = \{t\in[a,b]:\ m_{\Sigma_1,\Sigma_2}(t)\neq 0\} ,$$
is as well a finite set.
\end{lemm}
Note that Lemma~\ref{l:starting-point} is then a reformulation of Lemma~\ref{l:basic-simple}. The latter states that after some time $T_0$, all possible intersections in~\eqref{e:def-Et} are transverse and thus our counting problem is well--posed. This property can be visualized as follows. When pushing a convex hypersurface $\partial K_1$ by the flow of $V$ in the cover $\mathbb{R}^d$, it behaves more and more like a flat hypersurface; Therefore, after some time $T_0$, its sectional curvatures will be strictly less than the sectional curvatures of $\partial K_2$. In turn, this implies that the intersections occurring in $S\mathbb{T}^d$ of the corresponding normal bundles will become transverse right after $T_0$.

\begin{proof}[Proof of Lemma~\ref{l:basic-simple}]
Given $t_0\in\IR_+$, one can always find some $\epsilon>0$ such that 
$$\bigcup_{t\in(t_0-\eps,t_0+\eps)}e^{tV}(N_{\sigma_2}(\Sigma_2))$$
is a smooth submanifold (with boundary) of dimension $d$ inside $S\IT^d$. Moreover, thanks to~\eqref{e:transversal} and to~\eqref{e:tangentmap}, we know that, for $t_0$ large enough, one can find $\eps>0$ such that, for every $(x,\theta)\in N_{\sigma_1}(\Sigma_1)\cap \cup_{t\in(t_0-\eps,t_0+\eps)}e^{tV}(N_{\sigma_2}(\Sigma_2))$,
\begin{equation}\label{e:transversality}T_{x,\theta}S\IT^d= T_{x,\theta}N_{\sigma_1}(\Sigma_1)\oplus T_{x,\theta}\left(\cup_{t\in(t_0-\eps,t_0+\eps)}e^{tV}(N_{\sigma_2}(\Sigma_2))\right).\end{equation}
In other words, the two submanifolds are transversal and, by compactness, they intersect at only finitely many points.
Note that the transversal intersection implies that the boundary of $\cup_{t\in(t_0-\eps,t_0+\eps)}e^{tV}(N_{\sigma_2}(\Sigma_2))$ does not meet $N_{\sigma_1}(\Sigma_1)$.
\end{proof}

We now provide with an a priori polynomial upper bound on $m_{\Sigma_1,\Sigma_2}(t)$. This is essential to ensure that Epstein functions like~\eqref{e:epstein-zeta-intro} and Poincar\'e series like~\eqref{e:poincare-series-intro} do converge for $\Re(s)$ large enough.
\begin{lemm}\label{l:apriorigrowth} Let $\Sigma_1$ and $\Sigma_2$ be two admissible subsets of $\IT^d$. Then, for $T_0$ as in Lemma~\ref{l:basic-simple}, there is $C_0>0$ such that, for every $T\geq T_0$,
$$\sum_{T\leq t\leq T+1} m_{\Sigma_1,\Sigma_2}(t)\leq C_0 T^{d-1}.$$ 
In particular, as $T\rightarrow+\infty$,
$$\sum_{T_0\leq t\leq T} m_{\Sigma_1,\Sigma_2}(t)=\ml{O}(T^d).$$ 
\end{lemm}

\begin{proof} In order to obtain such an upper bound, it is more convenient to lift the problem to $\mathbb{R}^d$ and to recall that the lift of $\Sigma_j$ is by definition the smooth boundary of a compact and strictly convex set.
As a consequence, we have for any $t>0$ 
$$ 
m_{\Sigma_1,\Sigma_2}(t) =\sharp\, \mathcal{E}_t(\Sigma_1,\Sigma_2) = \sharp \, \mathcal{E}_t (\d K_1,\d K_2) , 
$$ where $$
 \mathcal{E}_t (\d K_1,\d K_2) : = \left\{(x,\theta)\in N_{\sigma_1}(\partial K_1+2\pi\IZ^d): (x-t\mathbf{v}(\theta),\theta)\in N_{\sigma_2}(\partial K_2)\right\} \subset S\R^d. $$
Recalling that, when lifting the problem to $\mathbb{R}^d$,
$$N_\pm(\partial K_j):=\left\{(x_{K_j}(\theta),\sigma_j\theta):\theta\in\IS^{d-1}\right\} =\left\{( x_{K_j}(\sigma_j\theta),\theta):\theta\in\IS^{d-1}\right\}  , $$
 we notice that 
 $$
 (x,\theta) \in  \mathcal{E}_t (\d K_1,\d K_2) \implies t \mathbf{v}(\theta) +x_{K_1}(\sigma_1 \theta)- x_{K_2}(\sigma_2\theta)  \in 2\pi\IZ^d ,
 $$
 whence 
$$
\sharp \, \mathcal{E}_t (\d K_1,\d K_2) = \sharp \left\{\theta\in\IS^{d-1}: t\left( \mathbf{v}(\theta)+\frac{ x_{K_1}(\sigma_1 \theta)- x_{K_2}(\sigma_2\theta)}{t}\right) \in 2\pi\IZ^d\right\} .
 $$
This implies that
\begin{multline*}
\sum_{T\leq t\leq T+1} \sharp \, \mathcal{E}_t (\d K_1,\d K_2)\\
\leq \sharp\left\{\xi\in\IZ^d:\exists T\leq t\leq T+1,\ \exists\theta\in\IS^{d-1},\ 2\pi\xi=t \mathbf{v}(\theta)+ x_{K_1}(\sigma_1 \theta)-x_{K_2}(\sigma_2\theta)\right\}.
\end{multline*}
As we made the assumption that $K$ contains $0$ in its interior, there exists some $c_0,C>0$ (depending only on $K_1$, $K_2$ and $K$) such that, for $T>0$ large enough,
$$\sum_{T\leq t\leq T+1} \sharp \, \mathcal{E}_t (\d K_1,\d K_2)\leq \sharp \left(2\pi\IZ^d \cap \mathcal{C}_T \right) \leq C \Vol_{\R^d}(\mathcal{C}_T ) , \quad \text{with} \quad   \mathcal{C}_T = (T+1+c_0)K \setminus (T-c_0)K .
$$

%For $t \geq 1$, we have $\left| \frac{{\red x_{K_1}}(\sigma_1 \theta)-{\red x_{K_2}}(\sigma_2\theta)}{t}\right| \leq \frac{C_0}{{\red t}}$ so that 
%$$
%\sum_{T\leq t\leq T+1} \sharp \, \mathcal{E}_t (\d K_1,\d K_2) \leq \sharp \left(2\pi\IZ^d \cap \mathcal{C}_T \right) \leq C \Vol_{\R^d}(\mathcal{C}_T ) , \quad \text{with} \quad   \mathcal{C}_T = {\red (T+1+C_0)K \setminus (T-C_0)K }.
%$$
We finally deduce from~\cite[Th.~1]{Herz62b} that, for $T_0>0$ as in Lemma~\ref{l:basic-simple}, 
$$
\sum_{T\leq t\leq T+1} m_{\Sigma_1,\Sigma_2}(t)  = \sum_{T\leq t\leq T+1}\sharp \, \mathcal{E}_t (\d K_1,\d K_2)  \leq C T^{d-1} , \quad T\geq T_0 ,
$$ 
which concludes the proof of the lemma.
\end{proof}

\subsubsection{Generalized Epstein zeta functions and Poincar\'e series} 
\label{general-ep-series-poin}
We may now come back to Epstein zeta functions and Poincar\'e series. We fix $T_0>0$ large enough to ensure that Lemma~\ref{l:basic-simple} (and thus~\ref{l:apriorigrowth}) apply. We also fix $\beta=\beta_0+df$ with 
$$\beta_0\in H^1(\IT^d,\IR):=\left\{\sum_{j=1}^d\beta_jdx_j:\ (\beta_1,\ldots,\beta_d)\in\IR^d\right\}.$$
In particular, for such a $T_0$ and such a $\beta$, we can define, for $\text{Re}(s)>d$, the \emph{generalized Epstein zeta function} as
\begin{align}
\label{e:def-zeta-general}
\zeta_\beta(K_2,K_1,s):=\sum_{t>T_0:\mathcal{E}_t(\Sigma_1,\Sigma_2)\neq\emptyset}t^{-s}\left(\sum_{(x,\theta) \in \mathcal{E}_t(\Sigma_1,\Sigma_2)}e^{-i\int_{-t}^0\beta(V)(x+\tau\mathbf{v}(\theta),\theta)|d\tau|}\right),
\end{align}
where $\mathcal{E}_t(\Sigma_1,\Sigma_2)$ is defined in~\eqref{e:def-Et} and is a finite set according to Lemma~\ref{l:basic-simple}. Lemma \ref{l:apriorigrowth} ensures that this defines a holomorphic function in $\{\text{Re}(s)>d\}$.

Similarly, for $\text{Re}(s)>0$, we define the \emph{generalized Poincar\'e series} as
\begin{align}
\label{e:def-Z-general}
\mathcal{Z}_\beta(K_2,K_1,s):=\sum_{t>T_0:\mathcal{E}_t(\Sigma_1,\Sigma_2)\neq\emptyset}e^{-st}\left(\sum_{(x,\theta)\in \mathcal{E}_t(\Sigma_1,\Sigma_2)}e^{-i\int_{-t}^0\beta(V)(x+\tau\mathbf{v}(\theta),\theta)|d\tau|}\right).
\end{align}
 Again Lemma~\ref{l:apriorigrowth} shows that this defines a holomorphic function in $\{\text{Re}(s)>0\}$.

\begin{rema}
\label{r:zeta-zeta-Z-Z}
 When $V=\theta\cdot\partial_x$ (i.e. for $K=B_d$ in the definition of $V$) and except for the role of $K_1$ and $K_2$ that are reversed compared with the introduction, these two functions are exactly the two series defined in~\eqref{e:epstein-zeta-intro} and~\eqref{e:poincare-series-intro} in the introduction. 
%{\vio
%In the general case $V=\mathbf{v}(\theta)\cdot\partial_x$, where $\mathbf{v}$ is associated to a strictly convex compact set in the sense of Definition~\ref{d:def-convex} with $0 \in \Int(K)$, then~\eqref{e:def-zeta-general}
% resp.~\eqref{e:def-Z-general} also rewrite as in~\eqref{e:epstein-zeta-intro} resp.~\eqref{e:poincare-series-intro}. Indeed, they are equal to~\eqref{e:epstein-zeta-intro} resp.~\eqref{e:poincare-series-intro}, in which $\mathcal{P}_{K_1,K_2}$ (resp. $\ell(\gamma)$) denotes the set of geodesic curves (resp. their length) for the translation invariant Finsler metric $F=h_K$ (see Lemma~\ref{l:finsler-hamilton} and Section~\ref{s:Finsler-Ham}) that are directly (with respect to the choice $\epsilon_1$ and $\epsilon_2$ respectively) orthogonal to $K_1$ and $K_2$ for the Euclidean inner product.
%}
\end{rema}

\begin{rema}
 When $K_1=K_2=\{0\}$, the times $t$ appearing in these zeta functions correspond to dilation parameters $t$ such that $tK$ intersects $2\pi\IZ^d$, with $K$ being the convex subset used to define $V=\mathbf{v}(\theta)\cdot\partial_x$. 
\end{rema}

\subsection{Sums of convex subsets}
\label{ss:convex-schneider}
Let $K$ be a compact and convex subset of $\IR^d$. In this paragraph, it is not necessarily strictly convex or with a smooth boundary. Following~\cite[\S1.7]{Schneider14}, we define the supporting hyperplane to $K$ with exterior normal $v\in\IR^d\setminus\{0\}$ as
$$H(K,v):=\left\{w\in\IR^d:  v \cdot w=\max_{x\in K} v \cdot x\right\}.$$
Note that the maximum in this definition is necessarily attained at a point $x_0\in\partial K$. In particular, if a point $x$ lies in $H(K,v)\cap K$, then it belongs to $\partial K$. For such a point, $v$ is called an outward normal vector of $K$ at $x$. Then, the normal bundle $N_x(K)$ to $K$ at the point $x$ consists in the collection of all the outward normal vectors of $K$ at $x$ together with the zero vector~\cite[\S 2.2]{Schneider14}. 
\begin{rema}
Note that for strictly convex smooth domains, as considered in the present article, $N_x(K)$ is a one-dimensional cone. This is not  necessarily the case for general convex sets, see for instance $K:=\{(x,y)\in\IR_+^2: x+y\leq 1\}$. 
%where the reader can convince himself that, at the vertices of the triangle, the normal bundle contains subcones of the tangent fiber.{\vio M: je comprends pas cette phrase ; mais peut etre pas la peine de trop s'etendre sur les convexes non lisses...}
\end{rema}
In that manner, we can extend the definition given in \S\ref{ss:convex} to any compact and convex subset of $\IR^d$:
$$N_+(K):=\bigcup_{x\in K}\{(x,v): v\in N_x(K)\}\cap\IS^{d-1}.$$
Given two compact and convex subsets $K$ and $L$ of $\IR^d$, one has according to~\cite[Th.~2.2.1]{Schneider14}
\begin{equation}\label{e:sum-normal-bundle}\forall (x,y)\in K\times L,\quad N_{x+y}(K+L)=N_x(K)\cap N_y(L)\end{equation}
In particular, if $K$ and $L$ are two strictly convex bodies with smooth boundaries $N_{x+y}(K+L)$ is not reduced to $0$ if and only if the outward unit normal vectors at $x\in \partial K$ and $y\in\partial L$ coincide.  
We may summarize the formula for the normal bundle of the sum of two strictly convex subsets as
\begin{equation}
N_+\left(K+L \right)=\{(x+y;\xi); (x;\xi)\in N_+(K),(y;\xi)\in N_+(L)  \} .
\end{equation}
%where we only take the outgoing normals. 
Thanks to~\eqref{e:sum-normal-bundle}, we deduce the following fact of independent interest 
\begin{lemm}\label{l:sum-convex}
Let $K_1,K_2$ be two  strictly convex subsets of $\mathbb{R}^d$ with smooth boundaries $\partial K_1,\partial K_2$ that are parametrized by their outward normals through the maps $ x_{K_1},x_{K_2}:\mathbb{S}^{d-1}\mapsto \mathbb{R}^d$. Then,
the Minkowski sum $K_1+K_2$ is such that its boundary $\partial \left( K_1+K_2 \right)$ is parametrized by the sum: 
$$x_{K_1}+ x_{K_2}:\mathbb{S}^{d-1}\mapsto \mathbb{R}^d.$$
\end{lemm}
\begin{rema}\label{r:remark-mixed-volume}
With the conventions of the previous paragraphs, the boundary of the compact convex set $K_2-K_1+tK$ is parametrized by the map $\theta\in\IS^{d-1}\mapsto x_{K_2}(\theta)- x_{K_1}(-\theta)+t \mathbf{v}(\theta)\in\IR^d.$
\end{rema}

\section{Correlation functions for differential forms}
\label{s:correlation-diff-forms}
In view of defining our functional setup for the vector field $V$ and of its applications to convex geometry, we need to describe as precisely as possible the long time behaviour of the flow $e^{tV}$ acting on differential forms of $S\IT^d$. More precisely, given $(\varphi,\psi)\in\Omega^{2d-1-k}(S\IT^d)\times\Omega^k(S\IT^d)$ (where $\Omega^l(S\IT^d)$ is the space of differential of degree $l$), we aim at describing the correlation function
\begin{equation}\label{e:correlation}\Cor_{\varphi,\psi}(t):=\int_{S\IT^d}\varphi\wedge e^{-tV*}(\psi),\quad\text{as}\quad t\rightarrow+\infty.\end{equation}
In this section, we will show how to write a Fourier decomposition of such integrals in view of reducing our analysis to the study of classical oscillatory integrals. Using the notion of current~\cite[Ch.IX]{Schwartz66}, we will also explain how such correlation functions naturally appears when studying the zeta functions of Section~\ref{s:background-convex}.

%which plays the role of distributions for differential forms. Among other things, we define their Fourier transform in the setting of $S\IT^d$ following~\cite[Ch.IX, \S6]{Schwartz66}. In particular, we fix some conventions that will be used all along  this work and we describe a few concrete examples that will be useful in view of our applications to convex geometry. This preliminary analysis is then used to rewrite the correlation function~\eqref{e:correlation} using the Fourier decomposition of $\varphi$ and $\psi$. This is the content of Lemma~\ref{l:reduction} which allows to reduce the problem to estimating some classical oscillatory integrals.

\subsection{Some conventions}
\label{ss:differential}

Given an open subset $U$ of $\IS^{d-1}$, we will denote by $\Omega^k_m(\IT^d\times U)$ the space of (complex valued) differential forms~\cite[Ch.~8]{Lee09} of degree $k$ compactly supported in $\IT^d\times U$ and having $\mathcal{C}^m$-regularity with respect to the variable $(x,\theta)\in\IT^d\times U$. When $U=\IS^{d-1}$, we will just write $\Omega^k(S\IT^d)$. For smooth forms, we will use the standard convention
$$\Omega^k(\IT^d\times U)=\bigcap_{m\geq 0}\Omega^k_m(\IT^d\times U),  \quad \Omega^k(S\IT^d)=\bigcap_{m\geq 0}\Omega^k_m(S\IT^d).
$$
We will decompose differential forms as follows:
$$\psi(x,\theta,dx,d\theta)=\sum_{I\subset\{1,\ldots,d\}, |I|\leq k}dx^I\wedge\psi^{I}(x,\theta,d\theta),$$
where $dx^I=dx_{i_1}\wedge dx_{i_2}\ldots \wedge dx_{i_l}$ for some $I=\{i_1<i_2<\ldots<i_l\}\subset\{1,\ldots,d\}$ and where $\psi^I:x\in\IT^d\mapsto \psi^I(x,\theta,d\theta)\in\Omega^{k-|I|}(U)$.

\begin{rema}
 The convention $dx$ and $d\theta$ in the arguments of $\psi\in\Omega^k_m(S\IT^d)$ will always indicate that we are considering differential forms in the $x$ and $\theta$ variables.
\end{rema}

Similarly to the case of distributions, one can define, for any $k \in \{0,\dots ,2d-1\}$, the set $\mathcal{D}^{\prime k}(\IT^d\times U)$ of currents of degree $k$~\cite[Ch.IX, \S2]{Schwartz66} as the topological dual to $\Omega^{2d-1-k}(\IT^d\times U)$
with respect to bilinear pairing
$$(\psi_1,\psi_2)\in\Omega^{k}(S\IT^d)\times \Omega^{2d-1-k}(S\IT^d)\mapsto \langle \psi_1,\psi_2\rangle:=\int_{S\IT^d}\psi_1\wedge\psi_2 \in\IC.$$
In the following, we let $\mathbf{V}=\mathcal{L}_V$ be the operator acting as the \textbf{Lie derivative} along $V$ on differential forms (or currents). In particular, we write the action by pullback as $e^{t\mathbf{V}}=e^{tV*}$.

Using the structure of the torus, we can decompose differential forms (and thus currents) into Fourier series along the $x$-variable. To that aim, we set, for every $\xi\in \IZ^d$,
$$\mathfrak{e}_\xi(x):=\frac{e^{i\xi\cdot x}}{(2\pi)^{\frac{d}{2}}}.$$
We can decompose any current $u$ in $\ml{D}^{\prime k}(S\IT^d)$ as
\begin{equation}\label{e:dec-fourier-diff-form}u(x,\theta,dx,d\theta)=\sum_{\xi\in\IZ^d}\sum_{I\subset\{1,\ldots,d\}}\mathfrak{e}_\xi(x)dx^I\wedge\widehat{u}^I_\xi(\theta,d\theta),
\end{equation}
where each $\widehat{u}^I_\xi$ is a current of degree $k-|I|$ on $\IS^{d-1}$ and $ \widehat{u}^I_\xi=0 $ if $\vert I\vert>k$. For the sake of compactness, we also set 
$$\pi_\xi^{(k)}(u)(\theta,dx,d\theta):=\sum_{I\subset\{1,\ldots,d\}}dx^I\wedge\widehat{u}^I_\xi(\theta,d\theta),$$
and
$$\Pi_\xi^{(k)}(u)(x,\theta,dx,d\theta):=\sum_{I\subset\{1,\ldots,d\}}\mathfrak{e}_\xi(x)dx^I\wedge\widehat{u}^I_\xi(\theta,d\theta)
= \mathfrak{e}_\xi(x)\pi_\xi^{(k)}(u)(\theta,dx,d\theta)
.$$
These operators are projectors in the sense that
\begin{equation}\label{e:projector-torus}\forall \xi,\xi'\in\IZ^d,\quad \Pi_\xi^{(k)}\Pi_{\xi'}^{(k)}=\delta_{\xi,\xi'}\Pi_\xi^{(k)},
 \end{equation}
where $\delta_{\xi,\xi'}$ is the Kronecker symbol. These conventions allow to decompose any element $u\in\ml{D}^{\prime k}(S\IT^d)$ as follows 
$$u=\sum_{\xi\in\IZ^d}\Pi_\xi^{(k)}(u)=\sum_{\xi\in\IZ^d}\pi_\xi^{(k)}(u)\mathfrak{e}_\xi.$$
\begin{def1}\label{r:norm} Suppose that we are given a Sobolev or H\"older type norm $\|.\|_{\ml{B}(\IS^d)}$ on some Banach space $\ml{B}(\IS^d)$ continuously included in 
$$\bigoplus_{l=0}^{d-1}\ml{D}^{\prime l}(\IS^{d-1}).$$
If we identify $ \pi_{\xi}^{(k)}(u)$ with some vector valued distribution $ (\widehat{u}^I_\xi)_{I\subset\{1,\dots,k\}}$ on $\mathbb{S}^{d-1}$, we then define
$$\|\pi_{\xi}^{(k)}(u)\|_{\ml{B}(\IS^{d-1})}:=\sup\left\{\|\widehat{u}_{\xi}^I\|_{\ml{B}(\IS^{d-1})}:\ I\subset\{1,\ldots,d\},\ |I|\leq k\right\}.$$
Similarly, for some open set $U$ of $\IS^{d-1}$, we set
$$\|\pi_{\xi}^{(k)}(u)\|_{\ml{B}(U)}:=\sup\left\{\|\widehat{u}_{\xi}^I\|_{\ml{B}(U)}:\ I\subset\{1,\ldots,d\},\ |I|\leq k\right\}.$$
\end{def1}

%For any $k \in \{0,\dots ,2d-1\}$, recall that $\Omega^{k}(S\IT^d)$ has a natural structure of Fr\'echet space since $S\T^d$ is compact.
%We now define the set of currents as the topological dual to differential forms~\cite[Ch.IX, \S2]{Schwartz66}.
%\begin{def1}
% The set of currents of degree $0\leq k\leq 2d-1$ on $S\IT^d$ is the topological dual of $\Omega^{2d-1-k}(S\IT^d)$ with respect to bilinear pairing
%$$(\psi_1,\psi_2)\in\Omega^{k}(S\IT^d)\times \Omega^{2d-1-k}(S\IT^d)\mapsto \langle \psi_1,\psi_2\rangle:=\int_{S\IT^d}\psi_1\wedge\psi_2 \in\IC,$$ and the Fr\'echet topology of $\Omega^{2d-1-k}(S\IT^d)$. We denote this set by $\ml{D}^{\prime k}(S\IT^d)$.
%\end{def1}
%{\red More generally, }
% More concretely, given any (small enough) open subset $U$ of $\IS^{d-1}$ and given any smooth orthonormal family $\theta\in U\mapsto (e_1(\theta),\ldots, e_{d-1}(\theta))$ as above, an element $u$ in $\ml{D}^{\prime k}(S\IT^d)$ can be written inside $\IT^d\times U$ as
%$$u(x,\theta,dx,d\theta)=\sum_{I\subset\{1,\ldots,d\}}\sum_{J\subset\{1,\ldots,d-1\}}\sum_{|I|+|J|=k}u^{I,J}(x,\theta)dx^I\wedge e_J^*(\theta,d\theta),$$
%where $u^{I,J}(x,\theta)$ is a distribution on $\IT^{d}\times U$, i.e. an element in $\ml{D}^\prime(\IT^d\times U,\IC)=\ml{D}^{\prime 0}(\IT^d\times U,\IC)$.
%\begin{rema}
% Note that in the extremal cases $k=0,2d-1$, the restriction to an open subset $U$ is not necessary in this discussion as there is a canonical volume form on $S\IT^d$ -- see~\S\ref{ss:volume-sphere}.
%\end{rema}

\subsection{Orientation conventions}\label{ss:volume-sphere} 

We note that we implicitely fix an orientation on $\IT^d$ by fixing the volume form $dx_1\wedge\ldots\wedge dx_d$ which can be identified with the Lebesgue measure on $\IT^d$. When we want to insist on the fact that we view it as the Lebesgue measure, we will use the convention $|dx|$ and we will use the same convention to distinguish volume forms and measures on $\IR^d$, $\IS^{d-1}$ or $\IR$.
In the following, we choose to orient $\IS^{d-1}$ with the standard $d-1$ form
$$\Vol_{\IS^{d-1}}(\theta,d\theta)=\sum_{p=1}^d(-1)^{p+1}\theta_p\bigwedge_{q=1,q\neq p}^dd\theta_q,\quad\theta\in\IS^{d-1},$$
and $S\IT^d$ with the $2d-1$-form
 $$dx_1\wedge\ldots \wedge dx_d\wedge \Vol_{\IS^{d-1}}(\theta,d\theta).$$
In view of alleviating notation, we will simply write $d\text{Vol}(\theta):=\Vol_{\IS^{d-1}}(\theta,d\theta)$ and, when we want to emphasize that we view it as a measure on $\IS^{d-1}$, we will write $|d\text{Vol}(\theta)|$ or $|d\text{Vol}|$.

%In the following, we choose to orient $\IS^{d-1}$ with the $d-1$ form
% $$e_{1}^*(\theta,d\theta)\wedge e_{2}^*(\theta,d\theta)\wedge\ldots\wedge e_{d-1}^*(\theta,d\theta),$$
% and $S\IT^d$ with the $2d-1$-form
% $$dx_1\wedge\ldots \wedge dx_d\wedge e_{1}^*(\theta,d\theta)\wedge \ldots\wedge e_{d-1}^*(\theta,d\theta).$$
% Note that $e_{1}^*(\theta,d\theta)\wedge e_{2}^*(\theta,d\theta)\wedge\ldots\wedge e_{d-1}^*(\theta,d\theta)$ is independent of the choice of the orthonormal family $(e_1(\theta),\ldots, e_{d-1}(\theta))$ and that it can be identified with the canonical volume form $\Vol_{\IS^{d-1}}(\theta,d\theta)$ on $\IS^{d-1}$:
 %$$\Vol_{\IS^{d-1}}(\theta,d\theta)=\sum_{p=1}^d(-1)^{p+1}\theta_p\bigwedge_{q=1,q\neq p}^dd\theta_q,\quad\theta\in\IS^{d-1}.$$
 %When we want to emphasize that we view it as a measure, we will write $\Vol_{\IS^{d-1}}(\theta,|d\theta|)$. 
 \begin{rema}[Orientation conventions for the sphere] If we denote by $B_d:=\{y\in\IR^d:|y|\leq 1\}$, the natural orientation on $B_d$ is given by that on $\IR^d$. In spherical coordinates $(r,\theta)$, the current of integration on $B_d$ reads $[B_d](r,\theta,dr,d\theta) =  \mathds{1}_{[0,1]}(r) $ so that
 $$[\IS^{d-1}](r,\theta,dr,d\theta)=\partial[B_d](r,\theta,dr,d\theta):=-d[B_d](r,\theta,dr,d\theta)=\delta_0(r-1)dr.$$
In particular, $\Vol_{\IS^{d-1}}(\theta,d\theta)$ is the orientation on $\IS^{d-1}=\partial ^B_d$ induced by the one on $\IR^d$ as
$$\left<[\IS^{d-1}] , \Vol_{\IS^{d-1}} \right> = \int_{\IR^d}[\IS^{d-1}]\wedge \Vol_{\IS^{d-1}}>0.$$
%(where  $\Vol_{\IS^{d-1}}$ has been extended smoothly in a small neighborhood of $\IS^{d-1}$).  
 \end{rema}

\subsection{Fundamental examples of currents of integration}\label{ss:example}

We now discuss an important example of current in view of our analysis. In the sequel, 
$[0]$ denotes the equivalence class of $0\ \text{mod}\ 2\pi\IZ^d$. We introduce the following current of degree $d$ on $\IT^d$:
$$\delta_{2\pi\IZ^d}(x,dx):=\frac{1}{(2\pi)^{\frac{d}{2}}}\sum_{\xi\in\IZ^d}\mathfrak{e}_\xi(x)dx_1\wedge\ldots\wedge dx_d ,$$
acting on functions $f \in \mathcal{C}^\infty(\T^d)$ by $\left< \delta_{2\pi\IZ^d}(x,dx), f \right> = f([0]).$ We will also use the notation
\begin{equation}\label{e:Dirac-Fourier}
\delta_{[0]}(x)=\frac{1}{(2\pi)^{\frac{d}{2}}}\sum_{\xi\in\IZ^d}\mathfrak{e}_\xi(x),
\end{equation}
so that we can write
\begin{equation}\label{e:Dirac-Fourier-bis}
\delta_{2\pi\IZ^d}(x,dx)=\delta_{[0]}(x)dx_1\wedge\ldots\wedge dx_d.
\end{equation}
If we view this current as a current\footnote{Again, we use the same notation for the current on the base and its pullback on $S\IT^d$. We keep this convention for simplicity and we will in fact mostly consider the pullback in the following.} on $S\IT^d$ via the pullback by the map $(x,\theta)\in S\IT^d\mapsto x\in \IT^d$, then it is in fact the current of integration on the fiber $S_{[0]}\IT^d$ viewed as a submanifold of dimension $d-1$. We can now slightly modify this example by fixing a smooth map $\tilde{x}:\IS^{d-1}\rightarrow \IR^d$ so that we can set 
\begin{equation}\label{e:example-current}
\boxed{
[\mathcal{N}]:=\delta_{[0]}(x-\tilde{x}(\theta))\bigwedge_{i=1}^dd\left(x_i-\tilde{x}_i(\theta)\right)\in\ml{D}^{\prime d}(S\IT^d)
.}\end{equation}
Note that this is the current of integration on the $d-1$ dimensional submanifold $$\mathcal{N}:=\{(\tilde{x}(\theta),\theta):\theta\in \IS^{d-1}\}\subset \IT^{d}\times\IS^{d-1}$$ that we have oriented with~$\Vol_{\IS^{d-1}}(\theta,d\theta)$~\cite[Cor.~D.4]{DangRiviere18}. This is typically the kind of currents to which we will apply $e^{-tV*}$ in our applications to convex geometry. See Lemma~\ref{l:truncated-poincare} for instance.

Thanks to~\eqref{e:Dirac-Fourier}, we have the following Fourier decomposition
\begin{eqnarray*} [\ml{N}](x,\theta,dx,d\theta)&=&\delta_0(x-\tilde{x}(\theta))\bigwedge_{i=1}^dd\left(x_i-\tilde{x}_i(\theta)\right)\in\ml{D}^{\prime d}(S\IT^d)\\ &=&\frac{1}{(2\pi)^{\frac{d}{2}}}\sum_{\xi\in\IZ^d}\mathfrak{e}_\xi(x) e^{-i\xi\cdot\tilde{x}(\theta)}\bigwedge_{i=1}^dd\left(x_i-\tilde{x}_i(\theta)\right)\in\ml{D}^{\prime d}(S\IT^d).
 \end{eqnarray*}
Thus, for every $\xi\in\IZ^d$, one finds using the conventions of \S\ref{ss:differential}:
\begin{equation}\label{e:fourier-coeff-current}\pi_\xi^{(d)}([\ml{N}])(\theta,dx,d\theta)=\mathfrak{e}_\xi(-\tilde{x}(\theta))\bigwedge_{i=1}^dd\left(x_i-\tilde{x}_i(\theta)\right)\in\Omega^d(S\IT^d).\end{equation}

\subsection{Relations with volumes in convex geometry}\label{ss:volume} 
These currents of integration are naturally linked with the volumes of convex sets:
\begin{lemm}\label{l:volume-convex} With the above conventions, for any strictly convex set $K$ in the sense of Definition~\ref{d:def-convex}, one has
 $$\operatorname{Vol}_{\IR^d}(K)=\frac{1}{d!}\int_{S\IT^d}[S_{[0]}\IT^d]\wedge \mathbf{V}^{d-1} \iota_V \left(dx_1\wedge\ldots\wedge dx_d\right).$$
\end{lemm}

\begin{rema}\label{r:mixed-volume}
 Note that the assumption of having $0$ in the interior of $K$ is not useful in this lemma. Recalling that $\mathbf{V}=\mathbf{v}(\theta)\cdot\partial_x$ and Lemma~\ref{l:sum-convex}, we recover that, for strictly convex subsets $(K_1,K_2)$ and for every $\lambda_1,\lambda_2>0$, $\text{Vol}_{\IR^d}(\lambda_1K_1+\lambda_2K_2)$ is a homogeneous polynomial in $(\lambda_1,\lambda_2)$~\cite{Schneider14}. The coefficients are the so-called mixed volumes of $K_1$ and $K_2$. In particular, if we take $\lambda_1=1$, $\lambda_2=t$, $K_1=K$ and $K_2=B_d$, we recover Steiner's formula~\eqref{e:steiner} from the introduction.
\end{rema}

\begin{proof} We let $t>0$. Suppose that, near a given point $x_0\in\partial tK$, $tK$ is parametrized by $\{f\leq 0\}$ with $\nabla f(x_0)\neq 0$. Then, the direct normal bundle writes locally near $x_0$:
$$N_+(\partial (tK)):=\left\{\left(x,\frac{\nabla f(x)}{|\nabla f(x)|}\right): f(x)=0\right\}.$$
Given $\psi\in\Omega^{d-1}(\IR^d)$ (not necessarily compactly supported), we write
$$\int_{S\IR^d}[N_+(\partial (tK))]\wedge P^*(\psi)=\int_{S\IR^d}\delta_0^{\IR}(f)df\wedge \left[\left\{\theta=\frac{\nabla f(x)}{|\nabla f(x)|}\right\}\right]\wedge \psi(x,dx)=\int_{\IR^d}[\partial (tK)]\wedge \psi,$$
where $P:S\IR^d\rightarrow\IR^d$ denotes the canonical projection. Recalling that $-d[K]=[\partial K]$, we find that
$$\int_{S\IR^d}[N_+(\partial (tK))]\wedge P^*(\psi)=\int_{tK} d\psi .
$$
Making use of~\eqref{e:example-current-pullback} below, we thus obtain
$$
\int_{tK} d\psi =\int_{S\IR^d}[N_+(\partial (tK))]\wedge P^*(\psi)=\int_{S\IR^d}e^{-tV*}[S_0\IR^d]\wedge P^*(\psi)=\int_{S\IR^d}[S_0\IR^d]\wedge e^{tV*}P^*(\psi).$$
Now, we choose $\psi=x_1dx_2\wedge\ldots \wedge dx_d$ and observe that $[S_0\IR^d]=\delta_0^{\IR^d}(x)dx_1\wedge\ldots\wedge dx_d$ and that
$$ \sum_{\ell=0}^{d}\frac{t^\ell}{\ell!}\mathbf{V}^\ell P^*(\psi)= e^{tV*}P^*(\psi)=\psi(x+t\mathbf{v}(\theta),d(x+t\mathbf{v}(\theta))).$$
See for instance the proof of Lemma~\ref{l:compute-forms} below to see that this is indeed a polynomial in $t$. Hence, when making the pairing of these two quantities, we end up with $$\int_{S\IR^d}[S_0\IR^d]\wedge e^{tV*}P^*(\psi)=\frac{t^d}{d!}\int_{S\IR^d}[S_0\IR^d]\wedge \mathbf{V}^d P^*(\psi)=\frac{t^d}{d!}\int_{S\IR^d}[S_0\IR^d]\wedge \mathbf{V}^{d-1} \iota_V P^*(d\psi),$$
where we used the Cartan formula and the Stokes formula to write down the last equality. Finally, remarking that $d\psi$ is translation invariant, we obtain  
$\int_{S\IR^d}[S_0\IR^d]\wedge \mathbf{V}^{d-1} \iota_V P^*(d\psi)=\int_{S\IT^d}[S_0\IR^d]\wedge \mathbf{V}^{d-1} \iota_V P^*(d\psi)$. Setting $t=1$ in the above formula yields the expected result.
\end{proof}
Replacing $tK$ by $K_0+tK$ in the above proof and recalling Lemma~\ref{l:sum-convex}, we also get the following useful equality from the above proof:
\begin{lemm}\label{r:volume-Rd} Let $K_0$ and $K$ be two compact and strictly convex sets in the sense of Definition~\ref{d:def-convex}. Then, one has  
 \begin{multline*}\operatorname{Vol}_{\IR^d}(K_0+tK)%= \sum_{\ell=0}^d\frac{t^\ell}{\ell!}\int_{S\IR^d}[S_{0}\IR^d]\wedge e^{V_{K_0}*} \mathbf{V}^\ell\left(x_1 dx_2\wedge\ldots\wedge dx_d\right)\\
 = \sum_{\ell=0}^d\frac{t^\ell}{\ell!}\int_{S\IT^d}[S_{0}\IT^d]\wedge e^{V_{K_0}*} \mathbf{V}^\ell\iota_V\left(dx_1\wedge dx_2\wedge\ldots\wedge dx_d\right)\\
  = \operatorname{Vol}_{\IR^d}(K_0)+\sum_{\ell=1}^d\frac{t^\ell}{\ell!}\int_{S\IT^d}[S_{[0]}\IT^d]\wedge e^{V_{K_0}*}\mathbf{V}^{\ell-1}\iota_V \left(dx_1\wedge\ldots\wedge dx_d\right).
 \end{multline*}
\end{lemm}
In particular, when $K=B_d$, this yields a somewhat explicit expression of Steiner's formula~\eqref{e:steiner}. We also record the following result that implicitely appeared in the proof of Lemma~\ref{l:volume-convex}:
\begin{lemm}
\label{l:compute-forms}
Setting $\omega := \iota_V(dx_1 \wedge \dots \wedge dx_d) \in \Omega^{d-1}(S\T^d)$, then the form $e^{-t\mathbf{V}}\omega$ is a polynomial of degree $d-1$ in $t$ with coefficients in $\Omega^{d-1}(S\T^d)$, with leading coefficient 
$\frac{1}{(d-1)!} \mathbf{V}^{d-1}\omega.$
% = \Vol_{\IS^{d-1}}(\theta,d\theta) 
%$
\end{lemm}

\begin{proof} We start from 
$$
\omega = \sum_{j=1}^d \mathbf{v}_j(\theta) (dx_1 \wedge \dots\wedge dx_d)(\d_{x_j}) = \sum_{j=1}^d (-1)^{j+1} \mathbf{v}_j(\theta)  \, dx_1 \wedge \dots \wedge \widehat{dx_j} \wedge \dots \wedge dx_d ,
$$
from which we deduce that 
$$
e^{-t\mathbf{V}}\omega =  \sum_{j=1}^d (-1)^{j+1}  \mathbf{v}_j(\theta)  \, d(x_1-t  \mathbf{v}_1(\theta) ) \wedge \dots \wedge \widehat{d(x_j-t  \mathbf{v}_j(\theta) )} \wedge \dots \wedge d(x_d-t \mathbf{v}_d(\theta) )
$$
is a polynomial of degree $d-1$ in $t$.
%, whence the first statement. As a consequence, we have
%$$
%e^{-t\mathbf{V}}\omega = \sum_{k=0}^{d-1} \frac{1}{k!}  \big(-t\mathbf{V} \big)^k\omega , 
%$$
%where the leading order coefficient in $t$ is thus $\frac{(-1)^{d-1}}{(d-1)!} \mathbf{V}^{d-1}\omega$. Identifying with the leading order in the previous computation yields
%$$\frac{\mathbf{V}^{d-1}\omega}{(d-1)!}  =  \sum_{j=1}^d (-1)^{j+1} \theta_j \, d \theta_1 \wedge \dots \wedge \widehat{d \theta_j} \wedge \dots \wedge d \theta_d 
%=  \iota_{\theta\cdot\partial_\theta}(d\theta_1 \wedge \dots\wedge d\theta_d) = \Vol_{\IS^{d-1}}(\theta,d\theta)  ,
%$$
%which concludes the proof of the lemma.
\end{proof}

\subsection{Translation maps}\label{ss:pullback}

If we fix a smooth map $\tilde{x}:\IS^{d-1}\rightarrow \IR^d$, then we can define a translation map on $S\IT^{d}$:
\begin{align}
\label{e:def-xtilde}
\mathbf{T}_{\tilde{x}}:(x,\theta)\in S\IT^d\mapsto (x+\tilde{x}(\theta),\theta)\in S\IT^d.
\end{align}
As we saw, these operators naturally appear when we consider the currents we are interested in. For instance, we can rewrite Example~\eqref{e:example-current} as
\begin{equation}\label{e:example-current-pullback}
\boxed{[\mathcal{N}]=\mathbf{T}_{-\tilde{x}}^*[S_{[0]}\IT^d]=[\mathbf{T}_{\tilde{x}}(S_{[0]}\IT^d)].}
\end{equation}
%We can think of the above transport map as generated by the flow $\mathbf{T}_{t\tilde{x}}=e^{t\tilde{x}(\theta)\cdot\partial_x}:S\IT^d\mapsto S\IT^d, t\in \mathbb{R}$ of the vector field $ \tilde{x}(\theta)\cdot\partial_x$. In section~\ref{s:convex}, we will make a crucial use of the above map to transport the cotangent fiber $ S_{[0]}\IT^d$ into the normal bundle $ \mathcal{N}:=\{(\tilde{x}(\theta),\theta):\theta\in \IS^{d-1}\}\subset \IT^{d}\times\IS^{d-1}$ when $\tilde{x}$ is the parametrization by the normal.

\begin{rema} An important example for our analysis will be of course given by the action of the flow, 
that is to say $e^{tV} = \mathbf{T}_{t\mathbf{v}}$ .
\end{rema}
%Such a diffeomorphism acts on currents by pullback in the following way. Write locally 
%$$u(x,\theta,dx,d\theta)=\sum_{I\subset\{1,\ldots,d\}}\sum_{J\subset\{1,\ldots,d-1\}}\sum_{|I|+|J|=k}u^{I,J}(x,\theta)dx^I\wedge e_J^*(\theta,d\theta),$$
%for some $u\in\ml{D}^{\prime k}(S\IT^d)$. The action by pullback on $u$ is defined as
%\begin{multline*}(\mathbf{T}_{\tilde{x}}^*u)(x,\theta,dx,d\theta)\\
% =\sum_{I\subset\{1,\ldots,d\}}\sum_{J\subset\{1,\ldots,d-1\}}\sum_{|I|+|J|=k}u^{I,J}(x+\tilde{x}(\theta),\theta)\left(\bigwedge_{l=1}^{|I|}d\left(x_{i_l}+\tilde{x}_{i_l}(\theta)\right)\right) \wedge e_J^*(\theta,d\theta).
%\end{multline*}
%For instance, this allows to rewrite example~\eqref{e:example-current} as
%\begin{equation}\label{e:example-current-pullback}
%\boxed{\delta_0(x-\tilde{x}(\theta))\bigwedge_{i=1}^dd\left(x_i-\tilde{x}_i(\theta)\right)=\mathbf{T}_{-\tilde{x}}^{*}\left(\delta_{2\pi\IZ^d}\right)(x,\theta, dx,d\theta)=\mathbf{T}_{-\tilde{x}}^*[S_{[0]}\IT^d]=[\mathbf{T}_{\tilde{x}}(S_{[0]}\IT^d)].}
%\end{equation}
%This can be explained visually as follows, we start from the cotangent fiber $S_{[0]}\mathbb{T}^d$ and transporting this fiber by the flow $\tilde{x}.\partial_\theta$ yields the integration current $[\mathcal{N}]$ on the submanifold $\mathcal{N}$.
We also record the following useful properties of these translation maps:

\begin{lemm}\label{l:commute-pullback} Let $\tilde{x}:\IS^{d-1}\rightarrow\IR^d$ be a smooth map and $\mathbf{T}_{\tilde{x}}$ defined by~\eqref{e:def-xtilde}, then one has
\begin{equation}\label{e:projector-pullback}
 \forall\xi\in\IZ^d,\quad \mathbf{T}_{\tilde{x}}^*\Pi_\xi^{(k)}(u)(x,\theta,dx,d\theta)=e^{i\xi\cdot\tilde{x}(\theta)}\Pi_\xi^{(k)}(u)(x,\theta,dx+D\tilde{x}(\theta),d\theta),
\end{equation}
and
\begin{equation}\label{e:commute-shift-lie}
 \mathbf{T}_{\tilde{x}}^*\mathbf{V}=\mathbf{V}\mathbf{T}_{\tilde{x}}^*,\quad\iota_V\mathbf{T}^*_{\tilde{x}}=\mathbf{T}^*_{\tilde{x}}\iota_V.
\end{equation}

%\begin{align}\label{e:projector-pullback}
 %\forall\xi\in\IZ^d,\quad \mathbf{T}_{\tilde{x}}^*\Pi_\xi^{(k)}(u)(x,\theta,dx,d\theta)=e^{i\xi\cdot\tilde{x}(\theta)}\Pi_\xi^{(k)}(u)(x,\theta,dx+D\tilde{x}(\theta),d\theta)\\
 %\label{e:commute-shift-lie}
 % \mathbf{T}_{\tilde{x}}^*\mathbf{V}=\mathbf{V}\mathbf{T}_{\tilde{x}}^*, \\
 % \label{e:commute-pullback}
 % \iota_V\mathbf{T}^*_{\tilde{x}}=\mathbf{T}^*_{\tilde{x}}\iota_V.
 %\end{align}
\end{lemm}
\begin{proof} 
The first point is a direct consequence of the action by pullback. For the second point, we notice that $e^{tV}\mathbf{T}_{\tilde{x}}(x,\theta) = (x+\tilde{x}(\theta)+ t\theta,\theta) =\mathbf{T}_{\tilde{x}}e^{tV}(x,\theta)$, from which the first equality in~\eqref{e:commute-shift-lie} follows.

Next, due to the specific forms of the operators $\mathbf{T}^*_{\tilde{x}}$, it is sufficient to verify the last on the smooth forms $dx^I$ with $I\subset\{1,\ldots,d\}.$ To that aim, we write
\begin{eqnarray*}
\iota_V\mathbf{T}^*_{\tilde{x}}(dx^I)&=&\iota_V\left(\bigwedge_{m=1}^ld(x_{i_m}+\tilde{x}_{i_m}(\theta))\right)\\
&=&\sum_{m=1}^l(-1)^{l+1} \mathbf{v}_{i_m}(\theta) \bigwedge_{m'=1, m'\neq m}^ld(x_{i_{m'}}+\tilde{x}_{i_{m'}}(\theta))\\
&=&\mathbf{T}^*_{\tilde{x}}\left(\sum_{m=1}^l(-1)^{l+1} \mathbf{v}_{i_m}(\theta) \bigwedge_{m'=1, m'\neq m}^ldx_{i_{m'}}\right)\\
&=&\mathbf{T}^*_{\tilde{x}}\iota_V(dx^I).
\end{eqnarray*}
\end{proof}

\subsection{Twisted correlations}

%\subsection{Transport equations on differential forms} We set
%$$H^1(\IT^d,\IR):=\left\{\beta:=\sum_{j=1}^d\beta_jdx_j:\ (\beta_1,\ldots,\beta_d)\in\IR^d\right\}.$$
%Given $u_0\in\ml{D}^{\prime k}(S\IT^d)$, we can verify that
%$$u(t):=e^{-it\beta_0(V)}e^{-tV*}(u_0)$$
%solves the transfort equation
%$$\partial_tu=-\mathbf{V}u-i\beta_0(V)u,\quad u(t=0)=u_0,$$
%where we recall that $\mathbf{V}=d\iota_V+\iota_Vd$ is the Lie derivative along the geodesic vector field. Equivalently, the transport equation can be written as
%\begin{equation}\label{e:twisted-transport}\partial_tu=-\mathbf{V}_{\beta_0}u,\quad\mathbf{V}_{\beta_0}:=(d+i\beta_0\wedge)\iota_V+\iota_V(d+i\beta_0\wedge).
% \end{equation}
%More generally, 

 We fix\footnote{The index $\IR$ means that the coefficients of the form are real valued.} $\beta\in\Omega^1_{\IR}(\IT^d)$ such that $d\beta=0$. Recall that any such $1$-form can be written as $\beta_0+df$ where $f\in\ml{C}^\infty(\IT^d,\IR)$ and $\beta_0$ is a constant $1$-form on $\IT^d$. Hence, we can define the twisted Lie derivative for such a general closed $1$-form $\beta$,
\begin{equation}\label{e:witten-twist}\mathbf{V}_{\beta}:=\mathbf{V}+i\beta(V)=(d+i\beta\wedge)\iota_V+\iota_V(d+i\beta\wedge)=e^{-if}\mathbf{V}_{\beta_0}e^{if}.
\end{equation}

%and consider the corresponding transport equation~\eqref{e:twisted-transport}. Yet, we note that it is ``conjugated'' to the transport equation induced by $\beta_0\in H^1(\IT^d,\IR)$ as
%\begin{equation}\label{e:witten-twist}\mathbf{V}_{\beta}=e^{-if}\mathbf{V}_{\beta_0}e^{if} .
%\end{equation}
%Intuitively, the reader should think that we are doing weighted transport. We transport some differential form by the flow and simultaneously, we multiply the the differential form by the integral of $\beta_0$ along the transport path.

We can record the following a priori estimate:
 \begin{lemm}\label{l:functional-calculus} Let $\beta\in\Omega^1_{\IR}(\IT^d)$ such that $d\beta=0$. Let $\chi\in\ml{C}^{\infty}(\IR)$ such that $t^{d-1}\chi(t)\in L^1(\IR)$. Then, for every $0\leq k\leq 2d-1$,
 $$\widehat{\chi}(-i\mathbf{V}_{\beta}):\psi\in\Omega^k_0(S\IT^d)\mapsto \int_{\IR}\chi(t)e^{-t\mathbf{V}_{\beta}}(\psi)|dt|\in\Omega^k_0(S\IT^d)$$
  is a bounded operator on the space of continuous $k$-forms $\Omega^k_0(S\IT^d)$.
 \end{lemm}
\begin{proof}
This follows from the fact that for any $k$--form $\psi\in \Omega^k(S\mathbb{T}^d)$, we have the polynomial bound $\Vert e^{-t\mathbf{V}_\beta}\psi \Vert_{L^\infty}=\mathcal{O}(t^{d-1})$ which comes from the definition of the flow exactly like in the proof of Lemma~\ref{l:compute-forms}.
\end{proof}

\begin{rema} Note that there is a slight abuse of notations when writing $\widehat{\chi}(-i\mathbf{V}_{\beta})$ as the operator $-i\mathbf{V}_{\beta}$ is not selfadjoint even on $L^2$-spaces (except if $k=0$ or $2d-1$). Yet, this convention from functional calculus is convenient and we will use it in the following except when it may create some confusions with the standard spectral Theorem~\cite[\S~12.7]{Sternberg05}.  
\end{rema}

%\subsection{Correlation function for differential forms} 
%We can now describe the correlation function $\Cor_{\varphi,\psi}(t)$ defined in~\eqref{e:correlation} using our Fourier decomposition. 
In fact, slightly more generally than the correlation function~\eqref{e:correlation}, we want to fix a smooth map $\tilde{x}:\IS^{d-1}\rightarrow \IR^d$ and some element $\beta_0\in H^1(\IT^d,\IR)$ and to describe the twisted correlations
\begin{align}
\label{e:def-correlation-T}
\boxed{\Cor_{\varphi,\mathbf{T}_{-\tilde{x}}^*(\psi)}(t,\beta_0) : =\int_{S\IT^d}\varphi\wedge e^{-t\mathbf{V}_{\beta_0}}\mathbf{T}_{-\tilde{x}}^*(\psi),}
\end{align}
as $t\rightarrow+\infty$ in terms of $\varphi$, $\psi$ and the map $\tilde{x}$ which will be needed later in our applications to convex geometry. The element $\beta_0$ plays the role of magnetic potential. 
Here, $\varphi$ and $\psi$ are two differential forms on $S\IT^d$ of respective degree $k_1$ and $k_2$ such that $k_1+k_2=2d-1$, and $\mathbf{T}_{\tilde{x}}$ is defined in~\eqref{e:def-xtilde}.

Coming back to our correlation function, we will also use at some points that, for every smooth map $\tilde{x}:\IS^{d-1}\rightarrow\IR^d$,
\begin{equation}\label{e:duality}\int_{S\IT^d}\varphi\wedge e^{-t\mathbf{V}_{\beta_0}}\mathbf{T}_{-\tilde{x}}^*(\psi)=\int_{S\IT^d}e^{t\mathbf{V}_{-\beta_0}}\mathbf{T}_{\tilde{x}}^*(\varphi)\wedge \psi ,
 \end{equation}
as a consequence of the change of variable formula and of the commutation of $e^{t\mathbf{V}_{-\beta_0}}$ and $\mathbf{T}_{\tilde{x}}$. With the above conventions, we can decompose differential forms using Fourier series. Following~\S\ref{ss:differential}, we write
\begin{equation}\label{e:fourier-phi}\varphi(x,\theta,dx,d\theta)=\sum_{\xi\in\IZ^d}\mathfrak{e}_{\xi}(x)\pi_\xi^{(k_1)}(\varphi)\left(\theta,dx,d\theta\right).\end{equation}
We collect in the following lemma several useful properties of $\Cor_{\varphi,\mathbf{T}_{-\tilde{x}}^*(\psi)}(t,\beta_0)$.
\begin{lemm}\label{l:reduction} Using the above conventions, one has, for every $(\varphi,\psi)\in\Omega^{k_1}(S\IT^d)\times \Omega^{k_2}(S\IT^d)$, and $t\in \R$,
 \begin{align}
 \label{e:decomposition1}
 \Cor_{\varphi,\mathbf{T}_{-\tilde{x}}^*(\psi)}(t,\beta_0) 
&  =\sum_{l=0}^{\min\{k_1,k_2\}}
 \Cor^l_{\varphi,\mathbf{T}_{-\tilde{x}}^*(\psi)}(t,\beta_0) \quad  \text{ with } \\
  \label{e:decomposition2}
\Cor^l_{\varphi,\mathbf{T}_{-\tilde{x}}^*(\psi)}(t,\beta_0) & = \frac{t^l}{l!}\sum_{\xi\in\IZ^d}\int_{\IS^{d-1}}e^{it(\xi-\beta_0)\cdot\mathbf{v}(\theta)}e^{i\xi\cdot\tilde{x}(\theta)}B_{\tilde{x},\xi}^{(k_2,l)}(\varphi,\psi)(\theta)d\Vol(\theta)
 \end{align}
 and  \begin{equation}\label{e:coeff-horrible} B_{\tilde{x},\xi}^{(k_2,l)}(\varphi,\psi)d\Vol:=(-1)^l\pi_{\xi}^{(k_1)}(\varphi)\wedge \mathbf{V}^l\mathbf{T}_{-\tilde{x}}^*\pi_{-\xi}^{(k_2)}(\psi).
 \end{equation}
Moreover, for every $m\in\IZ_+$ (resp. $s \in \R$), one can find some constant $C_m>0$ (resp. $C_s>0$) (depending also on $d$, $\tilde{x}$) such that, for every open set $U\subset \IS^{d-1}$, for every $\xi\in\IZ^d$, for every $k_1+k_2=2d-1$, for every $0\leq l\leq \min\{k_2,k_1\}$ and for every $(\varphi,\psi)\in\Omega^{k_1}(S\IT^d)\times \Omega^{k_2}(S\IT^d)$,
\begin{equation}\label{e:control-Cm-norm}\left\|B_{\tilde{x},\xi}^{(k_2,l)}(\varphi,\psi)\right\|_{W^{m,1}\left(U\right)}\leq C_m  \left\|\pi_{\xi}^{(k_1)}(\varphi)\right\|_{H^m\left(U\right)} \left\|\pi_{-\xi}^{(k_2)}(\psi)\right\|_{H^m\left(U\right)},\end{equation}
resp.
\begin{equation}\label{e:control-Cm-norm-bis}\left\|B_{\tilde{x},\xi}^{(k_2,l)}(\varphi,\psi)\right\|_{L^1\left(\IS^{d-1}\right)}\leq C_s  \left\|\pi_{\xi}^{(k_1)}(\varphi)\right\|_{H^s\left(\IS^{d-1}\right)} \left\|\pi_{-\xi}^{(k_2)}(\psi)\right\|_{H^{-s}\left(\IS^{d-1}\right)},\end{equation}
where $W^{m,1}(U)$ is the Sobolev norm of order $m$ and exponent $p=1$ and where $H^{m}(U)=W^{m,2}(U)$ is the Sobolev norm of order $m$ (and exponent $p=2$) on differential forms on the set $U$. 
\end{lemm}

The decomposition of the dynamical correlator $\Cor$ in terms of Fourier series and in powers of $t$ introduced in the above lemma will play a crucial role in the sequel, especially for our applications to convex geometry.

\begin{proof}
 We first expand $\psi$ in Fourier series using~\eqref{e:fourier-phi}. Then, we can make use of~\eqref{e:projector-pullback} so that each term of the sum over $\xi$ becomes of the form $e^{-i\xi\cdot\tilde{x}(\theta)}\mathfrak{e}_\xi(x)\pi_\xi^{(k)}(\psi)(\theta,d(x-\tilde{x}(\theta)),d\theta).$ Then, the action of $e^{-t\mathbf{V}_{\beta_0}}$ yields:
\begin{multline}\label{e:fourier-pulbback}(e^{-t\mathbf{V}_{\beta_0}}\mathbf{T}_{-\tilde{x}}^*\psi)(x,\theta,dx,d\theta)
 =\sum_{\xi\in\IZ^d}\mathfrak{e}_{\xi}(x)e^{-it(\xi+\beta_0)\cdot\mathbf{v}(\theta)}e^{-i\xi\cdot\tilde{x}(\theta)}\sum_{l=0}^{\min\{k_2,d-1\}}\frac{(-1)^lt^l}{l!}\\
 \times  \mathbf{V}^l\circ \mathbf{T}_{-\tilde{x}}^*\circ \pi_\xi^{(k_2)}(\psi)\left(\theta,dx,d\theta\right).
\end{multline}
Hence, using~\eqref{e:fourier-phi} applied to $\varphi$,~\eqref{e:commute-shift-lie},~\eqref{e:fourier-pulbback} and~\eqref{e:duality}, one has
$$\Cor_{\varphi,\mathbf{T}_{-\tilde{x}}^*(\psi)}(t,\beta_0)=\sum_{l=0}^{\min\{k_1,k_2\}}\frac{t^l}{l!}\sum_{\xi\in\IZ^d}\int_{\IS^{d-1}}e^{i\xi\cdot\tilde{x}(\theta)}e^{it(\xi-\beta_0)\cdot  \mathbf{v}(\theta)}B_{\tilde{x},\xi}^{(k_2,l)}(\varphi,\psi)(\theta) d\Vol(\theta),$$
where $B_{\tilde{x},\xi}^{(k_2,l)}(\varphi,\psi)(\theta)\Vol_{S\IT^d}(\theta,dx,d\theta)$ is given by~\eqref{e:coeff-horrible}. Note that, in order to verify that the sum over $l$ runs up to $\text{min}(k_1,k_2)$, we used~\eqref{e:duality} and considered the action of $e^{\pm t\mathbf{V}_{\beta_0}}$ on the form of smallest degree. This readily yields~\eqref{e:decomposition1}--\eqref{e:decomposition2}.

Finally, estimate~\eqref{e:control-Cm-norm} and~\eqref{e:control-Cm-norm-bis} follow from the fact that the coefficients $B_{\tilde{x},\xi}^{(k_2,l)}(\varphi,\psi)$ depend in a bilinear way on $(\pi_{\xi}^{(k_1)}(\varphi),\pi_{-\xi}^{(k_2)}(\psi))$, together with the Cauchy-Schwarz inequality.
\end{proof}

According to Lemma~\ref{l:reduction}, understanding the properties of the correlation function as $t\rightarrow+\infty$ amounts to describe the behaviour of the integrals
\begin{equation}\label{e:oscillatory-integral}
\boxed{I_F(\xi-\beta_0,t):=\int_{\IS^{d-1}}e^{i(\xi-\beta_0).(t\mathbf{v}(\theta)+\tilde{x}(\theta))}e^{i\beta_0\cdot\tilde {x}(\theta)}F(\theta) d\Vol(\theta),}
\end{equation}
as $t \to + \infty$, where $\xi\in\IZ^{d}\setminus\{0\}$, where $\beta_0$ plays the role of the magnetic potential and where $F$ is a smooth function. In view of applications, one needs to make this asymptotic description with a uniform control in terms of the $W^{m,1}$-norm of $F$ and of $\xi\in\IZ^d$.

\begin{rema} The extra oscillating term $e^{i\xi\cdot\tilde {x}(\theta)}$ makes things slightly more involved than when one treats the case of dilating convex sets as for instance in~\cite{Hlawka50, Herz62b, Randol66}. Indeed, in that setup, the parameter $t$ is also in factor of $\xi\cdot\tilde {x}(\theta)$ which allows to deal with $t|\xi|$ as a large parameter. Despite this technical issue, the strategy to analyze these integrals remains the same. See Section~\ref{s:analysis} below.             
            \end{rema}

\begin{rema}\label{r:correlation-function} In the case where $k_1=2d-1$ and $k_2=0$, one has
$$\varphi(x,\theta,dx,d\theta)=\sum_{\xi\in\IZ^d}\widehat{\varphi}_\xi(\theta)\mathfrak{e}_\xi(x)dx_1\wedge\ldots dx_d\wedge d\Vol(\theta) ,$$
and
$$\psi(x,\theta,dx,d\theta)=\sum_{\xi\in\IZ^d}\widehat{\psi}_\xi(\theta)\mathfrak{e}_\xi(x).$$
Hence, one gets the simpler expression
 $$\Cor_{\varphi,\mathbf{T}_{-\tilde{x}}^*(\psi)}(t,\beta_0)=\sum_{\xi\in\IZ^d}\int_{\IS^{d-1}}e^{it(\xi-\beta_0)\cdot\mathbf{v}(\theta)} e^{i\xi\cdot\tilde{x}(\theta)}\widehat{\varphi}_{\xi}(\theta)\widehat{\psi}_{-\xi}(\theta) d\Vol(\theta)$$
\end{rema}

\subsection{Back to generalized Epstein functions and Poincar\'e series}\label{ss:currents}
Before going further in the analysis of oscillatory integrals, let us explain how the zeta functions from~\S\ref{s:background-convex} (or at least truncated versions of it) can be naturally rewritten using the theory of currents in terms of certain twisted correlation functions. This will be achieved by considering the current of integration $[N_\pm(\Sigma)]$ on the unit normal bundle to some admissible $\Sigma$ (in the sense of Definition~\ref{d:admissible}). Recall that this current was explicitely defined in~\eqref{e:example-current} if one uses the Gauss coordinates~\eqref{e:gauss-coordinates}.

\subsubsection{Wavefront sets of the currents of integration}
The following lemma studies the wavefront set of this current in view of pairing two such objects.
\begin{lemm}\label{l:WF} For any admissible subset $\Sigma$ of $\IT^d$, we have
\begin{align}\label{e:WF-coord}
\WF([N_+(\Sigma)])  \cap  T^*_{(x,\theta)}S\mathbb{T}^d   & = E_0^*(\tilde{x}(\theta),\theta)\oplus \left\{\left(0,\xi,- (D\tilde{x}(\pm\theta))^T\xi \right): \xi \in \ml{H}^* \right\}  \nonumber\\
& \subset \left(E_0^*\oplus \ml{H}^*\oplus\ml{V}^*\right)(\tilde{x}(\theta),\theta).
\end{align}
Moreover, for any conical neighborhood $\Gamma$  of $E_0^*\oplus\ml{V}^*$, there exists $T_0>0$ such that, for every $T\geq T_0$, the wavefront set of $e^{-TV*}[N_+(\Sigma)]=[e^{TV}(N_+(\Sigma))]$ satisfies
\begin{align}
\label{e:first-WF}
\WF\left( e^{-TV*}[N_+(\Sigma)] \right) \subset \{(x,\theta;\xi)\in T^*S\IT^d:\ e^{-TV}(x,\theta)\in N_+(\Sigma),\ \xi\in\Gamma(x,\theta)\}.
\end{align}
%\begin{equation}\label{e:WF-coord}
%\WF([N_\pm(\Sigma)])  \cap  T^*_{(x,\theta)}S\mathbb{T}^d   = E_0^*(\tilde{x}(\pm\theta),\theta)\oplus \left\{\left(0,\xi,\mp (D\tilde{x}(\pm\theta))^T\xi \right): \xi \in \ml{H}^* \right\}\subset \left(E_0^*\oplus \ml{H}^*\oplus\ml{V}^*\right)(\tilde{x}(\pm\theta),\theta).\end{equation}
%Moreover, for any conical neighborhood $\Gamma$  of $E_0^*\oplus\ml{V}^*$, there exists $T_0>0$ such that, for every $T\geq T_0$, the wavefront set of $e^{-TV*}[N_\pm(\Sigma)]=[e^{TV}(N_\pm(\Sigma))]$ satisfies
%\begin{align}
%\label{e:first-WF}
%\WF\left( e^{-TV*}[N_\pm(\Sigma)] \right) \subset \{(x,\theta;\xi)\in T^*S\IT^d:\ e^{-TV}(x,\theta)\in N_\pm(\Sigma),\ \xi\in\Gamma(x,\theta)\}.
%\end{align}
The same holds for $N_-(\Sigma)$ with the appropriate sign changes. In particular, if $\Sigma_1$ and $\Sigma_2$ are two admissible subsets and if $\sigma_1,\ \sigma_2\in\{\pm\}$, then one can find $T_0>0$ such that, for every $T\geq T_0$,
\begin{align}
\label{e:second-WF}
\WF \left([N_{\sigma_1}(\Sigma_1)]\right) \cap \WF\left( e^{-TV*}[N_{\sigma_2}(\Sigma_2)] \right) \subset E_0^* .
\end{align}
\end{lemm}
\begin{proof}
 The wavefront set of the current $[N_+(\Sigma)]$ is given according to~\cite[Th.~8.1.5]{Hormander90}\footnote{The proof in that reference is given for a linear space and it can be transfered to submanifolds through a local chart thanks to~\cite[Th.~8.2.4]{Hormander90}. See Example~8.2.5 in that reference.} by the conormal bundle to $N_+(\Sigma)$, namely
$$\ml{N}^*(N_+(\Sigma)):=\left\{(x,\theta;\xi)\in T^*S\IT^d\setminus\underline{0}:\ (x,\theta)\in N_+(\Sigma),\ \forall v\in T_{x,\theta}N_+(\Sigma),\ \xi(v)=0\right\}.$$
In particular, thanks to~\eqref{e:legendrian}, the fiber $\ml{N}_{(x,\theta)}^*(N_+(\Sigma))\subset T^*_{(x,\theta)}S\mathbb{T}^d$ over $(x;\theta)$ always contains the annihilator $E_0^*(x,\theta)$ of $\ml{H}\oplus\ml{V}(x,\theta)$. Using the description of $T_{x,\theta}N_+(\Sigma)$ in~\eqref{e:tangent-space} in the coordinates~\eqref{e:gauss-coordinates}, this wavefront set can in fact be identified as
$$
\ml{N}_{(x,\theta)}^*(N_+(\Sigma)) =E_0^*(\tilde{x}(\theta),\theta)\oplus \left\{\left(0,\xi,- (D\tilde{x}(\theta))^T\xi \right): \xi \in \ml{H}^* \right\}\subset \left(E_0^*\oplus \ml{H}^*\oplus\ml{V}^*\right)(\tilde{x}(\theta),\theta),
$$
whence the first statement~\eqref{e:WF-coord}.
Hence, thanks to~\eqref{e:cotangentmap} and to~\cite[Th.~8.2.4]{Hormander90} (see also~\cite[Prop.~5.1]{BrouderDangHelein16}), we deduce the second statement~\eqref{e:first-WF}.
The last statement~\eqref{e:second-WF} follows from the first by choosing $\Gamma$ some sufficiently small conical neighborhood of $E_0^*\oplus\ml{V}^*$ so that 
$\left(\WF \left([N_{\sigma_1}(\Sigma_1)]\right) \cap \Gamma \right)\subset E_0^*$.
\end{proof}

\subsubsection{Representation of truncated series using currents} In the end, our  goal will be to use our fine analysis of the geodesic vector field to study the continuation of these series beyond their natural halfplane of definition. As in~\cite[Prop.~4.10]{DangRiviere20d}, one starts with the following result, relating the above discrete sums on intersection points between two submanifolds with the geodesic flow acting on currents.
\begin{lemm}\label{l:truncated-poincare}  Let $\Sigma_1$ and $\Sigma_2$ be two admissible subsets of $\IT^d$ and let $\sigma_1$ and $\sigma_2$ be elements in $\{\pm\}$. Let $\beta=\beta_0+df$ be a closed one-form with $\beta_0\in H^1(\IT^d,\IR)$ and $f\in\ml{C}^{\infty}(\IT^d,\IR)$. %There exists $T_0>0$ large enough such that, for every $T\geq T_0$, one can find $t_0>0$ such that, for every $\chi\in\mathcal{C}^{\infty}_c((T-t_0,+\infty))$,
There exists $T_0>0$ large enough such that for every $\chi\in\mathcal{C}^{\infty}_c((T_0,+\infty))$,
$$I_T(\chi):=(-1)^{d-1}\int_{S\IT^d}[N_{\sigma_1}(\Sigma_1)]\wedge \int_{\IR}\chi(t)e^{-t\mathbf{V}_\beta}\iota_V([N_{\sigma_2}(\Sigma_2)])|dt|$$
is well defined and is equal to
$$\sum_{t:\mathcal{E}_t(\Sigma_1,\Sigma_2)\neq\emptyset}\chi(t)\left(\sum_{(x,\theta)\in \mathcal{E}_t(\Sigma_1,\Sigma_2)}\epsilon_t(x)e^{-i\int_{-t}^0\beta(V)(x+\tau\mathbf{v}(\theta),\theta)|d\tau|}\right),$$
where $\epsilon_t(x)=1$ if
$$T_{(x,\theta)}N_{\sigma_1}(\Sigma_1)\oplus \IR V(x,\theta)\oplus D(e^{tV})(e^{-tV}(x,\theta))\left( T_{e^{-tV}(x,\theta)}N_{\sigma_2}(\Sigma_2)\right),$$
has the same orientation as $S\IT^d$ and $\epsilon_t(x)=-1$ otherwise.
\end{lemm}
This result follows from~\cite[\S2]{DangRiviere20d} together with Lemma~\ref{l:WF}. Strictly speaking, the proof from~\cite{DangRiviere20d} does not include exponential weights but it can be adapted directly to deal with such twisted transport equations. Note that, in view of applying the result from this reference, it is crucial here to have that $\mathbf{v}(\theta)\cdot\theta>0$ for every $\theta\in\IS^{d-1}$ as follows from~\eqref{e:interior-0}. Indeed, this property ensures that $\mathbf{v}(\theta)\cdot\partial_x $ is transverse to the contact plane where the tangent space to $N(\Sigma_i)$ always lies. Here, $[N_\pm(\Sigma_1)]$ is an element in $\ml{D}^{\prime d}(S\IT^d)$ and $\int_{\IR}\chi(t)e^{-tV}\iota_V([N_+(\Sigma_2)])|dt|$ is an element in $\ml{D}^{\prime d-1}(S\IT^d)$. The key point in the argument of~\cite{DangRiviere20d} is that the wavefront sets of these two currents are disjoint so that we can take their wedge product. Here, the wavefront set of $[N_{\sigma_1}(\Sigma_1)]$ is given by~\eqref{e:WF-coord} while the wavefront set of $\int_{\IR}\chi(t)e^{-tV}\iota_V([N_{\sigma_2}(\Sigma_2)])|dt|$ is contained in a small conical neighborhood of $\ml{V}^*$ thanks to Lemma~\ref{l:WF} and to the integration over time. See~\cite[Lemma~4.11]{DangRiviere20d} for more details. 
We would like to emphasize that this lemma states that the dynamical correlator
$$I_T: t\mapsto (-1)^{d-1}\int_{S\IT^d}[N_{\sigma_1}(\Sigma_1)]\wedge e^{-t\mathbf{V}_\beta}\iota_V([N_{\sigma_2}(\Sigma_2)])$$
is in fact a \textbf{distribution of the variable} $t$ which writes
as a weighted sum of $\delta$. The distribution $I_T$ is a weighted counting measure, its Mellin transform is the Epstein function and its Laplace transform is the Poincaré series.

\bigskip
In order to make the connection with the series of~\S\ref{s:background-convex}, we need to clarify the values of $\epsilon_t(x)$ in our case:
\begin{lemm}\label{l:orientation-index} There exists $T_0>0$ such that, for every $t\geq T_0$, $\epsilon_t(x)=1$. 
\end{lemm}
\begin{proof} Recall that we oriented $N_\pm(\Sigma_1)$ with the $d-1$-form $\Vol_{\IS^{d-1}}(\theta,d\theta)$, or equivalently with the polyvector $(e_1(\theta)\cdot\partial_\theta)\wedge \ldots\wedge (e_d(\theta)\cdot\partial_\theta)$ where $(\theta,e_1(\theta),\ldots,e_d(\theta))$ is a direct orthonormal basis of $\IR^d$. The same holds true for $\Sigma_2$ but we need to take into account the action of the tangent map as given by~\eqref{e:tangentmap} which transforms this polyvector into
$$\left(e_1(\theta)\cdot\partial_\theta+t D\mathbf{v}(\theta)e_1(\theta)\cdot\partial_x\right)\wedge \ldots\wedge \left(e_d(\theta)\cdot\partial_\theta+t D\mathbf{v}(\theta)e_d(\theta)\cdot\partial_x\right).$$
Finally, $\IR V(x)$ is oriented through the vector $ \mathbf{v}(\theta) \cdot\partial_x$ and it yields the following orientation
$$t^{d-1}\left(e_1(\theta)\cdot\partial_\theta\right)\wedge \ldots\wedge \left(e_d(\theta)\cdot\partial_\theta\right)\wedge\mathbf{v}(\theta) \cdot\partial_x\wedge \left( D\mathbf{v}(\theta)e_1(\theta)\cdot\partial_x\right)\wedge \ldots\wedge \left( D\mathbf{v}(\theta)e_d(\theta)\cdot\partial_x\right),$$
which is the same as the orientation on $S\IT^d$ up to the factor $(-1)^{d(d-1)}=1$ as the inverse Gauss map $\theta\mapsto \mathbf{v}(\theta)$ is orientation preserving. 
\end{proof}

Letting $\tilde{x}_j^\pm(\theta)=\tilde{x}_{ K_j}(\pm\theta)$, we derive the following corollary.
\begin{coro}\label{c:counting-current}
There is $T_0>0$ and $t_0\in(0,T_0)$ such that for every $\chi\in\mathcal{C}^{\infty}_c((T_0-t_0,+\infty))$ with $\chi =1$ on $[T_0,+\infty)$, we have
\begin{multline}\label{e:main-formula}
 \sum_{t>T_0-t_0}\chi(t)\left(\sum_{(x,\theta)\in \mathcal{E}_t(\Sigma_1,\Sigma_2)}e^{-i\int_{-t}^0\beta(V)(x+\tau\mathbf{v}(\theta),\theta)|d\tau|}\right)\\
 =(-1)^{d-1}\int_{S\IT^d}e^{ i\left(f(\tilde{x}_{2}^{\sigma_2})-f(\tilde{x}_{1}^{\sigma_1})\right)}\delta_{2\pi\IZ^d}\wedge \int_{\IR}\chi(t)\left(e^{-t\mathbf{V}_{\beta_0}}\mathbf{T}_{\tilde{x}_{1}^{\sigma_1}-\tilde{x}_{2}^{\sigma_2}}^*\right)\iota_V(\delta_{2\pi\IZ^d})|dt|,
\end{multline}
where we recall that the notation $(\delta_{2\pi\IZ^d})$ stands for a current of degree $d$. 
\end{coro}
We now recognize on the right hand-side the twisted dynamical correlations of Lemma~\ref{l:reduction} except that the test ``functions'' are much more singular.

\begin{proof}
According to Lemma~\ref{l:orientation-index}, for $T_0>0$ large enough and recalling~\eqref{e:witten-twist}, one deduces from Lemma~\ref{l:truncated-poincare} that, for every $\chi\in\mathcal{C}^{\infty}_c((T_0-t_0,+\infty))$,
\begin{multline}\label{e:equality-zeta-current}
\sum_{t>T_0-t_0}\chi(t)\left(\sum_{(x,\theta)\in \mathcal{E}_t(\Sigma_1,\Sigma_2)}e^{-i\int_{-t}^0\beta(V)(x+\tau\mathbf{v}(\theta),\theta)|d\tau|}\right)\\
=(-1)^{d-1}\int_{S\IT^d}e^{-if}[N_{\sigma_1}(\Sigma_1)]\wedge \int_{\IR}\chi(t)e^{-t\mathbf{V}_{\beta_0}}\iota_V(e^{if}[N_{\sigma_2}(\Sigma_2)])|dt|.
\end{multline}
Thanks to~\eqref{e:example-current}, we can write, for $j=1,2$,
\begin{align*}e^{\pm if}[N_{\sigma_j}(\Sigma_j)]=e^{\pm if(x)}\mathbf{T}_{-\tilde{x}_{j}^{\sigma_j}}^*(\delta_{2\pi\IZ^d})
=e^{\pm if(\tilde{x}_{j}^{\sigma_j})}\mathbf{T}_{-\tilde{x}_{j}^{\sigma_j}}^*(\delta_{2\pi\IZ^d}).
 \end{align*}
Combined with~\eqref{e:commute-shift-lie} in Lemma~\ref{l:commute-pullback} , this allows to rewrite~\eqref{e:equality-zeta-current} as~\eqref{e:main-formula}.
\end{proof}

%%%%%%%%%%%%%%%%%%%%%%%%%%%%%%%%%%%%%%%%%%%%
%%%%%%%%%%%%%%%%%%%%%%%%%%%%%%%%%%%%%%%%%%%%
%%%%%%%%%%%%%%%%%%%%%%%%%%%%%%%%%%%%%%%%%%%%

\section{Asymptotics of oscillatory integrals}\label{s:analysis}

In view of describing the long time asymptotics of the correlation function (or integrated versions of it), we need to describe with some accuracy the oscillatory integrals appearing in Lemma~\ref{l:reduction}. More precisely, we want to study the behaviour as $t\rightarrow+\infty$ (and for $\xi-\beta_0\neq 0$) of the oscillatory integral $I_F(\xi-\beta_0,t)$ in~\eqref{e:oscillatory-integral}, where $\beta_0\in\IR^d$, $F\in\mathcal{C}^{\infty}(\IS^{d-1},\IC)$ and $\tilde{x}\in\ml{C}^\infty(\IS^{d-1},\IR)$. Estimating these kind of integrals as $t\rightarrow+\infty$ is a classical topic in harmonic analysis. See~\cite{Herz62, Littman63} for a rough estimate, and~\cite[Th.~7.7.14]{Hormander90}, \cite[Section~VIII-3, p347]{Steinbook} and~\cite[Th.~3.38, p140]{DyatlovZworski19} for fine asymptotic expansions. The only additional difficulty compared with these references is that we need to have a good control in terms of $|\xi|$. Thus, we need to pay a little attention to the extra term $e^{i\xi\cdot\tilde{x}(\theta)}$ when revisiting the classical stationary phase arguments used to describe these integrals: this is the content of the present section. There is a subtle competition between the large times $t$ and the large momenta $\vert\xi\vert$ which is what we deal with in the next lemmas. 

As usual, we split these oscillatory integrals in two parts: one corresponding to nonstationary points and one corresponding to stationary ones. Recall that the phase function we are interested in is of the form $\theta\in\IS^{d-1}\mapsto \mathbf{v}(\theta)\cdot\omega$ for some fixed $\omega\in\IS^{d-1}$ (typically $\omega=\pm(\xi-\beta_0)/|\xi-\beta_0|$). This is equivalent to considering the function $v\in\partial K\mapsto v\cdot\omega$ which is nothing but the height function in the direction $\omega$ which has two critical points on the convex boundary $\partial K$. One can verify that the differential of this function at $v$ can be identified with the linear form $y \mapsto \left(\omega-(\theta_{\partial K}(v)\cdot\omega)\theta_{\partial K}(v)\right) \cdot y$ where $\theta_{\partial K}(v)$ is the normal to $\partial K$ at $v$. In particular, $\theta_{\partial K}(\mathbf{v}(\theta))=\theta$. Hence, the phase has exactly two critical points $v_\pm$ that correspond to the points where $\theta_{\partial K}=\pm\omega$. Recalling the definition of the shape operator $S(v)$ in Remark~\ref{r:shape-operator}, the Hessian of the phase at these points is given by $\mp S(v_\pm)$, all of whose eigenvalues are positive (resp. negative) thanks to the strict convexity of $K$. Coming back to $\IS^{d-1}$, the phase of our oscillatory integrals has two critical points located at $\pm\omega$ and the determinant of their Hessian is exactly the Gauss curvature $\kappa(\pm \omega)$ (up to sign). We will thus decompose our integrals accordingly.

\subsection{Splitting the oscillatory integral}\label{ss:split-oscillatory-integral}
We first define cutoff functions that will be used all along the paper and that we will fix once and for all in Lemma~\ref{c:stat-points}.
\begin{def1}
\label{def-chi-j}
We let $(\chi_j)_{j\in \{-1,0,1\}}$ be any smooth partition of unity of the closed interval $[-1,1]$ such that $\chi_{1}$ is equal to one in a neighborhood of $1$, $\supp(\chi_1) \subset (0, 1]$, $\chi_{-1}(s)=\chi_1(-s)$, and $\supp(\chi_0) \subset (-1,1)$.
\end{def1}
We will make a $\xi$--dependent partition of unity of the sphere $\theta \in \IS^{d-1}$ by letting
$$
\theta \mapsto \chi_j \left(\theta \cdot  \frac{\xi-\beta_0}{|\xi-\beta_0|} \right) , \quad j\in \{-1,0,1\}  , \quad \xi \neq \beta_0, 
$$
that is to say, for $j=\pm 1$, localization near the poles $\pm \frac{\xi-\beta_0}{|\xi-\beta_0|}$, and for $j=0$, localization near the equator $\left(\frac{\xi-\beta_0}{|\xi-\beta_0|}\right)^\perp$.

We split the integral $I_F$ defined in~\eqref{e:oscillatory-integral} accordingly as
\begin{align}
\label{e:split-3-int}
I_F(\xi-\beta_0,t) & = I_F^{(-1)}(\xi-\beta_0,t) +I_F^{(0)}(\xi-\beta_0,t)+ I_F^{(1)}(\xi-\beta_0,t),\quad\text{ with }\\
\label{e:split-3-int-bis} 
I_F^{(j)}(\xi-\beta_0,t)& :=\int_{\IS^{d-1}} \chi_j \left(\theta \cdot  \frac{\xi-\beta_0}{|\xi-\beta_0|} \right)e^{i(\xi-\beta_0)\cdot(t\mathbf{v}(\theta)+\tilde{x}(\theta))}e^{i\beta_0\cdot\tilde{x}(\theta)}F(\theta) d\Vol(\theta) .
\end{align}

%We denote by $(e_1, \dots,e_d)$ the canonical basis of $\R^d$. 
We first rewrite $\xi-\beta_0$ as $\xi-\beta_0 = \lambda \omega$ with $\lambda = |\xi-\beta_0| \neq 0$ and $\omega=\frac{\xi-\beta_0}{|\xi-\beta_0|} \in \IS^{d-1}$. The above splitting~\eqref{e:split-3-int} of $I_F$ into three pieces now reads
\begin{align}
I_F(\lambda\omega,t) & = I_F^{(-1)}(\lambda\omega,t) +I_F^{(0)}(\lambda\omega,t)+ I_F^{(1)}(\lambda\omega,t),\quad\text{ with }\\
\label{e:expression-I-pm1}
I_F^{(j)}(\lambda\omega,t)& :=\int_{\IS^{d-1}} \chi_j \left(\theta \cdot  \omega \right)e^{i\lambda\omega\cdot(t\mathbf{v}(\theta)+\tilde{x}(\theta))}e^{i\beta_0\cdot\tilde{x}(\theta)}F(\theta) d\Vol(\theta)  ,  
\end{align}

%, and let 
%$$
%R : \omega \ni \IS^{d-1} \to R(\omega)  = R_\omega \in SO(d)
%$$
%be so that $\omega = R_\omega e_d$. We then make a change of variable in the integral, setting $\theta'=R_\omega^{-1}\theta = R_\omega^T \theta$, and obtain 
%\begin{align*}
%I_F(\lambda\omega,t)& =\int_{\IS^{d-1}}e^{i\lambda R(\omega) e_d \cdot (t\theta+\tilde{x}(\theta))}e^{i\beta_0 \cdot \tilde{x}(\theta)}F(\theta)\Vol_{\IS^{d-1}}(\theta,|d\theta|) \\
% & =\int_{\IS^{d-1}}e^{i\lambda e_d \cdot (t\theta'+R_\omega^{-1} \tilde{x}(R_\omega \theta'))}e^{i\beta_0 \cdot \tilde{x}(R_\omega\theta')}F(R_\omega \theta')\Vol_{\IS^{d-1}}(\theta',|d\theta'|) 
%\end{align*}
%where we used the invariance of $\Vol_{\IS^{d-1}}$ under $\text{SO}(d)$.
%The above splitting~\eqref{e:split-3-int} of $I_F$ into three pieces now reads
%\begin{align}
%I_F(\lambda \omega,t) & = I_F^{(-1)}(\lambda \omega,t) +I_F^{(0)}(\lambda \omega,t)+ I_F^{(1)}(\lambda \omega,t),\quad\text{ with }\\
%\label{e:expression-I-pm1}
%I_F^{(j)}(\lambda \omega,t)& =\int_{\IS^{d-1}} \chi_j( \theta' \cdot e_d)e^{i\lambda e_d \cdot (t\theta'+R_\omega^{-1} \tilde{x}(R_\omega \theta'))}e^{i\beta_0 \cdot \tilde{x}(R_\omega\theta')}F(R_\omega \theta')\Vol_{\IS^{d-1}}(\theta',|d\theta'|)  ,  
%\end{align}
and we study each piece separately.
To state our results, we also define, for $\eta \in \R^d\setminus\{0\}$, 
\begin{align}
\label{e:}
\mathbf{C}_j(\eta):=\left\{\theta\in\IS^{d-1} \text{ such that } \theta \cdot  \frac{\eta}{|\eta|}  \in \supp(\chi_j)\right\}, \quad \text{for }j=-1,0,1.
\end{align}
These are three closed subsets of the sphere.
According to the above properties of the cutoff functions $\chi_j$, we have, for all $\eta \in \R^d\setminus\{0\}$, 
 $$\mathbf{C}_{-1}(\eta)\cup\mathbf{C}_{0}(\eta)\cup\mathbf{C}_{1}(\eta)=\IS^{d-1}, $$
$\mathbf{C}_{\pm 1}(\eta)$ is a neighborhood of $\pm \frac{\eta}{|\eta|}$ in $\IS^{d-1}$, and $\mathbf{C}_{0}(\eta)$ is a neighborhood of $\left(\frac{\eta}{|\eta|}\right)^\perp$ in $\IS^{d-1}$.
We also notice that these three sets only depend on $\frac{\eta}{|\eta|}=\omega$, whence $\mathbf{C}_j(\eta)=\mathbf{C}_j(\frac{\eta}{|\eta|})=\mathbf{C}_j(\omega)$.

\subsection{Nonstationary points}
We start with the integral $I_F^{(0)}$ for which we will use the following lemma.

\begin{lemm}[Nonstationary points]
\label{c:non-stat-points}
%For all $\chi_0 \in \mathcal{C}^\infty_c(-1,1)$, there exists $C(\chi_0)>0$ such that for all $N\in\N$, 
There is a function $r\mapsto C_N(r) >0$ and $\mathsf{t}_0>0$ depending only on $\tilde{x}, \mathbf{v}:\mathbb{S}^{d-1}\mapsto \mathbb{R}^d$ and $\chi_0$ from definition~\ref{def-chi-j}  such that 
for all $t\geq \mathsf{t}_0$
%$\tilde{x} \in \mathcal{C}^{N+1}(\IS^{d-1})$, $t> C(\chi_0)\|D\tilde{x} \|_{L^\infty(\IS^{d-1})}$, $\omega \in \IS^{d-1}, \lambda >0$ 
and $F \in \ml{C}^N(\IS^{d-1})$, we have 
\begin{align*}
I_F^{(0)}(\lambda \omega,t)= ( i\lambda t )^{-N} \int_{\IS^{d-1}}e^{i\lambda \omega\cdot(t\mathbf{v}(\theta)+\tilde{x}(\theta))}  \mathcal{L}_{N,t}^\omega \left( \chi_0 \left(\theta \cdot  \omega \right)e^{i\beta_0\cdot\tilde{x}(\theta)}F(\theta) \right) d\Vol(\theta),
\end{align*}
or equivalently for $\xi \in \R^d\setminus\{\beta_0\}$,
\begin{multline}
( i|\xi-\beta_0| t )^{N}I_F^{(0)}(\xi-\beta_0,t)\\
= \int_{\IS^{d-1}}e^{i(\xi-\beta_0)\cdot(t\mathbf{v}(\theta)+\tilde{x}(\theta))}  \mathcal{L}_{N,t}^{\frac{\xi-\beta_0}{|\xi-\beta_0|}} \left( \chi_0 \left(\theta \cdot  \frac{\xi-\beta_0}{|\xi-\beta_0|} \right)e^{i\beta_0\cdot\tilde{x}(\theta)}F(\theta) \right) d\Vol(\theta) ,
\end{multline}
where  $\mathcal{L}_{N,t}^\omega$ is a differential operator of order $N$ on $\IS^{d-1}$ with smooth coefficients only depending on $X, \omega$ and $t$ and such that for all $\psi \in \ml{C}^N(\IS^{d-1})$, for all $\omega \in \IS^{d-1}$, for all $\theta \in \mathbf{C}_0(\omega)$, and all $t \geq \mathsf{t}_0$ we have 
\begin{align*}
\left|  (\mathcal{L}_{N,t}^\omega \psi ) (\theta) \right| \leq 
C_N \left(\| \tilde{x} \|_{\ml{C}^{N+1}(\mathbf{C}_0(\omega))} \right)\left(\sum_{|\alpha| \leq N} \left|\nabla_\theta^\alpha \psi (\theta) \right| \right) .
%C_N\left(\frac{\| \tilde{x} \|_{W^{N+1,\infty}(\mathbf{C}_0(\omega))}}{t} \right) \left(\sum_{|\alpha| \leq N} \left|\nabla_\theta^\alpha \psi (\theta) \right| \right) .
\end{align*}
In particular, 
\begin{align*}
| I_F^{(0)}(\lambda \omega, t) |  \leq 
C_N \left(\| \tilde{x} \|_{\ml{C}^{N+1}(\mathbf{C}_0(\omega))}\right) (\lambda t )^{-N}    \|F\|_{W^{N,1}(\mathbf{C}_0(\omega))},
%C_N\left(\frac{\| \tilde{x} \|_{W^{N+1,\infty}(\mathbf{C}_0(\omega))}}{t} \right) (\lambda t )^{-N}    \|F\|_{W^{N,1}(\mathbf{C}_0(\omega))},
\end{align*}
or equivalently, for $\xi \in \R^d\setminus\{\beta_0\}$ 
\begin{align*}
| I_F^{(0)}(\xi-\beta_0, t) |  \leq C_N \left( \| \tilde{x} \|_{\ml{C}^{N+1}(\mathbf{C}_0(\xi-\beta_0))} \right) (|\xi-\beta_0| t )^{-N}   \|F\|_{W^{N,1}(\mathbf{C}_0(\xi-\beta_0))}.
%\left( \frac{\| \tilde{x} \|_{W^{N+1,\infty}(\mathbf{C}_0(\xi-\beta_0))}}{t} \right) (|\xi-\beta_0| t )^{-N}   \|F\|_{W^{N,1}(\mathbf{C}_0(\xi-\beta_0))}.
\end{align*}
\end{lemm}
Again, this kind of estimates is classical and the only novelty here is the explicit control in terms of the various parameters involved which will be obtained by a careful inspection of the usual arguments.

\begin{proof} We let $\phi_{t,\omega}:=\omega\cdot\left(\mathbf{v}(\theta)+\frac{\tilde{x}(\theta)}{t}\right)$ so that 
$$\frac{\nabla\phi_{t,\omega}}{\|\nabla\phi_{t,\omega}\|^2}\left(e^{i\lambda \omega\cdot(t\mathbf{v}(\theta)+\tilde{x}(\theta))} \right)=(i\lambda t)e^{i\lambda \omega\cdot(t\mathbf{v}(\theta)+\tilde{x}(\theta))}.$$ 
Note that 
$$\nabla\phi_{t,\omega}=\nabla\phi_{\infty,\omega}(\theta)+\frac{D\tilde{x}(\theta)\omega}{t}, \quad \phi_{\infty,\omega}(\theta)=\omega\cdot\mathbf{v}(\theta),  $$
and that there exists some constant $c_0>0$ depending only on $\mathbf{v}$ such that $\|\nabla\phi_{\infty,\omega}(\theta)\|\geq c_0 d_{\IS^{d-1}}(\theta,\pm\omega)=c_0(1-|\theta\cdot\omega|)$ (as the critical points are nondegenerate). In particular, this operator is well defined as soon as $\theta\cdot\omega$ lies on the support of $\chi_0$ and as $t\geq \mathsf{t}_0 :=  c(\chi_0,\mathbf{v}) \|D\tilde{x}\|_{L^{\infty}}$ (for some constant depending only on $\chi_0$ and $\mathbf{v}$). Hence, we get the first of the Lemma by letting
$$\mathcal{L}^\omega_{N,t}:=\left(\left(\frac{\nabla\phi_{t,\omega}}{i\|\nabla\phi_{t,\omega}\|^2} \right)^*\right)^N.$$
The second part follows from the fact that this is a differential operator of order $\leq N$ whose coefficients depend on a certain number of derivatives of $\tilde{x}$ (and $\mathbf{v}$).

\end{proof}
Once and for all, we will use the function $\chi_0\in\mathcal{C}^\infty_c(-1,1)$ from definition~\ref{def-chi-j} with large enough support to meet the requirements on the support of $\chi_{\pm 1}$ in Lemma~\ref{c:stat-points} below. Then, the reference time $\mathsf{t}_0=C(\chi_0,\mathbf{v})\|D\tilde{x} \|_{L^\infty(\IS^{d-1})}$ is fixed, depending only on $\chi_0$ and the parametrization $\tilde{x}$ of some given convex boundary by the outward normal. The time $\mathsf{t}_0$ will always only depend on these three ingredients appearing in the phase function part of the oscillatory integral and not on the amplitude unless specified otherwise, it only deals with the nonstationary part of the oscillatory integral and tells us when the effect of $\tilde{x}$ becomes negligible compared to the size of $t\omega\cdot\mathbf{v}(\theta)$.

\subsection{Stationary points}
We now turn to the terms $I_F^{(\pm 1)}$ and need the following lemma, which is again more or less classical (asymptotics of the Fourier transform of the surface measure on a convex set):

\begin{lemm}
\label{c:stat-points}
For all $\chi_{1} \in \mathcal{C}^\infty([-1,1])$ compactly supported in a small enough neighborhood of $1$ and equal to one on a slightly smaller neighborhood of $1$, with $\chi_{-1}(s) = \chi_1(-s)$, and for all $N\in \N^*$, we have, for all $t>0, \xi \in \R^d\setminus\{\beta_0\}$, $\tilde{x} \in \ml{C}^{2N+ d}(\IS^{d-1})$ and all $F \in \ml{C}^{2N+ d}(\mathbf{C}_{\pm1}(\xi-\beta_0))$,
\begin{multline*}
  I_F^{(\pm 1)}(\xi-\beta_0,t)\\
  =
  \frac{e^{ it(\xi-\beta_0)\cdot\mathbf{v}\left(\pm\frac{\xi-\beta_0}{|\xi-\beta_0|}\right)} e^{\mp i\frac{\pi}{4}(d-1)}}{\sqrt{\kappa\left(\mathbf{v}\left(\pm\frac{\xi-\beta_0}{|\xi-\beta_0|}\right)\right)}}
  \left(\frac{2\pi}{t|\xi-\beta_0|}\right)^{\frac{d-1}{2}}  \sum_{j=0}^{N-1} \frac{1}{(t|\xi-\beta_0|)^j} L_{j, \frac{\xi-\beta_0}{|\xi-\beta_0|}}^\pm (e^{i \xi \cdot \tilde{x}(\cdot)}F) \left(\pm \frac{\xi-\beta_0}{|\xi-\beta_0|}\right) \\
 +\ml{O}_{N,\beta_0}(1) \frac{ \max\left\{1,\| \tilde{x}\|_{W^{2N+ d,\infty}(\mathbf{C}_{\pm1}(\xi-\beta_0))}^{2N+ d}|\xi-\beta_0|^{2N+ d}\right\}}{(t|\xi-\beta_0|)^{N+\frac{d-1}{2}}}   \|F\|_{W^{2N+ d,1}(\mathbf{C}_{\pm1}(\xi))} ,
\end{multline*}
where the constant in the remainder $\ml{O}_{N,\beta_0}(1)$ depends only on $N$, $\mathbf{v}$ and the cutoff functions, where $L_{j, \omega}^\pm$ are differential operators of order $\leq 2j$ whose coefficients are uniformly bounded in $\omega\in\IS^{d-1}$, where $L_0=1$ and $\kappa(\mathbf{v}(\pm\theta))$ is the Gauss curvature at the point $\mathbf{v}(\pm\theta)$ of $\partial K$.

%\begin{multline*}
%  I_F^{(\pm 1)}(\xi-\beta_0,t)\\
%  =e^{\pm it|\xi-\beta_0|} e^{\mp i\frac{\pi}{4}(d-1)}\left(\frac{2\pi}{t|\xi-\beta_0|}\right)^{\frac{d-1}{2}}  \sum_{j=0}^{N-1} \frac{1}{(t|\xi-\beta_0|)^j} L_{j, \frac{\xi-\beta_0}{|\xi-\beta_0|}}^\pm (e^{i \xi \cdot \tilde{x}(\cdot)}F) \left(\pm \frac{\xi-\beta_0}{|\xi-\beta_0|}\right) \\
% +\ml{O}_{N,\beta_0}(1) \frac{ \max\left\{1,\| \tilde{x}\|_{W^{2N+ d,\infty}(\mathbf{C}_{\pm1}(\xi-\beta_0))}^{2N+ d}|\xi-\beta_0|^{2N+ d}\right\}}{(t|\xi-\beta_0|)^{N+\frac{d-1}{2}}}   \|F\|_{W^{2N+ d,1}(\mathbf{C}_{\pm1}(\xi))} ,
%\end{multline*}
%where the constant in the remainder $\ml{O}_{N,\beta_0}(1)$ depends only on $N$ and the cutoff functions and where $L_{j, \omega}^\pm=R_{\omega}^{-1*}\circ L_j^{\pm}\circ R_{\omega}^*$ with $L_j^{\pm}$ the differential operator from Lemma~\ref{l:Littman}.
\end{lemm}

\begin{rema} We note that the growth in $|\xi|$ is a priori quite bad (except if $\tilde{x}=0$, which is e.g. the case when studying dynamical correlations for functions, see Section~\ref{s:correlation} below) and we will have to pay attention to this problem in the upcoming sections. For instance, this reads, for $N=1$ and for a constant $C$ depending on $\tilde x$, 
\begin{align*}
  \left|I_F^{(\pm 1)}(\xi-\beta_0,t) - \frac{e^{it(\xi-\beta_0)\cdot\mathbf{v}\left(\pm\frac{\xi-\beta_0}{|\xi-\beta_0|}\right)} e^{-i\frac{\pi}{4}(d-1)}}{\sqrt{\kappa\left(\mathbf{v}\left(\pm\frac{\xi-\beta_0}{|\xi-\beta_0|}\right)\right)}}\left(\frac{2\pi}{t|\xi|}\right)^{\frac{d-1}{2}}  e^{i \xi \cdot \tilde{x}\left(\frac{\xi-\beta_0}{|\xi-\beta_0|}\right)}F \left(\frac{\xi-\beta_0}{|\xi-\beta_0|}\right) \right| \\
 \leq  C \frac{ |\xi-\beta_0|^{2+d}}{(t|\xi-\beta_0|)^{1+\frac{d-1}{2}}}   \|F\|_{W^{2+d,1}(\mathbf{C}_1(\xi))} .
\end{align*}
\end{rema}
Again, the proof of Lemma~\ref{c:stat-points} is classical. After using a convenient coordinate chart on the sphere of the form
\begin{equation*}
\begin{array}{rcl}
B_{\R^{d-1}}(0,1) & \to & \{ x \in \IS^{d-1} ,\pm x_d>0 \} \\
x' & \mapsto & \big( x' , \pm \sqrt{1-|x'|^2} \big) ,
\end{array}
\end{equation*}
we are in the set-up of the standard stationary phase asymptotics Lemma~\cite[Th.~3.16]{Zworski12}. Compared with the proof in this reference, only two points require a particular attention:
\begin{itemize}
 \item The first one is that we want some uniformity of the constants with respect to the parameter $\omega= \frac{\xi-\beta_0}{|\xi-\beta_0|}$. Equivalently, we need to apply some Morse Lemma with parameter $\omega$. Hence, we need to verify that the size of the chart when applying the Morse Lemma in the proof of~\cite[Th.~3.16]{Zworski12} can be made uniform and that the coordinate charts have derivatives uniformly bounded in terms of $\omega$. This can verified by recalling that the Morse lemma near a (nondegenerate) critical point $x_0=0$ of a function $\phi$ is obtained by writing the Taylor formula
 $$\phi(x)=\phi(0)+x^T\left(\int_0^1(1-t)d^2\phi(tx) dt\right)x.$$
 Then, recall that, given a symmetric matrix $Q_0$, the map $M\in\{M: Q_0M\in S_n(\IR)\}\mapsto M^TQ_0M\in S_n(\IR)$ is invertible in a neighborhood of $\Id$ (whose size depends on $Q_0$). Hence, one concludes that $\left(\int_0^1(1-t)d^2\phi(tx) dt\right)=\frac{1}{2}M(x)^Td^2\phi(0)M(x),$
 for some smooth function $x\mapsto M(x)$ defined in a neighborhood of $0$ whose size depends on $d^2\phi(0)$. The size of this neighborhood can be chosen in a way that depends linearly on the norm of $d^2\phi(0)$. As our critical points vary in a compact part, their corresponding Hessian varies in a compact part and the neighborhood (and thus the size of the support of $\chi_1$) can be chosen uniformly as well as the involved constants.
 \item The second point is that the statement of~\cite[Th.3.16]{Zworski12} involves the $\ml{C}^{2N+d}$-norm. Yet, inspecting the proof (namely, step 3 in the proof of p.43 together with the proof of Lemma 3.5(ii)), one finds a control by the $W^{2N+d,1}$-norm.
 \end{itemize}
 
 We refer to~\cite{DangLeautaudRiviere21} for detailed proofs of Lemmas~\ref{c:non-stat-points} and~\ref{c:stat-points} in case $\mathbf{v}(\theta)=\theta$.

For later purposes,  we introduce the translation invariant Hamiltoninan
\begin{align}
\label{e:def-lambda-pm}
\lambda_\pm(\xi):=(\xi-\beta_0)\cdot\mathbf{v}\left(\pm\frac{\xi-\beta_0}{|\xi-\beta_0|}\right) = \pm h_K(\pm (\xi-\beta_0)) , 
\end{align}
where $h_K$ is the support function of the convex set $K$ in the definition of $\mathbf{v}$
(see Section~\ref{s:Finsler-Ham} for the second equality). We also introduce the 
critical values of the height function that will play a special role in our analysis:
\begin{equation}\label{e:critical-values}\Lambda_{\beta_0}:=\left\{\lambda_\pm(\xi) , \xi\in\IZ^d\setminus\{\beta_0\}\right\}.
\end{equation}
Thanks to~\eqref{e:interior-0}, this forms a discrete and locally finite subset of $\IR$. Note also that, thanks to~\eqref{e:interior-0}, $\pm \lambda_\pm(\xi)\geq c_1|\xi-\beta_0|$ for some uniform $c_1>0$ (depending only on $\mathbf{v}$). In that direction, we record the following useful lemma.
\begin{lemm}\label{l:support-chi1} There exists $c_0>0$ such that, for $\chi_{1} \in \mathcal{C}^\infty([-1,1])$ compactly supported in a small enough neighborhood of $1$ and equal to one on a slightly smaller neighborhood of $1$, and for all $(\theta,\xi)$ in the support of $\chi_{\pm 1}\left(\theta\cdot\frac{\xi-\beta_0}{|\xi-\beta_0|}\right)$, one has
$$\left|\mathbf{v}(\theta)\cdot(\xi-\beta_0)\right|\geq c_0|\xi-\beta_0|.$$
\end{lemm}
\begin{proof} We let $\epsilon_0>0$ be such that $\supp(\chi_1)= [1- \epsilon_0,1]$. Then, for $(\theta,\xi)$ in the support of $\chi_{\pm 1}\left(\theta\cdot\frac{\xi-\beta_0}{|\xi-\beta_0|}\right)$ with $\xi\neq\beta_0$, one has
$$\left|\mathbf{v}(\theta)\cdot\frac{(\xi-\beta_0)}{|\xi-\beta_0|}\right|=|\lambda_\pm(\xi)/|\xi-\beta_0||+\mathcal{O}(\epsilon_0)\geq c_1- \mathcal{O}(\epsilon_0),$$
which proves the statement for $\epsilon_0$ small enough. 
\end{proof}

\section{Asymptotics of twisted dynamical correlations}\label{s:correlation}
In this section, as a first application of these fine stationary phase asymptotics, we give an accurate description of the correlation function as $t\rightarrow +\infty$. See Theorem~\ref{t:twisted-correlations} for a precise statement. As a byproduct, this shows how anisotropic Sobolev norms naturally appear when studying analytical properties of geodesic flows and it also proves Theorem~\ref{t:maintheo-correlations} from the introduction. 

For the sake of simplicity, we restrict ourselves to the case where $k_1=2d-1$, $k_2=0$ and $\tilde{x}(\theta)=0$. Namely, we fix two smooth functions $\varphi$ and $\psi$ in $\ml{C}^{\infty}(S\IT^d)$ and we want to analyze the behaviour as $t\rightarrow +\infty$ of
\begin{align*}\Cor_{\varphi,\psi}(t,\beta_0)&:=\int_{S\IT^d}\varphi(x,\theta)e^{-t\mathbf{V}_{\beta_0}}(\psi)(x,\theta)|dx| d\Vol(\theta)\\
&=\int_{S\IT^d}e^{-it\mathbf{v}(\theta)\cdot\beta_0}\varphi(x,\theta)\psi(x-t\mathbf{v}(\theta),\theta)|dx| d\Vol(\theta)
 \end{align*}
where $\beta_0\in H^{1}(\IT^d,\IR)\simeq \IR^d$. According to Remark~\ref{r:correlation-function}, this can be rewritten as
$$\Cor_{\varphi,\psi}(t,\beta_0)=\sum_{\xi\in\IZ^d}\int_{\IS^{d-1}}\widehat{\varphi}_{\xi}(\theta)\widehat{\psi}_{-\xi}(\theta)e^{it(\xi-\beta_0)\cdot\mathbf{v}(\theta)} d\Vol(\theta),$$
where
$$\varphi(x,\theta)=\sum_{\xi\in\IZ^d}\widehat{\varphi}_{\xi}(\theta)\mathfrak{e}_\xi(x),\quad\text{and}\quad\psi(x,\theta)=\sum_{\xi\in\IZ^d}\widehat{\psi}_{\xi}(\theta)\mathfrak{e}_\xi(x).$$
We will now implement the decomposition~\eqref{e:split-3-int}-\eqref{e:split-3-int-bis} together with Lemmas~\ref{c:non-stat-points} and~\ref{c:stat-points} in order to analyze the asymptotic expansion of $\Cor_{\varphi,\psi}(t,\beta_0)$ as $t\rightarrow+ \infty$.
\begin{rema} Modulo some tedious work, the analysis could be extended to the more general framework of Lemma~\ref{l:reduction} except that the terms in the asymptotic expansion will be slightly less explicit. 
\end{rema}
First, we write
$$\Cor_{\varphi,\psi}(t,\beta_0)=E_{\beta_0}(\varphi,\psi)+\sum_{\xi\in\IZ^d\setminus\{\beta_0\}}\int_{\IS^{d-1}}\widehat{\varphi}_{\xi}(\theta)\widehat{\psi}_{-\xi}(\theta)e^{it(\xi-\beta_0)\cdot\mathbf{v}(\theta)} d\Vol(\theta),$$
where
\begin{equation}\label{e:leading-term}E_{\beta_0}(\varphi,\psi):=\int_{\IS^{d-1}}\widehat{\varphi}_{\beta_0}(\theta)\widehat{\psi}_{-\beta_0}(\theta) d\Vol(\theta),\ \text{if}\ \beta_0\in\IZ^d,\end{equation}
and $E_{\beta_0}(\varphi,\psi)=0$ otherwise.

\subsection{Anisotropic Sobolev spaces of distributions, splitting the correlation function}\label{ss:splitting-correlation}

We decompose these correlations further by writing
\begin{align}
 \Cor_{\varphi,\psi}(t,\beta_0)&  =E_{\beta_0}(\varphi,\psi)+ \Cor_{-1,\varphi,\psi}(t) + \Cor_{0,\varphi,\psi}(t) +\Cor_{1,\varphi,\psi}(t) , \quad \text{with, for }j \in \{-1,0,1\} ,\label{e:decomp-C0}\\
 \Cor_{j,\varphi,\psi}(t)  & :=  \sum_{\xi\in\IZ^d\setminus\{\beta_0\}}\int_{\IS^{d-1}} \chi_j \left(\theta \cdot  \frac{\xi-\beta_0}{|\xi-\beta_0|} \right) \widehat{\varphi}_{\xi}(\theta)\widehat{\psi}_{-\xi}(\theta)e^{it(\xi-\beta_0)\cdot\mathbf{v}(\theta)} d\Vol(\theta),
 \nonumber
\end{align} 
where the $\chi_j$ are the cutoff functions defined in~\S\ref{ss:split-oscillatory-integral} and $E_{\beta_0}(\varphi,\psi)$ is defined in~\eqref{e:leading-term}. Below, we will drop the dependance in $(\varphi,\psi)$ in the index of $\Cor$ to avoid too heavy notations.
%{\red We put the partition of unity index $j\in \{0,\pm 1\}$ in subscript to differentiate from the  index of differential forms in superscript. }

\medskip
We first consider the term $\Cor_{0}(t)$. Applying Lemma~\ref{c:non-stat-points} to $\tilde{x} =0$ and $F = \widehat{\varphi}_{\xi}\widehat{\psi}_{-\xi}$ combined with the Cauchy-Schwarz inequality $\Vert \widehat{\varphi}_{\xi}\widehat{\psi}_{-\xi} \Vert_{W^{N,1}(\mathbf{C}_0(\xi-\beta_0)))}\leqslant \Vert \widehat{\varphi}_{\xi}\Vert_{H^{N}(\mathbf{C}_0(\xi-\beta_0))} \Vert \widehat{\psi}_{-\xi}\Vert_{H^{N}(\mathbf{C}_0(\xi-\beta_0))}  $, we have the following statement: 
\begin{lemm}\label{l:correlation-remainder}
%For all $\chi_0 \in \mathcal{C}^\infty_c(-1,1)$ and 
For all $N\in\N$, there is $C_N >0$ such that for every $t>0$, every $\beta_0\in\IR^d$, and every $\varphi,\psi\in\ml{C}^\infty(S\IT^d)$, we have 
\begin{align*}
\left| \Cor_{0}(t) \right| 
& \leq C_N t^{-N} \sum_{\xi\in\IZ^d\setminus\{\beta_0\}}\frac{\left\|\widehat{\varphi}_{\xi}\right\|_{H^{N}(\mathbf{C}_0(\xi-\beta_0))}\left\|\widehat{\psi}_{-\xi}\right\|_{H^{N}(\mathbf{C}_0(\xi-\beta_0))}}{|\xi-\beta_0|^{N}} .
\end{align*}
\end{lemm}

\medskip
We now consider the terms $\Cor_{\pm 1}(t)$. Applying similarly Lemma~\ref{c:stat-points} to $\tilde{x}(\theta) =0$ and $F = \widehat{\varphi}_{\xi}\widehat{\psi}_{-\xi}$, we have the following statement: 
\begin{lemm}\label{l:correlation-stat}
For all $\chi_{1} \in \mathcal{C}^\infty([-1,1])$ compactly supported in a small enough neighborhood of $1$ and equal to one on a slightly smaller neighborhood of $1$, with $\chi_{-1}(s) = \chi_1(-s)$, and for all $N\in \N^*$, we have for every $\beta_0 \in \R^d$,
\begin{multline*}
  \Cor_{\pm 1}(t)  = \sum_{\xi\in\IZ^d\setminus\{\beta_0\}}  \frac{e^{ it\lambda_\pm(\xi)} e^{\mp i\frac{\pi}{4}(d-1)}}{\sqrt{\kappa\circ\mathbf{v}\left(\pm\frac{\xi-\beta_0}{|\xi-\beta_0|}\right)}}\left(\frac{2\pi}{t|\xi-\beta_0|}\right)^{\frac{d-1}{2}}\\
  \times\sum_{j=0}^{N-1} \frac{1}{(t|\xi-\beta_0|)^j} L_{j, \frac{\xi-\beta_0}{|\xi-\beta_0|}}^\pm \left(\widehat{\varphi}_{\xi}\widehat{\psi}_{-\xi} \right) \left(\pm \frac{\xi-\beta_0}{|\xi-\beta_0|}\right) \\
 +\ml{O}_N\left(t^{-N-\frac{d-1}{2}}\right) \sum_{\xi\in\IZ^d\setminus\{\beta_0\}} \frac{\left\|\widehat{\varphi}_{\xi}\right\|_{H^{2N+ d}(\mathbf{C}_{\pm1}(\xi-\beta_0))}\left\|\widehat{\psi}_{-\xi}\right\|_{H^{2N+ d}(\mathbf{C}_{\pm1}(\xi-\beta_0))}}{|\xi-\beta_0|^{N+\frac{d-1}{2}}}  ,
\end{multline*}
as $t\to + \infty$, where the constant in the remainder $\ml{O}_N(t^{-N-\frac{d-1}{2}})$ depends also on $\beta_0\in\IR^d$ and where $\lambda_\pm(\xi)$ was defined in~\eqref{e:def-lambda-pm} and depends on $\beta_0$ and $K$. 
\end{lemm}

The above decomposition motivates the following definition of anisotropic Sobolev norms.
\begin{def1}[Anisotropic Sobolev spaces] 
\label{d:sobonorm}
Let $\gamma \in\IR^d$ and let $(s_0,s_1,N_0,N_1)$ be an element in $\IZ_+^2\times\IR^2$.
For every $\varphi(x,\theta)=\sum_\xi\widehat{\varphi}_{\xi}(\theta)\mathfrak{e}_\xi(x) \in\ml{C}^{\infty}(S\IT^d)$, we define the following anisotropic Sobolev norms:
$$
\|\varphi\|_{\mathcal{H}^{s_0,N_0,s_1,N_1}_\gamma}^2:=\sum_{\xi\in \IZ^d}
\langle \xi \rangle^{2 N_0} \left\|\widehat{\varphi}_{\xi}\right\|_{H^{s_0}\left(\mathbf{C}_0 (\xi-\gamma)\right)}^2
+\sum_{\xi\in \IZ^d ,\pm}\langle \xi \rangle^{2 N_1} \left\|\widehat{\varphi}_\xi\right\|_{H^{s_1}\left(\mathbf{C}_{\pm 1}(\xi-\gamma)\right)}^2.
$$ 
\end{def1}
In our applications, these norms are used for $\gamma = \pm \beta_0$. 
The geometric content of these anisotropic norms is discussed in Section~\ref{s:geometry-anisotropic} below.

\subsection{Asymptotics of the correlation function}

Now, combining this definition with the reduction made in~\S\ref{ss:splitting-correlation}, we find
\begin{theo}[Asymptotics of twisted correlations] 
\label{t:twisted-correlations}
Let $\beta_0\in H^1(\IT^d,\IR)$ and let $N\in \IZ_+^*$. For every $\varphi(x,\theta)=\sum_{\xi\in\IZ^d}\widehat{\varphi}_\xi(\theta)\mathfrak{e}_\xi(x)$ and $\psi(x,\theta)=\sum_{\xi\in\IZ^d}\widehat{\psi}_\xi(\theta)\mathfrak{e}_\xi(x)$ in $\mathcal{C}^{\infty}(S\IT^d)$, one has
\begin{align*}
 \Cor_{\varphi,\psi}(t,\beta_0) &:=\int_{S\IT^d}e^{-it\beta_0\cdot\mathbf{v}(\theta)}\varphi(x,\theta)\psi\circ e^{-tV}(x,\theta)|dx| d\Vol(\theta)\\
 &=E_{\beta_0}(\varphi,\psi)  \\
 &+\left(2\pi\right)^{\frac{d-1}{2}}\sum_{j=0}^{N-1}\sum_{\xi\in\IZ^d\setminus\{\beta_0\},\pm}\frac{e^{i\left(t\lambda_\pm(\xi)\mp\frac{\pi}{4}(d-1)\right)}L_{j,\frac{\xi-\beta_0}{|\xi-\beta_0|}}^\pm\left(\widehat{\varphi}_\xi\widehat{\psi}_{-\xi}\right)\left(\pm\frac{\xi-\beta_0}{|\xi-\beta_0|}\right)}{\sqrt{\kappa\circ\mathbf{v}\left(\pm\frac{\xi-\beta_0}{|\xi-\beta_0|}\right)}\left(t|\xi-\beta_0|\right)^{j+\frac{d-1}{2}}}\\
 &\quad +\mathcal{O}_{N,\varphi,\psi}\left(\frac{1}{t^{N+\frac{d-1}{2}}}\right),
\end{align*}
where $E_{\beta_0}(\varphi,\psi)$ was defined in~\eqref{e:leading-term},
$L_{j,\omega}^{\pm}$ is the differential operator of degree $2j$ appearing in Lemma~\ref{c:stat-points} and, for every integer $ s\geq N+d$ the constant in the remainder is controlled by
$$C_{N}\left\|\varphi\right\|_{\mathcal{H}^{ s,-\frac{ s}{2},2N+d,-\frac{N}{2}-\frac{d-1}{4}}_{\beta_0}}\left\|\psi\right\|_{\mathcal{H}^{s,-\frac{s}{2},2N+d,-\frac{N}{2}-\frac{d-1}{4}}_{-\beta_0}}$$
with $C_N>0$ depending only on $d$, $N$, $s$, $\beta_0$ and the cutoff functions $(\chi_{j})_{j\in\{0, \pm1\}}$ used in \S\ref{s:analysis}.
\end{theo}
Theorem~\ref{t:maintheo-correlations} from the introduction is a direct consequence of this result, obtained by taking $\beta_0=0$ and $N=1$. 

\begin{rema}
\label{rem:quantumquantum}
Note that this result states convergence of the correlation function towards a constant $E_{\beta_0}(\varphi,\psi)$ at rate $t^{-\frac{d-1}{2}}$. the fluctuations around the equilibrium are governed by the quantum evolution operators $e^{it \lambda_\pm(D)}$. The latter is the magnetic half-wave group for the translation invariant Finsler structure on  $T^*\T^d$, associated with the convex set $K$. Here, recall that the operators $\lambda_\pm(D) = \pm h_K(\pm(D-\beta_0))$ (see \eqref{e:def-lambda-pm} Section~\ref{s:Finsler-Ham} for the notation) are a translation invariant Fourier multipliers that, on account to the assumption that $0 \in \Int(K)$, are elliptic pseudodifferential operators of order one on $\T^d$. The critical set $\Lambda_{\beta_0}$ in~\eqref{e:critical-values} is precisely the union of the spectra of $\lambda_+(D)$ and $\lambda_-(D)$. 
See~\cite{Ratner87, FaureTsujii15, DyatlovFaureGuillarmou2015, FaureTsujii17, FaureTsujii17b, FaureTsujii21} for related considerations in the context of contact Anosov flows. 
 See also~\cite{DangLeautaudRiviereJEDP} and Section~\ref{s:magnetic-laplacian} below for a more explicit connection with the Laplacian in the case $\mathbf{v}(\theta)=\theta$.

 \end{rema}
  
Recalling from the proof of Lemma~\ref{c:stat-points} that the operators $L_j^\pm$ can be computed explicitly (up to some tedious work), this theorem provides an explicit asymptotic expansion of the twisted correlation function for smooth observables. Besides that, another interesting feature of this theorem is that it illustrates how anisotropic Sobolev norms naturally appears when studying the asymptotic behaviour of the geodesic flow on the torus. This is particularly clear in the case of the remainder while for the term in the asymptotic expansion, one can remark that, using the standard Sobolev inequalities~\cite[\S5.6.3]{Evans10},
\begin{eqnarray*}\left|L_{j,\frac{\xi-\beta_0}{|\xi-\beta_0|}}^\pm\left(\widehat{\varphi}_\xi\widehat{\psi}_{-\xi}\right)\left(\pm\frac{\xi-\beta_0}{|\xi-\beta_0|}\right)\right|&\leq& C_j\|\widehat{\varphi}_\xi\|_{\ml{C}^{2j}(\mathbf{C}_{\pm 1}(\xi-\beta_0))}\|\widehat{\psi}_{-\xi}\|_{\ml{C}^{2j}(\mathbf{C}_{\pm 1}(\xi_{-\beta_0}))}\\
& \leq &\tilde{C}_j\|\widehat{\varphi}_\xi\|_{H^{2j+d/2+1}(\mathbf{C}_{\pm 1}(\xi-\beta_0))}\|\widehat{\psi}_{-\xi}\|_{H^{2j+d/2+1}(\mathbf{C}_{\pm 1}(\xi-\beta_0))}.
 \end{eqnarray*}
Hence, each term in the sum over $j$ is controlled by some anisotropic Sobolev semi-norm (that depends on $j$). In summary, test functions can have a priori arbitrarily large polynomial growth in $|\xi|$ away from the direction of $\xi-\beta_0$. Close to $\xi-\beta_0$, the situations is not as good and one needs to have moderate growth in $|\xi|$ to ensure the convergence of the sums.

\subsubsection{Further comments}

Let us now comment a little bit more on Theorem~\ref{t:twisted-correlations}. First, we emphasize that our strategy can be viewed as an analogue on flat tori of the strategy used by Ratner~\cite{Ratner87} to describe the asymptotic behaviour of the correlation function for the geodesic flow on hyperbolic manifolds. Like in this reference, we use tools from harmonic analysis to describe accurately the correlations and we end up naturally with anisotropic Sobolev norms (see e.g.~\cite[Th.1]{Ratner87} for the use of spaces with anisotropic H\"older regularity). As in~\cite[Cor.1]{Ratner87}, it is interesting to look at the case where $\varphi$ and $\psi$ do not depend on $\theta$. In that case, the asymptotic expansion of Theorem~\ref{t:twisted-correlations} reads as follows
\begin{eqnarray*}
\Cor_{\varphi,\psi}(t,\beta_0)&:=&\int_{S\IT^d}e^{-it\beta_0\cdot\mathbf{v}(\theta)}\varphi(x)\psi(x-t\mathbf{v}(\theta))|dx|d\Vol(\theta)\\
&=&\frac{2\pi^{\frac{d-1}{2}}\delta_{\IZ^d,\beta_0}}{\Gamma\left(\frac{d-1}{2}\right)}\left(\int_{\IT^d}\varphi(x)\mathfrak{e}_{-\beta_0}(x)|dx|\right)\left(\int_{\IT^d}\psi(x)\mathfrak{e}_{\beta_0}(x)|dx|\right)\\
 &+&\left(2\pi\right)^{\frac{d-1}{2}}\sum_{j=0}^{N-1}\sum_{\xi\in\IZ\setminus\{\beta_0\},\pm}c_{j}^\pm( \xi-\beta_0 )\frac{e^{i\left(t\lambda_\pm(\xi)\mp\frac{\pi}{4}(d-1)\right)}}{\left(t|\xi-\beta_0|\right)^{j+\frac{d-1}{2}}}\widehat{\varphi}_\xi\widehat{\psi}_{-\xi}\\
 &+&\mathcal{O}_{N,\varphi,\psi}\left(\frac{1}{t^{N+\frac{d-1}{2}}}\right),
\end{eqnarray*} 
where $\delta_{\IZ^d,\beta_0}=1$ if $\beta_0\in\IZ^d$ and $\delta_{\IZ^d,\beta_0}=0$ otherwise and where the coefficients $c_{j}^\pm( \xi-\beta_0 )$ depend only on the geometry of $\IS^{d-1}$ and are uniformly bounded in terms of $\xi$. In particular, we can verify that the term of degree $j$ is controlled by the following quantity (up to some constant depending only on $j$ and $d$)
$$\left(\sum_{\xi\in\IZ^d\setminus\{\beta_0\}}|\xi-\beta_0|^{-j-\frac{d-1}{2}}|\widehat{\varphi}_\xi|^2\right)^{\frac12}\left(\sum_{\xi\in\IZ^d\setminus\{\beta_0\}}|\xi+\beta_0|^{-j-\frac{d-1}{2}}|\widehat{\psi}_\xi|^2\right)^{\frac12}\leq C_{\beta_0}\|\varphi\|_{L^2}\|\psi\|_{L^2}.$$
The same bound would hold on the remainder term. Hence, $L^2$ is the natural space to consider when considering observables depending only on $x$ as in the case of hyperbolic manifolds~\cite[Cor.1]{Ratner87}. See also~\cite[Prop.~2.1]{HezariRiviere17} for related results on Birkhoff averages in the case of flat tori.

\subsubsection{Relation with the magnetic Laplacian}
\label{s:magnetic-laplacian}
When $\mathbf{v}(\theta)=\theta$ and when the observables $\varphi,\psi$ depend only on $x$ and not on $\theta$, the above 
discussion can also be understood differently if we make the connection with the magnetic Laplacian
$$\Delta_{\beta_0}:=\sum_{j=1}^d\left(\partial_{x_j}+i\beta_{0,j}\right)^2.$$
Indeed, if we rewrite according to~\cite[Eq.(25),p.347]{Steinbook}
\begin{eqnarray*}
\Cor_{\varphi,\psi}(t,\beta_0)&:=&\int_{S\IT^d}e^{-it\beta_0\cdot\theta}\varphi(x)\psi(x-t\theta)|dx| d\Vol(\theta)\\
&=&\sum_{\xi\in\IZ^d}\widehat{\varphi}_{\xi}\widehat{\psi}_{-\xi}\int_{\IS^{d-1}}e^{it(\xi-\beta_0)\cdot\theta} d\Vol(\theta)\\
&=&2\pi\sum_{\xi\in\IZ^d}\widehat{\varphi}_{-\xi}\widehat{\psi}_{\xi}(t|\xi+\beta_0|)^{\frac{2-d}{2}}J_{\frac{d-2}{2}}\left(2\pi t|\xi+\beta_0|\right)\\
&=&2\pi\int_{\IT^d}\varphi(x)\left(t\sqrt{-\Delta_{\beta_0}}\right)^{\frac{2-d}{2}}J_{\frac{d-2}{2}}\left(2\pi t\sqrt{-\Delta_{\beta_0}}\right)\psi(x)|dx|,
\end{eqnarray*} 
where $J_{\nu}$ is the standard Bessel function of the first kind. In particular, if we denote by $\Pi:(x,\theta)\in S\IT^d\mapsto x\in\IT^d$ the canonical projection, we obtain the following relation between the twisted geodesic flow and the magnetic Laplacian
\begin{equation}\label{e:Bessel-Laplacian}
\Pi_*e^{it\left(iV-\beta_0(V)\right)}\Pi^*=2\pi\left(t\sqrt{-\Delta_{\beta_0}}\right)^{\frac{2-d}{2}}J_{\frac{d-2}{2}}\left(2\pi t\sqrt{-\Delta_{\beta_0}}\right).
\end{equation}

For observables depending also on the $\theta$ variable, the expressions are slightly less explicit. Yet, as in Theorem~\ref{t:maintheo-correlations}, we can for instance consider the first term in the asymptotic expansion of Theorem~\ref{t:twisted-correlations}, which is given by
\begin{align*}
&\left(2\pi\right)^{\frac{d-1}{2}}\sum_{\xi\in\IZ^d\setminus\{\beta_0\},\pm}\frac{e^{\pm i\left(t|\xi-\beta_0|-\frac{\pi}{4}(d-1)\right)}}{\left(t|\xi-\beta_0|\right)^{\frac{d-1}{2}}}\left(\widehat{\varphi}_\xi\widehat{\psi}_{-\xi}\right)\left(\pm\frac{\xi-\beta_0}{|\xi-\beta_0|}\right) \\
&=\left(2\pi\right)^{\frac{d-1}{2}}\sum_{\xi\in\IZ^d\setminus\{\beta_0\},\pm}\frac{e^{\pm i\left(t|\xi+\beta_0|-\frac{\pi}{4}(d-1)\right)}}{\left(t|\xi+\beta_0|\right)^{\frac{d-1}{2}}}\left(\widehat{\varphi}_{-\xi}\widehat{\psi}_{\xi}\right)\left(\mp\frac{\xi+\beta_0}{|\xi+\beta_0|}\right)
.\end{align*}
If we introduce the following map
$$\Pi_{\beta_0}^\pm(\varphi):=\sum_{\xi\in\IZ^d\setminus\{-\beta_0\}}\widehat{\varphi}_\xi\left(\pm \frac{\xi+\beta_0}{|\xi+\beta_0|}\right)\mathfrak{e}_\xi,$$
then the (first term) asymptotic expansion of $\Cor_{\varphi,\psi}(t,\beta_0)$ in Theorem~\ref{t:twisted-correlations} can be rewritten, modulo $\mathcal{O}_{\varphi,\psi}\left(\frac{1}{t^{1+\frac{d-1}{2}}}\right)$, as
\begin{align}\label{beta0-beta0}
 \Cor_{\varphi,\psi}(t,\beta_0)  & = \left(\frac{2\pi}{t}\right)^{\frac{d-1}{2}}\sum_{\pm}e^{\mp \frac{i\pi(d-1)}{4}}\int_{\IT^d} \left( \frac{e^{\pm i t\sqrt{-\Delta_{-\beta_0}}}}{(-\Delta_{-\beta_0})^{\frac{d-1}{4}}}\circ\Pi_{-\beta_0}^{\pm}(\varphi)\right)(x)\Pi_{\beta_0}^{\mp}(\psi)(x)|dx|  \nonumber \\
 & = \left(\frac{2\pi}{t}\right)^{\frac{d-1}{2}}\sum_{\pm}e^{\mp \frac{i\pi(d-1)}{4}}\int_{\IT^d} \Pi_{-\beta_0}^{\pm}(\varphi)(x) 
 \left( \frac{e^{\pm i t\sqrt{-\Delta_{\beta_0}}}}{(-\Delta_{\beta_0})^{\frac{d-1}{4}}}\circ \Pi_{\beta_0}^{\mp}(\psi)\right)(x)|dx| ,
 \end{align}
after having used the Plancherel Theorem. Similarly, all the terms in the asymptotic expansion can be written in the same fashion except that the expression will be slightly more involved.

%%%%%%%%%%%%%%%%%%%%%%%%%%%%%
\subsection{Geometry of the anisotropic Sobolev norms}
\label{s:geometry-anisotropic}
With the geometric description of \S\ref{ss:riemannian} at hand, we can give a rough geometric interpretation of our anisotropic spaces using the notion of pseudodifferential operators~\cite{Hormander3}. Usually, Sobolev spaces are designed using the quantization of a symbol of the form $(1+|(\xi,\Theta)|_{x,\theta}^2)^{\frac{s}{2}}$ where $(x,\theta,\xi,\Theta)$ is an element in $T^*S\IT^d$ and $s$ is the Sobolev regularity. Here, due to the explicit structure of the problem, we did not write exactly things in that fashion. Yet, our spaces would in principle correspond to replace $s$ by a function $s(x,\theta,\xi,\Theta)$ whose values depend on the directions in $T^*S\IT^d$ and thus to work with anisotropic symbols. More precisely, taking $\gamma=0$ for simplicity, we would in fact require using this pseudodifferential approach that
\begin{itemize}
 \item near $E_0^*$, the symbol is given by $(1+|\xi|^2)^{\frac{N_1}{2}} (1+|\Theta|_\theta^2)^{\frac{s_1}{2}}$. Thus, we are roughly requiring a Sobolev regularity $N_1$ along $E_0^*$.
 \item near $\ml{H}^*\oplus \ml{V}^*$, the symbol is given by $(1+|\xi|^2)^{\frac{N_0}{2}} (1+|\Theta|_\theta^2)^{\frac{s_0}{2}}$. In particular, on $\mathcal{V}^*$, this correspond to a Sobolev regularity of order $s_0$ while on $\mathcal{H}^*$, the Sobolev regularity is $N_0$.
\end{itemize}
See Figure~\ref{f:anisotropic}. 
\begin{figure}[ht]
\includegraphics[scale=0.8]{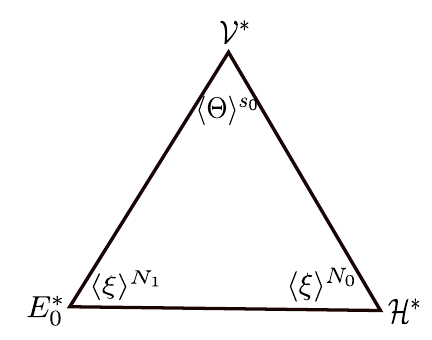}
\centering
\caption{\label{f:anisotropic}Schematic representation of the Sobolev regularity in the cotangent picture.}
\end{figure}

%%%%%%%%%%%%%%%%%%%%%%%%%%%%%%
%%%%%%%%%%%%%%%%%%%%%%%%%%%%%%
%%%%%%%%%%%%%%%%%%%%%%%%%%%%%%
%%%%%%%%%%%%%%%%%%%%%%%%%%%%%%
%%%%%%%%%%%%%%%%%%%%%%%%%%%%%%

\section{Anisotropic spaces of currents} 
\label{s:anis-spaces}

In Lemma~\ref{l:functional-calculus}, we saw that, for $\beta_0\in H^1(\IT^d,\IR)$ and for a smooth function $\chi:\IR\rightarrow\IR$ with enough decay at infinity, the operator
$$\widehat{\chi}(-i\mathbf{V}_{\beta_0}):=\int_{\IR}\chi(t)e^{-t\mathbf{V}_{\beta_0}}|dt|$$
is bounded when acting on the space of \emph{continuous} differential forms. Now we aim at describing anisotropic Sobolev spaces adapted to the dynamics of the geodesic flow on which $\widehat{\chi}(-i\mathbf{V}_{\beta_0})$ will still extend continuously.

\subsection{Anisotropic Sobolev spaces}\label{s:cutoff}

Motivated by the norms of Definition~\ref{d:sobonorm} appearing in the description of the correlation function, we introduce the following spaces of currents.
\begin{def1}[Anisotropic Sobolev spaces of currents]
\label{def-ani-spaces}
 Let $\beta_0\in H^1(\IT^d,\IR)$, let $0\leq k\leq 2d-1$ and let $(s_0,s_1,N_0,N_1)$ in $\mathbb{Z}_+^2\times\mathbb{R}^2$. We define the following anisotropic Sobolev norm:
$$
\|\varphi\|_{\mathcal{H}^{s_0,N_0,s_1,N_1}_{k,\beta_0}}^2:=\sum_{\xi\in\IZ^d}
\langle \xi\rangle^{2 N_0} \left\|\pi_{\xi}^{(k)}(\varphi)\right\|_{H^{s_0}\left(\mathbf{C}_0 (\xi-\beta_0)\right)}^2
+  \sum_{\xi\in\IZ^d,\pm}
\langle \xi \rangle^{2 N_1} \left\|\pi_{\xi}^{(k)}(\varphi)\right\|_{H^{s_1}\left(\mathbf{C}_{\pm1}(\xi-\beta_0)\right)}^2,
$$
where $\langle\eta\rangle :=(1+|\eta|^2)^{\frac12}$ and where the Sobolev norms $H^s$ on forms are understood in the sense of Remark~\ref{r:norm}. We define the space $\mathcal{H}^{s_0,N_0,s_1,N_1}_{k,\beta_0}$ to be the completion of $\Omega^{k}(S\IT^d)$ for this norm.
\end{def1}
As above, we note that these norms depend implicitely on the cutoff functions used in~\S\ref{ss:split-oscillatory-integral}. In particular, the conic neighborhood $\mathbf{C}_{\pm 1}(\omega)$ can be chosen arbitrarily close to $\omega\in\IS^{d-1}$ but it cannot be too large in order to apply Lemma~\ref{c:stat-points}. Using these spaces, one gets
\begin{theo}\label{t:integrated-correlation} Let $0\leq k\leq 2d-1$, $M, N$ be elements in $\IZ_+$, $\beta_0\in\ H^1(\IT^d,\mathbb{R})$ and $\chi\in\ml{C}^{\infty}_c(\IR_+^*)$. Then, 
$$\widehat{\chi}(-i\mathbf{V}_{\beta_0}): \mathcal{H}^{M,-M/2,0,-N/2}_{k,-\beta_0}\rightarrow (\mathcal{H}^{M,-M/2,0,-N/2}_{2d+1-k,\beta_0})'$$
 defines a continuous linear map, where $(\mathcal{H}^{M,-M/2,0,N/2}_{2d+1-k,\beta_0})'\subset \ml{D}^{\prime k}(S\IT^d)$ is the topological dual of $\mathcal{H}^{M,-M/2,0,N/2}_{2d+1-k,\beta_0}$. 
\end{theo}
Compared with the spaces appearing when describing the asymptotics of the correlation function, we now require that test currents are regular enough along the vertical space $\mathcal{V}$ while they can be singular along the horizontal space $\IR V\oplus \mathcal{H}$. See Figure~\ref{f:anisotropic} with $s_0=M$, $N_0=-M/2$ and $N_1=-N/2$.

\subsection{Mapping properties} 
For later applications to counting orthogeodesics, we also fix a smooth map
$$\tilde{x}:\IS^{d-1}\rightarrow \IR^d,$$
and our goal is to study more generally the analytical properties of the operator:
$$\widehat{\chi}(-i\mathbf{V}_{\beta_0})\mathbf{T}_{-\tilde{x}}^*:=\int_{\IR}\chi(t)e^{-t\mathbf{V}_{\beta_0}}\mathbf{T}_{-\tilde{x}}^{*}|dt|,$$
under appropriate assumptions on $\chi$. To that aim, we fix $0\leq k_1,k_2\leq 2d-1$ and two smooth forms $(\varphi,\psi)\in\Omega^{k_1}(S\IT^d)\times\Omega^{k_2}(S\IT^d)$ (with $k_1+k_2=2d-1$). Hence, for $\chi$ with enough regularity, we want to study the properties of
\begin{equation}\label{e:integrated-correlation}\int_{S\IT^d}\varphi\wedge \widehat{\chi}(-i\mathbf{V}_{\beta_0})\mathbf{T}_{-\tilde{x}}^*(\psi)=\int_{\IR}\chi(t)\Cor_{\varphi,\mathbf{T}_{-\tilde{x}}^*(\psi)}(t,\beta_0)|dt|\end{equation}
in terms of the anisotropic Sobolev norms we have just introduced. It is precisely in the present section that we will benefit from the regularization effect due to averaging on the time $t$ variable.

In order to state the main technical result of this section, let us introduce the following definition:
\begin{def1} Let $p\in\IR$ and let $N\in\IZ_+$. We say that $\chi$ is $(N,p)$-admissible if $\chi\in\mathcal{C}^{\infty}(\IR_+)$ and if it satisfies the following properties:
\begin{itemize}
 \item the support of $\chi$ does not contain $0$,
 \item for every $0\leq m\leq N$,
 $$\lim_{t\rightarrow +\infty}\frac{d^m}{dt^m}\left(t^p\chi(t)\right)=0,$$
 \item for every $0\leq m\leq N$, $t\mapsto  \frac{d^m}{dt^m}\left(t^p\chi(t)\right)\in L^1(\IR_+).$
\end{itemize}
\end{def1}

This definition includes the case of smooth compactly supported functions on $\IR_+^*$ and Theorem~\ref{t:integrated-correlation} is actually a corollary of the much more precise statement:
\begin{theo}
\label{th:correlation-cut-off}
Let $k_1+k_2=2d-1$, let $\beta_0\in H^1(\IT^d,\IR)$ and let $M,N$ be elements in $\IZ_+$. There exists a constant $C_{M,N}>0$ such that for all $(\varphi,\psi)\in\Omega^{k_1}(S\IT^d)\times\Omega^{k_2}(S\IT^d)$ and for all $\chi$ which is $(N,\min\{k_1,k_2\})$-admissible and which satisfies\footnote{The constant is the one from Lemma~\ref{c:non-stat-points}.} $\supp \chi \subset \left(  \mathsf{t}_0 , \infty \right)$, one has 
 \begin{eqnarray}
 \label{e:decomp-again}
 \int_{\IR}\chi(t) \Cor_{\varphi,\mathbf{T}_{-\tilde{x}}^*(\psi)}(t,\beta_0)|dt|  = \sum_{l=0}^{\min\{k_1,k_2\}} \mathscr{J}_\chi^{(l)}(\varphi,\psi) ,
  \end{eqnarray}
  where, for all $0\leq l \leq \min\{k_1,k_2\}$,
\begin{multline*}
\left|   \mathscr{J}_\chi^{(l)}(\varphi,\psi)  -   \frac{E^{(l)}_{\beta_0} }{l!} \int_{\IR}\chi(t) t^l  |dt| \right|\\ 
\leq  
C_{M,N} \max \left\{ \| \chi(t) t^{l-M} \|_{L^1(\R_+)} ,  \left\|\frac{d^N}{dt^N} \left(t^l\chi\right)\right\|_{L^1(\IR_+)}  \right\}\|\varphi\|_{\mathcal{H}_1}\|\psi\|_{\mathcal{H}_2}
\end{multline*}
with $\mathcal{H}_1 := \mathcal{H}^{M,-M/2,0,-N/2}_{k_1,\beta_0}$ and $\mathcal{H}_2 := \mathcal{H}^{M,-M/2,0,-N/2}_{k_2,-\beta_0}$ defined by Definition~\ref{def-ani-spaces}, and with
\begin{align}
\label{def-El}
 E^{(l)}_{\beta_0}  =   \int_{\IS^{d-1}} e^{i\beta_0\cdot\tilde {x}(\theta)}B_{\tilde{x},\beta_0}^{(k_2,l)}(\varphi,\psi)(\theta)d\Vol(\theta)  \ \text{if }\beta_0 \in  H^1(\T^d,\Z),  \quad  E^{(l)}_{\beta_0}  = 0 \  \text{otherwise},
\end{align}
where the explicit expression for $B_{\tilde{x},\beta_0}^{(k_2,l)}$ is given by~\eqref{e:coeff-horrible}.
\end{theo}
The function $\chi_0$ appearing in this Theorem is the one from Definition~\ref{def-chi-j} and we recall that each function $\chi_j$ implicitely appears in the definition of the anisotropic spaces.
The term $E_{\beta_0}^{(l)}$ is a generalization of the term $E_{\beta_0}$, defined in~\eqref{e:leading-term}, to the case of differential forms. 
Before entering the details of the proof, we start with the following observation which follows from a direct integration by parts argument:
\begin{lemm}\label{l:IPP-time} Suppose that $\chi$ is $(N,p)$-admissible. 
Then, for every $\lambda\neq 0$, one has
$$\left|\int_{\IR_+}t^p\chi(t)e^{-it\lambda }|dt|\right|\leq |\lambda|^{-N}\left\|\frac{d^N}{dt^N}\left(t^p\chi(t)\right)\right\|_{L^1(\IR_+)}.$$ 
\end{lemm}
This lemma will allow us to gain a decay in $|\xi|$ that is lacking in the region where the phase is stationary. In other words, it will allow us to take observables that may be singular along the direction of $V$ while for the correlation function we required to have some regularity along $V$; that is to say, we can now choose $N_1\ll -1$ in Figure~\ref{f:anisotropic}. Henceforth, in the proof of Theorem~\ref{th:correlation-cut-off}, we only make use of the non-stationary phase estimate of Lemma~\ref{c:non-stat-points} and do not rely on stationary phase estimates of Lemma~\ref{c:stat-points}.

\begin{proof}[Proof of Theorem~\ref{th:correlation-cut-off}]
According to Lemma~\ref{l:reduction}, we start from the decomposition~\eqref{e:decomposition1}--\eqref{e:decomposition2} of the dynamical correlator $\Cor$ according to the polynomial degree in the variable $t$.
Integrating the expression of $\Cor_{\varphi,\mathbf{T}_{-\tilde{x}}^*(\psi)}^{l}(t,\beta_0)$ in~\eqref{e:decomposition1} against $\chi(t)$ will then yield~\eqref{e:decomp-again} with 
\begin{align}
\label{e:def-J-l}
\mathscr{J}_\chi^{(l)}(\varphi,\psi) = \int_{\R}\chi(t)\Cor_{\varphi,\mathbf{T}_{-\tilde{x}}^*(\psi)}^{l}(t,\beta_0) |dt| .
\end{align}
We decompose $\Cor_{\varphi,\mathbf{T}_{-\tilde{x}}^*(\psi)}^{l}(t,\beta_0)$ further by writing
\begin{align}
\label{e:decomposition-de-ouf}
 \Cor_{\varphi,\mathbf{T}_{-\tilde{x}}^*(\psi)}^{l}(t,\beta_0) &  = \frac{t^l}{l!} E^{(l)}_{\beta_0}  +  \frac{t^l}{l!} \Cor^{l}_{-1}(t) +  \frac{t^l}{l!} \Cor^{l}_{0}(t) + \frac{t^l}{l!} \Cor^{l}_{1}(t) , \quad \text{with, for }j \in \{-1,0,1\}  
% ,\\
% E^{(l)}_{\beta_0} & =   \int_{\IS^{d-1}} e^{i\beta_0\cdot\tilde {x}(\theta)}B_{\tilde{x},\beta_0}^{(k_2,l)}(\varphi,\psi)(\theta){\red d\Vol(\theta)}  \quad \text{if }\beta_0 \in H^1(\T^d,\Z)\simeq\IZ^d,\nonumber\\  
% E^{(l)}_{\beta_0}  &= 0 \quad  \text{otherwise}, \nonumber
 \end{align}
where $E_{\beta_0}^{(l)}$ is defined in~\eqref{def-El} 
 and
\begin{align*}
 \Cor^{l}_j(t)   =   \sum_{\xi\in \IZ^d \setminus\{\beta_0\}}\int_{\IS^{d-1}}  \chi_j \left(\theta \cdot  \frac{\xi-\beta_0}{|\xi-\beta_0|} \right)e^{it(\xi-\beta_0)\cdot\mathbf{v}(\theta)}e^{i\xi\cdot\tilde{x}(\theta)}B_{\tilde{x},\xi}^{(k_2,l)}(\varphi,\psi)(\theta) d\Vol(\theta),
\end{align*} 
where the functions $\chi_j$ were introduced in Definition~\ref{def-chi-j}. Note that $E^{(l)}_{\beta_0}$ concerns the Fourier coefficient $\xi=\beta_0$ (in the case $\beta_0 \in \Z^d$ and it vanishes otherwise).
Moreover, $E^{(l)}_{\beta_0}$ is a time invariant quantity.  
The above decomposition indexed by $j=-1,0,1$ corresponds to the different integration regions of $\mathbb{S}^{d-1}$ on which we study the oscillatory integral.

\medskip
We now compute each term in~\eqref{e:decomposition-de-ouf}. The first term is nothing but 
\begin{align}
\label{e:terme-nul}
\int_{\IR}\chi(t) \frac{t^l}{l!} E^{(l)}_{\beta_0}  |dt| =  \frac{E^{(l)}_{\beta_0} }{l!} \int_{\IR}\chi(t) t^l  |dt| .
\end{align}

We next consider the term involving $\Cor^{l}_0(t)$. Applying Lemma~\ref{c:non-stat-points} to the function $F(\theta) = B_{\tilde{x} ,\xi}^{(k_2,l)}(\varphi,\psi)(\theta)$, we have the following statement: there exists $\mathsf{t}_0>0$ such that for all $M\in\IZ_+$, there is $C_M >0$ such that for all $t> \mathsf{t}_0 ,|\xi-\beta_0| \geq 1$, $(\varphi,\psi)\in\Omega^{k_1}(S\IT^d)\times \Omega^{k_2}(S\IT^d)$, we have 
\begin{align*}
\left| \Cor^{l}_0(t) \right| \leq  \sum_{\xi\in\IZ^d \setminus\{\beta_0\}}C_M \frac{1}{(|\xi-\beta_0| t)^{M}}  \left\|B_{\tilde{x} ,\xi}^{(k_2,l)}(\varphi,\psi)(\cdot)\right\|_{W^{M,1}(\mathbf{C}_0(\xi-\beta_0))} .
\end{align*}
According to~\eqref{e:control-Cm-norm}, this implies for every $k_1+k_2=2d-1$, for every $0\leq l\leq \min\{k_1,k_2\}$ and for every $(\varphi,\psi)\in\Omega^{k_1}(S\IT^d)\times \Omega^{k_2}(S\IT^d)$,
\begin{align*}
\left| \Cor^{l}_0(t) \right| \leq \frac{C_M}{t^{M}} \sum_{\xi\in\IZ^d \setminus\{\beta_0\}}\frac{1}{|\xi-\beta_0|^{M}} 
\left\|\pi_{\xi}^{(k_1)}(\varphi)\right\|_{H^M\left(\mathbf{C}_0(\xi-\beta_0) \right)} \left\|\pi_{-\xi}^{(k_2)}(\psi)\right\|_{H^M\left(\mathbf{C}_0(\xi-\beta_0) \right)}.
\end{align*}
We thus obtain, if $\supp \chi \subset \left( \mathsf{t}_0 , \infty \right)$, that 
\begin{multline}
\label{e:C0estimate-1}
\left| \int_{\IR}\chi(t) \frac{t^l}{l!} \Cor^{l}_0(t)   |dt| \right|  \leq C_M  \| \chi(t) t^{l-M} \|_{L^1(\R_+)}\\
 \times\sum_{\xi\in\IZ^d \setminus\{\beta_0\}}\frac{1}{|\xi-\beta_0|^{M}} 
\left\|\pi_{\xi}^{(k_1)}(\varphi)\right\|_{H^M\left(\mathbf{C}_0(\xi-\beta_0) \right)} \left\|\pi_{-\xi}^{(k_2)}(\psi)\right\|_{H^M\left(\mathbf{C}_0(\xi-\beta_0) \right)} .
\end{multline}
For the remaining two terms, we write
\begin{multline*}
 \int_{\IR}\chi(t) \frac{t^l}{l!} \Cor^{l}_{\pm 1}(t)   |dt| 
 =  \sum_{\xi\in\IZ^d \setminus\{\beta_0\}}\int_{\IS^{d-1}}\left(  \int_{\IR}\chi(t) \frac{t^l}{l!} e^{it(\xi-\beta_0) \cdot \mathbf{v}(\theta)}   |dt| \right)  \\ 
\times \chi_{\pm1} \left(\theta \cdot  \frac{\xi-\beta_0}{|\xi-\beta_0|} \right)e^{i\xi\cdot\tilde{x}(\theta)}B_{\tilde{x},\xi}^{(k_2,l)}(\varphi,\psi)(\theta)d\Vol(\theta) .
\end{multline*}
%We then remark from the properties of $\chi_{\pm 1}$ in Definition~\ref{def-chi-j} that $s \in \supp(\chi_{\pm 1}) \implies |s|\geq {\red 1-\eps_0}>0$. Lemma~\ref{l:IPP-time} then implies that, {\red there exists $c_0>0$ (depending only on $\mathbf{v}$ and $\beta_0$) such that,}
%for $(\theta,\xi)$ in $\supp \chi_{\pm1} \left(\theta \cdot  \frac{\xi-\beta_0}{|\xi-\beta_0|} \right)$:
 According to Lemma~\ref{l:support-chi1}, one has by integration by parts in time,  for all $(\theta,\xi)$ in the support of $\chi_{\pm 1}\left(\theta\cdot\frac{\xi-\beta_0}{|\xi-\beta_0|}\right)$,
\begin{align}\label{e:IPP-time}
\left| \int_{\IR}\chi(t) t^l e^{it(\xi-\beta_0) \cdot \mathbf{v}(\theta)}   |dt| \right| & \leq \frac{1}{|(\xi-\beta_0) \cdot \mathbf{v}(\theta)|^N}\left\|\frac{d^N}{dt^N}\left(t^l\chi\right)\right\|_{L^1(\IR_+)} \nonumber\\
& \leq  \frac{1}{|c_0 (\xi-\beta_0)|^N}\left\|\frac{d^N}{dt^N}\left(t^l\chi\right)\right\|_{L^1(\IR_+)} ,
\end{align}
%{\red where the last equality comes from the fact that $\mathbf{v}(\theta)\cdot\frac{\xi-\beta_0}{|\xi-\beta_0|}=\lambda(\xi)+\ml{O}(\epsilon_0)$ with our support properties. Recalling that $\lambda(\xi)$ is uniformly bounded from below, we get the expected bound (up to taking the value of $\epsilon_0$ small enough).}

Coming back to our problem, we can derive the estimate
\begin{align*}
&  \left| \int_{\IR}\chi(t) \frac{t^l}{l!} \Cor^{l}_{\pm 1}(t)   |dt|  \right| \\
&\leq   
  C_N \left\|\frac{d^N}{dt^N} \left(t^l\chi\right)\right\|_{L^1(\IR_+)} 
 \sum_{\xi\in\IZ^d \setminus\{\beta_0\}}  \frac{1}{|(\xi-\beta_0)|^N} \| B_{\tilde{x},\xi}^{(k_2,l)}(\varphi,\psi)(\theta)\|_{L^1\left(\mathbf{C}_{\pm1}(\xi-\beta_0) \right)} \\
&\leq   
  C_N \left\|\frac{d^N}{dt^N} \left(t^l\chi\right)\right\|_{L^1(\IR_+)} 
  \sum_{\xi\in\IZ^d \setminus\{\beta_0\}}  \frac{\left\|\pi_{\xi}^{(k_1)}(\varphi)\right\|_{L^2\left(\mathbf{C}_{\pm1}(\xi-\beta_0) \right)} \left\|\pi_{-\xi}^{(k_2)}(\psi)\right\|_{L^2\left(\mathbf{C}_{\pm1}(\xi-\beta_0) \right)}}{|(\xi-\beta_0)|^N} 
\end{align*}
Combining this together with~\eqref{e:terme-nul} and~\eqref{e:C0estimate-1} in~\eqref{e:def-J-l}--\eqref{e:decomposition-de-ouf}, and recalling the definition of the norms
$\mathcal{H}^{M,-M/2,0,-N/2}_{k,\beta_0}$ in Definition~\ref{def-ani-spaces}, we have obtained the expected bound.
\end{proof}

%%%%%%%%%%%%%%%%%%%%%%%%%%%%%%%
%%%%%%%%%%%%%%%%%%%%%%%%%%%%%%%
%%%%%%%%%%%%%%%%%%%%%%%%%%%%%%%
%%%%%%%%%%%%%%%%%%%%%%%%%%%%%%%
%%%%%%%%%%%%%%%%%%%%%%%%%%%%%%%
\section{Mellin and Laplace transforms}\label{s:laplace-mellin}

We will now apply the results of Section~\ref{s:analysis} to two fundamental cases which, besides their own interest, will be instrumental in our description of zeta functions associated with the length orthospectrum, defined in Section~\ref{general-ep-series-poin}. These two cases are the two main analytical statements of the article.

All along this section, we will take $\chi_{\infty}$ to be a smooth function on $\IR$ satisfying the following properties
\begin{equation}\label{e:cutoff-infini}
 \exists T_0\geq1,\ \exists t_0>0,\quad\text{such that}\quad\text{supp}(\chi_{\infty})\subset[T_0,\infty)\quad\text{and}\quad\chi_{\infty}(t)=1\quad\text{for}\quad t\geq T_0+t_0.
\end{equation}
Typically, for our applications, we will in fact work with \emph{nondecreasing} functions of this type. We now aim at refining the results of Section~\ref{s:cutoff} when the function $\chi$ depends on some extra complex parameter, e.g.
$$\chi_s^L(t):=\chi_{\infty}(t)e^{-st}\quad\text{and}\quad\chi_{s}^M(t):=\chi_\infty(t)t^{-s},$$
where $s\in\IC$ has large enough real part. Equivalently, this amounts to study the \emph{Laplace and the Mellin transforms} of $\chi_\infty(t)e^{-t\mathbf{V}_{\beta_0}}$:
\begin{equation}\label{e:cut-mellin-laplace}
\widehat{\chi_s^L}(-i\mathbf{V}_{\beta_0}):=\int_{0}^\infty e^{-st}\chi_\infty(t)e^{-t\mathbf{V}_{\beta_0}}|dt|\quad\text{and}\quad\widehat{\chi_s^M}(-i\mathbf{V}_{\beta_0}):=\int_{1}^\infty t^{-s}\chi_\infty(t)e^{-t\mathbf{V}_{\beta_0}}|dt|.
\end{equation}
Note that, for $\text{Re}(s)$ large enough, we are in the setting of application of Theorem~\ref{th:correlation-cut-off}. Hence, for such $s$, these operators are well defined on the anisotropic Sobolev spaces we have introduced in Section~\ref{s:cutoff}. Our goal is to show that these operators in fact extend to appropriate subsets of the complex plane when considered on these spaces. See Theorems~\ref{t:general-mellin} and~\ref{t:general-laplace} for precise statements. This section is divided in two main parts corresponding respectively to the analysis of $\widehat{\chi_s^M}(-i\mathbf{V}_{\beta_0})$ (\S\ref{ss:mellin}) and to the one of $\widehat{\chi_s^L}(-i\mathbf{V}_{\beta_0})$ (\S\ref{ss:laplace}).

\begin{rema}Besides applications to Poincar\'e series, note that the Laplace transform appears naturally when studying the resolvent of $\mathbf{V}_{\beta_0}$. In fact, one has
$$(s+\mathbf{V}_{\beta_0})^{-1}:=\int_0^{+\infty}e^{-st}e^{-t\mathbf{V}_{\beta_0}}|dt|=\int_0^{+\infty}(1-\chi_{\infty}(t))e^{-st}e^{-t\mathbf{V}_{\beta_0}}|dt|+\widehat{\chi_s^L}(-i\mathbf{V}_{\beta_0}),$$
 which defines a bounded operator from $\Omega^{k}(S\IT^d)$ to $\ml{D}^{\prime k}(S\IT^d)$ for $\text{Re}(s)$ large enough. Note that the first integral on the right hand side is over a compact interval. Hence, this part extends holomorphically to the whole complex plane as an operator from $\Omega^{k}(S\IT^d)$ to $\Omega^{k}(S\IT^d)$. Equivalently, understanding the extension of the resolvent amounts to understand the continuation of $\widehat{\chi}_s^L(-iV)$. The same remarks hold for the integral
 $$\int_1^{+\infty}t^{-s}e^{-t\mathbf{V}_{\beta_0}}|dt|.$$
 We refer to Section~\ref{s:terms-near-zero} below for precise statements. A refined analysis of the resolvent when $k=0$ and in the case of analytic regularity will be discussed in~\cite{BonthonneauDangLeautaudRiviere22}.
\end{rema}

\subsection{Mellin transform}\label{ss:mellin}
We begin with the case of the Mellin transform which is slightly easier to handle as it only requires nonstationary phase estimates.
\subsubsection{A preliminary lemma}
In the case of the Mellin transform, the analysis relies on the following elementary lemma:
\begin{lemm}
\label{l:f-phi-lam}
Let $T_0 \geq 1$ and $\phi \in \mathcal{C}^\infty(\R)$ be such that $\supp(\phi) \subset [T_0,+\infty)$ and $\phi$ is constant near infinity. Then, the following hold:
\begin{enumerate}
\item \label{i:debile} For any $\lambda \in \R$, the function 
$$
f_{\phi, \lambda} (s) := \int_\R \phi(t) t^{-s} e^{i\lambda t} |dt|  
$$
is a well-defined holomorphic function for $\Re(s) >1$ satisfying 
$$
\left| f_{\phi, \lambda} (s) \right| \leq \|\phi\|_{L^\infty(\R)}\frac{T_0^{-(\Re(s)-1)}}{\Re(s)-1} , \quad \text{ for }s \in \C ,  \Re(s) >1.
$$
\item \label{i:debile-cpct} If $\phi$ is compactly supported, $f_{\phi, \lambda}$ is actually defined on the whole complex plane and defines an entire function on $\C$ such that 
$$
\left| f_{\phi, \lambda} (s) \right| \leq C_\phi \frac{T_0^{-\Re(s)}}{\langle \Re(s) \rangle}, \quad \text{ for }s \in \C .
$$
\item \label{i:lambda=0} If $\lambda =0$ and $\phi = 1$ in a neighborhood of $+ \infty$, then $f_{\phi, 0}$ extends to $\C$ as a meromorphic function with a single simple pole at $s=1$ whose residue is equal to $1$. 
Moreover,
\begin{align}
\label{f-phi-0}
f_{\phi, 0}(s) = \frac{f_{\phi', 0}(s-1)}{s-1} , \quad \text{ for all } s \in \C\setminus\{1\} .
\end{align}

\item \label{i:lambda-neq0} If $\lambda \neq 0$ and $\phi = 1$ in a neighborhood of $+ \infty$, then $f_{\phi, \lambda}$ extends to $\C$ as an entire function. Moreover, this extended function satisfies, for all $m \in \N^*$ and all $s \in \C$,
\begin{align}
\label{f-phi-lambda-IPP}
f_{\phi, \lambda}(s)& =\frac{1}{(i\lambda)^m} \sum_{j=0}^{m-1} \begin{pmatrix} m\\ j\end{pmatrix} (-1)^{m-j}  P_j(s) f_{\phi^{(m-j)}, \lambda}(s+j)  + \frac{P_m(s)}{(i\lambda)^m}  f_{\phi, \lambda}(s+m) , 
\end{align}
where
\begin{align}
\label{e:def-Pj}
P_j(s)=  \prod_{k=0}^{j-1}(s+k) , \quad \text{for}\quad  j \in \IZ_+^* , \quad \text{and} \quad P_0(s)=1 ,
\end{align}
and, for all $m \in \IZ_+$, there is a constant $C_{\phi,m}>0$ such that for all $\lambda \neq 0$, 
\begin{align}
\label{e:estim-f-phi}
|f_{\phi, \lambda}(s)| \leq C_{\phi,m} \frac{\langle |s|\rangle^m}{|\lambda|^m} \frac{T_0^{-\Re(s)+1}}{\Re(s)+m-1}  , \quad \text{ for }   \Re(s) > -(m-1) . 
\end{align}
\end{enumerate}
\end{lemm}

\begin{proof}
Item~\ref{i:debile} follows from the rough estimate $|\phi(t) t^{-s} e^{i\lambda t} | \leq  \|\phi\|_{L^\infty} t^{-\Re(s)}$. In case $\supp \phi \subset [T_0,T_1]$, this yields in particular the estimate
\begin{equation} \label{e:debile-cpct}
 \left| f_{\phi, \lambda} (s) \right| \leq   \|\phi\|_{L^\infty} \int_{T_0}^{T_1}t^{-\Re(s)}|dt| = \frac{T_0^{-\Re(s)+1}- T_1^{-\Re(s)+1}}{\Re(s)-1} ,
\end{equation}
which, combined with holomorphy under the integral, provides a proof of Item~\ref{i:debile-cpct}. Item~\ref{i:lambda=0} consists in proving~\eqref{f-phi-0} for $\Re(s)>1$ by an integration by parts, and then observing that $f_{\phi', 0}$ is an entire function, whence the right hand-side of~\eqref{f-phi-0} has the sought properties. The result for all $s \in \C$ follows from analytic continuation and the residue is $f_{\phi', 0}(0) =  \int_\R \phi'(t)dt =\phi(+\infty) - \phi(0)  =1$.
 
The proof of Item~\ref{i:lambda-neq0} (in case $\lambda \neq 0$) also consists in proving first~\eqref{f-phi-lambda-IPP} for $\Re(s)>1$ by integration by parts. 
After $m$ integrations by parts, one finds for $\Re(s)>1$
 $$
f_{\phi, \lambda} (s) = \left(\frac{-1}{i\lambda} \right)^m \int_\R \d_t^m(\phi(t) t^{-s}) e^{i\lambda t} |dt| .
$$
The Leibniz formula together with the fact that $(t^{-s})^{(j)} = (-1)^j P_j(s)t^{-s-j}$ then implies~\eqref{f-phi-lambda-IPP} for $\Re(s)>1$.
 
 Next, we observe that the first term on the right hand-side of~\eqref{f-phi-lambda-IPP} is an entire function (as $\phi^{(m-j)}$ is compactly supported for $j\leq m-1$) and the second term is holomorphic on the half space $\Re(s)>-m+1$.
 Hence, for all $m \in \N^*$, the right hand-side of~\eqref{f-phi-lambda-IPP} is a holomorphic function on $\Re(s)>-m+1$, and all these functions coincide with $f_{\phi, \lambda}(s)$ on $\Re(s)>-m+1$.  
 As a consequence of analytic continuation, for any $m\in \N$, $f_{\phi, \lambda}$ can be extended uniquely to a holomorphic function on $\Re(s)>-m+1$ (still denoted $f_{\phi, \lambda}$), which satisfies~\eqref{f-phi-lambda-IPP} on $\Re(s)>-m+1$. 
 
 To prove the estimate, we use~\eqref{f-phi-lambda-IPP} and write 
 \begin{align*}
|\lambda|^m | f_{\phi, \lambda}(s)| & \leq   \sum_{j=0}^{m-1}  \begin{pmatrix} m\\ j\end{pmatrix}  | P_j(s) | | f_{\phi^{(m-j)}, \lambda}(s+j)|  + |P_m(s)| |f_{\phi, \lambda}(s+m)| .
 \end{align*}
 Taking $1\leq T_0 <T_1$ such that $\supp(\phi')\subset [T_0,T_1]$ and using item~\ref{i:debile} together with~\eqref{e:debile-cpct}, we deduce
 \begin{align*}
|\lambda|^m | f_{\phi, \lambda}(s)|  & \leq C_m    \sum_{j=0}^{m-1} \langle |s|\rangle^j \|\phi^{(m-j)}\|_{L^\infty}  \int_{T_0}^{T_1}t^{-\Re(s)-j}|dt| \\ 
& \quad +C_m \langle |s|\rangle^m  \|\phi\|_{L^\infty(\R)}\frac{T_0^{-(\Re(s)+m-1)}}{\Re(s)+m-1} ,
\end{align*}
 from which the statement follows.
\end{proof}

\subsubsection{Meromorphic continuation of $\widehat{\chi_s^M}(-i\mathbf{V}_{\beta_0})$} Before discussing the meromorphic continuation, let us first clarify its holomorphic properties on $\text{Re}(s)>d$:
\begin{prop}\label{p:mellin-holomorphe} Let $\chi_{\infty}$ be a function verifying assumption~\eqref{e:cutoff-infini}, let $\beta_0\in H^1(\IT^d,\IR)$ and let $\tilde{x}:\IS^{d-1}\rightarrow \IR^d$ be a smooth function. 

Then, for all $(\varphi,\psi)\in\Omega^{k_1}(S\IT^d)\times\Omega^{k_2}(S\IT^d)$ with $k_1+k_2=2d-1$, the function
  \begin{equation}
\label{e:Mellin-t}
s \mapsto \mathscr{M}_{(\varphi,\psi)}(s)  := \int_{S\IT^d}\varphi\wedge \widehat{\chi_{s}^M}(-i\mathbf{V}_{\beta_0})\mathbf{T}_{-\tilde{x}}^*(\psi) 
\end{equation}
is holomorphic on $\Re(s)>\min\{k_1,k_2\}+1$ and it satisfies
\begin{align}
\label{e:Mellin-rough-estim}
\left|  \mathscr{M}_{(\varphi,\psi)}(s)   \right|  
 \leq C \frac{T_0^{-(\Re(s)-\min\{k_1,k_2\}-1)}}{\Re(s) -\min\{k_1,k_2\}-1} \| \varphi \|_{L^2(S\T^d)} \| \psi \|_{L^2(S\T^d)} . 
 \end{align}
\end{prop}
Note that in this expression, $\min\{k_1,k_2\}+1$ can be always be replaced by $d$ (but it downgrades the statement).
\begin{proof}
 Recalling~\eqref{e:integrated-correlation}, we use again Lemma~\ref{l:reduction} and the decomposition~\eqref{e:decomposition1}--\eqref{e:decomposition2}.
Integrating~\eqref{e:decomposition1}--\eqref{e:decomposition2} against $\chi_\infty(t) t^{-s}$ then yields
 \begin{eqnarray}
 \label{e:decomp-again-again}
\mathscr{M}_{(\varphi,\psi)}(s)  =   \int_{S\IT^d}\varphi\wedge \widehat{\chi_{s}^M}(-i\mathbf{V}_{\beta_0})\mathbf{T}_{-\tilde{x}}^*(\psi) = \sum_{l=0}^{\min\{k_1,k_2\}} \mathscr{M}^{(l)}_{(\varphi,\psi)}(s) ,
  \end{eqnarray}
 with 
\begin{align}
\label{e:def-J-l-again}
 \mathscr{M}^{(l)}_{(\varphi,\psi)}(s) = \int_1^\infty\chi_\infty (t) t^{-s} \Cor_{\varphi,\mathbf{T}_{-\tilde{x}}^*(\psi)}^{l}(t,\beta_0) |dt| .
\end{align}
We then notice that the index $l$ is bounded by $l \leq \min \{k_1,k_2\}  \leq d-1$, and that 
\begin{align}
\label{e:rough-estim-corr}
\left|  \Cor_{\varphi,\mathbf{T}_{-\tilde{x}}^*(\psi)}^{l}(t,\beta_0) \right|&  = \left|
 \frac{t^l}{l!}\sum_{\xi\in\IZ^d}\int_{\IS^{d-1}}e^{it(\xi-\beta_0)\cdot\mathbf{v}(\theta)}e^{i\xi\cdot\tilde{x}(\theta)}B_{\tilde{x},\xi}^{(k_2,l)}(\varphi,\psi)(\theta) d\Vol(\theta) \right| \nonumber  \\
 &  \leq   \frac{t^l}{l!} 
\sum_{\xi\in\IZ^d}\int_{\IS^{d-1}} \left| B_{\tilde{x},\xi}^{(k_2,l)}(\varphi,\psi)(\theta) \right|  d\Vol(\theta)\nonumber \\
 &  \leq C t^l\| \varphi \|_{L^2(S\T^d)} \| \psi \|_{L^2(S\T^d)}
\end{align}
according to~\eqref{e:control-Cm-norm}. 
We deduce that 
$$
\left| \chi_\infty (t) t^{-s} \Cor_{\varphi,\mathbf{T}_{-\tilde{x}}^*(\psi)}^{l}(t,\beta_0) \right| \leq C\chi_\infty (t)  t^{l-s} \| \varphi \|_{L^2(S\T^d)} \| \psi \|_{L^2(S\T^d)} .
$$
Recalling that $l \leq \min\{k_1,k_2\}$, 
\eqref{e:def-J-l-again} then implies holomorphy of $\mathscr{M}^{(l)}_{(\varphi,\psi)}$ in $\Re(s)>l+1$ (and in particular in $\Re(s)>\min\{k_1,k_2\}+1$). Item~\ref{i:debile} in Lemma~\ref{l:f-phi-lam} finally yields
$$
 \left| \mathscr{M}^{(l)}_{(\varphi,\psi)}(s) \right| \leq \frac{T_0^{-(\Re(s)- \min\{k_1,k_2\}-1)}}{\Re(s) -\min\{k_1,k_2\}-1} 
 C \| \varphi \|_{L^2(S\T^d)} \| \psi \|_{L^2(S\T^d)} ,$$
from which we infer~\eqref{e:Mellin-rough-estim} thanks to~\eqref{e:decomp-again-again}.
\end{proof}

We now turn to our main statement on these regularized Mellin transforms.
\begin{theo}\label{t:general-mellin} Let $\chi_{\infty}$ be a function verifying assumption~\eqref{e:cutoff-infini}, let $\beta_0\in H^1(\IT^d,\IR)$ and let $\tilde{x}:\IS^{d-1}\rightarrow \IR^d$ be a smooth function. Suppose in addition that $T_0 \geq \mathsf{t}_0$ where $T_0\geq 1$ is the constant appearing in~\eqref{e:cutoff-infini} and $\mathsf{t}_0>0$ the one from Lemma~\ref{c:non-stat-points}. 
Recall that $E^{(l)}_{\beta_0} $ is defined in~\eqref{def-El}. 
%Set
%$$
% E^{(l)}_{\beta_0} :=   \int_{\IS^{d-1}} e^{i\beta_0\cdot\tilde {x}(\theta)} B_{\tilde{x},\beta_0}^{(k_2,l)}(\varphi,\psi)(\theta){\red d\Vol(\theta)},\quad  \text{if}\ \beta_0\in H^1(\IT^d,\IZ),
% $$
% and $E^{(l)}_{\beta_0}=0$ otherwise.

 Then, using the conventions of Proposition~\ref{p:mellin-holomorphe}, for any $N\in\IZ_+^*$, there exists $C_N>0$ such that, for every couple $(\varphi, \psi)$ in $\mathcal{H}^{N,-N/2,0,-N/2}_{k_1,\beta_0}\times \mathcal{H}^{N,-N/2,0,-N/2}_{k_2,-\beta_0}$ with $k_1+k_2=2d-1$, the function
 $$
  \mathscr{M}_{(\varphi,\psi)}(s)  -  \sum_{l=0}^{\min\{k_1,k_2\}}    \frac{1}{l!} \frac{E^{(l)}_{\beta_0}}{s-l-1} ,
 $$
originally defined on  $\Re(s)>\min\{k_1,k_2\}+1$ extends holomorphically to the half-plane $\Re(s) > -N +\min\{k_1,k_2\}+1$ with 
\begin{multline*}
\left|    \mathscr{M}_{(\varphi,\psi)}(s)  -  \sum_{l=0}^{\min\{k_1,k_2\}}   \frac{1}{l!} \frac{E^{(l)}_{\beta_0}}{s-l-1} \right|\\
\leq 
 \frac{ C_{N}\langle |s|\rangle^N }{\Re(s)-\min\{k_1,k_2\}-1+N} 
 \|\varphi\|_{\mathcal{H}^{N,-N/2,0,-N/2}_{k_1,\beta_0}} \|\psi\|_{\mathcal{H}^{N,-N/2,0,-N/2}_{k_2,-\beta_0}}.
\end{multline*}
  \end{theo}

The proof is very close to that of Theorem~\ref{th:correlation-cut-off} and we just need to pay attention to the dependence on the parameter $s\in\mathbb{C}$. Combined with Proposition~\ref{p:terms-near-zero}, this proves Theorem~\ref{t:maintheo-mellin-function} from the introduction by picking $\beta_0=0$, $k_1=2d-1$ and $k_2=0$. Note that the term $E^{(l)}_{\beta_0}$ is only bounded by the norm $\|\varphi\|_{L^2} \|\psi\|_{L^2}$.
\begin{proof} By bilinearity of the considered mappings with respect to $(\varphi,\psi)$ and by density, it is sufficient to prove these analytical estimates when $(\varphi,\psi)\in\Omega^{k_1}(S\IT^d)\times\Omega^{k_2}(S\IT^d)$. As in the proof of Proposition~\ref{p:mellin-holomorphe}, we can decompose $\Cor_{\varphi,\mathbf{T}_{-\tilde{x}}^*(\psi)}^{l}(t,\beta_0)$ using~\eqref{e:decomposition-de-ouf}. 
Then, we are left with describing the terms $\mathscr{M}^{(l)}_{(\varphi,\psi)}(s)$ in~\eqref{e:def-J-l-again} that can be decomposed accordingly as 
 \begin{align}
\label{e:decmposition-J}
 \mathscr{M}^{(l)}_{(\varphi,\psi)}(s) = \mathscr{M}^{(l,E)}_{(\varphi,\psi)}(s) +\mathscr{M}^{(l,-1)}_{(\varphi,\psi)}(s)+\mathscr{M}^{(l,0)}_{(\varphi,\psi)}(s)+\mathscr{M}^{(l,1)}_{(\varphi,\psi)}(s) , 
 \end{align}
with 
\begin{align*}
\mathscr{M}^{(l,E)}_{(\varphi,\psi)}(s)&=\int_{\R}\chi_\infty (t)t^{-s}\frac{t^l}{l!}E^{(l)}_{\beta_0}|dt|,\quad\text{ and } \\ 
 \mathscr{M}^{(l,j)}_{(\varphi,\psi)}(s) &= \int_{\R}\chi_\infty (t) t^{-s}  \frac{t^l}{l!} \Cor^{l}_{j}(t) |dt| , 
  \quad \text{ for }j \in \{-1,0,1\}  .
  \end{align*}
We now study each of these terms separately.

\medskip
Firstly, as $\supp(\chi_{\infty})\subset[1,\infty)$, we have 
\begin{align}
\label{e:value-JlE}
 \mathscr{M}^{(l,E)}_{(\varphi,\psi)}(s) =E^{(l)}_{\beta_0} \frac{1}{l!} \int_{1}^\infty\chi_\infty (t)t^{-s+l}|dt|=  \frac{1}{l!}\frac{E^{(l)}_{\beta_0}}{s-l-1} +\frac{E^{(l)}_{\beta_0} }{l!} \int_{1}^{\infty}(1-\chi_\infty (t))t^{-s+l}|dt|, 
\end{align}
where the second term on the right-hand side of the equation is an entire function. Hence, this term has the claimed properties.

\medskip
Secondly, we consider the term with $\mathscr{M}^{(l,0)}_{(\varphi,\psi)}(s)$ and proceed as in the proof of Theorem~\ref{th:correlation-cut-off}. According to  Lemma~\ref{c:non-stat-points}, we have, for $ t > \mathsf{t}_0 $
\begin{align*}
\Cor^{l}_{0}(t)  &= \sum_{\xi\in \IZ^d \setminus\{\beta_0\}}\int_{\IS^{d-1}}  \chi_0 \left(\theta \cdot  \frac{\xi-\beta_0}{|\xi-\beta_0|} \right)e^{it(\xi-\beta_0)\cdot\mathbf{v}(\theta)}e^{i\xi\cdot\tilde{x}(\theta)}B_{\tilde{x},\xi}^{(k_2,l)}(\varphi,\psi)(\theta) d\Vol(\theta)\\
& = \sum_{\xi\in \IZ^d \setminus\{\beta_0\}}  ( i|\xi-\beta_0| t )^{-N}\int_{\IS^{d-1}} F(t,\xi,\theta)d\Vol(\theta) ,
\end{align*}
with $$F(t,\xi,\theta) = e^{i(\xi-\beta_0)\cdot(t\mathbf{v}(\theta)+\tilde{x}(\theta))} 
 \mathcal{L}_{N,t}^{\frac{\xi-\beta_0}{|\xi-\beta_0|}} \left( \chi_0 \left(\theta \cdot  \frac{\xi-\beta_0}{|\xi-\beta_0|} \right)e^{i\beta_0\cdot\tilde{x}(\theta)}B_{\tilde{x},\xi}^{(k_2,l)}(\varphi,\psi)(\theta) \right) ,$$
and, for all $\omega \in \IS^{d-1}$, for all $\theta \in \mathbf{C}_0(\omega)$, and all $t > \mathsf{t}_0$,
\begin{align*}
\left|  (\mathcal{L}_{N,t}^\omega \psi ) (\theta) \right| \leq  C_N%\left(\frac{\| \tilde{x} \|_{W^{N+1,\infty}(\mathbf{C}_0(\omega))}}{t} \right) 
\left(\sum_{|\alpha| \leq N} \left|\nabla_\theta^\alpha \psi (\theta) \right| \right) .
\end{align*}
Coming back to $\mathscr{M}^{(l,0)}_{(\varphi,\psi)}(s)$, we have
\begin{align*}
\mathscr{M}^{(l,0)}_{(\varphi,\psi)}(s)& = \int_{\IR}\chi_\infty(t)t^{-s}\frac{t^l}{l!} \Cor^{l}_{ 0}(t)   |dt|  \\
& =  \int_\R \sum_{\xi\in\IZ^d \setminus\{\beta_0\}}\int_{\IS^{d-1}} \chi_\infty(t)t^{-s} \frac{t^l}{l!} ( i|\xi-\beta_0| t )^{-N} F(t,\xi,\theta)d\Vol(\theta) 
 |dt|  ,
\end{align*}
where, as $\supp(\chi_{\infty})\subset[T_0,+\infty)$ with $T_0\geq\max\{1, \mathsf{t}_0\}$, one has
\begin{align*}
& \left| \chi_\infty(t)t^{-s} \frac{t^l}{l!} ( i|\xi-\beta_0| t )^{-N}  F(t,\xi,\theta)\right| \\
& \leq  \chi_\infty(t)t^{-\Re(s)+l-N}  |\xi-\beta_0|^{-N}  C_N
%\left(\frac{\| \tilde{x} \|_{W^{N+1,\infty}(\mathbf{C}_0(\omega))}}{t} \right) 
\sum_{|\alpha| \leq N}\left|\nabla_\theta^\alpha B_{\tilde{x},\xi}^{(k_2,l)}
(\varphi,\psi)(\theta)\right|,
\end{align*}
uniformly for $t \geq \max\{1, \mathsf{t}_0\}$, $\xi \in \R^d \setminus\{\beta_0\}$ and $\theta \in \mathbf{C}_0(\xi-\beta_0)$.
We deduce from that bound and from~\eqref{e:control-Cm-norm} the holomorphy of the term 
$\mathscr{M}^{(l,0)}_{(\varphi,\psi)}(s)$ in $\Re(s) > -N +l + 1$ together with the estimate
\begin{align}
\label{e:C0estimate-1-bis}
& \left|  \mathscr{M}^{(l,0)}_{(\varphi,\psi)}(s) \right|  \nonumber\\
& \leq C_N  \| \chi_\infty(t) t^{l-\Re(s)-N} \|_{L^1(\R_+)} \sum_{\xi\in\IZ^d \setminus\{\beta_0\}}\frac{1}{|\xi-\beta_0|^{N}} 
\left\|B_{\tilde{x},\xi}^{(k_2,l)}(\varphi,\psi)\right\|_{W^{1,N}\left(\mathbf{C}_0(\xi-\beta_0) \right)} \nonumber\\
& \leq C_N  \| \chi_\infty(t) t^{l-\Re(s)-N} \|_{L^1(\R_+)} \sum_{\xi\in\IZ^d \setminus\{\beta_0\}}\frac{\left\|\pi_{\xi}^{(k_1)}(\varphi)\right\|_{H^N\left(\mathbf{C}_0(\xi-\beta_0) \right)} \left\|\pi_{-\xi}^{(k_2)}(\psi)\right\|_{H^N\left(\mathbf{C}_0(\xi-\beta_0) \right)}}{|\xi-\beta_0|^{N}} 
 \nonumber\\
& \leq C_N \frac{T_0^{-(\Re(s)+N-l-1)}}{\Re(s) +N-l-1} \sum_{\xi\in\IZ^d \setminus\{\beta_0\}}\frac{\left\|\pi_{\xi}^{(k_1)}(\varphi)\right\|_{H^N\left(\mathbf{C}_0(\xi-\beta_0) \right)} \left\|\pi_{-\xi}^{(k_2)}(\psi)\right\|_{H^N\left(\mathbf{C}_0(\xi-\beta_0) \right)}}{|\xi-\beta_0|^{N}} 
 .
\end{align}

\medskip
Thirdly, we consider the two terms $\mathscr{M}^{(l,\pm 1)}_{(\varphi,\psi)}(s)$ and proceed as in the proof of Theorem~\ref{th:correlation-cut-off} (i.e. take advantage of the integration over time). We write
\begin{multline*}
\mathscr{M}^{(l,\pm 1)}_{(\varphi,\psi)}(s) 
=  \sum_{\xi\in\IZ^d \setminus\{\beta_0\}}\int_{\IS^{d-1}}\left(  \int_{\IR}\chi_\infty(t)t^{-s} \frac{t^l}{l!} e^{it(\xi-\beta_0) \cdot \mathbf{v}(\theta)}   |dt| \right)\\
\times\chi_{\pm1} \left(\theta \cdot  \frac{\xi-\beta_0}{|\xi-\beta_0|} \right)e^{i\xi\cdot\tilde{x}(\theta)}B_{\tilde{x},\xi}^{(k_2,l)}(\varphi,\psi)(\theta)d\Vol(\theta) .
\end{multline*}
%We then remark from the properties of $\chi_{\pm 1}$ in Definition~\ref{def-chi-j} that $\sigma \in \supp(\chi_{\pm 1}) \implies |\sigma|\geq {\red 1-\eps_0}>0$ from which we infer that 
%$\left|\theta \cdot (\xi-\beta_0)\right| \geq \eps_0 |\xi-\beta_0|$ for $(\theta, \xi)$ in the support of $\chi_{\pm1} \left(\theta \cdot  \frac{\xi-\beta_0}{|\xi-\beta_0|} \right)$.
Item~\ref{i:lambda-neq0} in Lemma~\ref{l:f-phi-lam} together with Lemma~\ref{l:support-chi1} then imply that the integral 
$$
 \int_{\IR}\chi_\infty(t)t^{-s} \frac{t^l}{l!} e^{it(\xi-\beta_0) \cdot \mathbf{v}(\theta)}   |dt|  =\frac{1}{l !} f_{\chi_\infty,(\xi-\beta_0) \cdot \mathbf{v}(\theta)}(s-l)
$$
extends as an entire function in $s$ for any given $(\theta, \xi)$ in the support of $\chi_{\pm1} \left(\theta \cdot  \frac{\xi-\beta_0}{|\xi-\beta_0|} \right)$.  %{\red %Recall from the argument in~\eqref{e:IPP-time} that $|(\xi-\beta_0)\cdot\mathbf{v}(\theta)|\geq c_0|\xi-\beta_0|$ for some uniform $c_0$ thanks to our support properties.} 
According to Lemma~\ref{l:support-chi1} and to~\eqref{e:estim-f-phi}, it satisfies in addition, for any $m \in \IZ_+^*$,
\begin{align*}
| f_{\chi_\infty,(\xi-\beta_0) \cdot \theta }(s-l)| \leq C_{\chi_\infty,m} \frac{\langle |s|\rangle^m}{ c_0^m |\xi-\beta_0|^m} \frac{T_0^{-\Re(s)+l+1}}{\Re(s)-l+m-1}  , \quad \text{ for }   \Re(s) > -m +l + 1 . 
\end{align*}
Coming back to~$\mathscr{M}^{(l,\pm 1)}_{(\varphi,\psi)}(s)$, we find that it
 is holomorphic in $\Re(s) > -m +l + 1$ together with the estimate
\begin{align*}
  \Big| \mathscr{M}^{(l,\pm 1)}_{(\varphi,\psi)}(s)  \Big| 
&\lesssim_m   \frac{\langle |s|\rangle^m T_0^{-\Re(s)+l+1}}{\Re(s)-l+m-1}
  \sum_{\xi\in\IZ^d \setminus\{\beta_0\}} 
   \frac{1}{|\xi-\beta_0|^m} \| B_{\tilde{x},\xi}^{(k_2,l)}(\varphi,\psi)(\theta)\|_{L^1\left(\mathbf{C}_{\pm1}(\xi-\beta_0) \right)} \\
&\lesssim_m   \frac{\langle |s|\rangle^m T_0^{-\Re(s)+l+1}}{\Re(s)-l+m-1}
  \sum_{\xi\in\IZ^d \setminus\{\beta_0\}} 
   \frac{\left\|\pi_{\xi}^{(k_1)}(\varphi)\right\|_{L^2\left(\mathbf{C}_{\pm1}(\xi-\beta_0) \right)} \left\|\pi_{-\xi}^{(k_2)}(\psi)\right\|_{L^2\left(\mathbf{C}_{\pm1}(\xi-\beta_0) \right)}}{|\xi-\beta_0|^m}.
\end{align*}

\medskip
Finally, combining this together with~\eqref{e:C0estimate-1-bis} (and the statement preceding this estimate), and choosing $m=N$, we have obtained in the decomposition~\eqref{e:decmposition-J} that the function 
$$
 \mathscr{M}^{(l)}_{(\varphi,\psi)}(s) - \mathscr{M}^{(l,E)}_{(\varphi,\psi)}(s)  = \mathscr{M}^{(l,-1)}_{(\varphi,\psi)}(s)+\mathscr{M}^{(l,0)}_{(\varphi,\psi)}(s)+\mathscr{M}^{(l,1)}_{(\varphi,\psi)}(s) 
 $$
is a holomorphic function in $\Re(s) > -N +l + 1$. As long as $T_0 \geq \max\{1, \mathsf{t}_0\}$ and recalling the definition of the norm $\mathcal{H}^{N,-N/2,0,-N/2}_{k,\pm\beta_0}$ in Definition~\ref{def-ani-spaces}, we end up with the estimate
\begin{align*}
\left|  \mathscr{M}^{(l)}_{(\varphi,\psi)}(s) - \mathscr{M}^{(l,E)}_{(\varphi,\psi)}(s)  \right| \leq 
 C_N \frac{\langle |s|\rangle^N T_0^{-\Re(s)+l+1}}{\Re(s)-l+N-1} 
 \|\varphi\|_{\mathcal{H}^{N,-N/2,0,-N/2}_{k_1,\beta_0}} \|\psi\|_{\mathcal{H}^{N,-N/2,0,-N/2}_{k_2,-\beta_0}}.
\end{align*}
Coming back to the decomposition~\eqref{e:decomp-again-again}, recalling that $l \leq \min\{k_1,k_2\}$ and using the fact that $\mathscr{M}^{(l,E)}_{(\varphi,\psi)}(s)$ is entire by the identity ~\eqref{e:value-JlE} concludes the proof of the theorem.
\end{proof}

\subsection{Laplace transform}\label{ss:laplace}
 Here and in the whole section, we write
\begin{align*}
\overline{\C}_+:=\left\{z\in\C,\Re(z)\geq0\right\} , 
\end{align*}
and we say that a function is in $\mathcal{C}^k(\overline{\C}_+)$ if it is the restriction in $\overline{\C}_+$ of a function in $\mathcal{C}^k(\C)$.
In this part, we are going to use the notion of distributions obtained as boundary values of holomorphic functions~\cite[Th.~3.1.11]{Hormander90} in the most elementary way. We state a proposition which characterizes those distributions which arise as boundary values from the upper or lower half--plane of holomorphic functions: 
\begin{prop}\label{prop:boundaryvaluesholo} Let $T\in \mathcal{S}^\prime(\mathbb{R})$ be a tempered distribution supported in $\IR_{\pm}^*$. Then, denoting by $F(\cdot\mp iy)$ the Fourier transform of $x\mapsto T(x)e^{\mp yx}$, the function $(x+iy) \mapsto F(x\mp iy)$ is holomorphic on the half plane $\mathbb{H}_\mp=\{(x\mp iy): y>0\}$ and we have 
 $$
\widehat{T}=\lim_{\varepsilon\rightarrow 0^+}F(.\mp i\varepsilon)\ \text{in}\ \ml{S}^\prime(\IR) . $$ 
We will write $\widehat{T}=F(.\mp i0)$.
Conversely, if we are given a holomorphic function $F$ on $\mathbb{H}_\mp$, such that there exists $C,N$ and some polynomial $P$ such that $\forall y>0$:
\begin{eqnarray*}
\vert F(x\mp iy)\vert\leqslant C \vert P(x\mp iy)\vert \left(1+ y^{-N}\right)
\end{eqnarray*}
then 
the following limit 
\begin{eqnarray*}
\widehat{T}=\lim_{y\rightarrow 0^+} F(.\mp iy)
\end{eqnarray*}
exists in $\mathcal{S}^\prime(\mathbb{R})$ and it is the Fourier transform of a distribution $T$ carried on $\IR_\pm^*$.
\end{prop}
This is a particular case of a more general result valid on $\IR^d$ and described in~\cite[Thm IX.16 p.~23]{ReedSimonII}.
Given $\lambda<0$, we can now give the fundamental example of such boundary values together with their Fourier transform. Namely, if we consider
$$T(x):=\frac{ 2\pi e^{\mp i\lambda \frac{\pi}{2}}}{\Gamma(-\lambda)}\mathds{1}_{\IR_+^*}(\pm x) \vert x\vert^{-1-\lambda},$$
where $\mathds{1}_{\IR_+^*}$ is the indicator of the positive reals, then the Fourier transform $\widehat{T}(\xi)$ is given by $(\xi\mp i0)^{\lambda}$~\cite[p.~360]{GelfandShilovI}.
 The function $\mathds{1}_{\IR_+^*}(\pm x) \vert x\vert^{-1-\lambda}$ is in $L^1_{\text{loc}}(\IR)$ and bounded polynomially. Hence it defines a tempered distribution for all $\lambda<0$. In the litterature, one can also find the notation $\xi_\mp^{-1-\lambda}=\left(\mathds{1}_{\R_+^*}(\mp \xi) \vert\xi\vert^{-1-\lambda}\right)$~\cite[\S3.2, p.~48]{GelfandShilovI}. 
We will use the above proposition to describe the singularities of our Poincar\'e series.

The useful analogue to Lemma~\ref{l:f-phi-lam} is the following elementary result:
\begin{lemm}
\label{l:f-phi-lam-Laplace}
Let $T_1>T_0 >0$ and $\phi \in \mathcal{C}^\infty(\R)$ be such that $\supp(\phi) \subset [T_0,+\infty)$ and $\phi =1$ on $[T_1+\infty)$. 
Then, the following hold:
\begin{enumerate}
\item \label{i:dedebile-laplace} For any $\alpha \in \R$, the function 
$$
F_{\phi, \alpha} (z) := \int_{\R_+} \phi(t) t^{-\alpha} e^{-z t} |dt| , \quad \text{ for } \Re(z) >0,
$$
defines a holomorphic function in $\Re(z) >0$ which satisfies 
\begin{align*}
& \left| F_{\phi, \alpha} (z) \right| \leq  C_{\phi,\alpha}  , \quad \text{ for }z \in \overline{\C}_+ \text{ and extends continuously to this set}, \quad \text{ if } \alpha >1,\\
& \left| F_{\phi, \alpha} (z) \right| \leq  \frac{C_{\phi,\alpha}}{\Re(z)^{1-\alpha}} , \quad \text{ for }z \in \C ,  \Re(z) >0, \quad \text{ if } \alpha <1,\\
& \left| F_{\phi, 1} (z) \right| \leq  C_{\phi,1}(|\ln \Re(z)|+1)   , \quad \text{ for }z \in \C ,  \Re(z) > 0, \quad \text{ if } \alpha = 1 .
\end{align*}
\item \label{i:derivees} $\d_z^kF_{\phi, \alpha} (z) = (-1)^k F_{\phi, \alpha-k} (z)$ for all $k \in \Z_{+} ,z \in \C, \Re(z)>0$, and $F_{\phi, \alpha} \in \mathcal{C}^k(\overline{\C}_+)$ for all $k\in \Z_+$ such that $\alpha-k >1$.
\item \label{i:laplace-sing-gamma} If $\alpha  <1$, the function $F_{\phi,\alpha}$ can be extended to $\C \setminus \R_-$ (and even to $\C \setminus \{0\}$ if $\alpha \in \Z_-$) as a holomorphic function satisfying $F_{\phi,\alpha} (z) = \frac{\Gamma(1-\alpha)}{z^{1-\alpha}} + H_{\alpha}(z)$ where $H_{\alpha}$ is an entire function such that, for all $k\geq 0$, $|\partial_z^kH_{\alpha} (z)|\leq C_{\phi,\alpha,k}(e^{-T_1 \Re(z)}+1)$ on $\C$, for some $C_{\phi,\alpha,k}>0$.
Moreover, we have the following identity which holds in $\mathcal{S}^\prime(\mathbb{R})$: 
\begin{eqnarray}
\lim_{x\rightarrow 0^+} F_{\phi,\alpha}(x+iy)=\frac{\Gamma(1-\alpha)e^{i\frac{\pi}{2}(\alpha-1)}}{(y-i0)^{1-\alpha}} + H_{\alpha}(iy)
\end{eqnarray}

\item \label{i:laplace-sing-log} If $\alpha=1$, the function $F_{\phi,1}$ can be extended to the cut plane $\C \setminus \R^-$ as a holomorphic function satisfying $F_{\phi, 1} (z) = - \log(z) + H_1(z)$ where $\log$ is the principal determination of the logarithm and $H_1$ is an entire function such that, for all $k\geq 0$, $|\partial_z^kH_{1} (z)|\leq C_{\phi,k}(e^{-T_1 \Re(z)}+1)$ on $\C$, for some $C_{\phi,\alpha,k}>0$.
Moreover, we have the following identity which holds in $\mathcal{S}^\prime(\mathbb{R})$: 
\begin{eqnarray}
\lim_{x\rightarrow 0^+} F_{\phi,\alpha}(x+iy)=- \log(y-i0)-\frac{\pi}{2} + H_1(iy).
\end{eqnarray}

\item \label{i:rec-a-2-bal} For all $\beta \in \R , m \in \IZ_+^*$, 
\begin{align}
 \label{f-phi-lambda-IPP-bis}
F_{\phi,\beta}(z)& =\frac{1}{z^m}E_{\beta,m}(z) + (-1)^m  \frac{P_m(\beta)}{z^m}  F_{\phi,\beta+m}(z) , 
\end{align}
where $P_j(\beta)$ is defined in~\eqref{e:def-Pj} and $E_{\beta,m}(z)$ is an entire function such that 
\begin{equation}
\label{e:estim-E}
| E_{\beta,m}(z)| \leq C_{\phi,\beta,m}(e^{-T_1\Re(z)}+1) , \quad z \in \C .
\end{equation}

\item \label{i:laplace-sing-alpha} if $\alpha >1 , \alpha \notin \Z_+$, the function $F_{\phi,\alpha}$ can be extended to the cut plane $\C \setminus \R_{-}$ as a holomorphic function satisfying $F_{\phi,\alpha} (z) = \frac{\pi}{\sin(\pi \alpha)\Gamma(\alpha)} z^{\alpha-1} + H_\alpha(z)$ where $H_\alpha$ is an entire function such that, for all $k\geq 0$, $|\partial_z^kH_\alpha(z)|\leq C_{\phi,\alpha,k}\langle |z|\rangle^{\lfloor \alpha \rfloor}(e^{-T_1 \Re(z)}+1)$ on $\C$, for some $C_{\phi,\alpha,k}>0$.
Moreover, extended by the value zero at zero, we have $F_{\phi,\alpha} \in \mathcal{C}^{\lfloor\alpha\rfloor-1}(\overline{\C}_+)$.

\item \label{i:laplace-sing-log-n}  if $\alpha =n \in \Z_{+}^*$, $ n\geqslant 2$, the function $F_{\phi,n}$ can be extended to the cut plane $\C \setminus \R_-$ as a holomorphic function satisfying $F_{\phi,n} (z) = \frac{(-1)^n}{n!} z^{n-1} \log(z) + H_n(z)$ where $H_n$ is an entire function such that, for all $k\geq 0$, $|\partial^k_zH_n(z)|\leq C_{\phi,n,k}\langle |z|\rangle^{n-1}(e^{-T_1 \Re(z)}+1)$ on $\C$, for some $C_{\phi,n,k}>0$. Moreover, extended by the value zero at zero and for $n\geq 2$, we have $F_{\phi,n} \in \mathcal{C}^{n-2}(\overline{\C}_+)$.

\end{enumerate}
\end{lemm}

\begin{proof}
Let us first prove Item~\ref{i:dedebile-laplace}. 
The statement for $\alpha>1$ follows from a crude bound and continuity under the integral. For $\alpha<1$, we have 
\begin{align*}
\left| F_{\phi, \alpha} (z) \right| & \leq \|\phi\|_{\infty} \int_{T_0}^\infty t^{-\alpha} e^{-\Re(z) t} |dt| =  \|\phi\|_{\infty}\Re(z)^{1-\alpha} \int_{T_0\Re(z)}^\infty \sigma^{-\alpha} e^{-\sigma} |d\sigma| \\
 & \leq  \|\phi\|_{\infty}\Re(z)^{1-\alpha} \int_{0}^\infty \sigma^{-\alpha} e^{-\sigma} |d\sigma| ,
\end{align*}
which is the sought estimate. In the case $\alpha=1$, we have
\begin{align*}
\left| F_{\phi,1} (z) \right| & \leq \|\phi\|_{\infty} \int_{T_0}^\infty t^{-1} e^{-\Re(z) t} |dt| = \|\phi\|_{\infty}  \int_{T_0\Re(z)}^\infty \sigma^{-1} e^{-\sigma} |d\sigma| \\
& \leq \|\phi\|_{\infty}  \int_{T_0\Re(z)}^1 \sigma^{-1} |d\sigma|  + \|\phi\|_{\infty}  \int_1^\infty e^{-\sigma} |d\sigma| = -\|\phi\|_{\infty}  \ln(T_0 \Re(z)) + \|\phi\|_{\infty} e^{-1}.
\end{align*}
Item~\ref{i:derivees} is a straightforward consequence of  Item~\ref{i:dedebile-laplace} and differentiation under the integral.

\medskip
For Item~\ref{i:laplace-sing-gamma}, we first notice that for $\gamma := -\alpha >-1$ and $z\in \R^*_+$, we have  
\begin{multline*}
F_{\phi,-\gamma} =\int_0^{\infty}t^\gamma e^{-zt}|dt| + \int_0^{\infty}(\phi(t)-1)t^\gamma e^{-zt}|dt|\\
= \frac{1}{z^{\gamma+1}}\int_0^{\infty}\sigma^\gamma e^{-\sigma}|d\sigma| + \int_0^{\infty}(\phi(t)-1)t^\gamma e^{-zt}|dt| .
\end{multline*}
The last integral is an entire function satisfying the sought bound and the result follows from analytic continuation, where the cut plane $\C \setminus \R_-$ is chosen arbitrarily. These bounds give exactly the necessary moderate growth assumption so that the distributional limit
$F_{\phi,\alpha}(iy+0)$ exists by Proposition~\ref{prop:boundaryvaluesholo}.

\medskip
To prove Item~\ref{i:laplace-sing-log}, we differentiate $F_{\phi,1}$ in $\Re(z)>0$ to obtain
\begin{multline*}
\d_z F_{\phi,1}(z) = - \int_0^\infty \phi(t) e^{-zt} dt = - \int_0^\infty e^{-zt} dt + \int_0^\infty (1-\phi(t)) e^{-zt} dt \\
=-\frac{1}{z} + \int_0^\infty (1-\phi(t)) e^{-zt} dt .
\end{multline*}
Integrating this equation on the segment $[1,z]$ for $\Re(z)>0$ implies
$$
F_{\phi,1}(z)-F_{\phi,1}(1) = -\log(z) + \int_0^\infty (1-\phi(t)) \int_1^z e^{-st} ds dt  ,  
$$
whence, for $\Re(z)>0$,
$$
F_{\phi,1}(z)=  -\log(z) + \int_0^\infty \phi(t)t^{-1} e^{-t} dt  +  \int_0^\infty (1-\phi(t)) e^{-\frac{1+z}{2}t}\frac{\sinh \left(\frac{z-1}{2}t \right)}{t} dt .
$$
The right hand side continues holomorphically to $\C\setminus \R_-$, the last integral being on the compact set $[0,T_1]$. Using now that $\sinh(a+ib) = \sinh a \cos b + i \cosh a \sin b$, we have $|\frac{\sinh(t(a+ib))}{t}| \leq |\frac{\sinh(ta)}{t}| + \cosh(ta) \leq 2\cosh(ta) \leq e^{|ta|}$, whence 
\begin{multline*}
\left| \int_0^\infty (1-\phi(t)) e^{-\frac{1+z}{2}t}\frac{\sinh \left(\frac{z-1}{2}t \right)}{t} dt  \right| \\ 
\leq C_\phi \int_0^{T_1} e^{-\frac{1+\Re(z)}{2}t} e^{\frac{|\Re(z)-1|}{2}t}dt =C_\phi \left( \frac{e^{T_1 A}-1}{A}\right)\Big|_{A = \frac{|\Re(z)-1|-(\Re(z)+1)}{2}} ,
\end{multline*}
which implies the sought estimate. The distributional limit again follows from Proposition~\ref{prop:boundaryvaluesholo} and the bound from item~\ref{i:dedebile-laplace}.

\medskip
Item~\ref{i:rec-a-2-bal} is proved as in Lemma~\ref{l:f-phi-lam} and consists in integrating by parts $m$ times to obtain, for $\text{Re}(z)>0$,
$$
F_{\phi, \alpha}(z) =\int_0^\infty \left(\frac{1}{z} \frac{d}{dt}\right)^m (\phi(t)t^{-\alpha}) e^{-zt} |dt| ,
$$ and then expanding with the Leibniz formula. We obtain the formula~\eqref{f-phi-lambda-IPP-bis} with  
$$E_{\alpha,m}(z)=\sum_{j=0}^{m-1}(-1)^{j} \begin{pmatrix} m\\ j\end{pmatrix}   P_j(\alpha) F_{\phi^{(m-j)},\alpha+j}(z),$$ and the estimate~\eqref{e:estim-E} follows from the fact that $|F_{\phi^{(m-j)},\alpha}(z)|\leq C_{\phi,\alpha}(e^{-T_1\Re(z)}+1)$ for $z\in\C$ if $m-j>0$ since $\phi^{(m-j)}$ is compactly supported in $\IR_+^*$.

\medskip
Item~\ref{i:laplace-sing-alpha} is a consequence of Items~\ref{i:laplace-sing-gamma} and~\ref{i:rec-a-2-bal} for $m = \lfloor \alpha \rfloor \in \Z_+$ and $\beta = \alpha -\lfloor \alpha \rfloor \in (0,1)$. From~\eqref{f-phi-lambda-IPP-bis}, we obtain 
\begin{align*}
F_{\phi,\alpha}(z)  &= F_{\phi,\beta+m}(z)  = (-1)^m  \frac{z^m}{P_m(\beta)}  F_{\phi,\beta}(z) -  \frac{(-1)^m}{P_m(\beta)} E_{\beta,m}(z)   \\
& = (-1)^m  \frac{z^m}{P_m(\beta)}  \left(  \frac{\Gamma(1-\beta)}{z^{1-\beta}} + H_{\beta}(z) \right)-  \frac{(-1)^m}{P_m(\beta)} E_{\beta,m}(z)  ,
\end{align*}
where we have used Item~\ref{i:laplace-sing-gamma} in the second line. We further notice from $\Gamma(z+1)=z\Gamma(z)$ that $P_m(\beta) = \frac{\Gamma(\beta+m)}{\Gamma(\beta)}$ (see~\eqref{e:def-Pj})
whence
\begin{align*}
F_{\phi,\alpha}(z) = (-1)^m z^{m+\beta-1} \frac{\Gamma(1-\beta)\Gamma(\beta)}{\Gamma(\beta+m)} +  \frac{(-1)^mz^m}{P_m(\beta)}   H_{\beta}(z) -  \frac{(-1)^m}{P_m(\beta)} E_{\beta,m}(z)  ,
\end{align*}
and hence the sought formula recalling $\Gamma(1-\beta)\Gamma(\beta) = \frac{\pi}{\sin(\pi\beta)} = (-1)^m \frac{\pi}{\sin(\pi(m+\beta))} $.

\medskip
Item~\ref{i:laplace-sing-log-n} is a consequence of Items~\ref{i:laplace-sing-log} and~\ref{i:rec-a-2-bal} taken for $\beta=1$ and $m=n-1$, and the fact that $P_{n-1}(1)=n!$.
\end{proof}

%%%%%%%%%%%%%%%%%%%%%%%%%%%%%%%%%%%%%%%%%%%%%
\subsubsection{Continuous/$\mathcal{C}^k$ continuation of $\widehat{\chi_s^L}(-i\mathbf{V}_{\beta_0})$} Before discussing the Continuous/$\mathcal{C}^k$  continuation, let us first clarify its holomorphic properties on $\text{Re}(s)>0$:
\begin{prop}\label{p:laplace-holomorphe} Let $\chi_{\infty}$ be a function verifying assumption~\eqref{e:cutoff-infini}, let $\beta_0\in H^1(\IT^d,\IR)$ and let $\tilde{x}:\IS^{d-1}\rightarrow \IR^d$ be a smooth function. 
Then, for all $(\varphi,\psi)\in\Omega^{k_1}(S\IT^d)\times\Omega^{k_2}(S\IT^d)$ with $k_1+k_2=2d-1$, the function
  \begin{equation}
\label{e:Laplace-t}
s \mapsto \mathscr{L}_{(\varphi,\psi)}(s)  := \int_{S\IT^d}\varphi\wedge \widehat{\chi_{s}^L}(-i\mathbf{V}_{\beta_0})\mathbf{T}_{-\tilde{x}}^*(\psi) 
\end{equation}
is holomorphic on $\Re(s)>0$ and it satisfies
\begin{align}
\label{e:Laplace-rough-estim}
\left|  \mathscr{L}_{(\varphi,\psi)}(s)   \right|  
 \leq   \frac{C}{\Re(s)^{\min \{k_1,k_2\} +1}} \| \varphi \|_{L^2(S\T^d)} \| \psi \|_{L^2(S\T^d)} .
 \end{align}
\end{prop}
Recall that $\min \{k_1,k_2\} \leq d-1$ so that the latter estimate can always be roughly bounded by $\frac{C}{\Re(s)^d}$.

\begin{proof}
 We start with~\eqref{e:integrated-correlation} and we use again Lemma~\ref{l:reduction} to write
 \begin{eqnarray*}
 \Cor_{\varphi,\mathbf{T}_{-\tilde{x}}^*(\psi)}(t,\beta_0) = \sum_{l=0}^{\min\{k_1,k_2\}} \Cor_{\varphi,\mathbf{T}_{-\tilde{x}}^*(\psi)}^{l}(t,\beta_0) 
 \end{eqnarray*}
with 
$$
 \Cor_{\varphi,\mathbf{T}_{-\tilde{x}}^*(\psi)}^{l}(t,\beta_0)  = 
 \frac{t^l}{l!}\sum_{\xi\in\IZ^d}\int_{\IS^{d-1}}e^{it(\xi-\beta_0)\cdot\mathbf{v}(\theta)}e^{i\xi\cdot\tilde{x}(\theta)}B_{\tilde{x},\xi}^{(k_2,l)}(\varphi,\psi)(\theta) d\Vol(\theta). 
$$
Integrating against $\chi_\infty(t) e^{-st}$ then yields
 \begin{eqnarray}
 \label{e:decomp-again-again-laplace}
\mathscr{L}_{(\varphi,\psi)}(s)  =   \int_{S\IT^d}\varphi\wedge \widehat{\chi_{s}^M}(-i\mathbf{V}_{\beta_0})\mathbf{T}_{-\tilde{x}}^*(\psi) = \sum_{l=0}^{\min\{k_1,k_2\}}\mathscr{L}^{(l)}_{(\varphi,\psi)}(s) ,
  \end{eqnarray}
 with 
\begin{align}
\label{e:def-J-l-again-laplace}
\mathscr{L}^{(l)}_{(\varphi,\psi)}(s) = \int_{\R}\chi_\infty (t) e^{-st} \Cor_{\varphi,\mathbf{T}_{-\tilde{x}}^*(\psi)}^{l}(t,\beta_0) |dt| .
\end{align}
According to~\eqref{e:rough-estim-corr}, we have
\begin{align*}
\left|  \Cor_{\varphi,\mathbf{T}_{-\tilde{x}}^*(\psi)}^{l}(t,\beta_0) \right| \leq C t^l\| \varphi \|_{L^2(S\T^d)} \| \psi \|_{L^2(S\T^d)},
\end{align*}
which according to~\eqref{e:def-J-l-again-laplace} implies holomorphy of $\mathcal{L}^{(l)}_{(\varphi,\psi)}$ in $\Re(s)>0$. 
We further deduce that 
\begin{multline*}
\left|  \mathcal{L}^{(l)}_{(\varphi,\psi)}(s)\right| \leq 
 \int_{\R}    \left| \chi_\infty (t) e^{-st}  \Cor_{\varphi,\mathbf{T}_{-\tilde{x}}^*(\psi)}^{l}(t,\beta_0) \right| |dt|\\
 \leq C\int_\R \left| \chi_\infty (t)  t^{l}  e^{-st} \right|  |dt|\| \varphi \|_{L^2(S\T^d)} \| \psi \|_{L^2(S\T^d)} .
\end{multline*}
Recalling that $l \leq \min \{k_1,k_2\}$, Item~\ref{i:dedebile-laplace} in Lemma~\ref{l:f-phi-lam-Laplace} then yields, as $\text{Re}(s)\rightarrow 0^+$,
$$
 \left| \mathscr{L}^{(l)}_{(\varphi,\psi)}(s) \right| \leq \frac{C}{\Re(s)^{1+ \min \{k_1,k_2\} }} 
 \| \varphi \|_{L^2(S\T^d)} \| \psi \|_{L^2(S\T^d)} ,
 $$
from which we infer~\eqref{e:Laplace-rough-estim} thanks to~\eqref{e:decomp-again-again-laplace}.
\end{proof}

%%%%%%%%%%%%%%%%%%%%%%%%%%%%%%%%%%%%%%
We now turn to our main statement on these regularized Laplace transforms.
Lemma~\ref{l:f-phi-lam-Laplace} leads us to introduce the functions 
\begin{align}
\label{e:def-functionsF}
\mathsf{F}_{\alpha}(z) := 
\left\{
\begin{array}{ll}
\displaystyle  \frac{\Gamma(1-\alpha)}{z^{1-\alpha}}  , &  \text{ if } \quad \alpha <1   \text{ or } \alpha >1 , \alpha \notin \Z_+ ,\\
\displaystyle  \frac{(-1)^n}{n!} z^{n-1} \log(z) ,  & \text{ if } \quad\alpha = n \in \Z_{+} ,\end{array}
\right.
\end{align}
these functions are considered as holomorphic functions on the plane $\C\setminus \R_-$ (except if $\alpha \in \Z_-^*$ in which case $\mathsf{F}_{\alpha}$ is holomorphic in $\C^*$). We also associate the corresponding distributions obtained as boundary values that we still denote by $\mathsf{F}_{\alpha}$:
\begin{align}
\label{e:def-boundaryvaluesF}
\mathsf{F}_{\alpha}(iy+0) := 
\left\{
\begin{array}{ll}
\displaystyle  \frac{\Gamma(1-\alpha)e^{i\frac{\pi}{2}(\alpha-1)}}{(y-i0)^{1-\alpha}}  , &  \text{ if } \quad \alpha <1 \text{ or }\alpha >1 , \alpha \notin \Z_+\\
\displaystyle  \frac{(-1)^ne^{i\frac{\pi}{2}(n-1)}}{n!} y^{n-1} \left(\log(y-i0)+\frac{\pi}{2}\right) ,  & \text{ if } \quad\alpha = n \in \Z_{+}  
\end{array}
\right.
\end{align}
Both will describe the singularities of the Laplace transform of correlators up to the imaginary axis.
For a given $\alpha \in \R$, $\mathsf{F}_{\alpha}$ is essentially the Laplace transform of $t^{-\alpha}$ (near $t=+\infty$).

We also denote
$$
\CLam :=\left\{s \in \overline{\C}_+, |\Im(s)| \leq \Lambda \right\}  = \left\{s \in  \C ,\Re(z)\geq0 , |\Im(s)| \leq \Lambda \right\},
$$
and explain how the Laplace transform extends to this set.

\begin{theo}\label{t:general-laplace} Suppose that the assumptions of Theorem~\ref{t:general-mellin} on $\tilde{x}$, $\chi_\infty$ and $T_0$ are satisfied and let\footnote{In particular, it only depends on the choice of $\chi_{1}$ in the stationary phase Lemma.} $c_0>0$ be the constant from Lemma~\ref{l:support-chi1}.
Given $N,N_0 \in \IZ_+^2,m\geq 0, \Lambda \geq 1$, and $(\varphi,\psi)\in\Omega^{k_1}(S\IT^d)\times\Omega^{k_2}(S\IT^d)$ with $k_1+k_2=2d-1$, we define, for all $s\in \C$ with $\Re(s)>0$, 
\begin{align}
 \label{e:expansion-laplace}
\mathscr{R}^{(N,\Lambda )}_{(\varphi,\psi)}(s)&  : =\mathscr{L}_{(\varphi,\psi)}(s) - 
\sum_{l=0}^{\min\{k_1,k_2\}}  \frac{ E^{(l)}_{\beta_0}}{s^{l+1}} \nonumber \\
&  - \sum_{l=0}^{\min\{k_1,k_2\}}\sum_{\pm}  \sum_{j=0}^{N-1}
 \sum_{0< |\xi-\beta_0| \leq 2c_0^{-1} \Lambda}    
\frac{\mathsf{F}_{\frac{d-1}{2}+j-l}\big(s- i\lambda_\pm(\xi)\big)}{l!|\xi-\beta_0|^{\frac{d-1}{2}+j}} 
 P_{j,l,\xi}^\pm[\varphi,\psi]\left(\pm \frac{\xi-\beta_0}{|\xi-\beta_0|}\right)  ,
\end{align}
where $E^{(l)}_{\beta_0}$ is defined in~\eqref{def-El}, $\mathsf{F}_{\alpha}$ in~\eqref{e:def-functionsF},  and, letting $L_{j,\omega}^\pm$ being that of Lemma~\ref{c:stat-points}, with
\begin{align}
\label{e:def-P-j} 
P_{j,l,\xi}^\pm[\varphi,\psi]\left( \omega \right) 
= \frac{e^{\mp i\frac{\pi}{4}(d-1)}}{\sqrt{\kappa\circ\mathbf{v}(\pm\omega)}} (2\pi)^{\frac{d-1}{2}} L_{j,\omega}^\pm (e^{i \xi \cdot \tilde{x}(\cdot)} B_{\tilde{x},\xi}^{(k_2,l)}(\varphi,\psi)) \left(\omega \right) , \quad \omega \in \IS^{d-1}.
\end{align}
Both sides of~\eqref{e:expansion-laplace} extend for $s=x+iy$ when $x\rightarrow 0^+$ as tempered distributions of the variable $y$ where $\mathsf{F}_{\alpha}$ is defined in~\eqref{e:def-boundaryvaluesF}.

For any $N,N_0 \in \IZ_+^2$ such that 
$$
\mathsf{k}:=\mathsf{k}(N_0,N)=\min\left\{N_0,N + \left\lceil \frac{d-1}{2}\right\rceil\right\}- \min\{k_1,k_2\} - 2>0,
$$ 
and for any $m\geq 0$, there exists $C= C_{N_0,N,m}>0$ such that for any $\Lambda  \geq 1$ and 
for every $(\varphi, \psi) \in \mathcal{H}^{N_0,-N_0/2,2N+d,-m}_{k_1,\beta_0} \times \mathcal{H}^{N_0,-N_0/2,2N+d,-m}_{k_2,-\beta_0}$ with $k_1+k_2=2d-1$
 the function $\mathscr{R}^{(N,\Lambda)}_{(\varphi,\psi)}$ originally defined for $\Re(s)>0$, extends as a function $\mathscr{R}^{(N,\Lambda)}_{(\varphi,\psi)}\in \mathcal{C}^\mathsf{k}(\CLam)$ for $\mathsf{k}=\mathsf{k}(N_0,N)$ with 
\begin{align}
\label{e:estim-remainder-lap}
 \left\| \mathscr{R}^{(N,\Lambda)}_{(\varphi,\psi)} \right\|_{\mathcal{C}^\mathsf{k}(\CLam)} 
   \leq C \Lambda^{2m+N+\frac{d+1}{2}} \|\varphi\|_{\mathcal{H}^{N_0,-N_0/2,2N+d,-m}_{k_1,\beta_0}}\|\psi \|_{\mathcal{H}^{N_0,-N_0/2,2N+d,-m}_{k_2,-\beta_0}},
\end{align}
using the notation of Definition~\ref{def-ani-spaces}.
\end{theo}
In particular, this theorem states that for $(\varphi,\psi)\in\Omega^{k_1}(S\IT^d)\times\Omega^{k_2}(S\IT^d)$, the Laplace transform $\mathscr{L}_{(\varphi,\psi)}(s)$ extends as a $\mathcal{C}^\infty$ function in a neighborhood in $\overline{\C}_+$ of any point $z_0 \in i \R \setminus \left(i\Lambda_{\beta_0}\cup\{0\}\right)$, where $\Lambda_{\beta_0}$ was defined in~\eqref{e:critical-values}. %\footnote{Again this can be equivalently rephrased in terms of the spectrum of $\Delta_{\beta_0}$ instead of that of $\Delta_{-\beta_0}$ if we change variables $\xi \to -\xi$ as in~\eqref{beta0-beta0}.}, where $\Delta_{-\beta_0}= (\partial_x-i\beta_0)^2$ is the magnetic Laplacian acting on functions on $\mathbb{T}^d$.  
This follows from the fact that bookkeeping the regularities in the proof below, we may choose the regularity exponent $k=\inf( N_0-2-\inf(k_1,k_2), N-2-\inf(k_1,k_2)+\frac{d-1}{2} )$ and we see that $k\rightarrow +\infty$ when $N,N_0\rightarrow +\infty$.
 Moreover, when $\text{Re(s)}>0$ goes to zero, then $y\mapsto \mathscr{L}_{(\varphi,\psi)}(iy+0)$ makes sense as a tempered distribution obtained as boundary value of holomorphic function and Theorem~\ref{t:general-laplace} describes its singularity near any point in $\Lambda_{\beta_0}$ explicitly in terms of the distributions $\mathsf{F}_{\alpha}$ in~\eqref{e:def-boundaryvaluesF}. In particular, if $d$ is odd,~\eqref{e:expansion-laplace} gives an expansion of the limit Laplace transform $\lim_{x\rightarrow 0^+} \mathscr{L}_{(\varphi,\psi)}(x+iy)$ in terms of the distributions $\frac{1}{(y-i0-z_j)^m}$ and $(y-z_j)^n \log(y-i0-z_j)$, for $z_j \in \Lambda_{\beta_0}$ and $m ,n \in \Z_+ ,m\leq \min\{k_1,k_2\}-\frac{d-1}{2}$. If $d$ is even,~\eqref{e:expansion-laplace} is an expansion of the limit Laplace transform in terms of the distributions $\frac{1}{(y-i0-z_j)^{m/2}}$ for $z_j \in  \Lambda_{\beta_0}$ and $m \in \Z ,m\leq 2\min\{k_1,k_2\}- d+1$. 
The coefficients in this expansion are in principle explicit; recall for instance that $L_{0,\omega}^\pm=1$.
Note also that 
in~\eqref{e:estim-remainder-lap} the regularity of the resolvent up to the imaginary axis, given by the index $\mathsf{k}(N_0,N)$, depends explicitly on the anisotropic Sobolev regularity $\mathcal{H}^{N_0,-N_0/2,2N+d,-m/2}$ of the currents $(\varphi, \psi)$. The bigger $N_0,N$ are, the better the regularity of this background remainder term is.

Note finally that in case $\tilde{x}=0$ (and when computing the resolvent acting on functions), the proof simplifies slightly and the estimate~\eqref{e:estim-remainder-lap} of the remainder is better behaved in terms of spaces and powers of $\Lambda$. As far as the proof is concerned, it is worth noticing that, as opposed to the proofs of Theorems~\ref{th:correlation-cut-off} and~\ref{t:general-mellin}, \textbf{we do actually make use of both non-stationary and stationary phase estimates} of Lemmas~\ref{c:non-stat-points} and~\ref{c:stat-points}. 

\begin{proof}[Proof of Theorem~\ref{t:general-laplace} ]
As in the proof of Proposition~\ref{p:laplace-holomorphe}, we consider $(\varphi,\psi)\in\Omega^{k_1}(S\IT^d)\times\Omega^{k_2}(S\IT^d)$ and decompose $\Cor_{\varphi,\mathbf{T}_{-\tilde{x}}^*(\psi)}(t,\beta_0)$ as a sum of $\Cor_{\varphi,\mathbf{T}_{-\tilde{x}}^*(\psi)}^{l}(t,\beta_0)$ according to~\eqref{e:decomposition1}--\eqref{e:decomposition2}. Then decomposing $\Cor_{\varphi,\mathbf{T}_{-\tilde{x}}^*(\psi)}^{l}(t,\beta_0)$ according to~\eqref{e:decomposition-de-ouf}, we are left with describing the terms $\mathscr{L}^{(l)}_{(\varphi,\psi)}(s)$ in~\eqref{e:decomp-again-again-laplace}--\eqref{e:def-J-l-again-laplace}, which we again decompose accordingly as 
 \begin{align}
\label{e:decmposition-L}
\mathscr{L}^{(l)}_{(\varphi,\psi)}(s) =\mathscr{L}^{(l,E)}_{(\varphi,\psi)}(s) +\mathscr{L}^{(l,-1)}_{(\varphi,\psi)}(s)+\mathscr{L}^{(l,0)}_{(\varphi,\psi)}(s)+\mathscr{L}^{(l,1)}_{(\varphi,\psi)}(s) , 
 \end{align}
with 
\begin{align}
\mathscr{L}^{(l,E)}_{(\varphi,\psi)}(s)&= \int_0^\infty\chi_\infty (t)e^{-st}\frac{t^l}{l!}E^{(l)}_{\beta_0}|dt|,\quad\text{ and }  \nonumber \\ 
 \mathscr{L}^{(l,j)}_{(\varphi,\psi)}(s) &= \int_0^\infty\chi_\infty (t) e^{-st}  \frac{t^l}{l!} \Cor^{l}_{ j}(t) |dt| , 
  \quad \text{ for }j \in \{-1,0,1\}  .
  \label{e:Lj}
  \end{align}
We now study each of these terms separately.

\medskip
Firstly, we have using Item~\ref{i:laplace-sing-gamma} of Lemma~\ref{l:f-phi-lam-Laplace}
\begin{align}
\label{e:value-JlE-laplace}
\mathscr{L}^{(l,E)}_{(\varphi,\psi)}(s) =E^{(l)}_{\beta_0} \frac{1}{l!} \int_{0}^\infty\chi_\infty (t)t^{l}e^{-st }|dt|= 
 \frac{ E^{(l)}_{\beta_0}}{l!}
 F_{\chi_\infty,-l} (s)  =   \frac{ E^{(l)}_{\beta_0}}{s^{l+1}} + \frac{ E^{(l)}_{\beta_0}}{l!} H_{-l}(s)  
 \end{align}
 where $H_{-l}$ is an entire function such that $|\partial_s^kH_{-l} (s)|\leq C_{\chi_\infty,-l,k}(e^{-T_1 \Re(s)}+1)$ on $\C$.

\medskip
Secondly, we consider the term with $\mathscr{L}^{(l,0)}_{(\varphi,\psi)}(s)$ and proceed as in the proof of Theorem~\ref{th:correlation-cut-off}.
The index $j=0$ means that we consider a good term in the nonstationary phase region. According to Lemma~\ref{c:non-stat-points}, we have, for $t \geq \mathsf{t}_0$, all $N\in \IZ_+$,
\begin{align*}
| \Cor^{l}_{ 0}(t)  |  \leq C_N \sum_{\xi\in \IZ^d \setminus\{\beta_0\}}  (|\xi-\beta_0| t )^{-N}   \|B_{\tilde{x},\xi}^{(k_2,l)}(\varphi,\psi)\|_{W^{N,1}(\mathbf{C}_0(\xi-\beta_0))}. 
\end{align*}
We now use that $\supp(\chi_{\infty})\subset[T_0,+\infty)$ with $T_0\geq\max\{1, \mathsf{t}_0 \}$.
Integrating in~\eqref{e:Lj}, using Item~\ref{i:dedebile-laplace} in Lemma~\ref{l:f-phi-lam-Laplace}, we deduce that $\mathscr{L}^{(l,0)}_{(\varphi,\psi)}(s)$ extends as a function in $\mathcal{C}^\infty(\overline{\C}_+)$ with 
\begin{align}
\left| \d_s^k \mathscr{L}^{(l,0)}_{(\varphi,\psi)}(s) \right| & \leq C_{N_0}  \int_{0}^\infty\chi_\infty (t)t^{l+k} t^{-N_0} e^{-\Re(s)t }|dt| \sum_{\xi\in \IZ^d \setminus\{\beta_0\}}   \frac{ \|B_{\tilde{x},\xi}^{(k_2,l)}(\varphi,\psi)\|_{W^{N_0,1}(\mathbf{C}_0(\xi-\beta_0))}}{ |\xi-\beta_0|^{N_0}}\nonumber  \\
& \leq  C_{N_0,k}  \sum_{\xi\in\IZ^d \setminus\{\beta_0\}}\frac{\left\|\pi_{\xi}^{(k_1)}(\varphi)\right\|_{H^{N_0}\left(\mathbf{C}_0(\xi-\beta_0) \right)} \left\|\pi_{-\xi}^{(k_2)}(\psi)\right\|_{H^{N_0}\left(\mathbf{C}_0(\xi-\beta_0) \right)}}{|\xi-\beta_0|^{N_0}} 
 , 
\label{e:laplace-L0}
\end{align}
uniformly on $\Re(s)\geq 0$, as soon as $N_0-k -l >1$. Recalling that $l \leq \min \{k_1,k_2\}$, this holds for all $N_0,k$ such that $N_0 > k+ \min \{k_1,k_2\}+1$.

\medskip
Thirdly, we consider the term with $\mathscr{L}^{(l,\pm1)}_{(\varphi,\psi)}(s)$. According to~\eqref{e:Lj} and the expression of $\Cor^{(l,\pm1)}(t)$ in~\eqref{e:decomposition-de-ouf}, we have 
\begin{multline*}
\mathscr{L}^{(l,\pm1)}_{(\varphi,\psi)}(s) =   \sum_{\xi\in\IZ^d \setminus\{\beta_0\}}\int_{\IS^{d-1}}\left(  \int_{\IR}\chi_\infty(t) \frac{t^l}{l!} e^{-st }e^{it(\xi-\beta_0) \cdot \mathbf{v}(\theta)}   |dt| \right)  \\
\times  \chi_{\pm1} \left(\theta \cdot  \frac{\xi-\beta_0}{|\xi-\beta_0|} \right)e^{i\xi\cdot\tilde{x}(\theta)}B_{\tilde{x},\xi}^{(k_2,l)}(\varphi,\psi)(\theta)d\Vol(\theta) .
\end{multline*}

\medskip
\textbf{An extra decomposition in large and small Fourier modes.}

Given $\Lambda>0$, we recall that we always assume $|\Im(s)|\leq \Lambda$, and we split (further) this expression according to    
\begin{align*}
& \mathscr{L}^{(l,\pm1)}_{(\varphi,\psi)}(s)  = \mathscr{L}^{(l,\pm1)}_{\leq}(s) +\mathscr{L}^{(l,\pm1)}_{>}(s) , \quad \text{ with }   \\
&   \mathscr{L}^{(l,\pm1)}_{\leq}(s):= 
\sum_{\xi\in\IZ^d ,0< |\xi-\beta_0| \leq 2 c_0^{-1}\Lambda}\cdots , \quad  \text{ and } \quad \mathscr{L}^{(l,\pm1)}_{>}(s):= 
\sum_{\xi\in\IZ^d , |\xi-\beta_0| > 2c_0^{-1} \Lambda} \cdots,
\end{align*}
where $c_0>0$ is the constant from Lemma~\ref{l:support-chi1} (depending only on $\chi_1$ and $K$). In other words, we decomposed the sum $\sum_\xi$ into an infinite sum over large Fourier modes (i.e. $\vert \xi\vert$ far from $\vert s\vert$) and into a finite sum over small Fourier modes. We will apply stationary phase estimates only to the finite sum and use integration by parts with the infinite sum to get decay in $\xi$.
Let us first consider the good term $\mathscr{L}^{(l,\pm1)}_{>}(s)$. 
Recalling Lemma~\ref{l:support-chi1}, one has %the properties of $\chi_{\pm 1}$ in Definition~\ref{def-chi-j}, we have $\sigma \in \supp(\chi_{\pm 1}) \implies |\sigma|\geq {\red 1-\eps_0}>0$ with $\eps_0 <1/2$ {\red small enough}, from which we {\red showed in~\eqref{e:IPP-time}} that 
$\left|\mathbf{v}(\theta) \cdot (\xi-\beta_0)\right| \geq  c_0 |\xi-\beta_0|$ for $(\theta, \xi)$ in the support of $\chi_{\pm1} \left(\theta \cdot  \frac{\xi-\beta_0}{|\xi-\beta_0|} \right)$.
For $|\Im(s)|\leq \Lambda$, and $(\theta, \xi)$ in the support of $\chi_{\pm1} \left(\theta \cdot  \frac{\xi-\beta_0}{|\xi-\beta_0|} \right)$ and such that $ |\xi-\beta_0| > \frac{2}{c_0} \Lambda$, we thus have a lower bound on the phase factor:
$$
|s- i(\xi-\beta_0)\cdot\mathbf{v}(\theta) | \geq |\Im(s)  - (\xi-\beta_0)\cdot\mathbf{v}(\theta)| \geq c_0 |\xi-\beta_0| -  |\Im(s)| \geq \frac{c_0}{2} |\xi-\beta_0|  .
$$
According to Items~\ref{i:dedebile-laplace},~\ref{i:derivees} and~\ref{i:rec-a-2-bal} of Lemma~\ref{l:f-phi-lam-Laplace} applied with $z=s-i(\xi-\beta_0) \cdot \mathbf{v}(\theta)$, $\alpha=-l$ and $\beta=-(l+k)$, we deduce that for any $m, k \in \IZ_+$,
$$
\left| \d_s^k \int_{\IR}\chi_\infty(t) \frac{t^l}{l!}e^{-st }e^{it(\xi-\beta_0) \cdot \mathbf{v}(\theta)}|dt|  \right| \leq \frac{C_{m,k}}{|s- i(\xi-\beta_0)\cdot \mathbf{v}(\theta) |^m}\leq  \frac{C_{m,k}}{|\xi-\beta_0|^m} , 
$$
uniformly in $\Lambda>0$, $\Re(s) \geq 0, |\Im(s)|\leq \Lambda$, $ |\xi-\beta_0| > 2 c_0^{-1}\Lambda$, $(\theta, \xi)$ in the support of $\chi_{\pm1} \left(\theta \cdot  \frac{\xi-\beta_0}{|\xi-\beta_0|} \right)$. Moreover, this integral extends as a function in $\mathcal{C}^\infty(\CLam)$.
As a consequence, $\mathscr{L}^{(l,\pm1)}_{>}$ also extends as a function in $\mathcal{C}^\infty(\CLam)$ with, for $s \in \CLam$,
\begin{align}
\label{e:Lgeq}
\left|  \d_s^k \mathscr{L}^{(l,\pm1)}_{>}(s) \right| 
& \leq  C_{m,k}
  \sum_{\xi\in\IZ^d , |\xi-\beta_0| > 2c_0^{-1} \Lambda} \frac{1}{|\xi-\beta_0|^m} \left\| B_{\tilde{x},\xi}^{(k_2,l)}(\varphi,\psi) \right\|_{L^1\left(\mathbf{C}_{\pm1}(\xi-\beta_0) \right)} \nonumber \\
&   \leq  C_{m,k}
  \sum_{\xi\in\IZ^d , |\xi-\beta_0| > 2  c_0^{-1}\Lambda} \frac{1}{|\xi-\beta_0|^m}
  \left\|\pi_{\xi}^{(k_1)}(\varphi)\right\|_{L^2\left(\mathbf{C}_{\pm1}(\xi-\beta_0) \right)} \left\|\pi_{-\xi}^{(k_2)}(\psi)\right\|_{L^2\left(\mathbf{C}_{\pm1}(\xi-\beta_0) \right)}  .
  \end{align}

We next consider the term $\mathscr{L}^{(l,\pm1)}_{\leq}(s)$, which we rewrite as
\begin{align}
\label{e:L-express-ouf}
 \mathscr{L}^{(l,\pm1)}_{\leq}(s):= 
\sum_{\xi\in\IZ^d ,0< |\xi-\beta_0| \leq 2 c_0^{-1}\Lambda} 
\int_{\IR}\chi_\infty(t) \frac{t^l}{l!} e^{-st } I_{B_{\tilde{x},\xi}^{(k_2,l)}(\varphi,\psi)}^{(\pm 1)}(\xi-\beta_0,t) |dt| ,
\end{align}
where $I_F^{(\pm 1)}(\xi-\beta_0,t)$ is defined in~\eqref{e:split-3-int-bis}.
We then use the asymptotic expansion in Lemma~\ref{c:stat-points} which yields, with $P_{j,l,\xi}^\pm[\varphi,\psi]$ defined in~\eqref{e:def-P-j},
\begin{align*}
I_{B_{\tilde{x},\xi}^{(k_2,l)}(\varphi,\psi)}^{(\pm 1)}(\xi-\beta_0,t) =   \sum_{j=0}^{N-1} \frac{e^{ it\lambda_\pm(\xi)}}{(t|\xi-\beta_0|)^{\frac{d-1}{2}+j}} 
 P_{j,l,\xi}^\pm[\varphi,\psi]\left(\pm \frac{\xi-\beta_0}{|\xi-\beta_0|}\right) 
 + \frac{R_N^\pm [\varphi,\psi] (\xi,t)}{t^{N+\frac{d-1}{2}}} ,  
\end{align*}
where \begin{align}
\label{e:estim-RN-}
|R_N^\pm [\varphi,\psi](\xi,t) | & \leq C_N  |\xi-\beta_0|^{N+\frac{d+1}2}    \|B_{\tilde{x},\xi}^{(k_2,l)}(\varphi,\psi)\|_{W^{2N+ d,1}(\mathbf{C}_{\pm1}(\xi-\beta_0))}  
\end{align}
(except if $\tilde{x}=0$, in which case this remainder is much better behaved in terms of $\xi$ and it is not necessary to split this again). Note that when $\tilde{x}$ is non zero, the remainder in the stationary phase estimates has good decay properties in the $t$ variable but not in $\xi$, but this is not of our concern since the extra decomposition involves only a finite sum $\sum_{\vert\xi-\beta_0 \vert\leqslant 2\Lambda}$. Coming back to~\eqref{e:L-express-ouf}, we have obtained
\begin{align}
\label{e:decomp-Lleq}
 \mathscr{L}^{(l,\pm1)}_{\leq}(s)& = 
   \frac{\mathsf{c}_d^\pm}{l!}  \sum_{j=0}^{N-1}
   \mathsf{T}_j(s) +\mathsf{R}_N (s) ,
\end{align}
where
\begin{align*}
\mathsf{T}_j^\pm(s) & = \sum_{\xi\in\IZ^d ,0< |\xi-\beta_0| \leq 2 c_0^{-1} \Lambda} 
\left(\int_{\IR_+}\chi_\infty(t) 
\frac{e^{-st +it\lambda_\pm(\xi)}}{t^{\frac{d-1}{2}+j-l}}  |dt| \right)
\frac{1}{|\xi-\beta_0|^{\frac{d-1}{2}+j}} 
 P_{j,l,\xi}^\pm[\varphi,\psi]\left(\pm \frac{\xi-\beta_0}{|\xi-\beta_0|}\right) , \\
\mathsf{R}_N^\pm (s)&  = \sum_{\xi\in\IZ^d ,0< |\xi-\beta_0| \leq 2  c_0^{-1}\Lambda}
 \int_{\IR_+}\chi_\infty(t) \frac{R_N^\pm [\varphi,\psi] (\xi,t)}{t^{N+\frac{d-1}{2}-l}} e^{-st } |dt| .
\end{align*}
According to Items~\ref{i:dedebile-laplace} and~\ref{i:derivees} of Lemma~\ref{l:f-phi-lam-Laplace} and the uniform in $t$ estimate in~\eqref{e:estim-RN-}, $\mathsf{R}_N^\pm$ extends as a function in $\mathcal{C}^k(\CLam)$ as soon as $N+\frac{d-1}{2}-l -k >1$ (which, recalling $l\leq \min\{k_1,k_2\}$, holds for all $l$ if $N,k$ are such that $N > k+ \min\{k_1,k_2\} +1-\frac{d-1}{2}$), with the estimate for $m\geq 0$
\begin{multline}
\left|\d_s^k \mathsf{R}_N^\pm (s) \right| 
  \leq  C_{N,k} \sum_{\xi\in\IZ^d ,0< |\xi-\beta_0| \leq 2 c_0^{-1}\Lambda}   |\xi-\beta_0|^{N+\frac{d+1}2}    \|B_{\tilde{x},\xi}^{(k_2,l)}(\varphi,\psi)\|_{W^{2N+ d,1}(\mathbf{C}_{\pm1}(\xi))}     \\
   \leq  C_{N,m,k} \Lambda^{2m+N+\frac{d+1}{2}} \sum_{\xi\in\IZ^d\setminus \{\beta_0\}}     
  \frac{\left\|\pi_{\xi}^{(k_1)}(\varphi)\right\|_{H^{2N+d}\left(\mathbf{C}_{\pm1}(\xi-\beta_0) \right)} \left\|\pi_{-\xi}^{(k_2)}(\psi)\right\|_{H^{2N+d}\left(\mathbf{C}_{\pm1}(\xi-\beta_0) \right)}}{|\xi-\beta_0|^{2m}}  .
  \label{e:laplace-estim-RN}
  \end{multline}
Then, we have
\begin{align*}
\mathsf{T}_j^\pm(s) & = \sum_{\xi\in\IZ^d ,0< |\xi-\beta_0| \leq 2 c_0^{-1}\Lambda} 
F_{\chi_\infty ,\frac{d-1}{2}+j-l}\big(s- i\lambda_\pm(\xi)\big)
\frac{1}{|\xi-\beta_0|^{\frac{d-1}{2}+j}} 
 P_{j,l,\xi}^\pm[\varphi,\psi]\left(\pm \frac{\xi-\beta_0}{|\xi-\beta_0|}\right) ,
\end{align*}
where the singularity of $s\mapsto F_{\chi_\infty ,\frac{d-1}{2}+j-l}\big(s- i\lambda_\pm(\xi)\big)$ near $s =\pm i|\xi-\beta_0|$ is described in Items~\ref{i:laplace-sing-gamma}-\ref{i:laplace-sing-log}-\ref{i:laplace-sing-alpha}-\ref{i:laplace-sing-log-n} of Lemma~\ref{l:f-phi-lam-Laplace}, namely:
$$
F_{\chi_\infty ,\frac{d-1}{2}+j-l}(z) = \mathsf{F}_{\frac{d-1}{2}+j-l}(z)   + H_{\frac{d-1}{2}+j-l}(z),  
$$
with $\mathsf{F}_{\alpha}$ defined in~\eqref{e:def-functionsF}, 
and $H_{\alpha}$ are entire functions whose derivatives are uniformly bounded by a constant times $\langle z\rangle^{\frac{d-1}{2}}$ on $\Re(s)\geq -1$. Moreover, the terms $\mathsf{T}_j^\pm(s) $ have a limit 
$\lim_{x\rightarrow 0^+}\mathsf{T}_j^\pm(x+iy)  $ in $\mathcal{S}^\prime$ since the terms $\mathsf{F}_{\alpha}$ have boundary value distributions by Lemma~\ref{l:f-phi-lam-Laplace} defined by~\eqref{e:def-boundaryvaluesF}. This allows us to rewrite $\mathsf{T}_j(s)$ as 
\begin{multline*}
\mathsf{T}_j^\pm(s)  = \sum_{\xi\in\IZ^d ,0< |\xi-\beta_0| \leq 2 c_0^{-1}\Lambda} 
\mathsf{F}_{\frac{d-1}{2}+j-l}\big(s- i\lambda_\pm(\xi)\big)
\frac{1}{|\xi-\beta_0|^{\frac{d-1}{2}+j}} 
 P_{j,l,\xi}^\pm[\varphi,\psi]\left(\pm \frac{\xi-\beta_0}{|\xi-\beta_0|}\right) \\
 + \tilde{R}_j^\pm[\varphi,\psi](s),
\end{multline*}
where  
\begin{align*}
\tilde{R}_j^\pm[\varphi,\psi](s) = 
\sum_{\xi\in\IZ^d ,0< |\xi-\beta_0| \leq 2 c_0^{-1} \Lambda} 
 H_{\frac{d-1}{2}+j-l}\big(s- i\lambda_\pm(\xi)\big)
\frac{1}{|\xi-\beta_0|^{\frac{d-1}{2}+j}} 
 P_{j,l,\xi}^\pm[\varphi,\psi]\left(\pm \frac{\xi-\beta_0}{|\xi-\beta_0|}\right)   .
\end{align*}
The function $\tilde{R}_j^\pm$ is holomorphic on $\Re(s)>-1$. Moreover, recalling~\eqref{e:def-P-j}, the fact that the operators $L_{j, \frac{\xi-\beta_0}{|\xi-\beta_0|}}^\pm$ are of order $2j$, and using a Sobolev embedding, this can be estimated as follows, for $s\in \overline{\IC}_+^\Lambda\cap \{|\text{Re}(s)|\leq 1\}$,
\begin{align*}
\left|\d_s^k \tilde{R}_j^\pm[\varphi,\psi](s) \right|  
& \lesssim_k \langle \Lambda\rangle^{\frac{d-1}{2}}
\sum_{\xi\in\IZ^d ,0< |\xi-\beta_0| \leq 2 c_0^{-1} \Lambda} 
\frac{\left|L_{j, \frac{\xi-\beta_0}{|\xi-\beta_0|}}^\pm (e^{i \xi \cdot \tilde{x}(\cdot)} B_{\tilde{x},\xi}^{(k_2,l)}(\varphi,\psi)) \left(\pm \frac{\xi-\beta_0}{|\xi-\beta_0|}\right)  \right|}{|\xi-\beta_0|^{\frac{d-1}{2}+j}} 
 \\
& \lesssim_k \langle \Lambda\rangle^{\frac{d-1}{2}}
\sum_{\xi\in\IZ^d ,0< |\xi-\beta_0| \leq 2 c_0^{-1} \Lambda} 
\frac{\left\| e^{i \xi \cdot \tilde{x}(\cdot)} B_{\tilde{x},\xi}^{(k_2,l)}(\varphi,\psi)   \right\|_{W^{2j,\infty}\left(\mathbf{C}_{\pm1}(\xi-\beta_0) \right)}}{|\xi-\beta_0|^{\frac{d-1}{2}+j}} 
 \\
& \lesssim_k \langle \Lambda\rangle^{\frac{d-1}{2}}
\sum_{\xi\in\IZ^d ,0< |\xi-\beta_0| \leq 2  c_0^{-1}\Lambda} 
\frac{\langle \xi \rangle^{2j}}{|\xi-\beta_0|^{\frac{d-1}{2}+j}} 
\left\|  B_{\tilde{x},\xi}^{(k_2,l)}(\varphi,\psi)   \right\|_{W^{2j+ d+1,1}\left(\mathbf{C}_{\pm1}(\xi-\beta_0) \right)},  
\end{align*}
 which can be bounded one more time using the Cauchy-Schwarz inequality in terms of the norms of $\varphi$ and $\psi$. This yields actually a better bound than the estimate~\eqref{e:laplace-estim-RN} we already have on $\mathsf{R}_N^\pm$ as $j\leq N-1$.

\bigskip
Coming back to the definition of $\mathscr{L}^{(l)}_{(\varphi,\psi)}$ in~\eqref{e:decmposition-L} and collecting~\eqref{e:value-JlE-laplace}, \eqref{e:laplace-L0}, \eqref{e:Lgeq}, \eqref{e:decomp-Lleq}, \eqref{e:laplace-estim-RN} together with the last three lines, we obtain that if  $N+\frac{d-1}{2}-l -k >1$, $N_0-k -l >1$, and $m \geq 0$ , then the function 
\begin{align*}
\mathscr{R}^{(l,N)}_{(\varphi,\psi)}(s) & : =\mathscr{L}^{(l)}_{(\varphi,\psi)}(s) - 
 \frac{ E^{(l)}_{\beta_0}}{s^{l+1}} \\ 
& \quad - \sum_{\pm}  \sum_{j=0}^{N-1}
 \sum_{\xi\in\IZ^d ,0< |\xi-\beta_0| \leq 2 c_0^{-1} \Lambda}   \frac{1}{l!} 
\frac{\mathsf{F}_{\frac{d-1}{2}+j-l}\big(s-i\lambda_\pm(\xi)\big)}{|\xi-\beta_0|^{\frac{d-1}{2}+j}} 
 P_{j,l,\xi}^\pm[\varphi,\psi]\left(\pm \frac{\xi-\beta_0}{|\xi-\beta_0|}\right) 
\end{align*}
extends as $\mathscr{R}^{(l,N)}_{(\varphi,\psi)} \in \mathcal{C}^k(\CLam)$ with, for $s \in \CLam$ with $\text{Re}(s)\leq 1$,
\begin{align*}
& \left|\d_s^k  \mathscr{R}^{(l,N)}_{(\varphi,\psi)}(s)  \right| 
    \leq  C_{N_0,k}  \sum_{\xi\in\IZ^d \setminus\{\beta_0\}}\frac{\left\|\pi_{\xi}^{(k_1)}(\varphi)\right\|_{H^{N_0}\left(\mathbf{C}_0(\xi-\beta_0) \right)} \left\|\pi_{-\xi}^{(k_2)}(\psi)\right\|_{H^{N_0}\left(\mathbf{C}_0(\xi+\beta_0) \right)} }{|\xi-\beta_0|^{N_0}} 
\\
& + C_{N,m,k} \Lambda^{2m+N+\frac{d+1}{2}} \sum_{\xi\in\IZ^d\setminus \{\beta_0\}}    
  \frac{\left\|\pi_{\xi}^{(k_1)}(\varphi)\right\|_{H^{2N+d}\left(\mathbf{C}_{\pm1}(\xi-\beta_0) \right)} \left\|\pi_{-\xi}^{(k_2)}(\psi)\right\|_{H^{2N+d}\left(\mathbf{C}_{\pm1}(\xi+\beta_0) \right)}}{|\xi-\beta_0|^{2m} } \\
  & \leq C_{N_0,N,m,k} \Lambda^{2m+N+\frac{d+1}{2}} \|\varphi\|_{\mathcal{H}^{N_0,-N_0/2,2N+d,-m}_{k_1,\beta_0}}\|\varphi\|_{\mathcal{H}^{N_0,-N_0/2,2N+d,-m}_{k_2,-\beta_0}} ,
\end{align*}
where we have used the notation of Definition~\ref{def-ani-spaces}. When collecting all terms in~\eqref{e:decomp-again-again-laplace}--\eqref{e:def-J-l-again-laplace} this includes the statement of the theorem as $\mathscr{R}^{(N,\Lambda)}_{(\varphi,\psi)}(s) =  \sum_{l=0}^{\min\{k_1,k_2\}}\mathscr{R}^{(l,N)}_{(\varphi,\psi)}(s)$.
\end{proof}

%%%%%%%%%%%%%%%%%%%%%%%%%%%%%%%%%%%%%%%%%%%%%%%%
\subsection{Terms near zero in the Mellin and Laplace transforms}
\label{s:terms-near-zero}
To conclude the proofs of Theorems~\ref{t:maintheo-mellin-function} and~\ref{t:maintheo-resolvent-function} (together with their generalization to the case of forms/currents) it remains to describe the properties of the part of the integrals involving $\chi_0=1-\chi_\infty$:
\begin{prop}
\label{p:terms-near-zero}
Let $T_1\geq 1$, assume that $\chi_0\in \mathcal{C}^\infty_c(\R)$ has $\operatorname{supp}(\chi_0)\subset[-T_1,T_1]$, and let $\sigma ,N\in \R$. Consider the operator functions
$$A_{\mathscr{M}} : s \mapsto  \int_{1}^\infty t^{-s}\chi_0(t)e^{-t\mathbf{V}_{\beta_0}}|dt|\quad\text{and} \quad A_{\mathscr{L}} : s \mapsto \int_{0}^\infty e^{-st}\chi_0(t)e^{-t\mathbf{V}_{\beta_0}}|dt|  .$$
Then, there exists $C>0$ such that for all $(k_1,k_2)$ with $k_1+k_2=2d-1$, for every $(\varphi, \psi) \in \mathcal{H}^{\sigma,N,\sigma,N}_{k_1,\beta_0} \times \mathcal{H}^{-\sigma,-N,-\sigma,-N}_{k_2,-\beta_0}$,
$$
A_{\mathscr{M}, \varphi, \psi}  : s \mapsto   \int_{S\IT^d}\varphi\wedge A_{\mathscr{M}}(s)  \mathbf{T}_{-\tilde{x}}^*(\psi), 
\quad 
A_{\mathscr{L}, \varphi, \psi}  : s \mapsto   \int_{S\IT^d}\varphi\wedge A_{\mathscr{L}}(s)  \mathbf{T}_{-\tilde{x}}^*(\psi)
$$
are entire functions satisfying for all $s \in \C$, 
\begin{align*}
|A_{\mathscr{M}, \varphi, \psi} (s)| &  \leq C \frac{T_1^{-\Re(s)+d}+1}{\langle \Re(s) \rangle}  \|\varphi \|_{\mathcal{H}^{\sigma,N,\sigma,N}_{k_1,\beta_0}} \| \psi \|_{\mathcal{H}^{-\sigma,-N,-\sigma,-N}_{k_2,-\beta_0}} , \\ 
|A_{\mathscr{L}, \varphi, \psi} (s)| &  \leq C \frac{e^{-T_1\Re(s)}+1}{\langle \Re(s) \rangle} \|\varphi \|_{\mathcal{H}^{\sigma,N,\sigma,N}_{k_1,\beta_0}} \| \psi \|_{\mathcal{H}^{-\sigma,-N,-\sigma,-N}_{k_2,-\beta_0}} .
\end{align*}
\end{prop}
\begin{proof}
We prove the result for $A_{\mathscr{M}, \varphi, \psi} (s)$; the proof for $A_{\mathscr{L}, \varphi, \psi} (s)$ being the same. 
We start with the same decomposition as in~\eqref{e:decomp-again-again}-\eqref{e:def-J-l-again}, namely
\begin{align*}
A_{\mathscr{M}, \varphi, \psi} (s) &  =  \sum_{l=0}^{\min\{k_1,k_2\}} A_{\mathscr{M}, \varphi, \psi}^{(l)}(s) , \quad \text{ with }\\
A_{\mathscr{M}, \varphi, \psi}^{(l)}(s) &=   \int_{1}^\infty\chi_0 (t) t^{-s} \Cor_{\varphi,\mathbf{T}_{-\tilde{x}}^*(\psi)}^{l}(t,\beta_0) |dt|  \\
& = \int_{1}^\infty\chi_0 (t) t^{-s} 
  \frac{t^l}{l!}\sum_{\xi\in\IZ^d}\int_{\IS^{d-1}}e^{it(\xi-\beta_0)\cdot\mathbf{v}(\theta)}e^{i\xi\cdot\tilde{x}(\theta)}B_{\tilde{x},\xi}^{(k_2,l)}(\varphi,\psi)(\theta) d\Vol(\theta) |dt|, 
\end{align*}
which is an entire function as $\chi_0$ is compactly supported in $\IR$. We also have the rough bound 
\begin{align*}
\left| A_{\mathscr{M}, \varphi, \psi}^{(l)}(s) \right|  
& \leq C_l  \int_{1}^\infty| \chi_0 (t)| t^{-\Re(s)+l} \sum_{\xi\in\IZ^d}\int_{\IS^{d-1}} \left| B_{\tilde{x},\xi}^{(k_2,l)}(\varphi,\psi)(\theta) \right|  d\Vol(\theta) |dt|.
\end{align*}
The conclusion of the proposition follows
\begin{align*}
\left| A_{\mathscr{M}, \varphi, \psi}^{(l)}(s) \right|  
& \leq C \frac{T_1^{-\Re(s)+l+1}+1}{\langle \Re(s) \rangle} \sum_{\xi\in\IZ^d} \langle\xi \rangle^{N} \left\|\pi_{\xi}^{(k_1)}(\varphi)\right\|_{H^\sigma\left(\IS^{d-1}\right)} \langle\xi \rangle^{-N}\left\|\pi_{-\xi}^{(k_2)}(\psi)\right\|_{H^{-\sigma}\left(\IS^{d-1}\right)} ,
\end{align*}
where we have used~\eqref{e:control-Cm-norm-bis}, and the Cauchy-Schwarz inequality.
\end{proof}

%%%%%%%%%%%%%%%%%%%%%%%%%%%%%%%%%%%%%%%%%%%%%%%%
\section{Proofs of Theorems~\ref{t:maintheo-Mellin},~\ref{t:maintheo-laplace},~\ref{e:maintheo-realaxis} and~\ref{e:GuinandWeil}}\label{s:poincare}

Now that we have given a precise description of the analytical properties of our vector field, we are in position to derive, essentially as corollaries of this sharp analysis and of Corollary~\ref{c:counting-current}, the expected properties of generalized Epstein zeta functions and Poincar\'e series, as well as some asymptotics of counting functions. In this section, we take a general $\mathbf{v}(\theta)$ as defined in \S\ref{s:background-convex} and we thus prove (and state) the main Theorems from the introduction at this level of generality.

\subsection{Asymptotic of the counting function}\label{ss:couning} As a first application of this construction, we will refine the a priori bounds obtained in Lemma~\ref{l:apriorigrowth} and prove formula~\eqref{e:asymptotic-counting} from the introduction. Namely, we fix two admissible submanifolds $\Sigma_1,\Sigma_2\subset\IT^d$ and $\sigma_1,\sigma_2\in\{\pm\}$. We want to compute a precise asymptotic formula for
$$\sum_{T_0\leq t\leq T} m_{\Sigma_1,\Sigma_2}(t)=\sum_{T_0\leq t\leq T}  \sharp \, \mathcal{E}_t(\Sigma_1,\Sigma_2) , \quad \text{with} \quad  \mathcal{E}_t(\Sigma_1,\Sigma_2)=  N_{\sigma_1}(\Sigma_1)\cap e^{tV}(N_{\sigma_2}(\Sigma_2)),$$
where we refer to Section~\ref{ss:relation-resolvent-series} for the notation.
\begin{theo}\label{t:asymptotic-counting} Let $\Sigma_1,\Sigma_2\subset\IT^d$ be two admissible submanifolds, let $\sigma_1,\sigma_2\in\{\pm\}$. Then, there exists $T_0>0$ such that, as $T\rightarrow +\infty$
$$\sum_{T_0\leq t\leq T} \sharp \, \mathcal{E}_t(\Sigma_1,\Sigma_2)=\frac{\operatorname{Vol}_{\IR^d}(K)T^d}{(2\pi)^{d}}+\mathcal{O}(T^{d-1}).$$
\end{theo}
Even if it may not be the simplest manner to prove such a result, this discussion illustrates how our current-theoretical approach to this problem can be implemented. An interesting question would be to understand how the remainder term depends on $\Sigma_1$ and $\Sigma_2$. See~\cite{vanderCorput20, Hlawka50, Herz62b, Randol66, ColindeVerdiere77} for such results. We do not pursue this here and we rather focus on the application of this approach to zeta functions associated with our families of orthogeodesics.

\begin{proof}
First fix $T_0>0$ large enough so that all the above lemmas apply.
 In order to study this quantity, we take $\beta=0$ in~\eqref{e:main-formula} and we choose appropriate cutoff functions $\chi$ approximating the characteristic function of the interval $[T_0,T]$. More precisely, we fix $T>0$ (large enough) and $t_0>0$ (small enough). We define two smooth cutoff functions $\chi^{\pm}_{T}\in\mathcal{C}^\infty_c(\IR,[0,1])$ with the following properties:
\begin{itemize}
 \item $\chi^{+}_{T}$ is equal to $1$ on $[T_0,T]$, it is compactly supported on $(T_0-t_0,T+t_0)$, it is nonincreasing on $[T,T+t_0]$.
 \item $\chi^{-}_{T}$ is equal to $1$ on $[T_0+t_0,T-t_0]$, it is compactly supported on $(T_0,T)$, it is nonincreasing on $[T-t_0,T]$.
\end{itemize}
With such functions at hand, one has
\begin{equation}\label{e:sandwich}\sum_{t>T_0-t_0}\chi^{-}_{T}(t)m_{\Sigma_1,\Sigma_2}(t)\leq \sum_{T_0\leq t\leq T} m_{\Sigma_1,\Sigma_2}(t)\leq \sum_{t>T_0-t_0}\chi^{+}_{T}(t)m_{\Sigma_1,\Sigma_2}(t).\end{equation}
Hence, thanks to~\eqref{e:main-formula}, we have to study
$$
\sum_{t>T_0-t_0}\chi^\pm_{T}(t)m_{\Sigma_1,\Sigma_2}(t) =(-1)^{d-1}\int_{S\IT^d}\delta_{2\pi\IZ^d}\wedge \int_{\IR}\chi^{\pm}_{T}(t)\left(e^{-t\mathbf{V}}\mathbf{T}_{\tilde{x}_{1}^{\sigma_1}-\tilde{x}_{2}^{\sigma_2}}^*\right)\iota_V(\delta_{2\pi\IZ^d})|dt|,$$
which can be analyzed using Theorem~\ref{th:correlation-cut-off}. Indeed, the cutoff functions $\chi^{\pm}_{T}$ are compactly supported which implies that they satisfy the assumption of this Theorem (as they are $(N,p)$-admissible for every $(N,p)$). We just have to pay some attention to their dependence in $T$ as $T\rightarrow+\infty$. More precisely, we need to apply this theorem with $k_1=d$, $k_2=d-1$ and
$$\varphi=\delta_{2\pi\IZ^d},\quad\psi=\iota_V(\delta_{2\pi\IZ^d}).$$
We have 
\begin{align}
\label{e:sobo-space-dirac}
\varphi  \in \mathcal{H}_{d,0}^{+\infty, -\frac{d}{2}-, +\infty, -\frac{d}{2}-} :=  \bigcap_{\eps \in (0,1]} \mathcal{H}_{d,0}^{\frac{1}{\eps}, -\frac{d}{2}-\eps,\frac{1}{\eps}, -\frac{d}{2}-\eps}
, \quad \text{and}\quad   \psi  \in \mathcal{H}_{d-1,0}^{+\infty, -\frac{d}{2}-, +\infty, -\frac{d}{2}-} .
\end{align}
Thus, we are left with analyzing the size of the different integrals involving $\chi^{\pm}_{T}$. First, for $0\leq l\leq d-1$, one has
$$\frac{T^{l+1}}{l+1}+\mathcal{O}(1)\leq \int_{\IR}\chi_{T}^+(t)t^l|dt|\leq\frac{(T+t_0)^{l+1}}{l+1}+\mathcal{O}(1), \quad \text{ as } T \to +\infty , $$
and
$$\frac{(T-t_0)^{l+1}}{l+1}+\mathcal{O}(1)\leq \int_{\IR}\chi_{T}^-(t)t^l|dt|\leq\frac{T^{l+1}}{l+1}+\mathcal{O}(1),\quad \text{ as } T \to +\infty . $$
%where the constants in the remainders are independent of $T$. 
Similarly, we can treat the remainder terms involving terms of the type $\|(\chi^{\pm}_{T,\delta}t^{l})^{(2d+1)}\|_{L^1(\IR_+)}=\mathcal{O}(T^{d-1})$. Finally, the term $\|\chi^{\pm}_{T,\delta}t^{l-(2d+1)}\|_{L^1(\IR_+)}$ is bounded uniformly in terms of $T$ since $l\leq d-1$. Gathering these bounds, we find that
\begin{multline*} \int_{S\IT^d}\delta_{2\pi\IZ^d}\wedge \int_{\IR}\chi^{\pm}_{T}(t)\left(e^{-t\mathbf{V}}\mathbf{T}_{\tilde{x}_{1}^{\sigma_1}-\tilde{x}_{2}^{\sigma_2}}^*\right)\iota_V(\delta_{2\pi\IZ^d})|dt|\\
 =\frac{T^d}{d!}\int_{\IS^{d-1}}B_{\tilde{x}_2^{\sigma_2}-\tilde{x}_1^{\sigma_1},0}^{(d-1,d-1)}(\delta_{2\pi\IZ^d},\iota_V(\delta_{2\pi\IZ^d})) d\Vol(\theta)+ \mathcal{O}(T^{d-1}).
 \end{multline*}
 Recall now that $B_{\tilde{x}_2^{\sigma_2}-\tilde{x}_1^{\sigma_1},0}^{(d-1,d-1)}(\delta_{2\pi\IZ^d},\iota_V(\delta_{2\pi\IZ^d})) d\Vol(\theta)$ is defined in~\eqref{e:coeff-horrible} as
 \begin{multline*}
 B_{\tilde{x}_2^{\sigma_2}-\tilde{x}_1^{\sigma_1},0}^{(d-1,d-1)}(\delta_{2\pi\IZ^d},\iota_V(\delta_{2\pi\IZ^d}))   d\Vol(\theta) \\
  =(-1)^{d-1}\pi_{0}^{(d)}(\delta_{2\pi\IZ^d})\wedge \mathbf{T}_{\tilde{x}_1^{\sigma_1}-\tilde{x}_2^{\sigma_2}}^*\mathbf{V}^{d-1}\pi_{0}^{(d-1)}(\iota_V(\delta_{2\pi\IZ^d})) \\
  =  (-1)^{d-1}(2\pi)^{-2d} dx_1\wedge\ldots dx_d \wedge \mathbf{T}_{\tilde{x}_1^{\sigma_1}-\tilde{x}_2^{\sigma_2}}^*\mathbf{V}^{d-1}\pi_{0}^{(d-1)}(\iota_V(dx_1\wedge\ldots dx_d) ,
 \end{multline*}
 after having used~\eqref{e:Dirac-Fourier}--\eqref{e:Dirac-Fourier-bis}.
Inserting this expression in the above calculation and recalling Lemma~\ref{l:volume-convex} allows to conclude the proof. Technically speaking, in Lemma~\ref{l:volume-convex}, there is no map $\mathbf{T}_{\tilde{x}_1^{\sigma_1}-\tilde{x}_2^{\sigma_2}}^*$ but, inserting it in the proof of this lemma, we find that, as $t\rightarrow +\infty$,
$$\int_{S\IR^d}[S_0\IR^d]\wedge e^{tV*}P^*(x_1dx_2\wedge\ldots\wedge dx_d)\sim\int_{S\IR^d}[S_0\IR^d]\wedge \mathbf{T}_{\tilde{x}_1^{\sigma_1}-\tilde{x}_2^{\sigma_2}}^*e^{tV*}P^*(x_1dx_2\wedge\ldots\wedge dx_d).$$
Hence the leading term in $t^d$ is the same in both formula. As the leading term on the left-hand side computes the volume of $K$ according to the proof of Lemma~\ref{l:volume-convex}, we are done.
%that $\Vol(\IS^{d-1}) = \frac{2\pi^{d/2}}{\Gamma(d/2)} =  \frac{d \pi^{d/2}}{\Gamma(d/2+1)} $ concludes the proof of the theorem.
\end{proof}

%%%%%%%%%%%%%%%%%%%%%%%%%%%%%%%%%%%%%%%%
\subsection{Continuation of generalized Epstein zeta functions}\label{ss:epstein}

We now aim at proving Theorem~\ref{t:maintheo-Mellin} and its generalization in the case of a general $\mathbf{v}(\theta)$. This amounts to studying the meromorphic properties of the \emph{generalized Epstein zeta function} defined in~\eqref{e:def-zeta-general}. As an application of Lemma~\ref{l:apriorigrowth}, it defines a holomorphic function on $\{\text{Re}(s)>d\}$ and we want to understand its possible extension beyond the threshold $\text{Re}(s)>d$. To that aim, the first step is to use Lemma~\ref{l:truncated-poincare} to interpret $\zeta_{\Sigma_1,\Sigma_2,T_0}(s)$ as the Mellin transform of a correlation function of appropriate currents, and then make use of Theorem~\ref{t:general-mellin}. 

\begin{lemm}\label{l:dang-riviere-mellin}
There is $T_0^*>0$ such that for all $T_0>T_0^*$, one can find $\chi_\infty$ verifying assumption~\eqref{e:cutoff-infini} with $t_0>0$ small enough such that, for $\operatorname{Re}(s)$ large enough,
\begin{align}
\label{e:zeta=M}
\zeta_\beta(K_2,K_1,s) & =(-1)^{d-1}\int_{S\IT^d}e^{ i\left(f(\tilde{x}_{2}^{\sigma_2})-f(\tilde{x}_{1}^{\sigma_1})\right)}\delta_{2\pi\IZ^d}\wedge\widehat{\chi^M_s}\left(-i\mathbf{V}_{\beta_0}\right)\mathbf{T}_{\tilde{x}_{1}^{\sigma_1}-\tilde{x}_{2}^{\sigma_2}}^*\iota_V(\delta_{2\pi\IZ^d})|dt|,
\end{align}
where $\widehat{\chi_s^M}(-i\mathbf{V}_{\beta_0})$ was defined in~\eqref{e:cut-mellin-laplace}.
\end{lemm}
 In a more compact manner and using the conventions of Theorem~\ref{t:general-mellin} with $\tilde{x} = \tilde{x}_{2}^{\sigma_2} - \tilde{x}_{1}^{\sigma_1}$, the right hand-side of~\eqref{e:zeta=M} can be rewritten as
\begin{multline*}
 \mathscr{M}_{\big(e^{ i(f(\tilde{x}_{2}^{\sigma_2})-f(\tilde{x}_{1}^{\sigma_1}))}\delta_{2\pi\IZ^d} , \iota_V(\delta_{2\pi\IZ^d}) \big)}(s)\\
 =\int_{S\IT^d}e^{ i\left(f(\tilde{x}_{2}^{\sigma_2})-f(\tilde{x}_{1}^{\sigma_1})\right)}\delta_{2\pi\IZ^d}\wedge\widehat{\chi^M_s}\left(-i\mathbf{V}_{\beta_0}\right)\mathbf{T}_{\tilde{x}}^*\iota_V(\delta_{2\pi\IZ^d})|dt|.
\end{multline*}
The difficulty in making the connection with $\zeta_\beta(K_2,K_1,s)$ is that Lemma~\ref{l:truncated-poincare} is only valid for \emph{compactly} supported function while the function $\chi_{\infty}$ used to define $\widehat{\chi^M_s}\left(-i\mathbf{V}_{\beta_0}\right)$ has noncompact support. This lemma shows that we can in fact allow such a test function in Lemma~\ref{l:truncated-poincare} in view of connecting the zeta function with its integral representation.

\begin{proof}
First, we fix a smooth nondecreasing function $\chi_{\infty}$ which is equal to $1$ on $[T_0+t_0,\infty)$ and to $0$ on $(-\infty, T_0]$ for some small enough $t_0>0$ to ensure that $m_{\Sigma_1,\Sigma_2}(t)=0$ for all $t\in (T_0,T_0+t_0]$. Here $T_0>0$ is large enough to apply Corollary~\ref{c:counting-current}. We also fix a smooth function $\chi\in\mathcal{C}^{\infty}_c((-2,2),[0,1])$ such that
$$\forall t\in\mathbb{R},\quad\sum_{j\in\IZ}\chi(t+j)=1.$$
We let $\chi_{\infty,j}(t):=\chi_\infty(t)\chi(t+j)$. Using~\eqref{e:main-formula}, this leads to the following decomposition
\begin{multline}
 \zeta_\beta(K_2,K_1,s)
=\sum_{j\in\IZ}\sum_{t>T_0}\chi_{\infty,j}(t)t^{-s}\left(\sum_{(x,\theta)\in \mathcal{E}_t(\Sigma_1,\Sigma_2)}e^{-i\int_{-t}^0\beta(V)(x+\tau\theta,\theta)|d\tau|}\right)\\
 =(-1)^{d-1}\sum_{j\in\IZ}\int_{S\IT^d}e^{ i\left(f(\tilde{x}_{2}^{\sigma_2})-f(\tilde{x}_{1}^{\sigma_1})\right)}\delta_{2\pi\IZ^d}\wedge \int_{\IR}\chi_{\infty,j}(t)t^{-s}\left(e^{-t\mathbf{V}_{\beta_0}}\mathbf{T}_{\tilde{x}}^*\right)\iota_V(\delta_{2\pi\IZ^d})|dt|.
\end{multline}
These sums over $j$ are absolutely convergent and the only remaining difficulty is to check that the righthand-side indeed converges to the operator $\widehat{\chi^M_s}\left(-i\mathbf{V}_{\beta_0}\right)$ in the appropriate functional spaces. We now fix some $N$ large enough in order to apply Theorem~\ref{th:correlation-cut-off} (with $N=M$) to $\tilde{x}=\tilde{x}_2-\tilde{x}_1$, to  $k_1=d$, $k_2=d-1$,  and to the currents
$$\varphi=e^{ i\left(f(\tilde{x}_{2}^{\sigma_2})-f(\tilde{x}_{1}^{\sigma_1})\right)}\delta_{2\pi\IZ^d}\in\mathcal{D}^{\prime d}(S\IT^d),\quad\text{and}\quad\psi=\iota_V(\delta_{2\pi\IZ^d})\in\mathcal{D}^{\prime d-1}(S\IT^d).$$
In order to apply this theorem, we have (in practice) to split the sum over $j$ into a finite sum and an infinite one corresponding to the cutoff functions $\chi_{\infty}(t)\chi(t+j)$ in the range of application of this statement. We also have to control the growth of several integrals, namely for $0\leq l\leq d-1$ and for $\text{Re}(s)$ large enough,
$$\int_{\IR}\chi_{\infty}(t)\chi(t+j)t^{-s+l}|dt|=\mathcal{O}(j^{-s+l}),\ \int_{\IR}\chi_{\infty}(t)\chi(t+j)t^{-s+l-N}|dt|=\mathcal{O}(j^{-s+l-N}),$$
and
$$\int_{\IR}\left|\frac{d^N}{dt^N}\left(\chi_{\infty}(t)\chi(t+j)t^{-s+l}\right)\right||dt|=\mathcal{O}(j^{-s+l}).$$
These bounds allow to apply Theorem~\ref{th:correlation-cut-off} for $\text{Re}(s)$ large enough and thus the sums under consideration converge in the anisotropic Sobolev spaces of Section~\ref{s:cutoff} as long as $N$ is large enough to have $\varphi$ and $\psi$ in that space. 
\end{proof}
The next step is to make use of Theorem~\ref{t:general-mellin} to deduce the meromorphic continuation.
\begin{theo}
\label{l:prolongement-zeta}
If $\beta_0\notin\IZ^d$, the function $\zeta_\beta(K_2,K_1,s)$ extends holomorphically to the whole complex plane. 
If $\beta_0\in\IZ^d$, the function extends meromorphically to the whole complex plane with (at most) simple poles at $s=1,\ldots,d$ whose residues are given by
\begin{equation}\label{e:residues1}
\Res_{s=\ell}(\zeta_\beta(K_2,K_1,s)) = \frac{(-1)^{d-1}}{(\ell-1)!}  E^{(\ell-1)}_{\beta_0} , \quad \text{ for }\quad \ell \in \{1, \dots , d\} ,
\end{equation}
with 
\begin{equation}\label{e:residues2}
 E^{(\ell-1)}_{\beta_0} = \frac{(-1)^{d+\ell}}{(2\pi)^{2d}}\int_{S\IT^{d}}e^{i\int_{\tilde{x}_1^{\sigma_1}(\theta)\rightarrow\tilde{x}_2^{\sigma_2}(\theta)}\beta}dx_1\wedge\ldots \wedge dx_d\wedge\iota_V\mathbf{V}^{\ell-1}\mathbf{T}_{\tilde{x}_{1}^{\sigma_1}-\tilde{x}_{2}^{\sigma_2}}^*\left(dx_1\wedge\ldots \wedge dx_d\right).
\end{equation}
\end{theo}
Note that~\eqref{e:residues2} corresponds to~\eqref{def-El} specified to the currents associated with convex sets, see~\eqref{e:sobo-space-dirac-bis} below (but we kept the same notation $E^{(\ell-1)}_{\beta_0}$).
Recall that, for $\beta_0\notin\IZ^d$, we used the convention that $E_0^{(\ell)}=0$ for every $1\leq\ell \leq d$.

This lemma proves in particular Theorem~\ref{t:maintheo-Mellin} as a particular case (see Remark~\ref{r:zeta-zeta-Z-Z}). 
In order to prove Theorem~\ref{t:maintheo-residues}, we are left with giving an expression of these residues in terms of geometric quantities associated with our convex subsets. This will be the topic of Section~\ref{s:convex}.

\begin{rema}\label{r:top-residue} Before getting to this, let us already observe that, for $\beta_0\in\IZ^d$, the residue at $s=d$ can be written explicitly as
$$\frac{1}{(2\pi)^{2d}}\int_{S\IT^d}e^{i\int_{\tilde{x}_1^{\sigma_1}(\theta)\rightarrow\tilde{x}_2^{\sigma_2}(\theta)}\beta}[S_{[0]}\IT^d]\wedge\mathbf{V}^{d-1}\iota_V\left(dx_1\wedge\ldots \wedge dx_d\right),$$
which looks like the quantity appearing in Lemma~\ref{l:volume-convex}. Moreover, if $\tilde{x}_{1}^{\sigma_1}=\tilde{x}_{2}^{\sigma_2}$ or if $\Sigma_1$ and $\Sigma_2$ are both points, all these residues vanish except for the one at $s=d$ which is equal to
$$ \frac{e^{i\int_{\tilde{x}_1^{\sigma_1}\rightarrow\tilde{x}_2^{\sigma_2}}\beta}\Vol_{\IR^d}(K)}{(2\pi)^{d}}.$$
In that case and when $\mathbf{v}(\theta)=\theta$, we recall that the lengths of the geodesic arcs joining $\Sigma_1$ to $\Sigma_2$ are given explicitly by $(|2\pi\xi+\tilde{x}_2^{\sigma_2}-\tilde{x}_1^{\sigma_1}|)_{\xi\in\IZ^d}$ so that we end up with the classical Epstein zeta function~\cite{Epstein03}. Hence, we see that this property of having a single pole remains true for a general $\mathbf{v}(\theta)$ which corresponds to the case where one looks at dilations of a general strictly convex set.
\end{rema}

\begin{proof}[Proof of Theorem~\ref{l:prolongement-zeta}]
Given~\eqref{e:zeta=M}, we are now in position to apply Theorem~\ref{t:general-mellin} to 
\begin{align}
\label{e:sobo-space-dirac-bis}
\varphi=e^{ i\left(f(\tilde{x}_{2}^{\sigma_2})-f(\tilde{x}_{1}^{\sigma_1})\right)}\delta_{2\pi\IZ^d}  \in \mathcal{H}_{d,\beta_0}^{+\infty, -\frac{d}{2}-, +\infty, -\frac{d}{2}-} ,
 \quad \text{and}\quad   \psi=\iota_V(\delta_{2\pi\IZ^d})  \in \mathcal{H}_{d-1,-\beta_0}^{+\infty, -\frac{d}{2}-, +\infty, -\frac{d}{2}-} ,
\end{align}
where the notation is taken from~\eqref{e:sobo-space-dirac}.
Up to increasing slightly the value of $T_0$ to be in the setup of Theorem~\ref{t:general-mellin}, this result can be applied to these test functions if we pick $N> d/2$. 
This theorem implies that for $\beta_0\notin\IZ^d$, the function $\zeta_{\Sigma_1,\Sigma_2,T_0}(s)$ extends holomorphically to the whole complex plane (as $E_{\beta_0}^{(l)}=0$ in that case). When $\beta_0\in\IZ^d$, Theorem~\ref{t:general-mellin} implies that the function extends meromorphically to the whole complex plane with possibly some simple poles at $s=1,\ldots,d$. Moreover, together with~\eqref{e:zeta=M}, we deduce that if $\beta_0\in\IZ^d$, then
$$
\Res_{s=l+1}(\zeta_\beta(K_2,K_1,s)) = \frac{(-1)^{d-1}}{l!}  E^{(l)}_{\beta_0} , \text{ for }\quad l \in \{0, \dots , d-1\} ,
$$
with 
 $$
 E^{(l)}_{\beta_0}  =   \int_{\IS^{d-1}} e^{i\beta_0\cdot (\tilde{x}_{2}^{\sigma_2}(\theta) - \tilde{x}_{1}^{\sigma_1}(\theta))} B_{\tilde{x}_{2}^{\sigma_2} - \tilde{x}_{1}^{\sigma_1},\beta_0}^{(d-1,l)}\big(e^{ i(f(\tilde{x}_{2}^{\sigma_2})-f(\tilde{x}_{1}^{\sigma_1}))}\delta_{2\pi\IZ^d} , \iota_V(\delta_{2\pi\IZ^d}) \big)(\theta) d\Vol(\theta),
 $$
 where the bilinear operator $B$ is defined in~\eqref{e:coeff-horrible}. 
From the expression of $B$, one can verify that $e^{ i\left(f(\tilde{x}_{2}^{\sigma_2})-f(\tilde{x}_{1}^{\sigma_1})\right)}$ can be put in factor so that, in the resulting exponential, we obtain a term
$$\beta_0\cdot (\tilde{x}_2^{\sigma_2}(\theta)-\tilde{x}_1^{\sigma_1}(\theta))+f(\tilde{x}_{2}^{\sigma_2})-f(\tilde{x}_{1}^{\sigma_1})=\int_{\tilde{x}_{1}^{\sigma_1}(\theta)\rightarrow \tilde{x}_{2}^{\sigma_2}(\theta)}\beta,$$
which is independent of the choice of the path between $\tilde{x}_{1}^{\sigma_1}(\theta)$ and $\tilde{x}_{2}^{\sigma_2}(\theta)$ modulo $2\pi\IZ$. Therefore, the residue at $s=\ell :=l+1 \in \{1,\dots, d\}$ is given by
$$
 E^{(\ell-1)}_{\beta_0}  = \int_{\IS^{d-1}} e^{i\int_{\tilde{x}_1^{\sigma_1}(\theta)\rightarrow\tilde{x}_2^{\sigma_2}(\theta)}\beta}B_{\tilde{x}_2^{\sigma_2}-\tilde{x}_1^{\sigma_1},\beta_0}^{(d-1,\ell-1)}\left(\delta_{2\pi\IZ^d},\iota_V(\delta_{2\pi\IZ^d})\right)(\theta) d\Vol(\theta).$$
Using expression~\eqref{e:coeff-horrible} together with~\eqref{e:commute-shift-lie}  in Lemma~\ref{l:commute-pullback}, we finally obtain 
\begin{equation}\label{e:residues}
 E^{(\ell-1)}_{\beta_0} = 
 \frac{(-1)^{\ell-1}}{(2\pi)^{d}}\int_{S\IT^{d}}e^{i\int_{\tilde{x}_1^{\sigma_1}(\theta)\rightarrow\tilde{x}_2^{\sigma_2}(\theta)}\beta}dx_1\wedge\ldots \wedge dx_d\wedge\iota_V\mathbf{V}^{\ell-1}\mathbf{T}_{\tilde{x}_{1}^{\sigma_1}-\tilde{x}_{2}^{\sigma_2}}^*\left(\delta_{2\pi\IZ^d}\right).
\end{equation}
 for $1\leq \ell\leq d$. The latter can be rewriten as~\eqref{e:residues2} when recalling~\eqref{e:Dirac-Fourier}--\eqref{e:Dirac-Fourier-bis}.
\end{proof}

\subsection{Continuation of generalized Poincar\'e series}\label{ss:poincare} We now turn to the proof of Theorem~\ref{t:maintheo-laplace} which amounts to study the properties of the {\em generalized Poincar\'e series} defined in~\eqref{e:def-Z-general}. Arguing as in Lemma~\ref{l:dang-riviere-mellin} with $e^{-st}$ instead of $t^{-s}$, we can make use of Theorem~\ref{th:correlation-cut-off} and Lemma~\ref{l:truncated-poincare} to interpret $\mathcal{Z}_\beta(K_2,K_1,s)$ as the Laplace transform of a correlation function of appropriate currents.
\begin{lemm}
There is $T_0^*>0$ such that for all $T_0>T_0^*$, one can find $\chi_\infty$ verifying assumption~\eqref{e:cutoff-infini} with $t_0>0$ small enough such that
\begin{align}
\label{e:Z=L}
\mathcal{Z}_\beta(K_2,K_1,s) & =(-1)^{d-1}\int_{S\IT^d}e^{ i\left(f(\tilde{x}_{2}^{\sigma_2})-f(\tilde{x}_{1}^{\sigma_1})\right)}\delta_{2\pi\IZ^d}\wedge\widehat{\chi^L_s}\left(-i\mathbf{V}_{\beta_0}\right)\mathbf{T}_{\tilde{x}_{1}^{\sigma_1}-\tilde{x}_{2}^{\sigma_2}}^*\iota_V(\delta_{2\pi\IZ^d}) .
\end{align}
\end{lemm}
Again, using the conventions of Theorem~\ref{t:general-laplace} with $\tilde{x} = \tilde{x}_{2}^{\sigma_2} - \tilde{x}_{1}^{\sigma_1}$, the right-hand side can be rewritten as
\begin{multline*}
 \mathscr{L}_{\big(e^{ i(f(\tilde{x}_{2}^{\sigma_2})-f(\tilde{x}_{1}^{\sigma_1}))}\delta_{2\pi\IZ^d} , \iota_V(\delta_{2\pi\IZ^d}) \big)}(s)\\
 =\int_{S\IT^d}e^{ i\left(f(\tilde{x}_{2}^{\sigma_2})-f(\tilde{x}_{1}^{\sigma_1})\right)}\delta_{2\pi\IZ^d}\wedge\widehat{\chi^L_s}\left(-i\mathbf{V}_{\beta_0}\right)\mathbf{T}_{\tilde{x}}^*\iota_V(\delta_{2\pi\IZ^d})
\end{multline*}
We are thus in position to apply Theorem~\ref{t:general-laplace}. 
\begin{theo}\label{l:regularity-poincare}
Setting $S_{\beta_0}= i\Lambda_{\beta_0}\cup\{0\}$ if $\beta_0\in\IZ^d$ and $S_{\beta_0}=i\Lambda_{\beta_0}$ otherwise,
%\begin{align}
%S_{\beta_0} =\{\pm i |\xi - \beta_0| , \quad \xi \in \Z^d \} , 
%\end{align}
the following statements hold:
\begin{enumerate}
\item \label{i:hors-S} $\mathcal{Z}_\beta(K_2,K_1,s)$ extends as a function in $\mathcal{C}^\infty(\overline{\C}_+ \setminus S_{\beta_0})$ and the limit $\lim_{x\rightarrow 0^+}\mathcal{Z}_\beta(K_2,K_1,x+iy)$ exists in $\mathcal{S}^\prime(\mathbb{R})$ as boundary value of holomorphic function,
\item \label{i:near-0} the function
$$
\mathcal{Z}_\beta(K_2,K_1,s)- \sum_{\ell=1}^{d}  \frac{ E^{(\ell-1)}_{\beta_0}}{s^{\ell}}
$$
is a $\mathcal{C}^\infty$ function in a neighborhood of zero in $\overline{\C}_+$ where $E^{(\ell-1)}_{\beta_0}$ is given by~\eqref{e:residues} (recalling that $E^{(\ell-1)}_{\beta_0}=0$ if $\beta_0\notin\IZ^d$).
%The limit $\lim_{x\rightarrow 0^+} \mathcal{Z}_\beta(K_2,K_1,x+iy)- \sum_{\ell=1}^{d}  \frac{ E^{(\ell-1)}_{\beta_0}}{(x+iy)^{\ell}} $ exists as a smooth function near $y=0$.
%In case $\beta_0 \notin \Z^d$, then $\mathcal{Z}_\beta(K_2,K_1,s)$ is itself a $\mathcal{C}^\infty$ function in a neighborhood of zero in $\overline{\C}_+$.
\item \label{i:near-ir0} There exist constants $\mathsf{C}_{j,\ell-1}(\xi,\beta_0)$ for every $\ell \in \{1,\dots, d\}$ and $j \in  \IZ_+$, such that for any $ir_0^\pm =  i\lambda_\pm(\xi) 
\in S_{\beta_0} \setminus \{0\}$, the function 
\begin{align}
\label{e:expansion-singular}
\mathcal{Z}_\beta(K_2,K_1,s)
 -
 \sum_{\ell=1}^{d}
 \sum_{j=0}^{N-1}\left(\sum_{\xi \in \Z^d, \lambda_\pm(\xi) = r_0^\pm}\frac{1}{(\ell-1)!}\mathsf{C}_{j,\ell-1}^\pm\big(\xi,\beta_0\big)\right)\mathsf{F}_{\frac{d-1}{2}+j+1-\ell}\big(s-ir_0^\pm \big) , 
  \end{align}  
  extends as a $\mathcal{C}^{N-1- \left\lceil \frac{d+1}{2}\right\rceil}$ function in a neighborhood of $ir_0^\pm$ in $\overline{\C}_+$. 
  
Moreover, the most singular term in this expansion near $ir_0^\pm$ is given by ($j=0$ and $\ell=d$)
  \begin{multline*}
 \frac{(-1)^{d-1}e^{\mp i\frac{\pi}{4}(d-1)} \mathsf{F}_{-\frac{d-1}{2}}\big(s-ir_0^\pm \big)}{(2\pi)^{\frac{d+1}{2}}|r_0^\pm|^{\frac{d-1}{2}}} \\
 \times\left( \sum_{\xi \in \Z^d, \lambda_\pm(\xi) = r_0^\pm}
  \frac{e^{i \xi \cdot \left( \tilde{x}_{2}^{\sigma_2}-\tilde{x}_{1}^{\sigma_1}\right) \big(\pm \frac{\xi-\beta_0}{|\xi-\beta_0|}\big)}  
e^{ i\left(f(\tilde{x}_{2}^{\sigma_2})-f(\tilde{x}_{1}^{\sigma_1})\right) \big(\pm \frac{\xi-\beta_0}{|\xi-\beta_0|}\big)}}{\sqrt{\kappa\circ\mathbf{v}\left(\pm\frac{\xi-\beta_0}{|\xi-\beta_0|}\right)}}
 \right)   .
\end{multline*}
\end{enumerate}
\end{theo}
Recall that for $\alpha \in \R$, the distribution $\mathsf{F}_{\alpha}$ is defined in~\eqref{e:def-boundaryvaluesF} and it is essentially the Laplace transform of $t^{-\alpha}$ (near $t=+\infty$).
We also recall that $\Lambda_{\beta_0}$ is defined in~\eqref{e:critical-values} and that $S_{\beta_0}\setminus\{0\} = i \left( \Sp(\lambda_+(D)) \cup  \Sp(\lambda_-(D)) \right)$ where $\lambda_\pm(D)$ is the Fourier multiplier of symbol $\lambda_\pm(\xi)$ on $\T^d$ (see Remark~\ref{rem:quantumquantum}).

Note the important fact that the difference 
\begin{align}
\label{e:expansion-singular2}
\lim_{x\rightarrow 0^+}
\mathcal{Z}_\beta(K_2,K_1,x+iy)
 -
 \sum_{\ell=1}^{d}
 \sum_{j=0}^{N-1}\left(\sum_{\xi \in \Z^d, \lambda_\pm(\xi) = r_0^\pm}\frac{1}{(\ell-1)!}\mathsf{C}_{j,\ell-1}^\pm\big(\xi,\beta_0\big)\right)\mathsf{F}_{\frac{d-1}{2}+j+1-\ell}\big(x+iy-ir_0^\pm \big) , 
  \end{align} 
viewed as tempered distributions in $\mathcal{S}^\prime(\mathbb{R})$ of the variable $y$ is an element in  $\mathcal{C}^{N-1- \left\lceil \frac{d+1}{2}\right\rceil}$ near $y=r_0^\pm$. Here we view the difference as a distribution obtained as boundary values of holomorphic functions.

\begin{proof}
We apply Theorem~\ref{t:general-laplace} to $\tilde{x}=\tilde{x}_2^{\sigma_2}-\tilde{x}_1^{\sigma_1}$ and to the currents $\varphi, \psi$ in~\eqref{e:sobo-space-dirac-bis}.
 Theorem~\ref{t:general-laplace} thus applies to all $N_0>d, N >0, m>d/2$ and Item~\ref{i:hors-S} readily follows. 
As for Item~\ref{i:near-0}, Theorem~\ref{t:general-laplace} implies the expected result after recalling the definition of $E_{\beta_0}^{(l)}$.

We next prove Item~\ref{i:near-ir0}. We fix a point $r_0^\pm := \lambda_\pm(\xi)$ such that $ir_0^\pm  =  i\lambda_\pm(\xi) \in S_{\beta_0} \setminus \{0\}$, and describe $\mathcal{Z}_\beta(K_2,K_1,s)$ near $ir_0^\pm$. 
Theorem~\ref{t:general-laplace}, taken for $N_0>d$ large enough (compared to $N$), implies that 
\begin{align*}
\mathcal{Z}_\beta(K_2,K_1,s)
 -
 \sum_{l=0}^{d-1}
 \sum_{j=0}^{N-1}\left(\sum_{\xi \in \Z^d,\lambda_\pm(\xi)= r_0^\pm}\frac{1}{l!}\mathsf{C}_{j,l}^\pm\big( \xi,\beta_0\big)\right)\mathsf{F}_{\frac{d-1}{2}+j-l}\big(s-ir_0^\pm \big) , 
  \end{align*}  
  extends as a $\mathcal{C}^{\mathsf{k}}$ function in a neighborhood of $ir_0^\pm$ in $\overline{\C}_+$, where 
  $$
   \mathsf{k} = N+ \left\lceil \frac{d-1}{2}\right\rceil -\min\{k_1,k_2\}-2  = N-1- \left\lceil \frac{d+1}{2}\right\rceil ,
  $$
  (recall that $\min\{k_1,k_2\}=d-1$ here) and
  \begin{align}
  \label{e:bouffon-1}
 \mathsf{C}_{j,l}^\pm(\xi,\beta_0) = \frac{1}{|\xi_0-\beta_0|^{\frac{d-1}{2}+j}} 
 P_{j,l,\xi}^\pm \left[e^{ i\left(f(\tilde{x}_{2}^{\sigma_2})-f(\tilde{x}_{1}^{\sigma_1})\right)}\delta_{2\pi\IZ^d},\iota_V(\delta_{2\pi\IZ^d}) \right]\left(\pm \frac{\xi-\beta_0}{|\xi-\beta_0|}\right). 
\end{align}
Now, recalling the definition of $\mathsf{F}_{\alpha}$ in~\eqref{e:def-functionsF}, we notice that the most singular term in the expansion~\eqref{e:expansion-singular} is for $l=d-1$ and $j=0$, and is given by 
  \begin{align}
  \label{e:bouffon-2}
  \left( \sum_{\xi \in \Z^d, \lambda_\pm(\xi) = r_0^\pm} 
  \frac{1}{(d-1)!}\mathsf{C}_{0,d-1}^\pm \big(\xi,\beta_0\big) \right)  \mathsf{F}_{-\frac{d-1}{2}}\big(s-ir_0^\pm \big) .
\end{align}
We now compute it explicitly.
We first compute $\mathsf{C}_{0,l}^\pm(\omega)$ according to the definition of $P_{j,l,\xi}^\pm$ in~\eqref{e:def-P-j}, recalling from Lemma~\ref{c:stat-points} that $L_{0, \mp \omega}^\pm=1$, as 
\begin{multline}
  \label{e:bouffon-3}
P_{0,l,\xi}^\pm\left[e^{ i\left(f(\tilde{x}_{2}^{\sigma_2})-f(\tilde{x}_{1}^{\sigma_1})\right)}\delta_{2\pi\IZ^d},\iota_V(\delta_{2\pi\IZ^d}) \right]\left(\pm \omega \right) \\
= e^{\mp i\frac{\pi}{4}(d-1)} (2\pi)^{\frac{d-1}{2}}  e^{i \xi \cdot \tilde{x}(\omega)} B_{\tilde{x},\xi}^{(d-1,l)}\left( e^{ i\left(f(\tilde{x}_{2}^{\sigma_2})-f(\tilde{x}_{1}^{\sigma_1})\right)}\delta_{2\pi\IZ^d},\iota_V(\delta_{2\pi\IZ^d})\right) \left(\pm\omega \right) .
\end{multline}
Combining~\eqref{e:bouffon-1}--\eqref{e:bouffon-2}--\eqref{e:bouffon-3}, this concludes the proof of the statement in Item~\ref{i:near-ir0}.
\end{proof}

%%%%%%%%%%%%%%%%%%%%%%%%%%%%%%%%%%%%%%%%%
\subsection{A summation formula in the spirit of Guinand--Meyer}\label{ss:guinand}
We now turn to the proof of Theorem~\ref{e:GuinandWeil} and we emphasize that it is important here to have $\mathbf{v}(\theta) = \theta$. The proof would work as well for more general $\mathbf{v}$ if we suppose in addition that $\theta\cdot\mathbf{v}(-\theta)=-\theta\cdot\mathbf{v}(\theta)$ for every $\theta\in\IS^{d-1}$ (e.g. if the convex $K$ defining $\mathbf{v}$ is an ellipsoid with $0 \in \Int(K)$). Repeating the arguments in Section~\ref{ss:poincare} for certain variations of Poincar\'e series, we can in fact deduce a summation formula in the spirit of the recent results on crystalline measures by Meyer~\cite{Meyer2016}. More precisely, we set
\begin{align}
\label{e:def-Z-general-guinand}
\tilde{\mathcal{Z}}_\beta(K_2,K_1,s):=\sum_{t>T_0:\mathcal{E}_t(\Sigma_1,\Sigma_2)\neq\emptyset}\frac{e^{-st}}{t^{\frac{d-1}{2}}}\left(\sum_{(x,\theta)\in \mathcal{E}_t(\Sigma_1,\Sigma_2)}e^{-i\int_{-t}^0\beta(V)(x+\tau\theta,\theta)|d\tau|}\right),
\end{align}
and we emphasize that this function depends on the choice of orientation $(\sigma_1,\sigma_2)$ even if we drop this dependence for the moment. As for Poincar\'e series in Section~\ref{ss:poincare}, the limit as $x\rightarrow 0^+$ of $y\mapsto \tilde{\mathcal{Z}}_\beta(K_2,K_1,x+iy)$ exists as a tempered distribution on $\IR$ thanks to Lemma~\ref{l:apriorigrowth}. Arguing as in the proof of Theorem~\ref{l:regularity-poincare}, one can verify that the singular support of this distribution is the same as for $\lim_{x\rightarrow 0^+} \mathcal{Z}_\beta(K_2,K_1,x+iy)$ but the singularity are slightly simpler due to the renormalization factor $t^{-\frac{d-1}{2}}$. More precisely, using the conventions of this theorem, one finds that, near $y=r_0^+$, $\lim_{x\rightarrow 0^+} \tilde{\mathcal{Z}}_\beta(K_2,K_1,x+iy)$ is equal to
$$\lim_{x\rightarrow 0^+} \frac{(-1)^{d-1}e^{-i\frac{\pi}{4}(d-1)} }{(2\pi)^{\frac{d+1}{2}}|r_0^+|^{\frac{d-1}{2}}(x+iy-ir_0^+)} \\
 \times\left( \sum_{\xi \in \Z^d, |\xi-\beta_0| = |r_0^+|}
  e^{i \xi \cdot \left( \tilde{x}_{2}^{\sigma_2}-\tilde{x}_{1}^{\sigma_1}\right) \big( \frac{\xi-\beta_0}{|\xi-\beta_0|}\big)}  
e^{ i\left(f(\tilde{x}_{2}^{\sigma_2})-f(\tilde{x}_{1}^{\sigma_1})\right) \big( \frac{\xi-\beta_0}{|\xi-\beta_0|}\big)}
 \right),$$
 modulo some remainder belonging to $L^p((r_0^+-\delta,r_0^++\delta))$ for some positive $\delta$ and for every $1\leq p<\infty$. Similarly, one has, near $y=r_0^+$, $\lim_{x\rightarrow 0^+} \tilde{\mathcal{Z}}_\beta(K_2,K_1,x-iy)$ is equal to
$$\lim_{x\rightarrow 0^+} \frac{(-1)^{d-1}e^{i\frac{\pi}{4}(d-1)} }{(2\pi)^{\frac{d+1}{2}}|r_0^+|^{\frac{d-1}{2}}(x-iy+ir_0^+)} \\
 \times\left( \sum_{\xi \in \Z^d, |\xi-\beta_0| = |r_0^+|}
  e^{i \xi \cdot \left( \tilde{x}_{2}^{\sigma_2}-\tilde{x}_{1}^{\sigma_1}\right) \big( -\frac{\xi-\beta_0}{|\xi-\beta_0|}\big)}  
e^{ i\left(f(\tilde{x}_{2}^{\sigma_2})-f(\tilde{x}_{1}^{\sigma_1})\right) \big(- \frac{\xi-\beta_0}{|\xi-\beta_0|}\big)}
 \right),$$
 modulo some remainder belonging to $L^p((r_0^+-\delta,r_0^++\delta))$ for every $1\leq p<\infty$. Recalling from~\cite[Eq.~(3.2.11), p.72]{Hormander90} that
 $$\lim_{x\rightarrow 0^+}\left(\frac{1}{y+ix}-\frac{1}{y-ix}\right)=-2i\pi\delta_0(y),$$
 we finally find that, near $y=r_0^+$, the tempered distribution\footnote{We restablish the dependence in the orientation to get the expected cancellation.}
 $$\lim_{x\rightarrow 0^+} \left(e^{i\frac{\pi}{4}(d-1)}\tilde{\mathcal{Z}}_\beta^{\sigma_2,\sigma_1}(K_2,K_1,x+iy)+ e^{-i\frac{\pi}{4}(d-1)}\tilde{\mathcal{Z}}_\beta^{-\sigma_2,-\sigma_1}(K_2,K_1,x-iy)\right)$$
is equal to
$$\frac{(-1)^{d-1} \delta_0(y-r_0^+)}{(2\pi)^{\frac{d-1}{2}}|r_0^+|^{\frac{d-1}{2}}} \\
 \times\left( \sum_{\xi \in \Z^d, |\xi-\beta_0| = |r_0^+|}
  e^{i \xi \cdot \left( \tilde{x}_{2}^{\sigma_2}-\tilde{x}_{1}^{\sigma_1}\right) \big( \frac{\xi-\beta_0}{|\xi-\beta_0|}\big)}  
e^{ i\left(f(\tilde{x}_{2}^{\sigma_2})-f(\tilde{x}_{1}^{\sigma_1})\right) \big( \frac{\xi-\beta_0}{|\xi-\beta_0|}\big)}
 \right)+\mathcal{O}_{L^p}(1).$$
 The same discussion of course holds near $y=r_0^-$. Hence, if $\beta_0\notin H^1(\IT^d,\IZ)$, one finds that the distribution
 $$e^{i\frac{\pi}{4}(d-1)}\tilde{\mathcal{Z}}_\beta^{\sigma_2,\sigma_1}(K_2,K_1,iy)+ e^{-i\frac{\pi}{4}(d-1)}\tilde{\mathcal{Z}}_\beta^{-\sigma_2,-\sigma_1}(K_2,K_1,-iy)$$
 is a combination of Dirac masses modulo some $L^p_{\text{loc}}$ remainder which proves Theorem~\ref{e:GuinandWeil}.

%%%%%%%%%%%%%%%%%%%%%%%%%%%
\subsection{The case when $K_1,K_2$ are points}
In this section we finally discuss the particular case where the convex sets are reduced to points (or to balls) and we still suppose that $\mathbf{v}(\theta)=\theta$. In that case, the proofs are simpler and lead to very explicit formulas with connections to the magnetic Laplacian.

\subsubsection{Meromorphic continuation of Poincar\'e series}

\begin{prop}
Assume $K_1:=\{x\}$ and $K_2:=\{y\}$ where $x,y \in \R^d$ are two points and $\beta=\beta_0+df$ is a closed real valued one-form such that $[\beta]=\beta_0\in H^1(\IT^d,\IR)\simeq\IR^d$. Then we have 
\begin{equation}\label{e:laplacien-dirac}\sum_{\gamma\in\ml{P}_{x,y}} e^{i\int_\gamma\beta} \delta(t-\ell(\gamma))=2\pi t^{\frac{d}{2}}e^{i(f(y)-f(x))}\frac{J_{\frac{d-2}{2}}\left(2\pi t\sqrt{-\Delta_{\beta_0}}\right)}{\left(\sqrt{-\Delta_{\beta_0}}\right)^{\frac{d-2}{2}}}(x,y)\ \text{in}\ \mathcal{D}^\prime(\IR_+^*).\end{equation}
If moreover $x\neq y$, then~\eqref{e:laplacien-dirac} also holds in $\mathcal{D}^\prime((-t_0,\infty))$ for some small enough $t_0>0$, and we have 
\begin{equation}\label{e:poincare-point-laplacian-0}
\mathcal{Z}_{\beta}(x,y,s)=2^{d}\pi^{\frac{d-1}{2}}\Gamma\left(\frac{d+1}{2}\right)e^{i(f(y)-f(x))} s(s^2-4\pi^2\Delta_{\beta_0})^{-\frac{d+1}{2}}(x,y), \quad \Re(s) >0 .
\end{equation}\end{prop}
Recall that in the right hand-side of~\eqref{e:poincare-point-laplacian-0}, $(s^2-4\pi^2\Delta_{\beta_0})^{-\frac{d+1}{2}}(x,y)$ denotes the Schwartz kernel of the operator $(s^2-4\pi^2\Delta_{\beta_0})^{-\frac{d+1}{2}}$ taken at the point $(x,y)$.
The proof relies on the fact that the twisted counting measure $\sum_\gamma e^{i\int_\gamma \beta} \delta(t-\ell(\gamma))$ 
has an explicit relation with the Schwartz kernel of $\Pi_*e^{-t\left(V+i\beta_0(V)\right)}\Pi^*$ (acting on functions) at $(x,y)$. 

\begin{proof}
On the one hand, by a direct calculation, one has
$$\Pi_*e^{-t\left(V+i\beta_0(V)\right)}\Pi^*(x,y)=\frac{1}{(2\pi)^d}\sum_{\xi\in\IZ^d}e^{i\xi\cdot(y-x)}\int_{\IS^{d-1}}e^{it(\xi-\beta_0)\cdot \theta} d\Vol(\theta).$$
On the other hand, one can make use of Lemma~\ref{l:truncated-poincare} (applied either for $x \neq y$ or for $t>0$) to write the twisted counting measure when $K_2=\{x\}$ and $K_1=\{y\}$. This yields that this is equal to the previous quantity up to a normalization factor:
$$\sum_{\gamma\in\ml{P}_{x,y}} e^{i\int_\gamma\beta} \delta(t-\ell(\gamma))=e^{i(f(y)-f(x))}t^{d-1}\Pi_*e^{-t\left(V+i\beta_0(V)\right)}\Pi^*(x,y)\ \text{in}\ \mathcal{D}^\prime(\IR_+^*).$$
In particular, according to~\eqref{e:Bessel-Laplacian}, one has~\eqref{e:laplacien-dirac}.
Note that, as soon as $x\neq y$, the formula~\eqref{e:laplacien-dirac} still makes sense in $\mathcal{D}^\prime((-t_0,\infty))$ for some small enough $t_0>0$. For $x\neq y$, we can then make the Laplace transform of this equality:
$$\sum_{\gamma\in\ml{P}_{x,y}} e^{i\int_\gamma\beta} e^{-s\ell(\gamma)}=2\pi e^{i(f(y)-f(x))} \int_0^\infty t^{\frac{d}{2}}\frac{J_{\frac{d-2}{2}}\left(2\pi t\sqrt{-\Delta_{\beta_0}}\right)}{\left(\sqrt{-\Delta_{\beta_0}}\right)^{\frac{d-2}{2}}}(x,y)e^{-st}|dt|.$$
We now recall that, for every $\nu>-1$ and for every $a\in\IR$,
$$\int_0^\infty e^{-st} t^{\nu+1}J_\nu(at)|dt|=2^{\nu+1}\pi^{-\frac{1}{2}}\Gamma\left(\nu+\frac{3}{2}\right)a^\nu s(s^2+a^2)^{-\nu-\frac{3}{2}}, \quad \Re(s)>0 ,$$
see e.g.~\cite[Table~8, line~(8) p~182]{Erdelyi54}.
%%https://eqworld.ipmnet.ru/en/auxiliary/inttrans/laplace8.pdf
Combining the last two lines, we obtain
$$\sum_{\gamma\in\ml{P}_{x,y}} e^{i\int_\gamma\beta} e^{-s\ell(\gamma)}=2^{d}\pi^{\frac{d-1}{2}}\Gamma\left(\frac{d+1}{2}\right)e^{i(f(y)-f(x))} s(s^2-4\pi^2\Delta_{\beta_0})^{-\frac{d+1}{2}}(x,y), $$
which is the sought result.
\end{proof}
In particular, using the spectral properties of the operator $\Delta_{\beta_0}$, we can directly recover Theorem~\ref{t:maintheo-laplace} from the introduction in that case. Precisely, one has, for $x\neq y$,
\begin{equation}\label{e:poincare-point-laplacian}
\mathcal{Z}_{\beta}(x,y,s)=\pi^{-\frac{d+1}{2}}\Gamma\left(\frac{d+1}{2}\right)e^{i(f(y)-f(x))} s\sum_{\xi\in\IZ^d}\frac{e^{i\xi\cdot(x-y)}}{(s^2+4\pi^2|\xi+\beta_0|^2)^{\frac{d+1}{2}}}.
\end{equation}
We can even be slightly more precise as we can verify that 
\begin{itemize}
\item if $d$ is odd, this expression has a meromorphic extension to $\IC$ with poles located at $\Sp\left(\pm i\sqrt{-\Delta_{\beta_0}}\right)$; 
\item if $d$ is even, this expression has a meromorphic extension for instance to 
\begin{equation}
\label{e:region-extension-bar}
\IC \setminus \left\{ i \lambda +\R_- , \lambda \in \pm \Sp\left(\sqrt{-\Delta_{\beta_0}}\right)\setminus \{0\} \right\} , 
\end{equation} 
 due to the presence, in this case, of squareroot singularities at the points of $\Sp\left(\pm i\sqrt{-\Delta_{\beta_0}}\right)\setminus \{0\}$. Note that the only possible pole in the region described in~\eqref{e:region-extension-bar} is then at $0$ and that it only occurs if $\beta_0 \in \Z^d$.
\end{itemize}

Finally, when the convex sets $K_1$ and $K_2$ are two round balls, i.e. $K_1 = \overline{B}(x,r_1)$ and $K_1 = \overline{B}(y,r_2)$, with $x\neq y$ and small enough radii $r_1$ and $r_2$, the Poincar\'e series is slightly modified by a factor $e^{-s(r_1+r_2)}$ and the above formula yields
\begin{equation}
 \mathcal{Z}_{\beta}(K_1,K_2,s)=\pi^{-\frac{d+1}{2}}\Gamma\left(\frac{d+1}{2}\right) e^{i(f(y)-f(x))}se^{-s(r_2+r_1)}\sum_{\xi\in\IZ^d}\frac{e^{i\xi\cdot(x-y)}}{(s^2+4\pi^2|\xi+\beta_0|^2)^{\frac{d+1}{2}}} ,
\end{equation}
where we have taken $(\sigma_1,\sigma_2)=(+,-)$ for the (implicit) choice of orientations of the two balls.

\subsubsection{A Guinand-Meyer formula when $d$ is odd}\label{sss:guinand}
Let us now discuss a variant of Theorem~\ref{e:GuinandWeil} when $K_1:=\{x\}$ and $K_2:=\{y\}$ are reduced to points that are \emph{distinct}. 
Following~\cite[Th.~5]{Meyer2016}, we define, for $x\neq y$,
\begin{equation}
\label{e:def-muGM}
\mu_{GM}(t):=\sum_{\gamma\in\ml{P}_{x,y}} \frac{e^{i\int_\gamma\beta}}{\ell(\gamma)} \left(\delta(t-\ell(\gamma))-\delta(t+\ell(\gamma))\right),
\end{equation}
which is a Radon measure in $\mathcal{S}^\prime(\IR)$ carried by a discrete and locally finite set of $\IR$. 
Note that compared to the twisted counting measure in~\eqref{e:laplacien-dirac}, $\mu_{GM}$ is symmetrized and renormalized by $t^{-1}$. In particular, this is not the same renormalization as in Theorem~\ref{e:GuinandWeil} except if $d=3$.
\begin{prop}\label{p:LevReti}
Assume that $d$ is odd, that $x\neq y$ and that $\beta_0\notin H^1(\IT^d,\IZ)\simeq\IZ^d$. Then, the Fourier transform of $\mu_{GM}$ in~\eqref{e:def-muGM} is given by 
\begin{align*}
\widehat{\mu}_{GM}(\tau) & =i(-1)^{\frac{d-3}{2}} 2^{d}\pi^{\frac{d+1}{2}} e^{i(f(y)-f(x))}\\
& \quad \times \sum_{\xi\in\IZ^d}\left(\frac{\delta^{\left(\frac{d-3}{2}\right)}(\tau-2\pi|\xi+\beta_0|)}{\left(\tau+2\pi|\xi+\beta_0|\right)^{\frac{d-1}{2}}}
+ \frac{\delta^{\left(\frac{d-3}{2}\right)}(\tau+2\pi|\xi+\beta_0|)}{\left(\tau-2\pi|\xi+\beta_0|\right)^{\frac{d-1}{2}}}\right).
\end{align*}
\end{prop}

The assumption $\beta_0\notin H^1(\IT^d,\IZ)\simeq\IZ^d$ implies that the continuation of the Laplace transform of $\mu_{GM}$ has no pole at $s=0$ while the assumption $d$ odd ensures that the Poincar\'e series extends meromorphically to $\IC$ with poles located on the imaginary axis. When $d=3$, this proposition recovers~\cite[Th.~5]{Meyer2016} and, for $d\geq 5$, it corresponds to the more general statement from~\cite[Corollary~2.4]{LevReti2021}. In particular, if $d=3$, $\mu_{GM}$ is a crystalline measure: a measure in $\mathcal{S}'(\R)$ carried by a discrete and locally finite set of $\IR$ with Fourier transform having the same properties. In odd dimension $d \neq 3$, $\widehat{\mu}_{GM}$ is no longer a crystalline measure since its Fourier transform is not a measure (but a distribution of order $\frac{d-3}{2}$). Following~\cite{LevReti2021}, one says that the measure $\mu_{GM}$ is a \emph{crystalline distribution}, i.e. a distribution in $\mathcal{S}'(\R)$ carried by a discrete and locally finite set of $\IR$ with Fourier transform having the same properties. Compared to the measure in Theorem~\ref{e:GuinandWeil}, $\widehat{\mu}_{GM}$ has the drawback of not being a measure here. However, it has the advantage to be carried by a discrete and locally finite set of $\IR$ (i.e. there is no absolutely continuous remainder $r$ as in Theorem~\ref{e:GuinandWeil}).  

\begin{proof}
Dividing~\eqref{e:laplacien-dirac} by $t$ (which we may since $x\neq y$), we obtain
$$\sum_{\gamma\in\ml{P}_{x,y}} \frac{e^{i\int_\gamma\beta}}{\ell(\gamma)} \delta(t-\ell(\gamma))=2\pi t^{\frac{d-2}{2}}e^{i(f(y)-f(x))}\frac{J_{\frac{d-2}{2}}\left(2\pi t\sqrt{-\Delta_{\beta_0}}\right)}{\left(\sqrt{-\Delta_{\beta_0}}\right)^{\frac{d-2}{2}}}(x,y).$$
Taking the Laplace transform, one finds that, for $\text{Re}(s)>0$,
$$\sum_{\gamma\in\ml{P}_{x,y}} \frac{e^{i\int_\gamma\beta}}{\ell(\gamma)} e^{-s\ell(\gamma)}=2\pi e^{i(f(y)-f(x))} \int_0^\infty e^{-st}t^{\frac{d-2}{2}}\frac{J_{\frac{d-2}{2}}\left(2\pi t\sqrt{-\Delta_{\beta_0}}\right)}{\left(\sqrt{-\Delta_{\beta_0}}\right)^{\frac{d-2}{2}}}(x,y)|dt|.$$
Recalling that, for every $\nu>-1/2$ and for every $a\in\IR$,
$$\int_0^\infty e^{-st} t^{\nu}J_\nu(at)|dt|=2^{\nu}\pi^{-\frac{1}{2}}\Gamma\left(\nu+\frac{1}{2}\right)a^\nu (s^2+a^2)^{-\nu-\frac{1}{2}}, \quad \Re(s)>0 ,$$
see e.g.~\cite[Table~8, line~(7) p~182]{Erdelyi54}, 
%%https://eqworld.ipmnet.ru/en/auxiliary/inttrans/laplace8.pdf
we deduce that
\begin{equation}\label{e:Guinand-Meyer-key}
\sum_{\gamma\in\ml{P}_{x,y}} \frac{e^{i\int_\gamma\beta}}{\ell(\gamma)} e^{-s\ell(\gamma)}=2^{d-1}\pi^{\frac{d-1}{2}}\Gamma\left(\frac{d-1}{2}\right) e^{i(f(y)-f(x))} (s^2-4\pi^4\Delta_{\beta_0})^{-\frac{d-1}{2}}(x,y).
\end{equation}
Recalling the definition of $\mu_{GM}$ in~\eqref{e:def-muGM}, its Fourier transform is given by
$$\widehat{\mu}_{GM}(\tau)=\sum_{\gamma\in\ml{P}_{x,y}} \frac{e^{i\int_\gamma\beta}}{\ell(\gamma)} \left(e^{-i\tau\ell(\gamma)}-e^{i\tau\ell(\gamma)}\right)=\lim_{\alpha\rightarrow 0^+}\sum_{\gamma\in\ml{P}_{x,y}} \frac{e^{i\int_\gamma\beta}}{\ell(\gamma)} \left(e^{-(i\tau+\alpha)\ell(\gamma)}-e^{-(-i\tau+\alpha)\ell(\gamma)}\right).$$
This is an odd distribution and, according to~\eqref{e:Guinand-Meyer-key}, it is smooth near $\tau=0$ since $\beta_0 \notin \Z^d$. 

Next, from~\eqref{e:Guinand-Meyer-key}, one knows that
\begin{multline*}
\widehat{\mu}_{GM}(\tau)=2^{d-1}\pi^{\frac{d-1}{2}}\Gamma\left(\frac{d-1}{2}\right) e^{i(f(y)-f(x))}\\
\times\lim_{\alpha\rightarrow 0^+}\left(((i\tau+\alpha)^2-4\pi^4\Delta_{\beta_0})^{-\frac{d-1}{2}}-((-i\tau+\alpha)^2-4\pi^4\Delta_{\beta_0})^{-\frac{d-1}{2}}\right)(x,y).
 \end{multline*}
As $\widehat{\mu}_{GM}$ is odd and smooth near $0$, we just need to understand this distribution for $\tau>0$. To do that, we write that, for $\lambda>0$ and $\tau>0$,
\begin{multline*}
\lim_{\alpha\rightarrow 0^+}\left(((i\tau+\alpha)^2+\lambda^2)^{-\frac{d-1}{2}}-((-i\tau+\alpha)^2+\lambda^2)^{-\frac{d-1}{2}}\right)\\
=\frac{1}{(\tau+\lambda)^{\frac{d-1}{2}}}\lim_{\alpha\rightarrow 0^+}\left((\tau-\lambda-i\alpha)^{-\frac{d-1}{2}}-(\tau-\lambda+i\alpha)^{-\frac{d-1}{2}}\right).
\end{multline*}
Implementing~\cite[Eq.~(3.2.11), p.72]{Hormander90} one more time, one finds
$$\lim_{\alpha\rightarrow 0^+}\left(((i\tau+\alpha)^2+\lambda^2)^{-\frac{d-1}{2}}-((-i\tau+\alpha)^2+\lambda^2)^{-\frac{d-1}{2}}\right)
=\frac{(-1)^{\frac{d-3}{2}}2i\pi}{\Gamma\left(\frac{d-1}{2}\right)(\tau+\lambda)^{\frac{d-1}{2}}}\delta^{\left(\frac{d-3}{2}\right)}(\tau-\lambda).$$
We can now rewrite $\widehat{\mu}_{GM}(\tau)$ using this formula. It yields, for $\tau>0$,
\begin{equation*}
\widehat{\mu}_{GM}(\tau)=i(-1)^{\frac{d-3}{2}} 2^{d}\pi^{\frac{d+1}{2}} e^{i(f(y)-f(x))}\sum_{\xi\in\IZ^d}\frac{1}{\left(\tau+2\pi|\xi+\beta_0|\right)^{\frac{d-1}{2}}}\delta^{\left(\frac{d-3}{2}\right)}(\tau-2\pi|\xi+\beta_0|).
\end{equation*}
Recalling that $\widehat{\mu}_{GM}$ is smooth near $0$ and odd, this completely determines the Fourier transform and concludes the proof of the propotition.
\end{proof}
\begin{rema}
 Recall that in Theorem~\ref{e:GuinandWeil}, we are rather interested (in the more general setting of two convex sets) by a renormalized version of~\eqref{e:def-muGM}, namely
$$\tilde{\mu}_{GM}(t):=e^{\frac{i\pi}{4}(d-1)}\sum_{\gamma\in\ml{P}_{x,y}} \frac{e^{i\int_\gamma\beta}}{\ell(\gamma)^{\frac{d-1}{2}}} \delta(t-\ell(\gamma))+e^{-\frac{i\pi}{4}(d-1)}\sum_{\gamma\in\ml{P}_{x,y}} \frac{e^{i\int_\gamma\beta}}{\ell(\gamma)^{\frac{d-1}{2}}} \delta(t+\ell(\gamma)).$$
In particular, when $d=5$, one has
$$\widehat{\tilde{\mu}}_{GM}(\tau)=-\sum_{\gamma\in\ml{P}_{x,y}} \frac{e^{i\int_\gamma\beta}}{\ell(\gamma)^{2}}\left(e^{i\tau\ell(\gamma)}+e^{-i\tau\ell(\gamma)}\right),\quad \widehat{\tilde{\mu}}_{GM}'(\tau)=\widehat{\mu}_{GM}(\tau),$$
and, thanks to Proposition~\ref{p:LevReti}, the remainder $r$ in Theorem~\ref{e:GuinandWeil} is not identically $0$ (as $r'(\tau)$ is a combination of Dirac distributions). A similar remark holds for $d\geq 5$.
\end{rema}

%%%%%%%%%%%%%%%%%%%%%%%%%%%%%%%
%%%%%%%%%%%%%%%%%%%%%%%%%%%%%%%
%%%%%%%%%%%%%%%%%%%%%%%%%%%%%%%
%%%%%%%%%%%%%%%%%%%%%%%%%%%%%%%
%%%%%%%%%%%%%%%%%%%%%%%%%%%%%%%

\section{Geometric interpretation of the residues and proof of Theorem~\ref{t:maintheo-residues}}\label{s:convex}

In this Section, we aim at computing somehow explicitly the residues appearing in~\eqref{e:residues} in terms of geometric quantities. We will always suppose in the following that $\beta=0$ which will make the content of this residue more geometric. Along the way, we will prove Theorem~\ref{t:maintheo-residues} from the introduction. 

Recall that the admissible manifolds $\Sigma_1$ and $\Sigma_2$ used to define our generalized Epstein zeta functions are constructed from two compact and strictly convex subsets $K_1$ and $K_2$ of $\IR^d$ with smooth boundaries $\partial K_1$ and $\partial K_2$ (possibly reduced to a point). The submanifolds $\Sigma_i$ are parametrized through the inverse Gauss map $\tilde{x}_i:\IS^{d-1}\rightarrow\IT^d$ (using the convention that $\partial K_i$ is oriented using the outward normal vector to $K_i$). In order to take the various possibilities for our orthospectrum (outward or inward pointing geodesics), we introduced orientation parameters $\sigma_i\in\{\pm\}$ and we have set
$$\tilde{x}_i^{\sigma_i}(\theta)=\tilde{x}_i(\sigma_i\theta).$$

In this paragraph, we suppose that $\sigma_1=-$ and that $\sigma_2=+$. Equivalently, it means that we consider geodesic arcs that go from $\Sigma_2$ to $\Sigma_1$ and that point outside $K_2$ and inward $K_1$ (when lifted to $\IR^d$). Let $1\leq \ell\leq d$. Thanks to~\eqref{e:residues1}--\eqref{e:residues2}, we now compute the coefficients $E^{(\ell-1)}_0$ or equivalently the residues
$$\Res_{s=\ell}\left( \zeta_0(K_2,K_1,s)\right)=\frac{(-1)^{\ell+d}}{(2\pi)^{2d}(\ell-1)!}\int_{S\IT^{d}}dx_1\wedge\ldots\wedge dx_d\wedge\mathbf{T}_{\tilde{x}_{1}^{-}-\tilde{x}_{2}^{+}}^*\iota_V\mathbf{V}^{\ell-1}\left(dx_1\wedge\ldots\wedge dx_d\right).$$
 In view of emphasizing the dependence on the convex sets, we make use of \S\ref{ss:convex} and we set 
 $$V_{-K_1}^+-:=-\tilde{x}_1(-\theta)\cdot\partial_x,\quad V_{K_2}^+:=\tilde{x}_2(\theta)\cdot\partial_x,\quad\text{and}\quad V_K^+:=V=\mathbf{v}(\theta)\cdot\partial_x.$$
 In particular, the residues can be rewritten as
 $$\Res_{s=\ell}\left( \zeta_0(K_2,K_1,s)\right)=\frac{(-1)^{\ell+d}}{(2\pi)^{2d}(\ell-1)!}\int_{S\IT^{d}}dx_1\wedge\ldots\wedge dx_d\wedge e^{(V_{-K_1}-V_{K_2})^*}\mathbf{V}_K^{\ell-1}\iota_{V_K}\left(dx_1\wedge\ldots\wedge dx_d\right).$$
 Recalling Lemma~\ref{l:sum-convex}, $\theta\in\IS^{d-1}\mapsto-\tilde{x}_1(-\theta)+\tilde{x}_2(\theta)\in\IR^d$ is a parametrization of the convex set $-K_1+K_2$ so that the previous inequality rewrites as
 $$\Res_{s=\ell}\left( \zeta_0(K_2,K_1,s)\right)=\frac{1}{(2\pi)^{2d}(\ell-1)!}\int_{S\IT^{d}}e^{V_{K_2-K_1}^*}\mathbf{V}_K^{\ell-1}\iota_{V_K}\left(dx_1\wedge\ldots\wedge dx_d\right)\wedge dx_1\wedge\ldots\wedge dx_d.$$
 
In particular, one has
\begin{equation}\label{e:residue-mixedvolume}
 \frac{1}{(2\pi)^d}\sum_{\ell=1}^d\frac{t^{\ell-1}}{(\ell-1)!}\int_{S\IT^{d}}[S_{[0]}\IT^d]\wedge e^{V_{K_2-K_1}^*}\mathbf{V}_K^{\ell-1}\iota_{V_K}\left(dx_1\wedge\ldots\wedge dx_d\right)
=\sum_{\ell=1}^{d}t^{\ell-1}\Res_{s=\ell}\left( \zeta_0(K_2,K_1,s)\right).
\end{equation}
Recalling Lemma~\ref{r:volume-Rd}, we recognize on the left hand side the derivatives of the map $t\mapsto \Vol_{\IR^d}(K_2-K_1+tK)$ so that
\begin{equation}
 \sum_{\ell=1}^{d}t^{\ell-1}\Res_{s=\ell+1}\left( \zeta_0(K_2,K_1,s)\right)= \frac{1}{(2\pi)^d}\frac{d}{dt}\Vol_{\IR^d}\left(K_2-K_1+tK\right).   
     \end{equation}
Recalling Steiner's formula~\eqref{e:steiner}, this concludes the proof of Theorem~\ref{t:maintheo-residues}. It also shows how the residues can be expressed in terms of mixed volumes (see Remark~\ref{r:mixed-volume}) when we consider a more general vector $\mathbf{v}(\theta)\cdot\partial_x$ than $\theta\cdot\partial_x.$

\appendix

%%%%%%%%%%%%%%%%%%%%%
\section{Another formula for zeta functions}
\label{a:proof}

We now briefly explain how to prove~\eqref{e:formule-magnifique} without appealing the theory of De Rham currents and how it may slightly simplify the exposition of the proofs of Theorems~\ref{t:maintheo-Mellin},~\ref{t:maintheo-residues} and~\ref{t:maintheo-laplace}.
 Yet, this would be at the expense of loosing the dynamical pictures behind these results and thus the relation of these results with our other (more clearly dynamical) applications. Recall also that this formula only holds a priori for a specific choice of orientations for $K_1$ and $K_2$ while our current theoretic approach allows to handle any orientation convention and to easily implement exponential weights in our zeta functions.

First, from~\S\ref{ss:convex}, $\partial K_1$ and $\partial K_2$ can be parametrized by their outward normal vector $\theta\in\IS^{d-1}$ through the maps $x_{K_j}:\IS^{d-1}\rightarrow \IR^d$, $j\in\{1,2\}$. Moreover, according to~\S\ref{ss:convex-schneider}, the maps $$\theta\mapsto x_{K_1}(\theta)-x_{K_2}(-\theta)+t\theta$$ parametrize the boundary of the convex set $K_1-K_2+tB_d$ for every $0\leq t\leq T.$ Let us now remark that $\gamma$ belongs to $\mathcal{P}_{K_1,K_2}$ with $0<\ell(\gamma)=t\leq T$ if and only if there exist $\theta\in\IS^{d-1}$ and $\xi\in\IZ^d$ such that
$$x_{K_2}(-\theta)=x_{K_1}(\theta)+t\theta-2\pi\xi.$$
Equivalently, it means that there exists $\xi\in\IZ^d$ such that $2\pi\xi$ belongs to $\partial (K_1-K_2+t\theta)$. Hence, elements $\gamma$ in $\mathcal{P}_{K_1,K_2}$ are in one-to-one correspondance with the set
$$2\pi\IZ^d\cap\left(K_1-K_2+TB_d\right)\setminus\left(K_1-K_2\right).$$
Now observe that the restriction of the Lebesgue measure to the set $\left(K_1-K_2+TB_d\right)\setminus\left(K_2-K_1\right)$ can be disintegrated as follows
$$\int_0^T\delta_{\partial(K_1-K_2+tB_d)}(x,|dx|)|dt|,$$
so that
$$\sharp\left\{\gamma\in\ml{P}_{K_1,K_2}:0\leq \ell(\gamma)\leq T\right\}=\int_{\IR^d}\delta_{[0]}(x)\int_0^T\delta_{\partial(K_1-K_2+tB_d)}(x,|dx|)|dt|,$$
with $\delta_{[0]}$ defined in~\eqref{e:delta-periodique}. Similarly, if we weight the Lebesgue measure with $\chi(t)$ on each sublevel $\partial (K_1-K_2+t\theta)$, we derive formula~\eqref{e:formule-magnifique} from the introduction. Now, in order to prove our theorems on convex geometry from this formula, one would need to decompose $\delta_{[0]}(x)$ according to~\eqref{e:delta-periodique} and to make sense of the right side after this decomposition for the functions $t^{-s}$ and $e^{-st}$. For $\text{Re}(s)$ large enough, this is not a problem through a direct calculation. Then, one would need to make the meromorphic continuation of the right hand side through the natural threshold. This could be achieved by reducing to the oscillatory integrals of Section~\ref{s:analysis} (through the parametrization of $K_1-K_2+tB_d$ by $\theta$ as in Sections~\ref{s:poincare} and~\ref{s:convex}) and by arguing as in Section~\ref{s:laplace-mellin} with the simplifications that we only deal with $\delta$ functions rather than general test functions (as in~\S\ref{ss:epstein} and~\ref{ss:poincare}). Thus, the analytical difficulties would remain exactly the same through this approach. The main advantage would be that the fact that the residues involve the intrinsic volume would be more direct (from the analysis of the Fourier mode $0$). Finally, we considered here the case of $B_d$ but the proof could be adapted as well when $B_d$ is replaced by a strictly convex set $K$ (with $0$ in its interior) as we did all along the article.

\section*{Acknowledgements} We warmly thank B.~Chantraine, N.B.~Dang, F.~Faure, Y.~Guedes-Bonthonneau, D.~Han-Kwan and J.~Viola for interesting discussions related to this work.
We are grateful to the referee for her/his valuable comments that have helped us to improve the level of generality of the results and the presentation of the manuscript. 
NVD acknowledges the support of the Institut Universitaire de France.
ML is partially supported by the Agence Nationale de la Recherche under grants SALVE (ANR-19-CE40-0004) and ADYCT (ANR-20-CE40-0017).
GR acknowledges the support of the Institut Universitaire de France and of the PRC grants ADYCT (ANR-20-CE40-0017) and ODA (ANR-18-CE40-0020) from the Agence Nationale de la Recherche.

\bibliographystyle{alpha}
\bibliography{allbiblio}
\end{document}